\documentclass[12pt]{article}
\usepackage[english]{babel}

\usepackage[utf8]{inputenc}

\usepackage{amsfonts, amsmath, amssymb,latexsym,amsthm}
\usepackage{authblk}
\usepackage{textcomp}

\usepackage{enumitem}
\usepackage{moreenum}

\usepackage{tikz}
\usetikzlibrary{decorations.markings}
\usetikzlibrary{decorations.pathreplacing}
\usetikzlibrary{shapes.geometric,positioning,calc}
\usetikzlibrary{intersections}
\usetikzlibrary{positioning}

\usepackage{caption}

\theoremstyle{definition}
\newtheorem{definition}{Definition}[section]

\newtheorem{remark}{Remark}[section]

\newtheorem{example}{Example}[section]

\theoremstyle{plain}
\newtheorem{theorem}{Theorem}[section]

\newtheorem{lemma}{Lemma}[section]

\newtheorem{proposition}[lemma]{Proposition}

\newtheorem{corollary}[lemma]{Corollary}

\newcommand{\Fr}{\mathcal{F}}
\newcommand{\Ideal}{\mathcal{I}}
\newcommand{\LM}{\Lambda}
\newcommand{\Dp}{\mathrm{Dp}}
\newcommand{\Ft}{\mathrm{F}}
\newcommand{\Low}{\mathrm{L}}
\newcommand{\gr}{\mathrm{gr}}
\newcommand{\Gr}{\mathrm{Gr}}

\tikzset{arrow/.style={
        decoration={markings,
            mark= at position #1 with {\arrow{stealth}},
        },
        postaction={decorate}
    }
}

\tikzset{reversearrow/.style={
        decoration={markings,
            mark= at position #1 with {\arrow{stealth reversed}},
        },
        postaction={decorate}
    }
}

\newcommand{\fc}{\textbf{fc}}
\newcommand{\nfc}{\textbf{nfc}}

\begin{document}

\title{Construction of a Quotient Ring of $\mathbb{Z}_2\Fr$ in which a Binomial $1 + w$ is Invertible Using Small Cancellation Methods}

\author{A.\,Atkarskaya}
\affil{Department of Mathematics, Bar-Ilan University, 5290002 Ramat Gan, Israel \\ atkarskaya.agatha@gmail.com}

\author{A.\,Kanel-Belov \footnote{The paper was supported by Russian Science foundation grant \textnumero 17-11-01337.}}
\affil{Department of Mathematics, Bar-Ilan University, 5290002 Ramat Gan, Israel \\ beloval@macs.biu.ac.il}

\author{E.\,Plotkin}
\affil{Department of Mathematics, Bar-Ilan University, 5290002 Ramat Gan, Israel \\ plotkin.evgeny@gmail.com}

\author{E.\,Rips}
\affil{Department of Mathematics, The Hebrew University of Jerusalem, Givat Ram, 9190401 Jerusalem, Israel \\ eliyahu.rips@mail.huji.ac.il}

\date{}

\maketitle

\textit{Dedicated to Professor Boris\,I.\,Plotkin with love and respect on the occasion of his 90th Anniversary.}

\begin{abstract} We apply small cancellation methods originating from group theory to investigate the structure of a quotient ring $\mathbb{Z}_2\Fr / \Ideal$, where $\mathbb{Z}_2\Fr$ is the group algebra of the free group $\Fr$ over the field $\mathbb{Z}_2$, and the ideal $\Ideal$ is generated by a single trinomial $1 + v + vw$, where $v$ is a complicated word depending on $w$. In $\mathbb{Z}_2\Fr / \Ideal$ we have $(1 + w)^{-1} = v$, so $1 + w$ becomes invertible. We construct an explicit linear basis of $\mathbb{Z}_2\Fr / \Ideal$ (thus showing that $\mathbb{Z}_2\Fr / \Ideal \neq 0$). This is the first step in constructing rings with exotic properties.
\end{abstract}

\section{Introduction}
This paper describes the first step in a construction of a skew field with a finitely generated multiplicative group.

We will construct this skew field as a quotient ring of a group algebra of a finitely generated free group $\Fr$. The full construction of such a skew field would involve an iterative procedure, in which every step is similar to the one described in this paper, in a more complicated situation. The resulting injective limit will be either $0$ or a skew field.

Our main objective is to show that the resulting skew field is non trivial, and, moreover, to obtain a skew field of infinite dimension over its center in which every non zero element is equal to a monomial, that is, an element of the free group $\Fr$. Hence, when $\Fr$ is finitely generated, the multiplicative group of this skew field is finitely generated. In order to show that the resulting skew field is non trivial, we need to develop its structure theory.

In this paper we consider the following problem. Let $\mathbb{Z}_2\Fr$ be the group algebra of the free group $\Fr$ over the field $\mathbb{Z}_2$, and let $1 + w$ be a binomial. We would like to find a quotient ring of $\mathbb{Z}_2\Fr$ in which the image of $1 + w$ is invertible, without losing control of the quotient ring relations.

Our way to deal with this problem is influenced by an analogue with small cancellation in groups. Namely, we equate $(1 + w)^{-1}$ to a complicated word
\begin{equation*}
v = x^{\alpha}y x^{\alpha + 1}y \cdots x^{\beta - 1}y
\end{equation*}
with $\vert w \vert \ll \alpha \ll \beta$, where $\vert w \vert$ is a word length of $w$ in $\Fr$. The monomial $v$ exhibits small cancellation properties because every subword of $v$ that contains at least two letters $y$ appears in $v$ only once (\cite{LyndonSchupp}).

In the case of groups we consider semicanonical monomials, namely, monomials that do not contain parts of the relations (or of their cyclic conjugates and inverses) that are too big. It turns out that semicanonical monomials representing the same element of the group are connected by a so-called one-layer diagram. The transition between semicanonical monomials representing the same element is done by substituting subwords of the relations by the inverses of their complements repeatedly (such transitions are called turns).

Here we imitate this process with inevitable complications. Since our relations are not binomial but polynomial, we substitute a subword by the sum of the rest of the monomials of the relation. We call such transitions multi-turns. The simple-minded picture goes like this
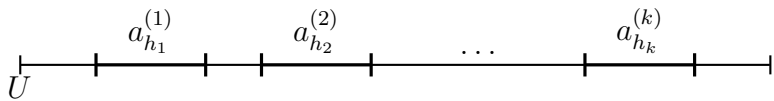
\begin{center}
\begin{tikzpicture}
\draw[|-|, black, thick] (0,0) to node[at start, below] {$U$} (10,0);
\draw[|-|, black, very thick] (1,0)--(2.5,0) node[midway, above] {$a_{h_1}^{(1)}$};
\draw[|-|, black, very thick] (3.2,0)--(4.7,0) node[midway, above] {$a_{h_2}^{(2)}$};
\draw[|-|, black, very thick] (7.5,0)--(9,0) node[midway, above] {$a_{h_k}^{(k)}$};
\path (4.7, 0) to node[midway, above] {$\ldots$} (7.5, 0);
\end{tikzpicture}
\end{center}
where $U$ is a monomial, and we have polynomial relations
\begin{equation*}
\sum\limits_{j_i = 1}^{l_i} a_{j_i}^{(i)} = 0,\ i = 1, \ldots, k,\ h_i \in \lbrace 1, \ldots, l_i \rbrace.
\end{equation*}
Performing a number of multi-turns, we obtain sums of monomials of the form
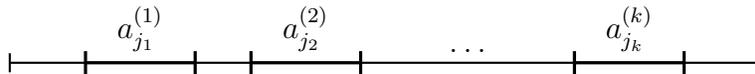
\begin{center}
\begin{tikzpicture}
\draw[|-|, black, thick] (0,0)--(10,0);
\draw[|-|, black, very thick] (1,0)--(2.5,0) node[midway, above] {$a_{j_1}^{(1)}$};
\draw[|-|, black, very thick] (3.2,0)--(4.7,0) node[midway, above] {$a_{j_2}^{(2)}$};
\draw[|-|, black, very thick] (7.5,0)--(9,0) node[midway, above] {$a_{j_k}^{(k)}$};
\path (4.7, 0) to node[midway, above] {$\ldots$} (7.5, 0);
\end{tikzpicture}
\end{center}
We have $l_1\cdot \ldots \cdot l_k$ monomials with linear dependencies induced by the multi-turns. In our (simple-minded) case we obtain a linear space of dimension $(l_1 - 1)\cdot\ldots\cdot (l_k - 1)$.

The actual situation is more complicated, due to a number of factors. Let us list some of them:
\begin{itemize}
\item
Some of the words $a^{(i)}_{j_i}$ might be short, causing certain degeneracies.
\item
There can be highly non-trivial interactions between the neighbour occurrences, so that the order of performing the multi-turns might matter.
\end{itemize}
This explains why pursuing this program turns out not entirely simple. We have to introduce a number of elaborate concepts to deal with the monomials and the subwords of the relations appearing in them (for this, see especially Section~\ref{monomials_chart_dependencies}).

To control the degenerations, we introduce a (decreasing) filtration and study its properties (Section~\ref{structure_calc}). Then we need to study the interaction between the linear dependencies and the filtration. Because of the properties like transversality and non-degeneracy, this interaction has nice properties (Section~\ref{fin_quotient_spaces_section}, Section~\ref{basis_descr_section}). Our final result (Theorem~\ref{whole_quotient_ring_structure}) supplies an explicit linear basis for the quotient ring. In particular, the quotient ring does not collapse to $0$.

Then we introduce a linear generating set of our quotient ring with particularly nice properties. For this generating set, the multiplication can be expressed as a sum of monomials forming thin triangles with the factors (Section~\ref{multiplication_geometry_section}, Theorem~\ref{thin_triangles_theorem}). Small cancellation groups are hyperbolic, so their multiplication is expressed by thin triangles. While we do not possess the concept of a hyperbolic ring (even for algebras over a field), the quotient ring we construct does display some features of what might be expected for a ``hyperbolic ring''.

\section{Basic definitions}
\label{basic_def}
Consider the group ring $\mathbb{Z}_2\Fr$, where $\Fr$ is a free group with at least 4 free generators. We will call elements of $\Fr$ words or monomials. We deal only with reduced monomials unless the converse is explicitly stated. For our purposes we can use any field $k$, but we choose the field $\mathbb{Z}_2$ to simplify our calculations. Let us fix $w \in \Fr$, an arbitrary cyclically reduced primitive (not a proper power) word from $\Fr$. Fix positive integers $\alpha$ and $\beta$ such that $\vert w \vert \ll \alpha \ll \beta$, where $\vert w \vert$ is a word length of $w$ in $\Fr$. Our aim is inverting of the binomial $1 + w$. Consider the word
\begin{equation}
\label{v_def}
v = x^{\alpha}y x^{\alpha + 1}y \cdots x^{\beta - 1}y
\end{equation}
such that $w$ does not start or end with the letters $x, y, x^{-1}, y^{-1}$. We are going to study a structure of the quotient ring $\mathbb{Z}_2\Fr / \Ideal$, where
\begin{equation}
\label{I_def}
\Ideal = \langle1 + v + vw\rangle.
\end{equation}
Clearly, in the quotient ring we have $v(1 + w) = 1$. Multiplying it by $v$ on the right side, we obtain $v(1 + w)v = v$. Since $v$ is a word, it is invertible in $\mathbb{Z}_2 \Fr$, so, it is also invertible in $\mathbb{Z}_2\Fr / \Ideal$ or the quotient ring is trivial. Multiplying the last equation by $v^{-1}$, we obtain $(1 + w)v = 1$. Thus, $v^{-1} = 1 + w$ in the ring $\mathbb{Z}_2\Fr / \Ideal$.

\begin{definition}
Suppose $M(v, w)$ is a non-commutative monomial over the words $v, w$. We call arbitrary subwords of monomials $M(v, w)$ \emph{(v, w)-generalized fractional powers} (or simply \emph{generalized fractional powers}).
\end{definition}

Let $v_1$ be an arbitrary subword of $v$ (or of $v^{-1}$). We define an additive (length) measure (aka $\LM$-measure) on subwords of $v$ (or $v^{-1}$) in the following way
\begin{equation}
\label{measure}
\LM(v_1) = \frac{m_y(v_1)}{\beta - \alpha},
\end{equation}
where $m_y(v_1)$ is the number of occurrences of the letter $y$ (or of the letter $y^{-1}$) in $v_1$. Obviously, $\LM(v) = 1$ ($\LM(v^{-1}) = 1$). For any subword $w_1$ of $w$ we put $\LM(w_1) = 0$. From now on, let us call $\LM$-measure of a generalized fractional power the sum of $\LM$-measure of its parts.

Let $U$ be a word and $U_1$ be its subword. We call the triple that consists of $U$, $U_1$ and the position of $U_1$ in $U$ \emph{an occurrence of $U_1$ in $U$}.

An arbitrary word from $\Fr$ may contain occurrences of generalized fractional powers. While the notion of a maximal occurrence seems to be intuitively simple, it turns out that the notion we actually need has considerable subtleties.

So, we define a maximal occurrence in the following way. Let $a$ be a generalized fractional power, that is, $a$ is an occurrence in a monomial $M(v, w) = M_LaM_R$. Assume that $M_L = M_1(v, w)M_L^{\prime}$, where $M_1(v, w)$ is a monomial in $v, w$, and $M_L^{\prime}$ has no initial subword equal to a monomial in $v, w$. In the same way, assume that $M_R = M_R^{\prime}M_2(v, w)$, where $M_2(v, w)$ is a monomial in $v, w$, and $M_R^{\prime}$ has no final subword equal to a monomial in $v, w$. Then we consider $a$ as an occurrence in the monomial
\begin{equation*}
M_0(v, w) = M_1(v, w)^{-1}M(v, w)M_2^{-1}(v, w) = M_L^{\prime}aM_R^{\prime}.
\end{equation*}
Consider a set of monomials
\begin{equation*}
\mathcal{A} = \lbrace M(v, w) \mid M(v, w) = M_i(v, w)M_0(v, w)M_j(v, w), i, j \in \mathbb{N} \rbrace,
\end{equation*}
where $M_i(v, w)$ and $M_j(v, w)$ are arbitrary monomials in $v$, $w$. Then the occurrence of $a$ in $M_0(v, w)$ can be considered as the occurrence of $a$ in $M(v, w) \in \mathcal{A}$ with the same position in $M_0(v, w)$.

Suppose $a$ is an occurrence in a word $U \in \Fr$, a letter $a_1$ prolongs $a$ in $U$ from the left side, a letter $a_2$ prolongs $a$ in $U$ from the right side. If $a$ is an initial subword of $U$, $a_1$ is absent; if $a$ is a final subword of $U$, $a_2$ is absent. If for any $M(v, w) \in \mathcal{A}$ the occurrence $a$ can be prolonged neither by $a_1$ on the left nor by $a_2$ on the right in $M(v, w)$, then we call $a$ \emph{a maximal occurrence of a generalized fractional power in $U$}.

Let $U = U_L a U_R$. The above definition in fact means the following. The occurrence $a$ is a maximal occurrence of a generalized fractional power in $U$ if and only if it satisfies one of the conditions:
\begin{enumerate}[label=(L\arabic*)]
\item
$U_L$ is the empty word;
\item
$U_L = U_L^{\prime}a_1$, $M_L^{\prime} = M_L^{\prime\prime}a_1^{\prime}$, $a_1 \neq a_1^{\prime}$;
\item
$U_L = U_L^{\prime}a_1$, $M_L^{\prime}$ is the empty word, the last letter in each $M_i(v, w)$ is different from $a_1$;
\end{enumerate}
and one of the conditions
\begin{enumerate}[label=(R\arabic*)]
\item
$U_R$ is the empty word;
\item
$U_R = a_2U_R^{\prime}$, $M_R^{\prime} = a_2^{\prime}M_R^{\prime\prime}$, $a_2 \neq a_2^{\prime}$;
\item
$U_R = a_2U_R^{\prime}$, $M_R^{\prime}$ is the empty word, the first letter in each $M_j(v, w)$ is different from $a_2$.
\end{enumerate}

\begin{definition}
Let $\tau$ be some small value, $U$ be an arbitrary word from $\Fr$. We call the set of maximal occurrences of generalized fractional powers in $U$ of $\LM$-measure greater or equal to a given threshold $\tau$ \emph{the chart of the word $U$}, and we call the corresponding occurrences of generalized fractional powers \emph{members of the chart}.
\end{definition}

Let us put $\varepsilon = \frac{1}{\beta - \alpha}$. Further we will assume that $\tau \geqslant 10\varepsilon$.

Any subword of $v$ containing at least two letters $y$ appears in a unique way in $v$. Moreover, a subword of $v^l$ containing at least two letters $y$ appears in $v^l$ uniquely modulo the period $v$. Therefore, if the $\LM$-measure of some maximal occurrence of a generalized fractional power is greater or equal to $2\varepsilon$, it can not be properly contained in another occurrence of a generalized fractional power (because otherwise it would not be maximal). In particular, given the chart of some word, one member of the chart can not be properly contained in another. However, members of the chart may have overlaps.

Consider an overlap of two members of the chart. If the overlap is a common part of two subwords of $v^l$, then according to the last remark, the $\LM$-measure of the overlap is not greater than $\varepsilon$. If the overlap contains a subword of $w^k$, it may have a more complicated form. Since $\alpha \gg \vert w \vert$, $w$ does not contain subwords $y^{\delta_1}x^{\gamma}y^{\delta_2}$, where $\vert \gamma\vert \geqslant \alpha$. Hence, since the $\LM$-measure of subwords of $w$ is equal to zero, the $\LM$-measure of the overlap still is not greater than $\varepsilon$ for both members of the chart. In particular, since $\tau \geqslant 10\varepsilon$, we never have a situation when one member of the chart is fully covered by others.

\begin{definition}
\label{v_diagram}
Consider oriented graphs with edges marked by generators of the group $\Fr$. Take such a graph of the form
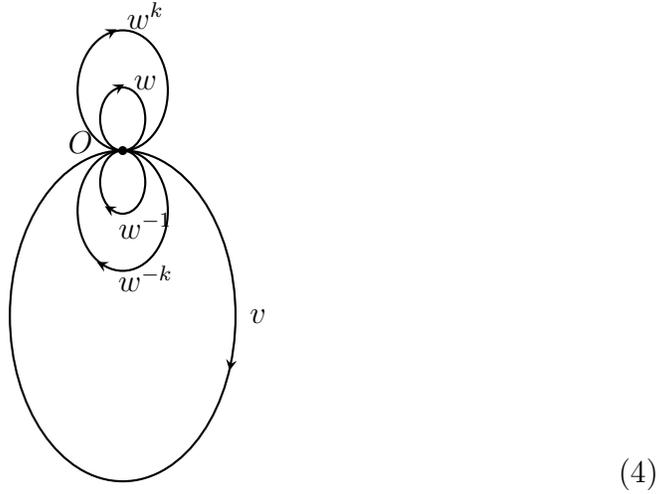
\begin{equation}
\begin{tikzpicture}
\begin{scope}
\coordinate (O) at ($(0,0)+(90:1.5 and 2.2)$);

\draw[black, thick, reversearrow=0.95] (0, 0) ellipse (1.5 and 2.2) node at (1.8, 0) {$v$};
\node[circle,fill,inner sep=1.2] at (O) {};
\path let \p1 = (O) in node at (\x1-16, \y1+3) {$O$};

\draw[black, thick, reversearrow=0.3] (0, 2.2+0.8) ellipse (0.6 and 0.8) node at (0.3, 2.2+1.8) {$w^k$};
\draw[black, thick, reversearrow=0.3] (0, 2.2+0.7*0.6) ellipse (0.5*0.6 and 0.7*0.6) node at (0.3, 1.3 + 1.8) {$w$};

\draw[black, thick, reversearrow=0.7] (0, 2.2-0.8) ellipse (0.6 and 0.8) node at (0.3, 2.2-1.7) {$w^{-k}$};
\draw[black, thick, reversearrow=0.7] (0, 2.2-0.7*0.6) ellipse (0.5*0.6 and 0.7*0.6) node at (0.3, 2.2-1) {$w^{-1}$};

\end{scope}
\end{tikzpicture}
\label{v_diagram_picture}
\end{equation}
Here to each integer power of $w$ corresponds a separate arc. That is, there are infinitely many arcs that correspond to different $w^{k}$. We call this graph a \emph{$v$-diagram}.
\end{definition}
Assume that we have an oriented path in the graph~\eqref{v_diagram_picture}. When we go along this path, we can write down the mark of an edge if we pass the edge in the positive direction, and we can write down the inverse to the mark of an edge if we pass the edge in the negative direction. As a result, we obtain a generalized fractional power (possibly after cancellations if there are any).

It is easy to see that to each monomial over $v, w$ there corresponds a path in the graph~\eqref{v_diagram_picture} with the initial and the final vertex $O$. By the definition, every generalized fractional power is a subword in a monomial over $v, w$. Hence, given a generalized fractional power $M$, one can specify two points $I$ and $F$ on the graph \eqref{v_diagram_picture} such that $M$ corresponds to the unique path starting at $I$ and ending at $F$. For the sake of uniqueness, we assume that we always choose an arc with maximal absolute value of degree of $w$ along the path. So, one can see that $M$ corresponds to the path in the graph of one of the types \eqref{path_type1} --- \eqref{path_type3}.

Note that if a generalized fractional power is of $\LM$-measure strictly greater than $\varepsilon$, then the positions of the initial and the final points that belong to the $v$-arc are uniquely determined.

In what follows we use for $v$ the notation $v = v_iv_mv_f$, where $v_i$ is some initial part, $v_m$ is some middle part and $v_f$ is some final part of $v$ (any part is allowed to be empty). Analogously to the notion of $v$-diagram, in this paper we represent all monomials (not only generalized fractional powers) as segments that consist of oriented edges marked by generators of the group $\Fr$; such segments we always read from left to right. So, for $v$ we have

\vspace{0.1cm}
\begin{center}
\begin{tikzpicture}
\begin{scope}[thick,decoration={
    markings,
    mark=at position 0.6 with {\arrow{stealth}}}
    ]

    \draw[|-|, black, postaction={decorate}] (-2,0)--(-0.7,0) node [below, near start] {$v$};
    \node at (-0.7*0.5 , 0) {$=$};
    \draw[|-, black, postaction={decorate}] (0,0)--(1,0) node [midway, below] {$v_i$};
    \draw[|-, black, postaction={decorate}] (1,0)--(2,0) node [midway, below] {$v_m$};
    \draw[|-|, black, postaction={decorate}] (2,0)--(3,0) node [midway, below] {$v_f$};
\end{scope}
\end{tikzpicture}
\end{center}
\vspace{0.1cm}
Similarly, for $w$ we use the notation $w = w_iw_mw_f$, where $w_i$ is some initial part, $w_m$ is some middle part and $w_f$ is some final part of $w$ (any part is allowed to be empty).

Let us enumerate all possible positions of points $I$ and $F$ on the $v$-diagram. Thereby, we enumerate all possible types of generalized fractional powers. In every case, we explicitly write the corresponding forms of generalized fractional powers. As above, $M(v, w)$ is a monomial over $v$ and $w$. First, both points $I$ and $F$ may lie on the $v$-arc. Then they divide the $v$-arc into three parts and, according to the notations introduced above, we denote them $v_i$, $v_m$ and $v_f$ and obtain the pictures
\begin{equation}
\begin{tikzpicture}
\begin{scope}
\coordinate (O) at ($(0,0)+(90:1.5 and 2.2)$);
\coordinate (Si) at ($(0,0)+(10:1.5 and 2.2)$);
\coordinate (Se) at ($(0,0)+(195:1.5 and 2.2)$);

\coordinate (vi) at (0.96, 1.69);
\coordinate (vm) at (0.32, -2.15);
\coordinate (vf) at (-1.19, 1.33);

\draw[black, thick, arrow=0.6] (O) arc (90:10:1.5 and 2.2);
\node[right] at (vi) {$v_i$};
\draw[black, thick, arrow=0.6] (Si) arc (10:-360+195:1.5 and 2.2);
\node[below] at (vm) {$v_m$};
\draw[black, thick, arrow=0.4] (Se) arc (195:90:1.5 and 2.2);
\node[left, yshift=2] at (vf) {$v_f$};

\node[circle,fill,inner sep=1.2] at (O) {};
\path let \p1 = (O) in node at (\x1-16, \y1+3) {$O$};

\node[circle,fill,inner sep=1.2] at (Si) {};
\node[right] at (Si) {$I$};
\node[circle,fill,inner sep=1.2] at (Se) {};
\node[left] at (Se) {$F$};

\draw[black, thick, reversearrow=0.3] (0, 2.2+0.8) ellipse (0.6 and 0.8) node at (0.3, 2.2+1.8) {$w^k$};
\draw[black, thick, reversearrow=0.3] (0, 2.2+0.7*0.6) ellipse (0.5*0.6 and 0.7*0.6);

\draw[black, thick, reversearrow=0.7] (0, 2.2-0.8) ellipse (0.6 and 0.8) node at (0.3, 2.2-1.7) {$w^{-k}$};
\draw[black, thick, reversearrow=0.7] (0, 2.2-0.7*0.6) ellipse (0.5*0.6 and 0.7*0.6);
\node at (0, -4) {$\begin{aligned}
&v_mv_fM(v, w)v_f^{-1},\\
&v_mv_fM(v, w)v_iv_m,\\
&v_i^{-1}M(v, w)v_iv_m,\\
&v_i^{-1}M(v, w)v_f^{-1};
\end{aligned}$
};
\end{scope}
\begin{scope}
\coordinate (O) at ($(5,0)+(90:1.5 and 2.2)$);
\coordinate (Si) at ($(5,0)+(10:1.5 and 2.2)$);
\coordinate (Se) at ($(5,0)+(195:1.5 and 2.2)$);

\coordinate (vi) at (5 + 0.96, 1.69);
\coordinate (vm) at (5 + 0.32, -2.15);
\coordinate (vf) at (5 + -1.19, 1.33);

\draw[black, thick, arrow=0.6] (O) arc (90:10:1.5 and 2.2);
\node[right] at (vi) {$v_i$};
\draw[black, thick, arrow=0.6] (Si) arc (10:-360+195:1.5 and 2.2);
\node[below] at (vm) {$v_m$};
\draw[black, thick, arrow=0.4] (Se) arc (195:90:1.5 and 2.2);
\node[left, yshift=2] at (vf) {$v_f$};

\node[circle,fill,inner sep=1.2] at (O) {};
\path let \p1 = (O) in node at (\x1-16, \y1+3) {$O$};

\node[circle,fill,inner sep=1.2] at (Si) {};
\node[right] at (Si) {$F$};
\node[circle,fill,inner sep=1.2] at (Se) {};
\node[left] at (Se) {$I$};

\draw[black, thick, reversearrow=0.3] (5, 2.2+0.8) ellipse (0.6 and 0.8) node at (5.3, 2.2+1.8) {$w^k$};
\draw[black, thick, reversearrow=0.3] (5, 2.2+0.7*0.6) ellipse (0.5*0.6 and 0.7*0.6);

\draw[black, thick, reversearrow=0.7] (5, 2.2-0.8) ellipse (0.6 and 0.8) node at (5.3, 2.2-1.7) {$w^{-k}$};
\draw[black, thick, reversearrow=0.7] (5, 2.2-0.7*0.6) ellipse (0.5*0.6 and 0.7*0.6);
\end{scope}
\node at (5, -4) {$\begin{aligned}
&v_f M(v, w)v_i,\\
&v_f M(v, w)v_f^{-1}v_m^{-1},\\
&v_m^{-1}v_i^{-1}M(v, w)v_i,\\
&v_m^{-1}v_i^{-1}M(v, w)v_f^{-1}v_m^{-1}.
\end{aligned}$
};
\end{tikzpicture}
\label{path_type1}
\end{equation}

The next configuration is when the point $I$ lies on the $v$-arc and the point $F$ lies on a $w$-arc or vice versa. Then, according to the above notations, we put $v = v_iv_f$, $w = w_iw_f$ and obtain the pictures
\begin{equation}
\begin{tikzpicture}
\begin{scope}
\coordinate (O) at ($(0,0)+(90:1.5 and 2.2)$);
\coordinate (Se) at ($(0,0)+(320:1.5 and 2.2)$);
\coordinate (Si) at ($(0, 2.2+0.7*0.6)+(40:0.5*0.6 and 0.7*0.6)$);

\coordinate (vi) at (1.36, 0.93);
\coordinate (vf) at (-1.36, -0.93);

\draw[black, thick, arrow=0.6] (O) arc (90:-360+320:1.5 and 2.2);
\node[right] at (vi) {$v_i$};
\draw[black, thick, arrow=0.4] (Se) arc (320:90:1.5 and 2.2);
\node[left, yshift=2] at (vf) {$v_f$};

\node[circle,fill,inner sep=1.2] at (O) {};
\path let \p1 = (O) in node at (\x1-16, \y1+3) {$O$};

\node[circle,fill,inner sep=1.2] at (Si) {};
\node[above] at (Si) {$I$};
\node[circle,fill,inner sep=1.2] at (Se) {};
\node[below] at (Se) {$F$};

\draw[black, thick, reversearrow=0.3] (0, 2.2+0.8) ellipse (0.6 and 0.8) node at (0.3, 2.2+1.8) {$w^k$};
\draw[black, thick, arrow=0.5] (Si) arc (40:-85:0.5*0.6 and 0.7*0.6);
\draw[black, thick, reversearrow=0.3] (Si) arc (40:360-40:0.5*0.6 and 0.7*0.6);

\draw[black, thick, reversearrow=0.7] (0, 2.2-0.8) ellipse (0.6 and 0.8) node at (0.3, 2.2-1.7) {$w^{-k}$};
\draw[black, thick, reversearrow=0.7] (0, 2.2-0.7*0.6) ellipse (0.5*0.6 and 0.7*0.6);
\node at (0, -4) {$\begin{aligned}
&w_fM(v, w)v_i,\\
&w_fM(v, w)v_f^{-1},\\
&w_i^{-1}M(v, w)v_i,\\
&w_i^{-1}M(v, w)v_f^{-1};
\end{aligned}$
};
\end{scope}
\begin{scope}
\coordinate (O) at ($(5,0)+(90:1.5 and 2.2)$);
\coordinate (Se) at ($(5,0)+(320:1.5 and 2.2)$);
\coordinate (Si) at ($(5, 2.2+0.7*0.6)+(40:0.5*0.6 and 0.7*0.6)$);

\coordinate (vi) at (5 + 1.36, 0.93);
\coordinate (vf) at (5 + -1.36, -0.93);

\draw[black, thick, arrow=0.6] (O) arc (90:-360+320:1.5 and 2.2);
\node[right] at (vi) {$v_i$};
\draw[black, thick, arrow=0.4] (Se) arc (320:90:1.5 and 2.2);
\node[left, yshift=2] at (vf) {$v_f$};

\node[circle,fill,inner sep=1.2] at (O) {};
\path let \p1 = (O) in node at (\x1-16, \y1+3) {$O$};

\node[circle,fill,inner sep=1.2] at (Si) {};
\node[above] at (Si) {$F$};
\node[circle,fill,inner sep=1.2] at (Se) {};
\node[below] at (Se) {$I$};


\draw[black, thick, reversearrow=0.3] (5, 2.2+0.8) ellipse (0.6 and 0.8) node at (5.3, 2.2+1.8) {$w^k$};
\draw[black, thick, arrow=0.5] (Si) arc (40:-85:0.5*0.6 and 0.7*0.6);
\draw[black, thick, reversearrow=0.3] (Si) arc (40:360-40:0.5*0.6 and 0.7*0.6);

\draw[black, thick, reversearrow=0.7] (5, 2.2-0.8) ellipse (0.6 and 0.8) node at (5.3, 2.2-1.7) {$w^{-k}$};
\draw[black, thick, reversearrow=0.7] (5, 2.2-0.7*0.6) ellipse (0.5*0.6 and 0.7*0.6);
\node at (5, -4) {$\begin{aligned}
&v_fM(v, w)w_i,\\
&v_i^{-1}M(v, w)w_i,\\
&v_fM(v, w)w_f^{-1},\\
&v_i^{-1}M(v, w)w_f^{-1}.
\end{aligned}$
};
\end{scope}
\end{tikzpicture}
\label{path_type2}
\end{equation}
In picture \eqref{path_type2} points $I$ and $F$ lie on a positive $w$-arc, but they may lie on a negative $w$-arc within this type of paths as well.

The last configuration is when both points $I$ and $F$ lie on a $w$-arc. Then, according to the above notations, we put $w = w_iw_mw_f$ and obtain the pictures
\begin{equation}
\begin{tikzpicture}
\begin{scope}
\coordinate (O) at ($(0,0)+(90:1.5 and 2.2)$);
\coordinate (Si) at ($(0, 2.2+0.7*0.6)+(10:0.5*0.6 and 0.7*0.6)$);
\coordinate (Se) at ($(0, 2.2+0.7*0.6)+(170:0.5*0.6 and 0.7*0.6)$);

\draw[black, thick, reversearrow=0.8] (0, 0) ellipse (1.5 and 2.2) node at (1.8, 0) {$v$};

\node[circle,fill,inner sep=1.2] at (O) {};
\path let \p1 = (O) in node at (\x1 - 16, \y1+3) {$O$};

\node[circle,fill,inner sep=1.2] at (Si) {};
\node[above, xshift=3] at (Si) {$I$};
\node[circle,fill,inner sep=1.2] at (Se) {};
\node[above, xshift=-3] at (Se) {$F$};

\draw[black, thick, reversearrow=0.3] (0, 2.2+0.8) ellipse (0.6 and 0.8) node at (0.3, 2.2+1.8) {$w^k$};
\draw[black, thick, arrow=0.4] (Se) arc (170:10:0.5*0.6 and 0.7*0.6);
\draw[black, thick, arrow=0.5] (Si) arc (10:-85:0.5*0.6 and 0.7*0.6);
\draw[black, thick, reversearrow=0.4] (Se) arc (170:170+85:0.5*0.6 and 0.7*0.6);

\draw[black, thick, reversearrow=0.7] (0, 2.2-0.8) ellipse (0.6 and 0.8) node at (0.3, 2.2-1.7) {$w^{-k}$};
\draw[black, thick, reversearrow=0.7] (0, 2.2-0.7*0.6) ellipse (0.5*0.6 and 0.7*0.6);
\node at (0, -4) {$\begin{aligned}
&w_f M(v, w)w_i,\\
&w_f M(v, w)w_f^{-1}w_m^{-1},\\
&w_m^{-1}w_i^{-1}M(v, w)w_i,\\
&w_m^{-1}w_i^{-1}M(v, w)w_f^{-1}w_m^{-1};
\end{aligned}$
};
\end{scope}
\begin{scope}
\coordinate (O) at ($(5,0)+(90:1.5 and 2.2)$);
\coordinate (Si) at ($(5, 2.2+0.7*0.6)+(10:0.5*0.6 and 0.7*0.6)$);
\coordinate (Se) at ($(5, 2.2+0.7*0.6)+(170:0.5*0.6 and 0.7*0.6)$);

\draw[black, thick, reversearrow=0.8] (5, 0) ellipse (1.5 and 2.2) node at (6.8, 0) {$v$};

\node[circle,fill,inner sep=1.2] at (O) {};
\path let \p1 = (O) in node at (\x1 - 16, \y1+3) {$O$};

\node[circle,fill,inner sep=1.2] at (Si) {};
\node[above, xshift=3] at (Si) {$F$};
\node[circle,fill,inner sep=1.2] at (Se) {};
\node[above, xshift=-3] at (Se) {$I$};

\draw[black, thick, reversearrow=0.3] (5, 2.2+0.8) ellipse (0.6 and 0.8) node at (5.3, 2.2+1.8) {$w^k$};
\draw[black, thick, arrow=0.4] (Se) arc (170:10:0.5*0.6 and 0.7*0.6);
\draw[black, thick, arrow=0.5] (Si) arc (10:-85:0.5*0.6 and 0.7*0.6);
\draw[black, thick, reversearrow=0.4] (Se) arc (170:170+85:0.5*0.6 and 0.7*0.6);

\draw[black, thick, reversearrow=0.7] (5, 2.2-0.8) ellipse (0.6 and 0.8) node at (5.3, 2.2-1.7) {$w^{-k}$};
\draw[black, thick, reversearrow=0.7] (5, 2.2-0.7*0.6) ellipse (0.5*0.6 and 0.7*0.6);
\node at (5, -4) {$\begin{aligned}
&w_mw_fM(v, w)w_f^{-1},\\
&w_mw_fM(v, w)w_iw_m,\\
&w_i^{-1}M(v, w)w_iw_m,\\
&w_i^{-1}M(v, w)w_f^{-1}.
\end{aligned}$
};
\end{scope}
\end{tikzpicture}
\label{path_type3}
\end{equation}
In picture \eqref{path_type3} points $I$ and $F$ lie on a positive $w$-arc, but each of them may lie on a negative $w$-arc within this type of paths as well. Also the point $I$ and the point $F$ may lie on different $w$-arcs.

Given a set of generalized fractional powers $M_j$, $j = 1, \ldots, k$, that correspond to the paths with the same initial and the same final point in the diagram of type \eqref{path_type1} --- \eqref{path_type3}, their sum $\sum_{j = 1}^{k}M_j$ corresponds to the collection of these paths.

Let $\mathbb{Z}_2(w)$ be the field of rational functions in one variable $w$ over $\mathbb{Z}_2$. Consider a non-commutative Laurent polynomial $P(x_1, x_2)$ over $\mathbb{Z}_2$ such that
\begin{equation}
\label{vanish_in_field}
P((1 + w)^{-1}, w) = 0 \emph{ as an element of } \mathbb{Z}_2(w).
\end{equation}
Notice that if $P(x_1, x_2)$ satisfies condition~\eqref{vanish_in_field}, then $P(v, w) \in \langle 1 + v + vw \rangle = \Ideal$.

\begin{example}
Let us give an example of such polynomials. The first example comes from the obvious equality in $ \mathbb{Z}_2(w)$
\begin{equation*}
w \cdot \frac{1}{1 + w} = \frac{1}{1 +w} \cdot w.
\end{equation*}
From this equality it follows that the polynomial
\begin{equation*}
P(x_1, x_2) = x_1x_2 + x_2x_1
\end{equation*}
satisfies condition~\eqref{vanish_in_field} and
\begin{equation*}
P(v, w) = vw + wv \in \Ideal.
\end{equation*}

Recalling the reduction of rational functions of the form $\frac{w^{\pm k}}{1 + w}$ in $\mathbb{Z}_2(w)$ to elementary fractions, we obtain the following equalities in $\mathbb{Z}_2(w)$ for $k > 0$:
\begin{align*}
\frac{w^k}{1 + w} & = \frac{w^k + w^{k-1} + w^{k-1} + \ldots + w + w + 1 + 1}{1 + w} \\ &= w^{k-1} + w^{k-2} + \ldots + w + 1 + \frac{1}{1 + w}
\end{align*}
and
\begin{align*}
\frac{w^{-k}}{1 + w}& = \frac{w^{-k} + w^{-k+1} + w^{-k+1} + \ldots + w^{-1} + w^{-1} + 1 + 1}{1 + w} \\ &= w^{-k} + w^{-k+1} + \ldots + w^{-1} + \frac{1}{1 + w}.
\end{align*}
Hence, the polynomials
\begin{align*}
&P_1(x_1, x_2) = x_1x_2^{k} + x_2^{k - 1} + x_2^{k - 2} + \ldots + x_2 + 1 + x_1,\\
&P_2(x_1, x_2) = x_1x_2^{-k} + x_2^{-k} + x_2^{-k + 1} + \ldots + x_2^{-1} + x_1
\end{align*}
satisfy condition~\eqref{vanish_in_field} and
\begin{align*}
&P_1(v, w) = vw^{k} + w^{k - 1} + w^{k - 2} + \ldots + w + 1 + v \in \Ideal,\\
&P_2(v, w) = vw^{-k} + w^{-k} + w^{-k + 1} + \ldots + w^{-1} + v \in \Ideal.
\end{align*}
\end{example}

According to types \eqref{path_type1} --- \eqref{path_type3} of $v$-diagrams that correspond to possible forms of generalized fractional powers, we consider the list of expressions \eqref{mturn1} --- \eqref{mturn6}, where $P(x_1, x_2)$ is a non-commutative Laurent polynomial such that $P((1 + w)^{-1}, w) = 0$ as an element of $\mathbb{Z}_2(w)$. We allow possibility of cancellations in monomials in \eqref{mturn1} --- \eqref{mturn6}. First, consider expressions
\begin{align}
\begin{split}
\label{mturn1}
&v_f P(v, w)v_i,\\
&v_f P(v, w)v_f^{-1}v_m^{-1},\\
&v_m^{-1}v_i^{-1}P(v, w)v_i,\\
&v_m^{-1}v_i^{-1}P(v, w)v_f^{-1}v_m^{-1}.
\end{split}
\end{align}
To each expression of \eqref{mturn1}, there corresponds a collection of paths in the graph~\eqref{path_type1}. Similarly, to each expression
\begin{align}
\begin{split}
\label{mturn2}
&v_mv_fP(v, w)v_f^{-1},\\
&v_mv_fP(v, w)v_iv_m,\\
&v_i^{-1}P(v, w)v_iv_m,\\
&v_i^{-1}P(v, w)v_f^{-1},
\end{split}
\end{align}
there corresponds a collection of paths in the graph~\eqref{path_type1}. To each expression
\begin{align}
\begin{split}
\label{mturn3}
&w_fP(v, w)v_i,\\
&w_fP(v, w)v_f^{-1},\\
&w_i^{-1}P(v, w)v_i,\\
&w_i^{-1}P(v, w)v_f^{-1},
\end{split}
\end{align}
there corresponds a collection of paths in the graph~\eqref{path_type2}. Similarly, to each expression
\begin{align}
\begin{split}
\label{mturn4}
&v_fP(v, w)w_i,\\
&v_i^{-1}P(v, w)w_i,\\
&v_fP(v, w)w_f^{-1},\\
&v_i^{-1}P(v, w)w_f^{-1},
\end{split}
\end{align}
there corresponds a collection of paths in the graph~\eqref{path_type2}. Finally, consider the expressions
\begin{align}
\begin{split}
\label{mturn5}
&w_f P(v, w)w_i,\\
&w_f P(v, w)w_f^{-1}w_m^{-1},\\
&w_m^{-1}w_i^{-1}P(v, w)w_i,\\
&w_m^{-1}w_i^{-1}P(v, w)w_f^{-1}w_m^{-1}.
\end{split}
\end{align}
To each expression of \eqref{mturn5} there corresponds a collection of paths in the graph~\eqref{path_type3}. Similarly, to each expression
\begin{align}
\begin{split}
\label{mturn6}
&w_mw_fP(v, w)w_f^{-1},\\
&w_mw_fP(v, w)w_iw_m,\\
&w_i^{-1}P(v, w)w_iw_m,\\
&w_i^{-1}P(v, w)w_f^{-1},
\end{split}
\end{align}
there corresponds a collection of paths in the graph~\eqref{path_type3}. Since $P(v, w) \in \Ideal$, the expressions \eqref{mturn1} --- \eqref{mturn6} vanish in $\mathbb{Z}_2\Fr / \Ideal$.

Every expression \eqref{mturn1} --- \eqref{mturn6} is, in fact, a linear combination of generalized fractional powers. For any linear combination $\sum_{j = 1}^{k}M_j$ of such type (where the cancellations in monomials are already performed) we call the monomials $M_{j_1}, M_{j_2}$, $j_1, j_2 \in \{1, \ldots, k\}$, \emph{incident monomials}. Notice that the paths in the $v$-diagram corresponding to incident monomials always have the same initial points and the same final points. Clearly, one generalized fractional power has an infinite number of incident monomials because they may contain different powers of $v$ and $w$.

Suppose we have a group $G$ with a relator $R = M_1M_2^{-1}$. Let $U = LM_1R$. Recall that the transition from $U = LM_1R$ to $LM_2R$ representing the same element of $G$
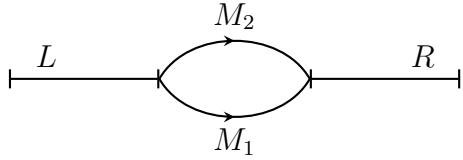
\vspace{0.1cm}
\begin{center}
\begin{tikzpicture}
\draw[|-|, black, thick] (0,0)--(2,0) node [near start, above] {$L$};
\draw[black, thick, arrow=0.5] (2,0) to [bend left=60] node [above] {$M_2$} (4,0);
\draw[black, thick, arrow=0.5] (2,0) to [bend right=60] node [below] {$M_1$} (4,0);
\draw[|-|, black, thick] (4,0)--(6,0) node [near end, above] {$R$};
\end{tikzpicture}
\vspace{0.1cm}
\end{center}
is called a \emph{turn} of an occurrence of a subrelation $M_1$ (to its complement $M_2$). We generalize this notion to a multi-turn. Multi-turns will play a central role in our work.
\begin{definition}
\label{multiturn_def}
Let $\sum_{j = 1}^{k}M_j$ be one of the expressions \eqref{mturn1} --- \eqref{mturn6} (where the cancellations in monomials are already performed). We call the transition
\begin{equation*}
M_h\longmapsto \sum\limits_{\substack{j = 1 \\ j\neq h}}^{k} M_j,
\end{equation*}
 an \emph{elementary multi-turn of $M_h$}.

Let $U_h = LM_hR$
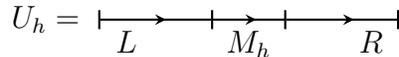
\vspace{0.1cm}
\begin{center}
\begin{tikzpicture}
\node at (-0.7 , 0) {$U_h=$};
\draw[|-, black, thick, arrow=0.6] (0,0)--(1.5,0) node [near start, below] {$L$};
\draw[|-|, thick, black, arrow=0.6] (1.5,0)--(2.5,0) node [midway, below] {$M_h$};
\draw[-|, black, thick, arrow=0.6] (2.5,0)--(4,0) node [near end, below] {$R$};
\end{tikzpicture}
\end{center}
\vspace{0.1cm}
be a monomial in $\mathbb{Z}_2\Fr$, where $M_h$ is a generalized fractional power, and $\sum_{j = 1}^{k} M_j$ be one of the expressions \eqref{mturn1} --- \eqref{mturn6}, then the transition
\begin{equation*}
U_h \longmapsto \sum\limits_{\substack{j = 1 \\ j\neq h}}^{k} U_j,
\end{equation*}
where $U_j = LM_jR$,
\vspace{0.1cm}
\begin{center}
\begin{tikzpicture}
\node at (-0.7 , 0) {$U_j = $};
\draw[|-, black, thick, arrow=0.6] (0,0)--(1.5,0) node [near start, below] {$L$};
\draw[|-|, thick, black, arrow=0.6] (1.5,0)--(2.5,0) node [midway, below] {$M_j$};
\draw[-|, black, thick, arrow=0.6] (2.5,0)--(4,0) node [near end, below] {$R$};
\node at (5.9, 0) {$j = 1, \ldots, k, j\neq h,$};
\end{tikzpicture}
\end{center}
\vspace{0.1cm}
is called a \emph{multi-turn of the occurrence $M_h$ in $U_h$}. Notice that since $\sum_{j = 1}^{k} M_j \in \Ideal$, $U_h = \sum_{\substack{j = 1 \\ j\neq h}}^{k} U_j$ in the ring $\mathbb{Z}_2\Fr / \Ideal$.

We call the corresponding sum $\sum_{j = 1}^{k} U_j$ \emph{the support of the multi-turn}. Then the expressions~\eqref{mturn1} --- \eqref{mturn6} are all types of supports of elementary multi-turns.
\end{definition}

\section{Linear dependencies on $\mathbb{Z}_2\Fr$ induced by multi-turns. The description of the ideal $\Ideal$ as a linear subspace of $\mathbb{Z}_2\Fr$}
\label{monomials_chart_dependencies}

\subsection{How multi-turns influence the chart}
\label{mt_configurations}
Let $U$ be a word and $a$ and $b$ be members of its chart, i.e., maximal occurrences of generalized fractional powers. Since the monomial $U$ is represented as a segment, it is natural to represent members of the chart of $U$ as its subsegments. Let us describe possible configurations of the members $a$ and $b$.
\begin{enumerate}
\item
We say that the members of the chart $a$ and $b$ are \emph{separated} if there exists some non empty subword between them in the word $U$.
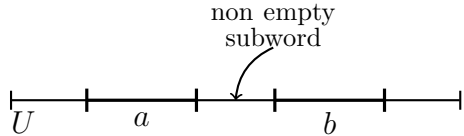
\begin{center}
\begin{tikzpicture}
\draw[|-|, black, thick] (0,0)--(6,0);
\node[below, xshift=5] at (0, 0) {$U$};
\draw[|-|, black, very thick] (1,0)--(2.5,0) node[midway, below] {$a$};
\draw[|-|, black, very thick] (3.5,0)--(5,0) node[midway, below] {$b$};
\node[text width=2cm, align=center] at (3.5,1) {\footnotesize{\baselineskip=10ptnon empty subword \par}};
\draw[->, black, thick] (3.5, 0.7) to [bend right] (3, 0);
\end{tikzpicture}
\end{center}
\item
We say that the members $a$ and $b$ \emph{touch at a point} if $a$ and $b$ are adjacent and have no common non empty subword.
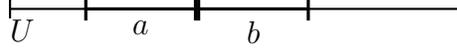
\begin{center}
\begin{tikzpicture}
\draw[|-|, black, thick] (0,0)--(6,0);
\node[below, xshift=5] at (0, 0) {$U$};
\draw[|-|, black, very thick] (1,0)--(2.5,0) node[midway, below] {$a$};
\draw[|-|, black, very thick] (2.5,0)--(4,0) node[midway, below] {$b$};
\end{tikzpicture}
\end{center}
\item
We say that the members $a$ and $b$ \emph{have an overlap} if the members have non empty common subword. Recall that the $\LM$-measure of the overlap is not greater than $\varepsilon$.
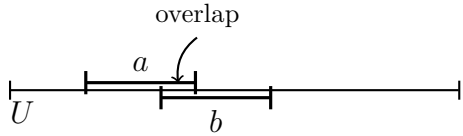
\begin{center}
\begin{tikzpicture}
\draw[|-|, black, thick] (0,0)--(6,0);
\node[below, xshift=5] at (0, 0) {$U$};
\draw[|-|, black, very thick] (1,0.1)--(2.5,0.1) node[midway, above] {$a$};
\draw[|-|, black, very thick] (2,-0.1)--(3.5,-0.1) node[midway, below] {$b$};
\node[text width=2cm, align=center] at (2.5,1) {\footnotesize{\baselineskip=10ptoverlap \par}};
\draw[->, black, thick] (2.5, 0.7) to [bend right] (2.25, 0.1);
\end{tikzpicture}
\end{center}
\end{enumerate}

We will call the member of the chart of $U$ closest to $a$ from the left side \emph{the left neighbour of $a$ in the chart of $U$}. Similarly, we will call the member of the chart of $U$ closest to $a$ from the right side \emph{the right neighbour of $a$ in the chart of $U$}. Recall that maximal occurrences of generalized fractional powers are considered as members of the chart only when they have $\LM$-measure greater or equal to a given threshold $\tau$ and we assume that $\tau \geqslant 10\varepsilon$. The restriction on the $\LM$-measure allows us to avoid situations when one member of the chart is fully covered by other members. So, we avoid the situation when there exist two members that are not separated from $a$ on one side. For example, we never obtain configurations like the following:
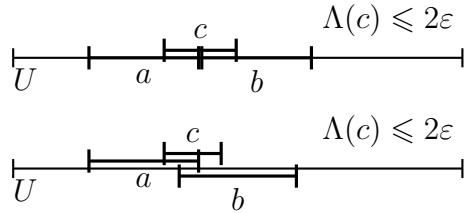
\begin{center}
\begin{tikzpicture}
\draw[|-|, black, thick] (0,0)--(6,0);
\node[below, xshift=5] at (0, 0) {$U$};
\draw[|-|, black, very thick] (1,0)--(2.5,0) node[midway, below] {$a$};
\draw[|-|, black, very thick] (2.5,0)--(4,0) node[midway, below] {$b$};
\draw[|-|, black, very thick] (2,0.1)--(3,0.1) node[midway, above] {$c$};
\node at (5,0.5) {$\LM(c) \leqslant 2\varepsilon$};
\end{tikzpicture}

\begin{tikzpicture}
\draw[|-|, black, thick] (0,0)--(6,0);
\node[below, xshift=5] at (0, 0) {$U$};
\draw[|-|, black, very thick] (1,0.1)--(2.5,0.1) node[midway, below] {$a$};
\draw[|-|, black, very thick] (2.2,-0.1)--(3.8,-0.1) node[midway, below] {$b$};
\draw[|-|, black, very thick] (2,0.2)--(2.8,0.2) node[midway, above] {$c$};
\node at (5,0.5) {$\LM(c) \leqslant 2\varepsilon$};
\end{tikzpicture}
\end{center}

Suppose $U_h = La_hR$ is a word, $a_h$ is a member of its chart, in particular $\LM(a_h) \geqslant \tau$. Suppose $U_h \mapsto \sum_{\substack{j = 1 \\ j\neq h}}^{k}U_j$ is a multi-turn of $U_h$ that comes from an elementary multi-turn $a_h \mapsto \sum_{\substack{j = 1 \\ j\neq h}}^{k}a_j$. So, $U_h = \sum_{\substack{j = 1\\ j\neq h}}^{k}U_j$ in the ring $\mathbb{Z}_2\Fr / \Ideal$. Let us study how the chart of $U_j = La_jR$, $j = 1, \ldots, k$, $j\neq h$, is related to the chart of $U_h = La_hR$. We distinguish three types of monomials in the sum $\sum_{\substack{j = 1\\ j\neq h}}^{k}U_j$:
\begin{enumerate}
\item
\label{keep_structure}
$La_jR$, where $\LM(a_j) > \varepsilon$;
\item
\label{donot_keep_structure1}
$La_jR$, where $a_j = 1$.
\item
\label{donot_keep_structure2}
$La_jR$, where $\LM(a_j) \leqslant \varepsilon$ but $a_j \neq 1$;
\end{enumerate}

Let $U_j = La_jR$ be a monomial of type~\ref{keep_structure}, that is, $\LM(a_j) > \varepsilon$. Let us study the chart of $La_jR$. First of all notice that since $a_h$ is a maximal occurrence of a generalized fractional power in $La_hR$, the monomial $U_j = La_jR$ has no cancellations and $a_j$ is a maximal occurrence of a generalized fractional power in $U_j$. Clearly, all members of the chart of $U_h = La_hR$ that are separated from $a_h$ remain unchanged in the chart of $U_j = La_jR$. Let $b_h$ be the left neighbour of $a_h$ in the chart of $U_h$. Assume $b_h$ is not separated from $a_h$, and $b_h^{\prime}$ is its initial subword such that $L = L^{\prime}b_h^{\prime}$, $U_h = L^{\prime}b_h^{\prime}a_hR$. Then there is a corresponding occurrence of $b_h^{\prime}$ in $U_j$. Let $b_j$ be the member of the chart of $U_j$ that prolongs $b_h^{\prime}$. Clearly, the member of the chart of $U_h$ that prolongs $b_h^{\prime}$ is precisely $b_h$. The member $b_j$ may differ from $b_h$ when $j\neq h$. There are four possibilities:
\begin{enumerate}[label=\ref{keep_structure}.\arabic*]
\item
\label{overlap_appears}
The members $b_h$ and $a_h$ touch at a point in $U_h$. The members $b_j$ and $a_j$ have a non-empty overlap in $U_j$. Assume $c$ is the overlap between $b_j$ and $a_j$. In this case $b_j = b_hc$. For example, we may obtain this effect when a beginning of $v_m$ is replaced by an end of $v_i^{-1}$ or vice versa.
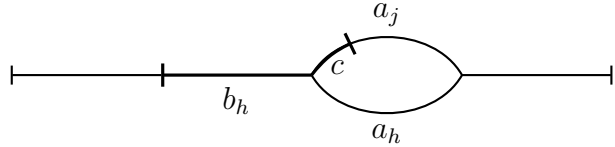
\begin{center}
\begin{tikzpicture}
\draw[|-, black, thick] (0,0)--(2,0);
\draw[|-,black, very thick] (2,0)--(4,0) node [midway, below] {$b_h$};
\draw[black, thick] (4,0) to [bend left=60] coordinate[pos=0.3] (A) node [above] {$a_j$} (6,0);
\draw[black, thick] (4,0) to [bend right=60] node [below] {$a_h$} (6,0);
\draw[-|, black, very thick] (4,0) to [bend left=13] node [below, xshift=3, yshift=3] {$c$} (A);
\draw[-|, black, thick] (6,0)--(8,0);
\end{tikzpicture}
\end{center}
\item
\label{overlap_increases}
The members $b_h$ and $a_h$ have an overlap in $U_h$. The member $b_h$ is enlarged in $U_j$ with the use of $a_j$ and the overlap between $b_j$ and $a_j$ increases in $U_j$. If the overlap increases by a piece $c$, then $b_j = b_hc$.
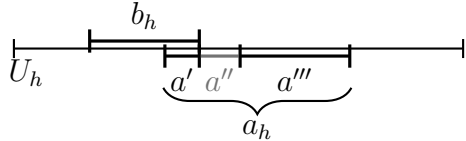
\begin{center}
\begin{tikzpicture}
\draw[|-|, black, thick] (0,0)--(6,0);
\node[below, xshift=5] at (0, 0) {$U_h$};
\draw[|-|, black, very thick] (1,0.1)--(2.5,0.1) node[midway, above] {$b_h$};
\draw[|-|, black, very thick] (2,-0.1)--(2.5,-0.1) node[midway, below] {$a^{\prime}$};
\draw[black!10!gray, very thick] (2.5,-0.1)--(3,-0.1) node[midway, below] {$a^{\prime\prime}$};
\draw[|-|, black, very thick] (3,-0.1)--(4.5,-0.1) node[midway, below] {$a^{\prime\prime\prime}$};
\draw [thick, decorate, decoration={brace, amplitude=10pt, raise=12pt, mirror}] (2, -0.1) to node[midway, below, yshift=-20pt] {$a_h$} (4.5, -0.1);
\end{tikzpicture}

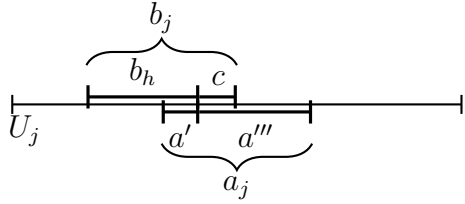
\begin{tikzpicture}
\draw[|-|, black, thick] (0,0)--(6,0);
\node[below, xshift=5] at (0, 0) {$U_j$};
\draw[|-|, black, very thick] (1,0.1)--(2.5,0.1) node[midway, above] {$b_h$};
\draw[-|, black, very thick] (2.5,0.1)--(3,0.1) node[midway, above] {$c$};
\draw [thick, decorate, decoration={brace, amplitude=10pt, raise=12pt}] (1, 0.1) to node[midway, above, yshift=20pt] {$b_j$} (3, 0.1);
\draw[|-|, black, very thick] (2,-0.1)--(2.5,-0.1) node[midway, below] {$a^{\prime}$};
\draw[-|, black, very thick] (2.5,-0.1)--(4,-0.1) node[midway, below] {$a^{\prime\prime\prime}$};
\draw [thick, decorate, decoration={brace, amplitude=10pt, raise=12pt, mirror}] (2, -0.1) to node[midway, below, yshift=-20pt] {$a_j$} (4, -0.1);
\end{tikzpicture}
\end{center}
\item
\label{overlap_disappears}
The members $b_h$ and $a_h$ have an overlap in $U_h$. The member $b_h$ is shortened in $U_j$ such that $b_j$ and $a_j$ touch at a point in $U_j$. For example, we obtain this effect when a beginning of $v_m$ is replaced by an end of $v_i^{-1}$ or vice versa. In this case $b_h = b_jc$.
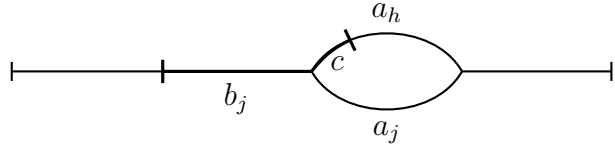
\begin{center}
\begin{tikzpicture}
\draw[|-, black, thick] (0,0)--(2,0);
\draw[|-,black, very thick] (2,0)--(4,0) node [midway, below] {$b_j$};
\draw[black, thick] (4,0) to [bend left=60] coordinate[pos=0.3] (A) node [above] {$a_h$} (6,0);
\draw[black, thick] (4,0) to [bend right=60] node [below] {$a_j$} (6,0);
\draw[-|, black, very thick] (4,0) to [bend left=13] node [below, xshift=3, yshift=3] {$c$} (A);
\draw[-|, black, thick] (6,0)--(8,0);
\end{tikzpicture}
\end{center}
\item
\label{overlap_decreases}
The members $b_h$ and $a_h$ have an overlap in $U_h$. The member $b_h$ is shortened in $U_j$ and the overlap between $b_j$ and $a_j$ becomes shorter in $U_j$. If the overlap decreases on a piece $c$, then $b_h = b_jc$.
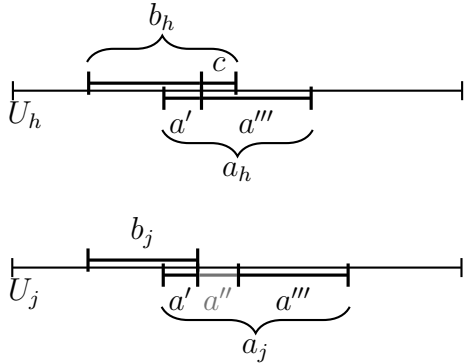
\begin{center}
\begin{tikzpicture}
\draw[|-|, black, thick] (0,0)--(6,0);
\node[below, xshift=5] at (0, 0) {$U_h$};
\draw[|-|, black, very thick] (2.5,0.1)--(3,0.1) node[midway, above] {$c$};
\draw[|-, black, very thick] (1,0.1)--(2.5,0.1);
\draw [thick, decorate, decoration={brace, amplitude=10pt, raise=8pt}] (1, 0.1) to node[midway, above, yshift=16pt] {$b_h$} (3, 0.1);
\draw[|-, black, very thick] (2,-0.1)--(2.5,-0.1) node[midway, below] {$a^{\prime}$};
\draw[|-|, black, very thick] (2.5,-0.1)--(4,-0.1) node[midway, below] {$a^{\prime\prime\prime}$};
\draw [thick, decorate, decoration={brace, amplitude=10pt, raise=12pt, mirror}] (2, -0.1) to node[midway, below, yshift=-20pt] {$a_h$} (4, -0.1);
\end{tikzpicture}

\begin{tikzpicture}
\draw[|-|, black, thick] (0,0)--(6,0);
\node[below, xshift=5] at (0, 0) {$U_j$};
\draw[|-|, black, very thick] (1,0.1)--(2.5,0.1) node[midway, above] {$b_j$};
\draw[|-|, black, very thick] (2,-0.1)--(2.5,-0.1) node[midway, below] {$a^{\prime}$};
\draw[black!10!gray, very thick] (2.5,-0.1)--(3,-0.1) node[midway, below] {$a^{\prime\prime}$};
\draw[|-|, black, very thick] (3,-0.1)--(4.5,-0.1) node[midway, below] {$a^{\prime\prime\prime}$};
\draw [thick, decorate, decoration={brace, amplitude=10pt, raise=12pt, mirror}] (2, -0.1) to node[midway, below, yshift=-20pt] {$a_j$} (4.5, -0.1);
\end{tikzpicture}
\end{center}
\end{enumerate}
The essential observation is that either $b_h = b_jc$, or $b_j = b_hc$; in both cases the $\LM$-measure of $c$ can not exceed $\varepsilon$. The possibilities for being changed for the right neighbour of $a_h$ in the chart are the same.

The cases described in \ref{overlap_appears}---\ref{overlap_decreases} may occur even if $\LM(b_h) < \tau$ or $\LM(b_j) < \tau$. Therefore, the chart of $U_j$ also may change as a result of the following effect. Assume $\tau - \varepsilon \leqslant \LM(b_h) < \tau$. Then we may obtain $\LM(b_j) \geqslant \tau$ in $U_j$, that is, $b_j$ is counted as a member of the chart of $U_j$. So, a new member appears in the chart of $U_j$. It can also happen that $\LM(b_j) < \tau$, while $\LM(b_h) \geqslant \tau$.

Consider the chart of a monomial $U_j = LR$ of type~\ref{donot_keep_structure1}. Cancellations between $L$ and $R$ may occur. Suppose $L = L^{\prime}C$, $R = C^{-1}R^{\prime}$ and $L^{\prime}R^{\prime}$ does not have further cancellations. Let $P$ be the meeting point of $L^{\prime}$ and $R^{\prime}$. Members of the chart of $L^{\prime}R^{\prime}$ that are separated from the point $P$ remain unchanged. The possible configurations are as follows:
\begin{enumerate}[label=\ref{donot_keep_structure1}.\arabic*]
\item
\label{all_separated2}
$L^{\prime}$ does not have a terminal subword equal to some generalized fractional power and $R^{\prime}$ does not have an initial subword equal to some generalized fractional power. Then the chart of $L^{\prime}R^{\prime}$ consists of members that lie left to $P$ in $L^{\prime}$ and right to $P$ in $R^{\prime}$.
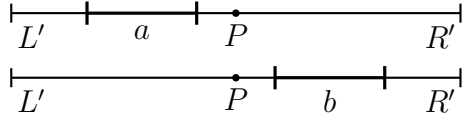
\begin{center}
\begin{tikzpicture}
\draw[|-|, black, thick] (0,0)--(6,0);
\path[|-|, black, thick] (0,0) to node[below, at start, xshift=8] {$L^{\prime}$} (6,0);
\path[|-|, black, thick] (0,0) to node[below, at end, xshift=-8] {$R^{\prime}$} (6,0);
\draw[|-|, black, very thick] (1,0)--(2.5,0) node[midway, below] {$a$};
\node[circle,fill,inner sep=1] at (3,0) {};
\node[below] at (3,0) {$P$};
\end{tikzpicture}

\begin{tikzpicture}
\draw[|-|, black, thick] (0,0)--(6,0);
\path[|-|, black, thick] (0,0) to node[below, at start, xshift=8] {$L^{\prime}$} (6,0);
\path[|-|, black, thick] (0,0) to node[below, at end, xshift=-8] {$R^{\prime}$} (6,0);
\draw[|-|, black, very thick] (3.5,0)--(5,0) node[midway, below] {$b$};
\node[circle,fill,inner sep=1] at (3,0) {};
\node[below] at (3,0) {$P$};
\end{tikzpicture}
\end{center}
\end{enumerate}
Suppose $L^{\prime} = L^{\prime}_1a$, $R^{\prime} = bR^{\prime}_1$, $a$ and $b$ are maximal occurrences of generalized fractional powers in $L^{\prime}$ and in $R^{\prime},$ respectively, possibly with $\LM$-measure less than $\tau$. We allow one of $a$ or $b$ to be equal to $1$. The occurrence of $a$ in $L^{\prime}R^{\prime}$ does not have to be maximal. Let us denote by $a^{\prime}$ the maximal occurrence of a generalized fractional power that prolongs $a$ in $L^{\prime}R^{\prime}$. Similarly, let us denote by $b^{\prime}$ the maximal occurrence of a generalized fractional power that prolongs $b$ in $L^{\prime}R^{\prime}$. Then interactions between $a^{\prime}$ and $b^{\prime}$ are as follows:
\begin{enumerate}[label=\ref{donot_keep_structure1}.\arabic*]
\setcounter{enumi}{1}
\item
\label{strictly_touch2}
$a^{\prime}$ and $b^{\prime}$ touch at a point in $L^{\prime}R^{\prime}$. Then like in the case~\ref{all_separated2} the chart of $L^{\prime}R^{\prime}$ consists precisely of the members contained in $L^{\prime}$ and the members contained in $R^{\prime}$.
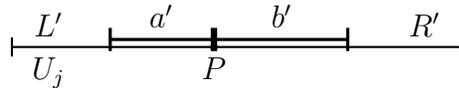
\begin{center}
\begin{tikzpicture}
\draw[|-|, black, thick] (0,0)--(6,0);
\node at (0.5, 0) [above] {$L^{\prime}$};
\node at (5.5, 0) [above] {$R^{\prime}$};
\node at (0.5, 0) [below] {$U_j$};
\draw[|-|, black, very thick] (1.3,0.1)--(2.7,0.1) node[midway, above] {$a^{\prime}$};
\draw[|-|, black, very thick] (2.7,0.1)--(4.5,0.1) node[midway, above] {$b^{\prime}$};
\node at (2.7, 0) [below] {$P$};
\node[circle,fill,inner sep=1] at (2.7, 0) {};
\end{tikzpicture}
\end{center}
\item
\label{slightly_enlarge2}
$a^{\prime}$ and $b^{\prime}$ in $L^{\prime}R^{\prime}$ have an overlap. Then $a^{\prime}$ is at most slightly enlarged with respect to $a$, and $b^{\prime}$ is at most slightly enlarged with respect to $b$. Since the $\LM$-measure of their overlap does not exceed $\varepsilon$, their total increase does not exceed $\varepsilon$.
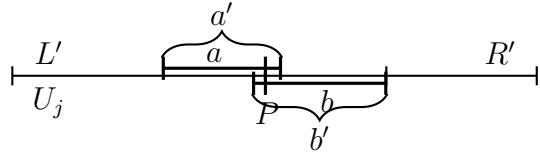
\begin{center}
\begin{tikzpicture}
\coordinate (S) at (0,0);
\coordinate (Li) at (2, 0);
\coordinate (ai) at (3.5, 0);
\coordinate (Rf) at (5, 0);
\coordinate (E) at (7, 0);
\draw[|-, black, thick] (S)--(Li) node [near start, above] {$L^{\prime}$} node[near start, below] {$U_j$};
\draw[|-|,black, very thick] (2, 0.1)--(3.4, 0.1) node [midway, above, yshift=-2] {$a$};
\draw[-|,black, very thick] (3.4, 0.1)--(3.6, 0.1);
\draw [thick, decorate, decoration={brace, amplitude=10pt, raise=4pt}] (2, 0.1) to node[midway, above, yshift=12] {$a^{\prime}$} (3.6, 0.1);
\node at (3.4, -0.5) {$P$};
\draw[|-|,black, very thick] (3.2, -0.1)--(3.4, -0.1);
\draw[-|,black, very thick] (3.4, -0.1)--(5, -0.1) node [midway, below, yshift=2] {$b$};
\draw [thick, decorate, decoration={brace, mirror, amplitude=10pt, raise=4pt}] (3.2, -0.1) to node[midway, below, yshift=-12] {$b^{\prime}$} (5, -0.1);
\draw[-|, black, thick] (Li)--(Rf);
\draw[-|, black, thick] (Rf)--(E) node [near end, above] {$R^{\prime}$};
\end{tikzpicture}
\end{center}
\item
\label{merge2}
$a^{\prime}$ and $b^{\prime}$ merge to one generalized fractional power $ab$ in $L^{\prime}R^{\prime}$.
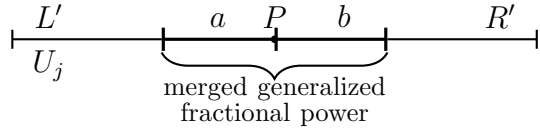
\begin{center}
\begin{tikzpicture}
\coordinate (S) at (0,0);
\coordinate (Li) at (2, 0);
\coordinate (ai) at (3.5, 0);
\coordinate (Rf) at (5, 0);
\coordinate (E) at (7, 0);
\draw[|-, black, thick] (S)--(Li) node [near start, above] {$L^{\prime}$} node[near start, below] {$U_j$};
\draw[|-,black, very thick] (Li)--(ai) node [midway, above] {$a$};
\node at (ai) [above]{$P$};
\node[circle,fill,inner sep=1] at (ai) {};
\draw[|-|,black, very thick] (ai)--(Rf) node [midway, above, xshift=5] {$b$};
\draw[-|, black, thick] (Rf)--(E) node [near end, above] {$R^{\prime}$};
\draw [thick, decorate, decoration={brace, mirror, amplitude=10pt, raise=5pt}] (Li) to (Rf);
\node[text width=3cm, align=center, below, yshift=-10] at (3.5, 0) {\footnotesize{\baselineskip=10ptmerged generalized fractional power\par}};
\end{tikzpicture}
\end{center}
\end{enumerate}

Consider the chart of monomials $U_j = La_jR$ of type~\ref{donot_keep_structure2}. First of all notice that since $a_h$ is a maximal occurrence of a generalized fractional power in $U_h$, the monomial $U_j = La_jR$ has no cancellations. The possible configurations are as follows:
\begin{enumerate}[label=\ref{donot_keep_structure2}.\arabic*]
\item
\label{all_separated1}
$L$ does not have a terminal subword equal to some generalized fractional power and $R$ does not have an initial subword equal to some generalized fractional power. Then all members of the chart of $La_jR$ are separated from $a_j$. In this case no new members are added; hence, the chart of $U_j$ consists of occurrences that lie strictly left to $a_j$ in $L$ and strictly right to $a_j$ in $R$.
\end{enumerate}

Suppose $L = L_1a$, $R = bR_1$, $a$ and $b$ are maximal occurrences of generalized fractional powers in $L$ and $R$, respectively, possibly with $\LM$-measure less than $\tau$. We allow one of $a$ or $b$ to be equal to $1$. As above, denote by $a^{\prime}$ the maximal occurrence of a generalized fractional power that prolongs $a$ in $La_jR$, and denote by $b^{\prime}$ the maximal occurrence of a generalized fractional power that prolongs $b$ in $La_jR$. We have the possibilities similar to the above \ref{strictly_touch2}---\ref{merge2}.
\begin{enumerate}[label=\ref{donot_keep_structure2}.\arabic*]
\setcounter{enumi}{1}
\item
\label{strictly_touch1}
$a^{\prime}$ remains equal to $a$ and $b^{\prime}$ remains equal to $b$ in $U_j$. Hence, $a^{\prime}$ and $b^{\prime}$ are separated in $La_jR$.
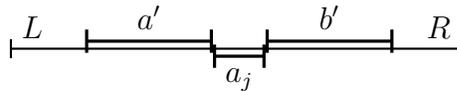
\begin{center}
\begin{tikzpicture}
\draw[|-|, black, thick] (0,0)--(6,0);
\node at (0.3, 0) [above] {$L$};
\node at (5.7, 0) [above] {$R$};
\draw[|-|, black, very thick] (1,0.1)--(2.7,0.1) node[midway, above] {$a^{\prime}$};
\draw[|-|, black, very thick] (2.7,-0.1)--(3.4,-0.1) node[midway, below] {$a_j$};
\draw[|-|, black, very thick] (3.4,0.1)--(5.1,0.1) node[midway, above] {$b^{\prime}$};
\end{tikzpicture}
\end{center}
\item
\label{slightly_enlarge1}
$a^{\prime}$ slightly enlarges from the right side with respect to $a,$ and $b^{\prime}$ slightly enlarges from the left side with respect to $b$. Let us illustrate possible configurations of $a^{\prime}$ and $b^{\prime}$.
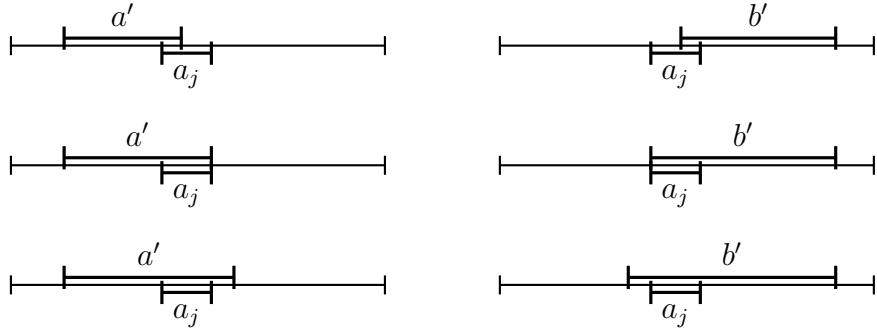
\begin{center}
\begin{tikzpicture}
\draw[|-|, black, thick] (0, 0)--(5, 0);
\draw[|-|, black, very thick] (0.7, 0.1)--(2.3, 0.1) node[midway, above] {$a^{\prime}$};
\draw[|-|, black, very thick] (2, -0.1)--(2.7, -0.1) node[midway, below] {$a_j$};
\draw[|-|, black, thick] (6.5,0)--(11.5,0);
\draw[|-|, black, very thick] (6.5 + 2, -0.1)--(6.5 + 2.7, -0.1) node[midway, below] {$a_j$};
\draw[|-|, black, very thick] (6.5 + 2.4, 0.1)--(11, 0.1) node[midway, above] {$b^{\prime}$};
\end{tikzpicture}

\begin{tikzpicture}
\draw[|-|, black, thick] (0, 0)--(5, 0);
\draw[|-|, black, very thick] (0.7, 0.1)--(2.7, 0.1) node[midway, above] {$a^{\prime}$};
\draw[|-|, black, very thick] (2, -0.1)--(2.7, -0.1) node[midway, below] {$a_j$};
\draw[|-|, black, thick] (6.5,0)--(11.5,0);
\draw[|-|, black, very thick] (6.5 + 2, -0.1)--(6.5 + 2.7, -0.1) node[midway, below] {$a_j$};
\draw[|-|, black, very thick] (6.5 + 2, 0.1)--(11, 0.1) node[midway, above] {$b^{\prime}$};
\end{tikzpicture}

\begin{tikzpicture}
\draw[|-|, black, thick] (0, 0)--(5, 0);
\draw[|-|, black, very thick] (0.7, 0.1)--(3, 0.1) node[midway, above] {$a^{\prime}$};
\draw[|-|, black, very thick] (2, -0.1)--(2.7, -0.1) node[midway, below] {$a_j$};
\draw[|-|, black, thick] (6.5,0)--(11.5,0);
\draw[|-|, black, very thick] (6.5 + 2, -0.1)--(6.5 + 2.7, -0.1) node[midway, below] {$a_j$};
\draw[|-|, black, very thick] (6.5 + 1.7, 0.1)--(11, 0.1) node[midway, above] {$b^{\prime}$};
\end{tikzpicture}
\end{center}

We may obtain a combination of any two configurations for $a^{\prime}$ and for $b^{\prime}$ in $La_jR$; hence, $a^{\prime}$ and $b^{\prime}$ may have any mutual position in $La_jR$ (i.e., remain either separate, or touch at a point, or have an overlap). Let us illustrate one possibility as an example.
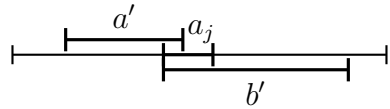
\begin{center}
\begin{tikzpicture}
\draw[|-|, black, thick] (0, 0)--(5, 0);
\draw[|-|, black, very thick] (0.7, 0.2)--(2.3, 0.2) node[midway, above] {$a^{\prime}$};
\draw[|-|, black, very thick] (2, 0)--(2.7, 0) node[near end, above] {$a_j$};
\draw[|-|, black, very thick] (2, -0.2)--(4.5, -0.2) node[midway, below] {$b^{\prime}$};
\end{tikzpicture}
\end{center}

Since $\LM$-measure of a possible overlap of $a^{\prime}$ and $b^{\prime}$ does not exceed $\varepsilon$, the total increase of $a^{\prime}$ and $b^{\prime}$ with respect to $a$ and $b$ does not exceed $2\varepsilon$. Hence, if $\LM$-measure of both $a^{\prime}$ and $b^{\prime}$ increases, $\LM$-measure of each of them increases at most by $\varepsilon$.
\item
\label{slightly_enlarge_precisely_one}
$a^{\prime}$ enlarges with respect to $a$, $b^{\prime}$ remains the same, or $b^{\prime}$ enlarges with respect to $b$, $a^{\prime}$ remains the same (the case when both of them enlarge is treated in \ref{slightly_enlarge1}). As above, $a^{\prime}$ and $b^{\prime}$ may have any mutual position in $La_jR$ (i.e.,either remain separate, or touch at a point, or have an overlap). In this case $\LM$-measure of $a^{\prime}$ or $b^{\prime}$ increases at most by $2\varepsilon$. Let us illustrate one important possibility, the case when $\LM$-measure of $a^{\prime}$ increases precisely by $2\varepsilon$.
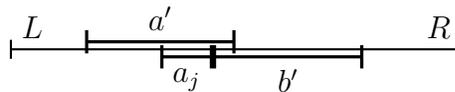
\begin{center}
\begin{tikzpicture}
\node at (0.3, 0) [above] {$L$};
\node at (5.7, 0) [above] {$R$};
\draw[|-|, black, thick] (0,0)--(6,0);
\draw[|-|, black, very thick] (1,0.1)--(3,0.1) node[midway, above] {$a^{\prime}$};
\draw[|-|, black, very thick] (2,-0.1)--(2.7,-0.1) node[midway, below] {$a_j$};
\draw[|-|, black, very thick] (2.7,-0.1)--(4.7,-0.1) node[midway, below] {$b^{\prime}$};
\end{tikzpicture}
\end{center}
\end{enumerate}
The most important case is the following:
\begin{enumerate}[label=\ref{donot_keep_structure2}.\arabic*]
\setcounter{enumi}{4}
\item
\label{merge1}
$a^{\prime}$ and $b^{\prime}$ merge to one generalized fractional power $aa_jb$ in $U_j$.
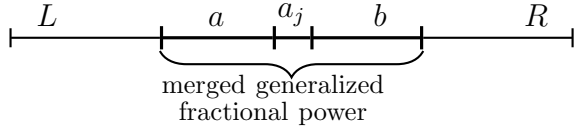
\begin{center}
\begin{tikzpicture}
\coordinate (S) at (0,0);
\coordinate (Li) at (2, 0);
\coordinate (ai) at (3.5, 0);
\coordinate (Ri) at (4, 0);
\coordinate (Rf) at (5.5, 0);
\coordinate (E) at (7.5, 0);
\draw[|-, black, thick] (S)--(Li) node [near start, above] {$L$};
\draw[|-,black, very thick] (Li)--(ai) node [midway, above] {$a$};
\draw[|-, black, very thick] (ai)--(Ri) node[midway, above] {$a_j$};
\draw[|-|,black, very thick] (Ri)--(Rf) node [midway, above, xshift=5] {$b$};
\draw[-|, black, thick] (Rf)--(E) node [near end, above] {$R$};
\draw [thick, decorate, decoration={brace, mirror, amplitude=10pt, raise=5pt}] (Li) to (Rf);
\node[text width=3cm, align=center, below, yshift=-10] at (3.5, 0) {\footnotesize{\baselineskip=10ptmerged generalized fractional power\par}};
\end{tikzpicture}
\end{center}
\end{enumerate}

\begin{remark}
Notice that, using the explicit description of generalized fractional powers~\eqref{path_type1}---\eqref{path_type3} and the explicit description of elementary multi-turns \eqref{mturn1}---\eqref{mturn6}, one can show that the case~\ref{merge1} may take place in at most one monomial of the corresponding sum~$\sum_{\substack{j = 1 \\ j\neq h}}^{k} La_jR$.
\end{remark}

\subsection{Linear dependencies on $\mathbb{Z}_2\Fr$ induced by multi-turns}
\label{linear_dependencies}
First we prove the important property of the supports of elementary multi-turns, that are the expressions~\eqref{mturn1} --- \eqref{mturn6}. Then we will show that all the supports of multi-turns linearly generate the ideal $\Ideal$. This fact plays the key role for the description of the linear structure of the quotient ring $\mathbb{Z}_2\Fr / \Ideal$.

\begin{proposition}[Transversality]
\label{transversality}
Let $Q_1, \ldots , Q_n \in \mathbb{Z}_2\Fr$ be linear combinations of generalized fractional powers of the form~\eqref{mturn1} --- \eqref{mturn6}. Consider $\sum_{l = 1}^n Q_l$. If in this sum, after additively cancelling out the identical monomials, we obtain non-zero, then at least one of the remaining monomials (generalized fractional powers) has $\LM$-measure $\geqslant \tau$.
\end{proposition}
\begin{proof}
Monomials of $Q_l$ can be represented as paths in the corresponding $v$-diagram with the same initial and the same final point. If some initial or some final point lies on a $w$-arc, we fix it on the arc marked by the first power of $w$. Divide the sum $\sum_{l = 1}^n Q_l$ into subsums corresponding to the same initial and final points $(I, F)$ in the $v$-diagram
\begin{equation*}
\sum\limits_{l = 1}^n Q_l = \sum\limits_{(I, F)}\sum\limits_{l = 1}^{n(I, F)} Q_l^{(I, F)}.
\end{equation*}
Let $a$ be a member of $\sum_{l = 1}^n Q_l $ such that $\LM(a) \geqslant \tau$.
Notice that since $\LM(a) \geqslant \tau > \varepsilon$, $a$ has uniquely defined points $I$ and $F$. So, $a$ is a member of $\sum_{l = 1}^{n(I, F)} Q_l^{(I, F)}$ and $a$ is not a member of any $\sum_{l = 1}^{n(I^{\prime}, F^{\prime})} Q_l^{(I^{\prime}, F^{\prime})}$ with $(I^{\prime}, F^{\prime}) \neq (I, F)$. Hence, if $a$ is not additively cancelled out in the sum $\sum_{l = 1}^{n(I, F)} Q_l^{(I, F)}$, then $a$ is not additively cancelled out in the sum $\sum_{l = 1}^n Q_l$. It follows that it is enough to prove Proposition~\ref{transversality} for a sum $\sum_{l = 1}^{n(I, F)} Q_l^{(I, F)}$ with some fixed $(I, F)$.

Assume the contrary. Suppose all members of $\sum_{l = 1}^{n(I, F)} Q_l^{(I, F)}$ after additively cancelling out the identical monomials are of $\LM$-measure less than $\tau$. Let us study possible forms of monomials in this sum. Suppose $b$ is a monomial in this sum, $\LM(b) < \tau$, $I$ and $F$ are the initial and the final points of the path corresponding to $b$. There are the following possible positions of $I$ and $F$:
\begin{enumerate}
\item
\label{small_v_arc}
The points $I$ and $F$ lie on the $v$-arc.
\item
\label{small_w_arc}
The points $I$ and $F$ lie on a $w$-arc.
\item
\label{small_v_w_arc}
The point $I$ lies on the $v$-arc, the point $F$ lies on a $w$-arc.
\item
\label{small_w_v_arc}
The point $I$ lies on a $w$-arc, the point $F$ lies on the $v$-arc.
\end{enumerate}
Since $\LM(b) < \tau$, the smallest path between $I$ and $F$ is necessarily of $\LM$-measure less than $\tau$.

Consider case~\ref{small_v_arc}. Let us denote the smallest path from $I$ to $F$ by $b^{\prime}$, the smallest path from $I$ to $O$ by $b_1$, and the smallest path from $O$ to $F$ by $b_2$. Recall our notion that $v = v_iv_mv_f$. Then there are the following possibilities:
\begin{enumerate}[label=\ref{small_v_arc}.\arabic*]
\item
\label{small_v_arc1}
$b^{\prime}$ contains the point $O$;
\begin{center}
\begin{tikzpicture}
\begin{scope}
\coordinate (O) at ($(0,0)+(90:1.5 and 2.2)$);
\coordinate (Si) at ($(0,0)+(60:1.5 and 2.2)$);
\coordinate (Se) at ($(0,0)+(130:1.5 and 2.2)$);

\coordinate (vi) at (0.39, 2.125);
\coordinate (vm) at (0.13, -2.19);
\coordinate (vf) at (-0.513, 2.06);

\draw[black, thick, arrow=0.6] (O) arc (90:60:1.5 and 2.2);
\node[right, yshift=3] at (vi) {$v_i$};
\draw[black, thick, arrow=0.6] (Si) arc (60:-360+130:1.5 and 2.2);
\node[below] at (vm) {$v_m$};
\draw[black, thick, arrow=0.5] (Se) arc (130:90:1.5 and 2.2);
\node[left, yshift=3] at (vf) {$v_f$};

\node[circle,fill,inner sep=1.2] at (O) {};
\path let \p1 = (O) in node at (\x1, \y1+8) {$O$};

\node[circle,fill,inner sep=1.2] at (Si) {};
\node[right] at (Si) {$I$};
\node[circle,fill,inner sep=1.2] at (Se) {};
\node[left] at (Se) {$F$};

\draw[black, thick, reversearrow=0.3] (0, 2.2+0.8) ellipse (0.6 and 0.8) node at (0.3, 2.2+1.8) {$w^k$};
\draw[black, thick, reversearrow=0.3] (0, 2.2+0.7*0.6) ellipse (0.5*0.6 and 0.7*0.6);

\draw[black, thick, reversearrow=0.7] (0, 2.2-0.8) ellipse (0.6 and 0.8) node at (0.3, 2.2-1.7) {$w^{-k}$};
\draw[black, thick, reversearrow=0.7] (0, 2.2-0.7*0.6) ellipse (0.5*0.6 and 0.7*0.6);
\node at (0, -3.5) {$\begin{aligned}
&b_1 = v_i^{-1}, b_2 = v_f^{-1},\\
&b^{\prime} = v_i^{-1}v_f^{-1};
\end{aligned}$};
\end{scope}
\begin{scope}
\coordinate (O) at ($(5,0)+(90:1.5 and 2.2)$);
\coordinate (Si) at ($(5,0)+(60:1.5 and 2.2)$);
\coordinate (Se) at ($(5,0)+(130:1.5 and 2.2)$);

\coordinate (vi) at (5 + 0.39, 2.125);
\coordinate (vm) at (5 + 0.13, -2.19);
\coordinate (vf) at (5 + -0.513, 2.06);

\draw[black, thick, arrow=0.6] (O) arc (90:60:1.5 and 2.2);
\node[right, yshift=3] at (vi) {$v_i$};
\draw[black, thick, arrow=0.6] (Si) arc (60:-360+130:1.5 and 2.2);
\node[below] at (vm) {$v_m$};
\draw[black, thick, arrow=0.4] (Se) arc (130:90:1.5 and 2.2);
\node[left, yshift=3] at (vf) {$v_f$};

\node[circle,fill,inner sep=1.2] at (O) {};
\path let \p1 = (O) in node at (\x1, \y1+8) {$O$};

\node[circle,fill,inner sep=1.2] at (Si) {};
\node[right] at (Si) {$F$};
\node[circle,fill,inner sep=1.2] at (Se) {};
\node[left] at (Se) {$I$};

\draw[black, thick, reversearrow=0.3] (5, 2.2+0.8) ellipse (0.6 and 0.8) node at (5.3, 2.2+1.8) {$w^k$};
\draw[black, thick, reversearrow=0.3] (5, 2.2+0.7*0.6) ellipse (0.5*0.6 and 0.7*0.6);

\draw[black, thick, reversearrow=0.7] (5, 2.2-0.8) ellipse (0.6 and 0.8) node at (5.3, 2.2-1.7) {$w^{-k}$};
\draw[black, thick, reversearrow=0.7] (5, 2.2-0.7*0.6) ellipse (0.5*0.6 and 0.7*0.6);
\node at (5, -3.5) {$\begin{aligned}
&b_1 = v_f, b_2 = v_i,\\
&b^{\prime} = v_fv_i.
\end{aligned}$};
\end{scope}
\end{tikzpicture}
\end{center}
\item
\label{small_v_arc2}
$b^{\prime}$ does not contain the point $O$ and $\LM(b_1) + \LM(b_2) < \tau$
\begin{center}
\begin{tikzpicture}
\begin{scope}
\coordinate (O) at ($(0,0)+(90:1.5 and 2.2)$);
\coordinate (Si) at ($(0,0)+(55:1.5 and 2.2)$);
\coordinate (Se) at ($(0,0)+(40:1.5 and 2.2)$);

\coordinate (vi) at (0.45, 2.1);
\coordinate (vm) at (1.013, 1.63);
\coordinate (vf) at (-0.63, -1.99);

\draw[black, thick, arrow=0.7] (O) arc (90:55:1.5 and 2.2);
\node[right, yshift=3] at (vi) {$v_i$};
\draw[black, thick, arrow=0.7] (Si) arc (55:40:1.5 and 2.2);
\node[below, xshift=-2] at (vm) {$v_m$};
\draw[black, thick, arrow=0.5] (Se) arc (40: -360 + 90:1.5 and 2.2);
\node[below] at (vf) {$v_f$};

\node[circle,fill,inner sep=1.2] at (O) {};
\path let \p1 = (O) in node at (\x1, \y1+8) {$O$};

\node[circle,fill,inner sep=1.2] at (Si) {};
\node[right, yshift=3] at (Si) {$I$};
\node[circle,fill,inner sep=1.2] at (Se) {};
\node[right] at (Se) {$F$};

\draw[black, thick, reversearrow=0.3] (0, 2.2+0.8) ellipse (0.6 and 0.8) node at (0.3, 2.2+1.8) {$w^k$};
\draw[black, thick, reversearrow=0.3] (0, 2.2+0.7*0.6) ellipse (0.5*0.6 and 0.7*0.6);

\draw[black, thick, reversearrow=0.7] (0, 2.2-0.8) ellipse (0.6 and 0.8) node at (0.3, 2.2-1.7) {$w^{-k}$};
\draw[black, thick, reversearrow=0.7] (0, 2.2-0.7*0.6) ellipse (0.5*0.6 and 0.7*0.6);
\node at (0, -3.5) {$\begin{aligned}
&b_1 = v_i^{-1}, b_2 = v_iv_m,\\
&b^{\prime} = v_m;
\end{aligned}$};
\end{scope}
\begin{scope}
\coordinate (O) at ($(5,0)+(90:1.5 and 2.2)$);
\coordinate (Si) at ($(5,0)+(55:1.5 and 2.2)$);
\coordinate (Se) at ($(5,0)+(40:1.5 and 2.2)$);

\coordinate (vi) at (5 + 0.45, 2.1);
\coordinate (vm) at (5 + 1.013, 1.63);
\coordinate (vf) at (5 + -0.63, -1.99);

\draw[black, thick, arrow=0.7] (O) arc (90:55:1.5 and 2.2);
\node[right, yshift=3] at (vi) {$v_i$};
\draw[black, thick, arrow=0.7] (Si) arc (55:40:1.5 and 2.2);
\node[below, xshift=-2] at (vm) {$v_m$};
\draw[black, thick, arrow=0.5] (Se) arc (40:-360 + 90:1.5 and 2.2);
\node[below] at (vf) {$v_f$};

\node[circle,fill,inner sep=1.2] at (O) {};
\path let \p1 = (O) in node at (\x1, \y1+8) {$O$};

\node[circle,fill,inner sep=1.2] at (Si) {};
\node[right, yshift=3] at (Si) {$F$};
\node[circle,fill,inner sep=1.2] at (Se) {};
\node[right] at (Se) {$I$};

\draw[black, thick, reversearrow=0.3] (5, 2.2+0.8) ellipse (0.6 and 0.8) node at (5.3, 2.2+1.8) {$w^k$};
\draw[black, thick, reversearrow=0.3] (5, 2.2+0.7*0.6) ellipse (0.5*0.6 and 0.7*0.6);

\draw[black, thick, reversearrow=0.7] (5, 2.2-0.8) ellipse (0.6 and 0.8) node at (5.3, 2.2-1.7) {$w^{-k}$};
\draw[black, thick, reversearrow=0.7] (5, 2.2-0.7*0.6) ellipse (0.5*0.6 and 0.7*0.6);
\node at (5, -3.5) {$\begin{aligned}
&b_1 = v_m^{-1}v_i^{-1}, b_2 = v_i,\\
&b^{\prime} = v_m^{-1};
\end{aligned}$};
\end{scope}
\end{tikzpicture}
\begin{tikzpicture}
\begin{scope}
\coordinate (O) at ($(0,0)+(90:1.5 and 2.2)$);
\coordinate (Si) at ($(0,0)+(-210:1.5 and 2.2)$);
\coordinate (Se) at ($(0,0)+(-225:1.5 and 2.2)$);

\coordinate (vi) at (0.75, -1.9);
\coordinate (vm) at (-1.19, 1.34);
\coordinate (vf) at (-0.58, 2.0325);

\draw[black, thick, arrow=0.6] (O) arc (90:-210:1.5 and 2.2);
\node[below] at (vi) {$v_i$};
\draw[black, thick, arrow=0.7] (Si) arc (-210:-225:1.5 and 2.2);
\node[left] at (vm) {$v_m$};
\draw[black, thick, arrow=0.5] (Se) arc (-225: -270:1.5 and 2.2);
\node[left] at (vf) {$v_f$};

\node[circle,fill,inner sep=1.2] at (O) {};
\path let \p1 = (O) in node at (\x1, \y1+8) {$O$};

\node[circle,fill,inner sep=1.2] at (Si) {};
\node[right, yshift=-3] at (Si) {$I$};
\node[circle,fill,inner sep=1.2] at (Se) {};
\node[right] at (Se) {$F$};

\draw[black, thick, reversearrow=0.3] (0, 2.2+0.8) ellipse (0.6 and 0.8) node at (0.3, 2.2+1.8) {$w^k$};
\draw[black, thick, reversearrow=0.3] (0, 2.2+0.7*0.6) ellipse (0.5*0.6 and 0.7*0.6);

\draw[black, thick, reversearrow=0.7] (0, 2.2-0.8) ellipse (0.6 and 0.8) node at (0.3, 2.2-1.7) {$w^{-k}$};
\draw[black, thick, reversearrow=0.7] (0, 2.2-0.7*0.6) ellipse (0.5*0.6 and 0.7*0.6);
\node at (0, -3.5) {$\begin{aligned}
&b_1 = v_mv_f, b_2 = v_f^{-1},\\
&b^{\prime} = v_m;
\end{aligned}$};
\end{scope}
\begin{scope}
\coordinate (O) at ($(5,0)+(90:1.5 and 2.2)$);
\coordinate (Si) at ($(5,0)+(-210:1.5 and 2.2)$);
\coordinate (Se) at ($(5,0)+(-225:1.5 and 2.2)$);

\coordinate (vi) at (5 + 0.75, -1.9);
\coordinate (vm) at (5 + -1.19, 1.34);
\coordinate (vf) at (5 + -0.58, 2.0325);

\draw[black, thick, arrow=0.7] (O) arc (90:-210:1.5 and 2.2);
\node[below] at (vi) {$v_i$};
\draw[black, thick, arrow=0.7] (Si) arc (-210:-225:1.5 and 2.2);
\node[left] at (vm) {$v_m$};
\draw[black, thick, arrow=0.5] (Se) arc (-225: -270:1.5 and 2.2);
\node[left] at (vf) {$v_f$};

\node[circle,fill,inner sep=1.2] at (O) {};
\path let \p1 = (O) in node at (\x1, \y1+8) {$O$};

\node[circle,fill,inner sep=1.2] at (Si) {};
\node[right, yshift=-3] at (Si) {$F$};
\node[circle,fill,inner sep=1.2] at (Se) {};
\node[right] at (Se) {$I$};

\draw[black, thick, reversearrow=0.3] (5, 2.2+0.8) ellipse (0.6 and 0.8) node at (5.3, 2.2+1.8) {$w^k$};
\draw[black, thick, reversearrow=0.3] (5, 2.2+0.7*0.6) ellipse (0.5*0.6 and 0.7*0.6);

\draw[black, thick, reversearrow=0.7] (5, 2.2-0.8) ellipse (0.6 and 0.8) node at (5.3, 2.2-1.7) {$w^{-k}$};
\draw[black, thick, reversearrow=0.7] (5, 2.2-0.7*0.6) ellipse (0.5*0.6 and 0.7*0.6);
\node at (5, -3.5) {$\begin{aligned}
&b_1 = v_f, b_2 = v_f^{-1}v_m^{-1},\\
&b^{\prime} = v_m^{-1};
\end{aligned}$};
\end{scope}
\end{tikzpicture}
\end{center}
\item
\label{small_v_arc3}
$b^{\prime}$ does not contain the point $O$ and $\LM(b_1) + \LM(b_2) \geqslant \tau$.
\begin{center}
\begin{tikzpicture}
\begin{scope}
\coordinate (O) at ($(0,0)+(90:1.5 and 2.2)$);
\coordinate (Si) at ($(0,0)+(-110:1.5 and 2.2)$);
\coordinate (Se) at ($(0,0)+(-140:1.5 and 2.2)$);

\coordinate (vi) at (1.477, -0.382);
\coordinate (vm) at (-0.86, -1.802);
\coordinate (vf) at (-1.36, 0.93);

\draw[black, thick, arrow=0.6] (O) arc (90:-110:1.5 and 2.2);
\node[right] at (vi) {$v_i$};
\draw[black, thick, arrow=0.7] (Si) arc (-110:-140:1.5 and 2.2);
\node[left] at (vm) {$v_m$};
\draw[black, thick, arrow=0.5] (Se) arc (-140: -270:1.5 and 2.2);
\node[left] at (vf) {$v_f$};

\node[circle,fill,inner sep=1.2] at (O) {};
\path let \p1 = (O) in node at (\x1, \y1+8) {$O$};

\node[circle,fill,inner sep=1.2] at (Si) {};
\node[above] at (Si) {$I$};
\node[circle,fill,inner sep=1.2] at (Se) {};
\node[right] at (Se) {$F$};

\draw[black, thick, reversearrow=0.3] (0, 2.2+0.8) ellipse (0.6 and 0.8) node at (0.3, 2.2+1.8) {$w^k$};
\draw[black, thick, reversearrow=0.3] (0, 2.2+0.7*0.6) ellipse (0.5*0.6 and 0.7*0.6);

\draw[black, thick, reversearrow=0.7] (0, 2.2-0.8) ellipse (0.6 and 0.8) node at (0.3, 2.2-1.7) {$w^{-k}$};
\draw[black, thick, reversearrow=0.7] (0, 2.2-0.7*0.6) ellipse (0.5*0.6 and 0.7*0.6);
\node at (0, -3) {$\begin{aligned}
&b^{\prime} = v_m;
\end{aligned}$};
\end{scope}
\begin{scope}
\coordinate (O) at ($(5,0)+(90:1.5 and 2.2)$);
\coordinate (Si) at ($(5,0)+(-110:1.5 and 2.2)$);
\coordinate (Se) at ($(5,0)+(-140:1.5 and 2.2)$);

\coordinate (vi) at (5 + 1.477, -0.382);
\coordinate (vm) at (5 + -0.86, -1.802);
\coordinate (vf) at (5 + -1.36, 0.93);

\draw[black, thick, arrow=0.7] (O) arc (90:-110:1.5 and 2.2);
\node[right] at (vi) {$v_i$};
\draw[black, thick, arrow=0.7] (Si) arc (-110:-140:1.5 and 2.2);
\node[left] at (vm) {$v_m$};
\draw[black, thick, arrow=0.5] (Se) arc (-140: -270:1.5 and 2.2);
\node[left] at (vf) {$v_f$};

\node[circle,fill,inner sep=1.2] at (O) {};
\path let \p1 = (O) in node at (\x1, \y1+8) {$O$};

\node[circle,fill,inner sep=1.2] at (Si) {};
\node[above] at (Si) {$F$};
\node[circle,fill,inner sep=1.2] at (Se) {};
\node[right] at (Se) {$I$};

\draw[black, thick, reversearrow=0.3] (5, 2.2+0.8) ellipse (0.6 and 0.8) node at (5.3, 2.2+1.8) {$w^k$};
\draw[black, thick, reversearrow=0.3] (5, 2.2+0.7*0.6) ellipse (0.5*0.6 and 0.7*0.6);

\draw[black, thick, reversearrow=0.7] (5, 2.2-0.8) ellipse (0.6 and 0.8) node at (5.3, 2.2-1.7) {$w^{-k}$};
\draw[black, thick, reversearrow=0.7] (5, 2.2-0.7*0.6) ellipse (0.5*0.6 and 0.7*0.6);
\node at (5, -3) {$\begin{aligned}
&b^{\prime} = v_m^{-1};
\end{aligned}$};
\end{scope}
\end{tikzpicture}
\end{center}
\end{enumerate}

In case~\ref{small_v_arc3} there is only one path of $\LM$-measure $< \tau$ ($b = b^{\prime}$). Hence, if the sum $\sum_{l = 1}^{n(I, F)} Q_l^{(I, F)}$ contains several different monomials of $\LM$-measure $< \tau$, then configuration~\ref{small_v_arc3} is not possible.

Consider configurations~\ref{small_v_arc1} and~\ref{small_v_arc2}. Since all members of the sum $\sum_{l = 1}^{n(I, F)} Q_l^{(I, F)}$ correspond to paths in the $v$-diagram with the initial point $I$ and the final point $F$ and of $\LM$-measure less than $\tau$, the members are of the form
\begin{equation*}
b_1 w^{k_j} b_2.
\end{equation*}
Notice that the monomial with $k_j = 0$ may have cancellations and equals to $b^{\prime}$ after cancellations.

Let us return to the consideration of case~\ref{small_v_arc} in general. Recall that all $Q_l^{(I, F)}$ are of one of the types \eqref{mturn1} --- \eqref{mturn6}. So, if we multiply $\sum_{l = 1}^{n(I, F)} Q_l^{(I, F)}$ by $b_1^{-1}$ from the left side and by $b_2^{-1}$ from the right side, we obtain the Laurent polynomial $P(v, w) = b_1^{-1}\sum_{j = 1}^n Q_j b_2^{-1}$ such that $P(v, w)$ satisfies~\eqref{vanish_in_field}. But, on the other hand, in cases~\ref{small_v_arc1} and~\ref{small_v_arc2}
\begin{equation*}
P(v, w) = b_1^{-1}\sum\limits_{l = 1}^{n(I, F)} Q_l^{(I, F)} b_2^{-1} = \sum\limits_{j = 1}^s w^{k_l}.
\end{equation*}
In case~\ref{small_v_arc3}, the polynomial consists of one monomial
\begin{equation*}
P(v, w) = b_1^{-1}\sum\limits_{l = 1}^{n(I, F)} Q_l^{(I, F)} b_2^{-1} = b_1^{-1}b^{\prime} b_2^{-1}.
\end{equation*}
Evidently, this polynomials does not satisfy~\eqref{vanish_in_field}. This contradiction completes the proof. The remaining cases are considered in the same way.
\end{proof}

Let us define a subspace of $\mathbb{Z}_2\Fr$ of linear dependencies induced by multi-turns. For every monomial from $\Fr$, we do all possible multi-turns of all members of the chart. As a result, we obtain a set of expressions
\begin{multline}
\label{linear_dep_members}
\mathcal{T} = \Bigg\lbrace \sum\limits_{j = 1}^k U_j \mid U_h\in \Fr, U_h = La_hR, \textit{ where } a_h \textit{ is a member of the chart of } U_h, \\
 U_h \mapsto \sum\limits_{\substack{j = 1 \\ j\neq h}}^{k} U_j \textit{ is a multi-turn coming from an elementary multi-turn } a_h \mapsto \sum\limits_{\substack{j = 1 \\ j\neq h}}^{k} a_j \Bigg\rbrace.
\end{multline}
That is, $\mathcal{T}$ consists of the supports of all the multi-turns of members of the chart. We denote by $\langle\mathcal{T}\rangle$ the linear span of this set.

\begin{proposition}
\label{the_ideal_characterisation}
The linear subspace $\langle\mathcal{T} \rangle \subseteq \mathbb{Z}_2\Fr$ is equal to the ideal $\Ideal = \langle 1 + v + vw \rangle$.
\end{proposition}
\begin{proof}
First let us show that $\langle \mathcal{T} \rangle$ is an ideal in $\mathbb{Z}_2\Fr$. Assume that $T$ is the support of a multi-turn of $M_h = La_hR$ that comes from an elementary multi-turn $a_h \mapsto \sum_{\substack{j = 1 \\ j\neq h}}^{k}a_j$, that is, $T = \sum_{j = 1}^k U_j$, $U_j = La_jR$. We have to check that if $Z$ is a monomial, then $ZT \in \langle \mathcal{T} \rangle$ and $TZ \in \langle \mathcal{T} \rangle$. Clearly, it is sufficient to check this property only for a monomial $Z$ that consists of only one letter $z$. We will show an even stronger property, namely that $z(\sum_{j = 1}^k U_j) \in \mathcal{T}$, $(\sum_{j = 1}^k U_j)z \in \mathcal{T}$.

Consider $z(\sum_{j = 1}^k U_j)$, $U_j = La_jR$. Suppose $\LM(a_h) < \tau$. From Proposition~\ref{transversality}, it follows that in the sum $\sum_{j = 1}^{k} a_j$ there exists a monomial $a_{h^{\prime}}$ such that $\LM(a_{h^{\prime}}) \geqslant \tau$. Then, by definition, $a_{h^{\prime}}$ is a member of the chart of $La_{h^{\prime}}R$. Hence, the sum $\sum_{j = 1}^{k} U_j = \sum_{j = 1}^{k} La_jR$ can be considered as the support of the multi-turn $M_{h^{\prime}} = La_{h^{\prime}}R \mapsto \sum_{\substack{j = 1 \\ j\neq h^{\prime}}}^{k} La_jR$. So, in the sequel we can assume that $\LM(a_h) \geqslant \tau$.

First consider the case when $L$ is not empty. Since $\LM(a_h) \geqslant \tau$, $a_h$ is a member of the chart of $(zL)a_hR$ both when $z$ does not cancel out with $L$, or when it does. Since $z(La_jR) = (zL)a_jR$, $\sum_{j = 1}^k zU_j = \sum_{j = 1}^k z(La_jR) = \sum_{j = 1}^k (zL)a_jR$. Then, clearly, $\sum_{j = 1}^k zU_j$ is the support of the multi-turn $(zL)a_hR \mapsto \sum_{\substack{j = 1 \\ j\neq h}}^{k} (zL)a_jR$. Hence, $\sum_{j = 1}^k zU_j \in \mathcal{T}$.

Consider the case when $L$ is empty, that is, $U_j = a_jR$. If $z$ does not cancel with $a_h$ and $z$ does not prolong $a_h$ to a generalized fractional power from the left, then $a_h$ is a maximal occurrence of a generalized fractional power in $zU_h = za_hR$. Since $\LM(a_h) \geqslant \tau$, $a_h$ is a member of the chart of $za_hR$. Hence, $za_hR \mapsto \sum_{\substack{j = 1 \\ j\neq h}}^{k} za_jR$ is a multi-turn of a member of the chart and its support $\sum_{j = 1}^{k} za_jR = \sum_{j = 1}^{k} zU_j$ belongs to~$\mathcal{T}$.

Assume that $z$ does not cancel with $a_h$ and $z$ prolongs $a_h$ to a generalized fractional power from the left. Then for every $j \neq h$ either $z$ does not cancel with $a_j$ (and since $\LM(a_h) \geqslant \tau > \varepsilon$, $z$ prolongs it to a generalized fractional power), or $z$ cancels with $a_j$. Hence, $za_h \mapsto \sum_{\substack{j = 1 \\ j\neq h}}^{k} za_j$ is an elementary multi-turn (after the cancellations from the right hand side). Clearly, $za_h$ is a maximal occurrence of a generalized fractional power in $zU_h = za_hR$. Since $\LM(za_h) \geqslant \LM(a_h) \geqslant \tau$, $za_h$ is a member of the chart of $za_hR$. Therefore, $za_hR \mapsto \sum_{\substack{j\neq h \\ j = 1}}^{k} za_jR$ is a multi-turn of a member of the chart and its support $\sum_{j = 1}^{k} za_jR = \sum_{j = 1}^{k} zU_j$ belongs to~$\mathcal{T}$.

Assume that $z$ cancels with $a_h$. Then for every $j \neq h$ either $z$ does not cancel with $a_j$ (and since $\LM(a_h) \geqslant \tau > \varepsilon$, $z$ prolongs it to a generalized fractional power), or $z$ cancels with $a_j$. Hence, $za_h \mapsto \sum_{\substack{j = 1 \\ j \neq h}}^{k} za_j$ is an elementary multi-turn (after the cancellations). Since $\LM(za_h) \leqslant \LM(a_h)$, we distinguish the following two possibilities:
\begin{enumerate}
\item
$\LM(za_h) \geqslant \tau$. Then $za_h$ is a member of the chart of $za_hR$. Therefore, $za_hR \mapsto \sum_{\substack{j = 1 \\ j\neq h}}^{k} za_jR$ is a multi-turn of a virtual member of the chart and its support $\sum_{j = 1}^{k} za_jR = \sum_{j = 1}^{k} zU_j$ belongs to $\mathcal{T}$.
\item
$\LM(za_h) \geqslant \tau$. Then $za_h$ is not a member of the chart of $za_hR$. Then we argue as at the beginning of the proof. From Proposition~\ref{transversality} it follows that in the sum $\sum_{j = 1}^{k} za_j$ there exists a monomial $za_{h^{\prime}}$ such that $\LM(za_{h^{\prime}}) \geqslant \tau$. So, $za_{h^{\prime}}$ is a member of the chart of $za_{h^{\prime}}R$. Hence, the sum $\sum_{j = 1}^{k} za_jR$ can be considered as the support of the multi-turn $za_{h^{\prime}}R \mapsto \sum_{\substack{j\neq h^{\prime} \\ j = 1}}^{k} za_jR$. Thus, the sum $\sum_{j = 1}^{k} zU_j = \sum_{j = 1}^{k} za_jR$ belongs to~$\mathcal{T}$.
\end{enumerate}

Summarising all of the above, we obtain $zT = z(\sum_{j = 1}^k) \in \mathcal{T}$. Clearly, for the same reason we obtain $Tz = (\sum_{j = 1}^k)z \in \mathcal{T}$. Hence, $\langle \mathcal{T}\rangle$ is an ideal in $\mathbb{Z}_2\Fr$.

Evidently, every $\sum_{j = 1}^{k} a_j$ of the form~\eqref{mturn1}--\eqref{mturn6} belongs to $\Ideal$. Hence, $\sum_{j = 1}^{k} La_jR$ belongs to $\Ideal$ for every $\sum_{j = 1}^{k} a_j$ of the form~\eqref{mturn1}--\eqref{mturn6}. So, $\mathcal{T} \subseteq \Ideal$ and $\langle \mathcal{T} \rangle \subseteq \Ideal$. Since $vw + v + 1$ is the support of the multi-turn $vw \mapsto v + 1$ and $\LM(vw) = 1 \geqslant \tau$, $vw + v + 1 \in \mathcal{T}$. Therefore, since $\langle\mathcal{T}\rangle$ is an ideal in $\mathbb{Z}_2\Fr$, we obtain $\Ideal \subseteq \langle\mathcal{T}\rangle$. Thus, $\Ideal =\langle\mathcal{T}\rangle$.
\end{proof}

\subsection{Virtual members of the chart}
\label{virtual_members_section}
Recall that the chart of a word consists of maximal occurrences of generalized fractional powers with $\LM$-measure $\geqslant \tau$ ($\tau$ is our threshold). When we perform multi-turns, the measure of the occurrences may increase on the right or decrease on the right by at most $\varepsilon$ in the resulting monomials of type~\ref{keep_structure}. Similarly, it may increase on the left by at most $\varepsilon$ or decrease on the left by at most $\varepsilon$. It follows that occurrences with measure $\geqslant \tau$ (above the threshold) may turn into occurrences $< \tau$ (below the threshold) and vice versa. Therefore we need to modify our notion of a member of the chart to a notion with more stable properties with respect to multi-turns. We call such a notion \emph{a virtural member of the chart}. In this section, we define this notion and study its properties.

Namely, we may have the following effect. As above, let $U_h$ be a monomial, $a_h$ be a member of its chart, and $U_h = La_hR$. Let $U_h \mapsto \sum_{\substack{j = 1 \\ j\neq h}}^{k}U_j$ be a multi-turn of $a_h$ that comes from an elementary multi-turn $a_h \mapsto \sum_{\substack{j = 1 \\ j\neq h}}^{k}a_j$, $U_j = La_jR$. Assume $b_h$ is a maximal occurrence of a generalized fractional power in $U_h$ different from $a_h$, $\LM(b_h) < \tau$, that is, $b_h$ is not counted as a member of the chart of $U_h$. According to the previous section, the element $b_h$ may be prolonged in $U_j$ and the $\LM$-measure of the corresponding prolonged element $b_j$ in $U_j$ may increase and become $\geqslant \tau$. So, $b_j$ may become a member of the chart of $U_j$ and the number of members of the chart of $U_j$ may become greater than the number of members of the chart of $U_h$. We may obtain this effect for both neighbours of $a_h$ simultaneously.
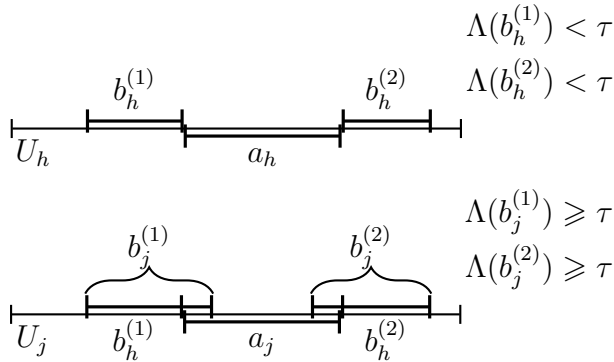
\begin{center}
\begin{tikzpicture}
\node at (0.3, 0) [below] {$U_h$};
\draw[|-|, black, thick] (0,0)--(6,0);
\draw[|-|, black, very thick] (1,0.1)--(2.3,0.1) node[midway, above] {$b^{(1)}_h$};
\draw[|-|, black, very thick] (2.3,-0.1)--(4.4,-0.1) node[midway, below] {$a_h$};
\draw[|-|, black, very thick] (4.4,0.1)--(5.6,0.1) node[midway, above] {$b^{(2)}_h$};
\node at (7, 1) {$\begin{aligned} &\LM(b^{(1)}_h) < \tau \\ &\LM(b^{(2)}_h) < \tau \end{aligned}$};
\end{tikzpicture}

\begin{tikzpicture}
\node at (0.3, 0) [below] {$U_j$};
\draw[|-|, black, thick] (0,0)--(6,0);
\draw[|-|, black, very thick] (1,0.1)--(2.3,0.1) node[midway, below] {$b^{(1)}_h$};
\draw[-|, black, very thick] (2.3,0.1)--(2.7,0.1);
\draw [thick, decorate, decoration={brace, amplitude=10pt, raise=5pt}] (1,0.1) to node[midway, above, yshift=10] {$b^{(1)}_j$} (2.7,0.1);
\draw[|-|, black, very thick] (2.3,-0.1)--(4.4,-0.1) node[midway, below] {$a_j$};
\draw[|-, black, very thick] (4,0.1)--(4.4,0.1);
\draw[|-|, black, very thick] (4.4,0.1)--(5.6,0.1) node[midway, below] {$b^{(2)}_h$};
\draw [thick, decorate, decoration={brace, amplitude=10pt, raise=5pt}] (4,0.1) to node[midway, above, yshift=10] {$b^{(2)}_j$} (5.6,0.1);
\node at (7, 1) {$\begin{aligned} &\LM(b^{(1)}_j) \geqslant \tau \\ &\LM(b^{(2)}_j) \geqslant \tau \end{aligned}$};
\end{tikzpicture}
\end{center}
In this case, the number of members of the chart of $U_j$ become greater than the number of members of the chart of $U_h$ even if $\LM(a_j) < \tau$.

Assume that $U_j = La_jR$ is a resulting monomial of a multi-turn such that, roughly speaking, $a_j$ is of small $\LM$-measure (for example $\LM(a_j) < \tau$, so $a_j$ is not counted as a member of the chart). In Section~\ref{structure_calc}, we will use an inductive argument for such resulting monomials of multi-turns. It looks natural to use induction by the number of members of the chart. But the example described above shows that the number of members of the chart may increase in $U_j$; hence, it is not an appropriate parameter for the induction. In order to prove that the induction is nevertheless finite, we refine a notion of a member of the chart in this section. After that we introduce a function that guarantees finiteness of the inductive process (see Corollary~\ref{decreasing_property}).

In this section, we consider the set $\mathcal{M}(U_h)$ of all maximal occurrences of generalized fractional powers in $U_h$, regardless of their $\LM$-measure, such that they are not properly contained in other occurrences of generalized fractional powers in $U_h$. There are two types of such occurrences: occurrences that are not fully covered by other occurrences from $\mathcal{M}(U_h)$ and occurrences that are fully covered by other occurrences from $\mathcal{M}(U_h)$. Denote the first set by $\mathcal{M}^{\nfc}(U_h)$ and the second set by $\mathcal{M}^{\fc}(U_h)$. Clearly, all maximal occurrences of a generalized fractional powers in $U_h$ of $\LM$-measure greater than $\varepsilon$ are contained in the set $\mathcal{M}(U_h)$. It is also clear that elements of $\mathcal{M}^{\fc}(U_h)$ are of $\LM$-measure not greater than $2\varepsilon$.

For any maximal occurrence of a generalized fractional power in $U_h$ we define \emph{a set of its images} in a resulting monomial of a multi-turn.
\begin{definition}
\label{corresponding_occurence}
Let $U_h = La_hR$ be a monomial, and $a_h$ be a maximal occurrence of a generalized fractional power such that $\LM(a_h) \geqslant \tau - 2\varepsilon$. Let $a_j$ and $a_h$ be incident monomials. Assume $b_h$ is a maximal occurrence of a generalized fractional power in $U_h$ that is different from $a_h$ and is not properly contained in $a_h$. Since $\LM(a_h) \geqslant \tau - 2\varepsilon > \varepsilon$, $a_h$ can not be properly contained in any other occurrence, particularly in $b_h$. To be precise, assume that the end of $b_h$ lies strictly left to the end of $a_h$. We give the definition in the cases $a_j \neq 1$ and $a_j = 1$ separately.
\begin{enumerate}
\item
Assume $a_j \neq 1$, $U_j = La_jR$. Then we call an element of $\mathcal{M}(U_j)$ that contains $a_j$ we call \emph{an image of $a_h$ in $U_j$}. We call the set of all these elements \emph{the set of images of $a_h$ in $U_j$}.

Denote by $b_h^{\prime}$ a subword of $b_h$ that is an intersection of $b_h$ and $L$. We have $L = L_1^{\prime}b_h^{\prime}L_2^{\prime}$, where $L_2^{\prime}$ is possibly equal to $1$. Hence,
\begin{align*}
&U_h = La_hR = L_1^{\prime}b_h^{\prime}L_2^{\prime}a_hR,\\
&U_j = La_jR = L_1^{\prime}b_h^{\prime}L_2^{\prime}a_jR.
\end{align*}
We call an element of $\mathcal{M}(U_j)$ that contains $b_h^{\prime}$ \emph{an image of $b_h$ in $U_j$}. We call the set of all these elements \emph{the set of images of $b_h$ in $U_j$}.

\item
Assume $a_j = 1$. Then we say that \emph{the set of images of $a_h$ in $U_j$ is empty}.

Let $L = L^{\prime}C$, $R = C^{-1}R^{\prime}$ and $L^{\prime}R^{\prime}$ has no further cancellations. If there is no non-empty subword of $b_h$ that is contained in $L^{\prime}$, we say that \emph{the set of images of $b_h$ in $U_j$ is empty}. Otherwise, let $b_h^{\prime}$ be a subword of $b_h$ that is an intersection of $b_h$ and $L^{\prime}$. We have $L^{\prime} = L_1^{\prime}b_h^{\prime}L_2^{\prime}$, where $L_2^{\prime}$ is possibly equal to $1$. Hence,
\begin{align*}
&U_h = La_hR = L^{\prime}Ca_hC^{-1}R^{\prime} = L_1^{\prime}b_h^{\prime}L_2^{\prime}Ca_hC^{-1}R^{\prime},\\
&U_j = L^{\prime}R^{\prime} = L_1^{\prime}b_h^{\prime}L_2^{\prime}R^{\prime}.
\end{align*}
We call an element of $\mathcal{M}(U_j)$ that contains $b_h^{\prime}$ \emph{an image of $b_h$ in $U_j$}. We call the set of all these elements \emph{the set of images of $b_h$ in $U_j$}.
\end{enumerate}
Clearly, we have a similar correspondence for $b_h$ when its beginning lies strictly right from the beginning of $a_h$.
\end{definition}

If $\LM(b_h^{\prime}) \leqslant \varepsilon$, it can be contained in several elements of $\mathcal{M}(U_j)$ and the set of images of $b_h$ may contain several elements. If $\LM(b_h^{\prime}) > \varepsilon$, then its prolongation in $U_j$ is uniquely determined and the set of images of $b_h$ consists of one element. In fact, in the previous section we described in detail all possible images of maximal occurrences of generalized fractional powers.

\begin{example}
We observe the following effect for elements of $\mathcal{M}^{\fc}(U_j)$. Let a maximal occurrence $b_h$ touch $a_h$ at a point from the left side. Suppose $b_h = b_h^{(1)}c_h^{\prime}$, where $c_h^{\prime}$ is a maximal occurrence of a generalized fractional power, $\LM(c_h^{\prime}) \leqslant \varepsilon$. Suppose $a_j = c_j^{\prime}a_j^{(1)}$, where $c_j^{\prime}$ is a maximal occurrence of a generalized fractional power, $\LM(c_j^{\prime}) \leqslant \varepsilon$. Assume $c_j = c_h^{\prime}c_j^{\prime}$ is a maximal occurrence of a generalized fractional power in $U_j$, then we obtain a new element in $\mathcal{M}^{\fc}(U_j)$, which grows from two maximal occurrences $c_h^{\prime} \notin \mathcal{M}(U_h)$ and $c_j^{\prime} \notin \mathcal{M}(U_j)$.
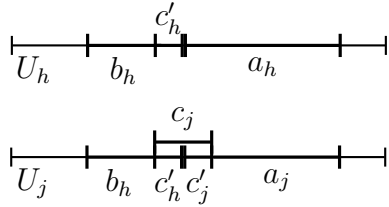
\begin{center}
\begin{tikzpicture}
\node at (0.3, 0) [below] {$U_h$};
\draw[|-|, black, thick] (0,0)--(5,0);
\draw[|-, black, very thick] (1, 0)--(2, 0) node[midway, below] {$b_h$};
\draw[|-|, black, very thick] (1.9, 0)--(2.3, 0) node[midway, above] {$c_h^{\prime}$};
\draw[|-|, black, very thick] (2.3, 0)--(4.4, 0) node[midway, below] {$a_h$};
\end{tikzpicture}

\begin{tikzpicture}
\node at (0.3, 0) [below] {$U_j$};
\draw[|-|, black, thick] (0,0)--(5,0);
\draw[|-, black, very thick] (1, 0)--(1.9, 0) node[midway, below] {$b_h$};
\draw[|-|, black, very thick] (1.9, 0)--(2.3, 0) node[midway, below] {$c_h^{\prime}$};
\draw[-|, black, very thick] (2.7, 0)--(4.4, 0) node[midway, below] {$a_j$};
\draw[|-|, black, very thick] (2.3, 0)--(2.7, 0) node[midway, below] {$c_j^{\prime}$};
\draw[|-|, black, very thick] (1.9, 0.2)--(2.7, 0.2) node[midway, above] {$c_j$};
\end{tikzpicture}
\end{center}
This example shows that we can not use just the size of $\mathcal{M}(U_j)$ as an inductive parameter. We will introduce special notions in order to control the behaviour of elements both from $\mathcal{M}^{\nfc}(U_j)$ and from $\mathcal{M}^{\fc}(U_j)$.
\end{example}

All elements of $\mathcal{M}(U_h)$ can be divided into sets such that every set covers a part of $U_h$ and different parts are separated from each other. We denote them by $\mathcal{M}_1(U_h), \ldots , \mathcal{M}_{k(U_h)}(U_h)$. We consider subcoverings of all parts, that are subsets of $\mathcal{M}_1(U_h), \ldots , \mathcal{M}_{k(U_h)}(U_h),$ and call the union of this subcoverings \emph{a covering of $U_h$}. There are finitely many different coverings of $U_h$, we denote them by $\{\mathcal{C}_i(U_h) \mid i = 1, \ldots , n(U_h)\}$. If $Z$ is a subword of $U_h$ and $\mathcal{C}_i(U_h)$ is a covering of $U_h$, we consider elements of $\mathcal{C}_i(U_h)$ that have non empty intersection with $Z$. We call the set of this elements \emph{a covering of a subword $Z$} and denote it by $\mathcal{C}_i(Z, U_h)$.

Let us consider a subcovering of every set $\mathcal{M}_k(U_h)$ that consists of the smallest number of elements and denote this number by $N_k(U_h)$. Let us call the union of this subcoverings \emph{a minimal covering of $U_h$}. Clearly, a minimal covering of $U_h$ may not be uniquely defined. We define
\begin{equation*}
N(U) = \sum\limits_{k = 1}^{k(U_h)}N_k(U_h),
\end{equation*}
that is, the number of elements in a minimal covering of $U_h$. Notice that all elements of $\mathcal{M}^{\nfc}(U_h)$ are contained in any covering of $U_h$. If some element of $\mathcal{M}^{\fc}(U_h)$ is fully covered by elements of $\mathcal{M}^{\nfc}(U_h)$, it is never contained in a minimal covering of $U_h$.

Let us study a general structure of positions of elements of $\mathcal{M}^{\fc}(U_h)$. Consider $c \in \mathcal{M}^{\fc}(U_h)$, that is, $c$ is fully covered by other elements of $\mathcal{M}(U_h)$. Consider the set of all maximal occurrences $\{c_k\}$ from $\mathcal{M}(U_h)$ that have non empty intersection with $c$. Then for every $c_k$ either its beginning point or its end point belong to $c$, since $c$ is not properly contained in another occurrence. Consider $c_{k_1}$ with the leftmost beginning with respect to the beginning of $c$ and $c_{k_2}$ with the rightmost beginning with respect to the beginning of $c$.
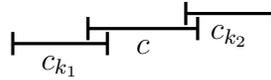
\begin{center}
\begin{tikzpicture}
\draw[|-|, black, very thick] (0, 0)--(1.5, 0) node[midway, below] {$c$};
\draw[|-|, black, very thick] (-1, -0.2)--(0.3, -0.2) node[midway, below] {$c_{k_1}$};
\draw[|-|, black, very thick] (1.3, 0.2)--(2.5, 0.2) node[midway, below] {$c_{k_2}$};
\end{tikzpicture}
\end{center}
By the definition of $\mathcal{M}(U_h)$ neither of $c_k$ is contained in another one. Hence, if $k \neq k_1$, then $c_{k}$ begins after $c_{k_1}$, if $k \neq k_2$, then $c_{k}$ ends before $c_{k_2}$. Therefore, if $k \neq k_1$, $k \neq k_2$, then $c_{k}$ is fully covered by $c$, $c_{k_1}$, $c_{k_2}$, so, $c_k \in \mathcal{M}^{\fc}(U_h)$.
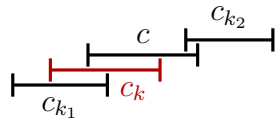
\begin{center}
\begin{tikzpicture}
\draw[|-|, black, very thick] (0, 0)--(1.5, 0) node[midway, above] {$c$};
\draw[|-|, black, very thick] (-1, -0.4)--(0.3, -0.4) node[midway, below] {$c_{k_1}$};
\draw[|-|, black!30!red, very thick] (-0.5, -0.2)--(1, -0.2) node[midway, below, near end] {$c_{k}$};
\draw[|-|, black, very thick] (1.3, 0.2)--(2.5, 0.2) node[midway, above] {$c_{k_2}$};
\end{tikzpicture}

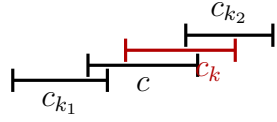
\begin{tikzpicture}
\draw[|-|, black, very thick] (0, 0)--(1.5, 0) node[midway, below] {$c$};
\draw[|-|, black, very thick] (-1, -0.2)--(0.3, -0.2) node[midway, below] {$c_{k_1}$};
\draw[|-|, black!30!red, very thick] (0.5, 0.2)--(2, 0.2) node[midway, below, near end] {$c_{k}$};
\draw[|-|, black, very thick] (1.3, 0.4)--(2.5, 0.4) node[midway, above] {$c_{k_2}$};
\end{tikzpicture}
\end{center}
If $c_{k_1}$ is also fully covered by other elements of $\mathcal{M}(U_h)$, we consider an element of $\mathcal{M}(U_h)$ that has the end at the leftmost side of $c_{k_1}$ and repeat the argument. If $c_{k_2}$ is fully covered by other elements of $\mathcal{M}(U_h)$, we consider an element of $\mathcal{M}(U_h)$ that has non-empty intersection with $c_{k_2}$ with the rightmost beginning with respect to the beginning of $c_{k_2}$ and repeat the argument. We continue the process until we find an element of $\mathcal{M}^{\nfc}(U_h)$ that has the beginning from the left of $c$ and an element of $\mathcal{M}^{\nfc}(U_h)$ that has the end from the right of $c$. Denote the first element by $d_1$ and the second element by $d_2$. We see that all elements of $\mathcal{M}(U_h)$ that have the beginning after the beginning of $d_1$ and the end before the end of $d_2$ belong to $\mathcal{M}^{\fc}(U_h)$.

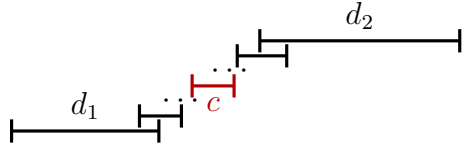
\begin{center}
\begin{tikzpicture}
\draw[|-|, black, very thick] (0, 0)--(2, 0) node[midway, above] {$d_1$};
\draw[|-|, black, very thick] (1.7, 0.2)--(2.3, 0.2);
\path[|-|, black, very thick] (2, 0.4)--(2.6, 0.4) node[midway] {$\cdots$};
\draw[|-|, black!30!red, very thick] (2.4, 0.6)--(3, 0.6) node[midway, below] {$c$};
\path[|-|, black, very thick] (2.7, 0.8)--(3.2, 0.8) node[midway] {$\cdots$};
\draw[|-|, black, very thick] (3, 1)--(3.7, 1);
\draw[|-|, black, very thick] (3.3, 1.2)--(6, 1.2) node[midway, above] {$d_2$};
\end{tikzpicture}
\end{center}

\begin{lemma}
\label{minimal_coverings_property}
Let $U_h$ be a monomial, $a_h \in \mathcal{M}^{\nfc}(U_h)$, $U_h = La_hR$. Assume $a_h$ and $a_j$ are incident monomials, $U_j = La_jR$. Then we have $N(U_j) \leqslant N(U_h)$. If, moreover, $a_j$ is fully covered by images of elements of $\mathcal{M}^{\nfc}(U_h) \setminus \{a_h\}$ (for instance, if $a_j = 1$), then $N(U_j) < N(U_h)$.

\end{lemma}
\begin{proof}
Let $\mathcal{C}_{i_1}(U_h)$ be a minimal covering of $U_h$. First consider the case $a_j \neq 1$. Since $a_h \in \mathcal{M}^{\nfc}(U_h)$, it necessarily belongs to $\mathcal{C}_{i_1}(U_h)$. Hence, the covering $\mathcal{C}_{i_1}(U_h)$ can be written as
\begin{equation*}
\mathcal{C}_{i_1}(U_h) = \mathcal{C}_{i_1}(L, U_h) \sqcup \{a_h\} \sqcup \mathcal{C}_{i_1}(R, U_h),
\end{equation*}
where $\sqcup$ is a disjoint union. Denote by $a_j^{\prime}$ an image of $a_h$. Denote by $\mathcal{C}^{\prime}$ the set of images of elements of $\mathcal{C}_{i_1}(L, U_h)$, by $\mathcal{C}^{\prime\prime}$ the set of images of elements of $\mathcal{C}_{i_1}(R, U_h)$. For elements that have more than one image, we take one arbitrary image. Every letter of $L$ that is covered by $\mathcal{C}_{i_1}(U_h)$ is covered by $\mathcal{C}^{\prime}$ in $U_j$. Every letter of $R$ that is covered by $\mathcal{C}_{i_1}(U_h)$ is covered by $\mathcal{C}^{\prime\prime}$ in $U_j$. The occurrence of $a_j$ in $U_j$ is covered by $a_j^{\prime}$. Therefore,
\begin{equation*}
\mathcal{C}^{\prime} \cup \{ a_j^{\prime}\} \cup \mathcal{C}^{\prime\prime} \textit{ is a covering of } U_j.
\end{equation*}
Denote it by $\mathcal{C}_{i_2}(U_j)$.

Since we take one image for every element of $\mathcal{C}_{i_1}(U_h)$,
\begin{equation*}
\vert \mathcal{C}^{\prime}\vert = \vert \mathcal{C}_{i_1}(L, U_h)\vert, \vert \mathcal{C}^{\prime\prime}\vert = \vert \mathcal{C}_{i_1}(R, U_h)\vert,
\end{equation*}
where $\vert \cdot \vert$ is the number of elements in a set. Hence, $\vert \mathcal{C}_{i_2}(U_j)\vert \leqslant \vert \mathcal{C}_{i_1}(U_h) \vert = N(U_h)$, so we obtain $N(U_j) \leqslant N(U_h)$.

Let us show that $b_h \in \mathcal{M}^{\nfc}(U_h) \setminus \{a_h\}$ has one image in $U_j$. If $b_h$ is separated from $a_h$, it is obvious. Suppose $b_h$ is not separated from $a_h$ from the left. Let $b_h^{\prime}$ be a subword of $b_h$ that is an intersection of $b_h$ and $L$. We have $L = L_1^{\prime}b_h^{\prime}$. Since $b_h$ is not fully covered by elements of $\mathcal{M}(U_h)$, $b_h^{\prime}$ is also not fully covered by elements of $\mathcal{M}(U_h)$. This means that $b_h^{\prime}$ can not be prolonged from the left neither in $U_h$ nor in $U_j$. Hence, $b_h^{\prime}$ in $U_j$ may be fully contained only in some element of $\mathcal{M}(U_j)$ that is contained in the word $b_h^{\prime}a_jR$. Since $b_h^{\prime}$ is an initial subword of $b_h^{\prime}a_jR$, there exists only one such element of $\mathcal{M}(U_j)$ that contains $b_h^{\prime}$ (possibly $b_h^{\prime}$ itself). Hence, $b_h$ has one image in $U_j$. The case of $b_h$ being not separated from $a_h$ from the right is considered similarly.

Every element of $\mathcal{M}^{\nfc}(U_h)$ is contained in any covering of $U_h$. Therefore, every element of $\mathcal{M}^{\nfc}(U_h) \setminus \{a_h\}$ is contained either in $\mathcal{C}_{i_1}(L, U_h)$, or in $\mathcal{C}_{i_1}(R, U_h)$. Assume $a_j$ is fully covered by images of elements of $\mathcal{M}^{\nfc}(U_h) \setminus \{a_h\}$. Since every element of $\mathcal{M}^{\nfc}(U_h) \setminus \{a_h\}$ has only one image, $a_j$ is fully covered by $\mathcal{C}^{\prime} \cup \mathcal{C}^{\prime\prime}$. Therefore, $\mathcal{C}^{\prime} \cup \mathcal{C}^{\prime\prime}$ is a covering of $U_j$ that has strictly less elements than $\mathcal{C}_{i_1}(U_h)$. Thus, we obtain $N(U_j) < N(U_h)$.

Now consider the case $a_j = 1$. Assume $U_j = LR = L^{\prime}CC^{-1}R^{\prime} = L^{\prime}R^{\prime}$, where $L^{\prime}R^{\prime}$ has no further cancellations. We have $U_h = L^{\prime}Ca_hC^{-1}R^{\prime}$, hence the covering $\mathcal{C}_{i_1}(U_h)$ can be written as
\begin{align*}
\mathcal{C}_{i_1}(U_h)& = \mathcal{C}_{i_1}(L^{\prime}, U_h) \sqcup \{\mathcal{C}_{i_1}(L, U_h) \setminus \mathcal{C}_{i_1}(L^{\prime}, U_h)\} \sqcup \{a_h\} \\ &\quad \sqcup \{\mathcal{C}_{i_1}(R, U_h) \setminus \mathcal{C}_{i_1}(R^{\prime}, U_h)\} \sqcup \mathcal{C}_{i_1}(R^{\prime}, U_h).
\end{align*}
Denote by $\mathcal{C}^{\prime}$ the set of images of elements of $\mathcal{C}_{i_1}(L^{\prime}, U_h)$, by $\mathcal{C}^{\prime\prime}$ the set of images of elements of $\mathcal{C}_{i_1}(R^{\prime}, U_h)$. Again, for elements that have more than one image, we take one arbitrary image. Every letter of $L^{\prime}$ that is covered by $\mathcal{C}_{i_1}(U_h)$ is covered by $\mathcal{C}^{\prime}$ in $U_j$. Every letter of $R^{\prime}$ that is covered by $\mathcal{C}_{i_1}(U_h)$ is covered by $\mathcal{C}^{\prime\prime}$ in $U_j$. Hence,
\begin{equation*}
\mathcal{C}^{\prime} \cup \mathcal{C}^{\prime\prime} \textit{ is a covering of } U_j.
\end{equation*}
Denote it by $\mathcal{C}_{i_2}(U_j)$.

Since we take one image for every element of $\mathcal{C}_{i_1}(U_h)$, we have
\begin{equation*}
\vert \mathcal{C}^{\prime}\vert = \vert \mathcal{C}_{i_1}(L^{\prime}, U_h)\vert, \vert \mathcal{C}^{\prime\prime}\vert = \vert \mathcal{C}_{i_1}(R^{\prime}, U_h)\vert.
\end{equation*}
Hence, $\vert \mathcal{C}_{i_2}(U_j)\vert < \vert \mathcal{C}_{i_1}(U_h) \vert = N(U_h)$, so we obtain $N(U_j) < N(U_h)$.
\end{proof}

Assume $a_h\in \mathcal{M}^{\nfc}(U_h)$. Consider the set of elements of $\mathcal{M}(U_h)$ that are not separated from $a_h$ from the left side. Evidently, there exists at most one element of $\mathcal{M}^{\nfc}(U_h)$ in this set. Similarly, there exists at most one element of $\mathcal{M}^{\nfc}(U_h)$ that is not separated from $a_h$ from the right side. Let us call them \emph{the essential left neighbour of $a_h$ and the essential right neighbour of $a_h$}, respectively.

\medskip

We define a special sequence of transformations of monomials.
\begin{definition}
\label{replacements}
Let $U$ be a word, $b \in \mathcal{M}^{\nfc}(U)$. We denote $U$ by $U^{(1)}$ and $b$ by $b^{(1)}$ and define a sequence of transformations inductively. Assume monomials $U^{(i)}$, maximal occurrence of a generalized fractional powers $b^{(i)}$ in $U^{(i)}$ and transformations $r_i: U^{(i - 1)} \mapsto U^{(i)}$, $1 \leqslant i \leqslant k$, are already defined (we assume that the transformation $r_1$ is identical). Let $a^{(k)}_h$ be a maximal occurrence of a generalized fractional power in $U^{(k)} = L^{(k)}a^{(k)}_hR^{(k)}$ such that
\begin{enumerate}
\item
\label{replacement_condition1}
$a^{(k)}_h$ differs from $b^{(k)}$;
\item
\label{replacement_condition2}
$\LM(a^{(k)}_h) \geqslant \tau - 2\varepsilon$ (hence, $a^{(k)}_h \in \mathcal{M}^{\nfc}(U^{(k)})$);
\item
\label{replacement_condition3}
if the beginning of $b^{(k)}$ lies from the right of the beginning $a^{(k)}_h$, $a^{(k)}_h$ has the right essential neighbour; if the beginning $b^{(k)}$ lies from the left of the beginning $a^{(k)}_h$, $a^{(k)}_h$ has the left essential neighbour.
\end{enumerate}
Suppose $a^{(k)}_h$ and $a^{(k)}_j$ are incident monomials. If $a^{(k)}_j$ in $L^{(k)}a^{(k)}_jR^{(k)}$ is not covered by images of $\mathcal{M}^{\nfc}(U^{(k)}) \setminus \{a^{(k)}_h\}$, then we say that
\begin{align*}
&U^{(k+1)} = L^{(k)}a^{(k)}_jR^{(k)},\\
&b^{(k + 1)} \emph{ is an image of } b^{(k)} \emph{ in } U^{(k + 1)},\\
&r_{k + 1} : U^{(k)} \mapsto U^{(k + 1)} \emph{ is a replacement of } a^{(k)}_h \emph{ by } a^{(k)}_j.
\end{align*}
\end{definition}

\begin{lemma}
\label{replacements_property}
Let $U$ be a word, $b \in \mathcal{M}^{\nfc}(U)$. Assume we have a sequence of transformations defined above that starts from $U$. If $U^{(K)}$ is a monomial in this sequence, then $b^{(K)}$ can increase or decrease at most by a piece of $\LM$-measure $\varepsilon$ from the right side and at most by a piece of $\LM$-measure $\varepsilon$ from the left side with respect to $b$.
\end{lemma}
\begin{proof}
In this proof, we use notations from Definition~\ref{replacements}. Let us prove the lemma by induction on $K$.

Assume $K = 2$, that is, the sequence of transformations has only one step $r_1$. If $a_h^{(1)}$ is separated from $b^{(1)}$, then the occurrence $b^{(2)}$ in $U^{(2)}$ is equal to $b^{(1)}$. Let $a_h^{(1)}$ be not separated from $b^{(1)}$. Since $b^{(1)} = b \in \mathcal{M}^{\nfc}(U)$, $b^{(1)}$ is an essential neighbour of $a_h^{(1)}$. According to the classification given in Section~\ref{mt_configurations}, $b^{(2)}$ can decrease at most by a piece of $\LM$-measure $\varepsilon$ after the replacement $a_h^{(1)} \mapsto a_j^{(1)}$. Since $b^{(1)}$ is an essential neighbour of $a_h^{(1)}$, $ a_j^{(1)}$ is not covered by the image of $b^{(1)}$, by Definition~\ref{replacements}. Hence, according to the classification given in Section~\ref{mt_configurations}, $b^{(2)}$ can increase at most by a piece of $\LM$-measure $\varepsilon$ after the replacement $a_h^{(1)} \mapsto a_j^{(1)}$.

Let us prove the step of the induction. Consider the replacement $r_2: U^{(1)} \mapsto U^{(2)}$, $a_h^{(1)} \mapsto a_j^{(1)}$. Suppose $b^{(2)} \in \mathcal{M}^{\fc}(U)$. Then any element of $\mathcal{M}^{\nfc}(U)$ that is not separated from $b^{(2)}$ does not have the essential neighbour from the necessary side. Hence, there are no more possible replacements $U^{(i)} \mapsto U^{(i + 1)}$, $i \geqslant 2$, of occurrences that are not separated from $b^{(i)}$. So, $b^{(i)}$ remains equal to $b^{(2)}$ after any further replacement.

Further we consider only the case $b^{(2)} \in \mathcal{M}^{\nfc}(U)$. If $a_h^{(1)}$ is separated from $b^{(1)}$, then the occurrence $b^{(2)}$ in $U^{(2)}$ is equal to $b^{(1)}$. The sequence of transformations starting from $U^{(2)}$ has a fewer number of steps. Hence, $b^{(K)}$ can increase or decrease at most by a piece of $\LM$-measure $\varepsilon$ from the right and the left sides with respect to $b^{(2)}$ ($ = b^{(1)}$) by the induction hypothesis.

Suppose $a_h^{(1)}$ is not separated from $b^{(1)}$. If $b^{(2)}$ remains equal to $b^{(1)}$, then as above $b^{(K)}$ can increase or decrease at most by a piece of $\LM$-measure $\varepsilon$ from the right and the left sides with respect to $b^{(1)}$ by the induction hypothesis.

Assume $\LM(b^{(2)})$ changes with respect to $\LM(b^{(1)})$. Then there are the following cases:
\begin{enumerate}
\item
\label{change1}
The beginning of $b^{(1)}$ lies from the left of the beginning of $a_h^{(1)}$, $\LM(b^{(2)}) < \LM(b^{(1)})$.
\item
\label{change2}
The beginning of $b^{(1)}$ lies from the right of the beginning of $a_h^{(1)}$, $\LM(b^{(2)}) < \LM(b^{(1)})$.
\item
\label{change3}
The beginning of $b^{(1)}$ lies from the left of the beginning of $a_h^{(1)}$, $\LM(b^{(2)}) > \LM(b^{(1)})$.
\item
\label{change4}
The beginning of $b^{(1)}$ lies from the right of the beginning of $a_h^{(1)}$, $\LM(b^{(2)}) > \LM(b^{(1)})$.
\end{enumerate}

Let us consider case~\ref{change1}. There are the following two configurations for $b^{(2)}$:
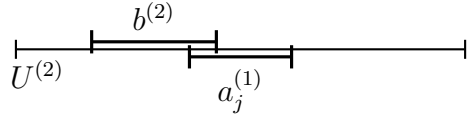
\begin{enumerate}[label=\textbf{(\greek*)}]
\item
\label{aj_not_covered}
$a_j^{(1)} \in \mathcal{M}^{\nfc}(U^{(2)})$
\begin{center}
\begin{tikzpicture}
\node at (0.3, 0) [below] {$U^{(2)}$};
\draw[|-|, black, thick] (0,0)--(6,0);
\draw[|-|, black, very thick] (1,0.1)--(2.7,0.1) node[midway, above] {$b^{(2)}$};
\draw[|-|, black, very thick] (2.3,-0.1)--(3.7,-0.1) node[midway, below] {$a_j^{(1)}$};
\end{tikzpicture}
\end{center}
\item
\label{aj_covered}
$a_j^{(1)} \in \mathcal{M}^{\fc}(U^{(2)})$
\begin{center}
\begin{tikzpicture}
\node at (0.3, 0) [below] {$U^{(2)}$};
\draw[|-|, black, thick] (0,0)--(6,0);
\draw[|-|, black, very thick] (1,0.1)--(2.7,0.1) node[midway, above] {$b^{(2)}$};
\draw[|-|, black, very thick] (2.3,-0.1)--(3,-0.1) node[midway, below] {$a_j^{(1)}$};
\draw[|-|, black, very thick] (2.7,0.1)--(3.5,0.1) node[midway, above] {$c$};
\node at (7, 1) {$\LM(c) \leqslant 4\varepsilon$};
\end{tikzpicture}
\end{center}
If $a_j^{(1)} \in \mathcal{M}^{\fc}(U^{(2)})$, then an element $c \in \mathcal{M}(U^{(2)})$ covers its terminal subword. Since $a_j^{(1)}$ is not covered by images of $\mathcal{M}^{\nfc}(U^{(1)}) \setminus \{a^{(1)}_h\}$, $c$ is not an image of an element of $\mathcal{M}^{\nfc}(U^{(1)}) \setminus \{a^{(1)}_h\}$. Hence, $c$ is an image of some generalized fractional power of $\LM$-measure not greater than $2\varepsilon$. Then, by the classification from Section~\ref{mt_configurations}, $\LM(c) \leqslant 4\varepsilon$.
\end{enumerate}

Let us study case~\ref{aj_covered} first. Since $\LM(c) \leqslant 4\varepsilon < \tau - 2\varepsilon$, $c$ can not be used for the next replacement $r_3$. The sequence of replacements starting from $U^{(2)}$ has $K - 1$ steps. Hence, if $c\in \mathcal{M}^{\nfc}(U^{(2)})$, then by the induction hypothesis the $\LM$-measure of an image of $c$ can increase at most by $2\varepsilon$ with respect to $\LM(c)$ in every $U^{(i)}$, $3 \leqslant i \leqslant K$. So, an image of $c$ is of less $\LM$-measure than $\tau - 2\varepsilon$ and can not be used for a replacement in any $U^{(i)}$, $3 \leqslant i \leqslant K$. Hence, all the next replaced occurrences are separated from $b^{(i)}$ from the right side. So, $b^{(i)}$ has no more changes from the right side with respect to $b^{(2)}$, $3 \leqslant i \leqslant K$.

Suppose $c\in \mathcal{M}^{\fc}(U^{(2)})$. Denote the element closest to $c$ from the right such that it belongs to $\mathcal{M}^{\nfc}(U^{(2)})$ by $d$. Then $d$ does not have the left essential neighbour. By the induction hypothesis, $b^{(i)}$ can decrease at most by a piece of $\LM$-measure $\varepsilon$ from the left side. Hence, the image of $d$ can not have the left essential neighbour in every $U^{(i)}$. Thus, $d$ can not be used for a replacement in any $U^{(i)}$, $3 \leqslant i \leqslant K$. Therefore, all the next replaced occurrences are separated from the image of $c$ from the right side. Hence, $b^{(i)}$ has no more changes from the right side with respect to $b^{(2)}$, $3 \leqslant i \leqslant K$.

Consider case~\ref{aj_not_covered}. Consider the next replacement $r_3 : U^{(2)} \mapsto U^{(3)}$, $a_h^{(2)} \mapsto a_j^{(2)}$. With a slight abuse of language, we use the same indices in different incident monomials. By the induction hypothesis, $b^{(i)}$, $3 \leqslant i \leqslant K$, may increase or decrease from the left at most by a piece of $\LM$-measure $\varepsilon$ with respect to $b^{(2)}$. This means that it remains to consider only replacements from the right of $b^{(2)}$. Therefore, without loss of generality, we can assume that the beginning $a_h^{(2)}$ lies from the right of the beginning of $b^{(2)}$. Then there are the following possibilities:
\begin{enumerate}
\item
$a_h^{(2)}$ is separated from $a_j^{(1)}$. Then we again obtain configuration \ref{aj_not_covered} for $b^{(3)}$, that is, the image of $a_j^{(1)}$ belongs to $ \mathcal{M}^{\nfc}(U^{(3)})$.
\item
$a_h^{(2)}$ coincides with $a_j^{(1)}$. In this case if $a_j^{(2)} \in \mathcal{M}^{\nfc}(U^{(3)})$, we obtain configuration \ref{aj_not_covered} for $b^{(3)}$, possibly with a smaller overlap or without an overlap between $b^{(3)}$ and $a_j^{(2)}$. If $a_j^{(2)} \in \mathcal{M}^{\fc}(U^{(3)})$, we obtain configuration \ref{aj_covered} for $b^{(3)}$.
\item
$a_h^{(2)}$ is not separated from $a_j^{(1)}$. If the image of $a_j^{(1)}$ belongs to $ \mathcal{M}^{\nfc}(U^{(3)})$, we obtain configuration \ref{aj_not_covered} for $b^{(3)}$. If the image of $a_j^{(1)}$ belongs to $ \mathcal{M}^{\fc}(U^{(3)})$, we obtain configuration \ref{aj_covered} for $b^{(3)}$.
\end{enumerate}
So, we may obtain for $b^{(3)}$ either configuration~\ref{aj_covered} that was already considered, or again configuration~\ref{aj_not_covered}. If we obtain configuration~\ref{aj_not_covered}, $b^{(3)}$ increases or decreases at most by a piece of $\LM$-measure $\varepsilon$ from the right side with respect to $b^{(1)}$.

Cases~\ref{change2} --- \ref{change4} are studied in the same way as case~\ref{change1}. So, if we continue the description for the next transformations $r_i$, we observe that possible changes of $b^{(i)}$ with respect to $b^{(1)}$ are restricted by a piece of $\LM$-measure not greater than $\varepsilon$ from each side.
\end{proof}

\begin{remark}
\label{replacements_property_remark}
In fact, in Lemma~\ref{replacements_property} we also proved the following. Assume that some maximal occurrence $a \in \mathcal{M}^{\nfc}(U)$ has an overlap of measure $\varepsilon$ with its left essential neighbour. Then after a sequence of replacements from Definition~\ref{replacements}, the $\LM$-measure of the image of $a$ can not increase from the left with respect to the $\LM$-measure of $a$. Similarly, if $a$ has an overlap of measure $\varepsilon$ with its right essential neighbour, then after a sequence of replacements from Definition~\ref{replacements} the $\LM$-measure of the image of $a$ can not increase from the right with respect to the $\LM$-measure of $a$.
\end{remark}

\begin{definition}
\label{virtual_members}
Let $U$ be a word, $b \in \mathcal{M}^{\nfc}(U)$. If there exists a sequence of replacements constructed in Definition \ref{replacements} such that $U^{(K)}$ is the last monomial in the sequence and $\LM(b^{(K)}) \geqslant \tau$, we call $b$ \emph{a virtual member of the chart of $U$}. We call the set of such maximal occurrences from $\mathcal{M}(U)$ \emph{the virtual $\tau$-chart of $U$}.

We denote the number of virtual members of the chart of $U$ by $K_{\tau}(U)$.
\end{definition}

From Lemma~\ref{replacements_property} it follows that if $b$ is a virtual member of the chart of $U$, then $\LM(b) \geqslant \tau - 2\varepsilon$. Clearly, every member of the chart of $U$ is also a virtual member of the chart of $U$.

From now on, when we speak about the chart of some word, we use only virtual members of the chart even if the qualification virtual is omitted.

\medskip

Let us prove important properties of virtual members of the chart that we will use in the further argument.
\begin{lemma}
\label{virtual_members_property1}
Let $U_h$ be a monomial, $a_h$ be a virtual member of its chart, $U_h = La_hR$. Assume $a_h$ and $a_j$ are incident monomials, $a_j\neq 1$, $U_j = La_jR$. If $a_j$ is not fully covered by images of $\mathcal{M}^{\nfc}(U_h) \setminus \{a_h\}$, then $K_{\tau}(U_j) \leqslant K_{\tau}(U_h)$. If, moreover, $a_j$ is not a virtual member of the chart of $U_j$, then $K_{\tau}(U_j) < K_{\tau}(U_h)$.
\end{lemma}
\begin{proof}
Let $U_h = La_hR$, $U_j = La_jR$. Suppose $b$ is a virtual member of the chart of $U_j$ different from $a_j$. Since $a_j$ is not fully covered by images of $\mathcal{M}^{\nfc}(U_h) \setminus \{a_h\}$, $a_j$ is not fully contained in $b$. To be precise, assume that the beginning of $b$ lies from the left side of the beginning of $a_j$, that is, $b$ is an occurrence in the subword $La_j$. Since $b$ is a virtual member of the chart of $U_j$, there exists a sequence of transformations
\begin{equation}
\label{virt_member_sequence}
U_j = U_j^{(1)} \mapsto \ldots \mapsto U_j^{(K)}
\end{equation}
defined above such that the image of $b$ in $U_j^{(K)}$ is of $\LM$-measure $\geqslant \tau$. Denote by $b^{\prime}$ the inverse image of $b$ in $U_h$.

First assume that $a_h$ in $U_h$ has the left essential neighbour. Since $a_j$ is not fully covered by images of $\mathcal{M}^{\nfc}(U_h) \setminus \{a_h\}$, the replacement $a_h \mapsto a_j$ can be used in the definition of a virtual member of the chart. Hence, we can add the transformation $U_h \mapsto U_j$ to the beginning of the sequence~\eqref{virt_member_sequence}. Therefore, we obtain the sequence that starts from $U_h$ and defines $b^{\prime}$ as a virtual member of the chart of $U_h$.

Now assume that $a_h$ in $U_h$ does not have the left essential neighbour. Then it is possible that there exist elements of $\mathcal{M}^{\fc}(U_h)$ that are not separated from $a_h$ from the left side. The other possibility is that there are no elements of $\mathcal{M}(U_h)$ that are not separated from $a_h$ from the left side. Let us study the first case. Since $a_h$ in $U_h$ does not have the left essential neighbour, $b^{\prime}$ is separated from $a_h$.
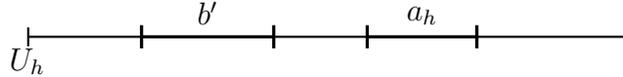
\begin{center}
\begin{tikzpicture}
\draw[|-|, black, thick] (-1,0) to node[at start, below] {$U_h$} (7,0);
\draw[|-|, black, very thick] (0.5,0)--(2.3,0) node[midway, above] {$b^{\prime}$};
\draw[|-|, black, very thick] (3.5,0)--(5,0) node[midway, above] {$a_h$};
\end{tikzpicture}
\end{center}

Consider the first possibility, namely, there exist elements of $\mathcal{M}^{\fc}(U_h)$ that are not separated from $a_h$ from the left side. Let us show that there exists an element of $\mathcal{M}(U_j)$ of $\LM$-measure $\leqslant 5\varepsilon$ such that its beginning lies between the beginning of $b$ and the beginning of $a_j$ and it is not fully covered by its essential neighbours. Let $d^{\prime} \in \mathcal{M}(U_h)$ be not separated from $a_h$ from the left and has the leftmost beginning among such elements. Let $d$ be an image of $d^{\prime}$ in $U_j$. Since $a_h$ does not have the left essential neighbour, $d^{\prime}\in \mathcal{M}^{\fc}(U_h)$. Hence, $\LM(d^{\prime}) \leqslant 2\varepsilon$, so, $\LM(d) \leqslant 5\varepsilon$. If $d$ is not covered by its essential neighbours, then $d$ is the necessary element.

Suppose $d$ is covered by its essential neighbours. In particular, this means that $d$ has both the left and the right essential neighbour. First assume that $a_j$ is contained in $d$. Then $a_j$ is covered by essential neighbours of $d$. Since $a_j$ is not covered by images of $\mathcal{M}^{\nfc}(U_h) \setminus \{a_h\}$, at least one essential neighbour of $d$ is not equal to an image of the element of $\mathcal{M}^{\nfc}(U_h)\setminus \{a_h\}$. Hence it is of $\LM$-measure $\leqslant 5\varepsilon$ and it satisfies necessary conditions.

Assume that $a_j$ is not contained in $d$. Since $d^{\prime}$ has the rightmost beginning among elements that are not separated from $a_h$ from the left, its image $d$ in $U_j$ has the same beginning position. Therefore, the left essential neighbour of $d$ in $U_j$ is the image of the left essential neighbour of $d^{\prime}$ in $U_h$. Hence, the left essential neighbour of $d$ is separated from $a_j$, since $a_h$ does not have the left essential neighbour. So, $d$ can not be covered by its left essential neighbour and $a_j$. Then consider the right essential neighbour of $d$ and denote it by $c$. Since $d$ is not covered by its left essential neighbour and $a_j$, the occurrence $c$ is not equal to the occurrence $a_j$. If some element of $\mathcal{M}(U_j)$ has the beginning from the right of the beginning of $d$ and the end from the left of the end of $a_j$, it is covered by $d$ and $a_j$. Hence, it does not belong to $\mathcal{M}^{\nfc}(U_j)$. Therefore, since $c\in \mathcal{M}^{\nfc}(U_j)$, $a_j$ is contained in $c$. Since $a_j$ is not covered by images of $\mathcal{M}^{\nfc}(U_h) \setminus \{a_h\}$, $c$ is not equal to an image of an element of $\mathcal{M}^{\nfc}(U_h)\setminus \{a_h\}$. Since the occurrence $c$ is not equal to the occurrence $a_j$, $\LM(c) \leqslant 5\varepsilon$ and it satisfies the necessary conditions.

We denote the obtained element by $g$. We put $g = g_lg_mg_r$, where $g_l$ is an overlap with the left essential neighbour (if there is any) and $g_r$ is an overlap with the right essential neighbour (if there is any). Since $g$ is not covered by its left and right essential neighbours, $g_m$ is not empty.

We put $U_j = U_j^{(1)} = L^{(1)}g_mR^{(1)}$. Then $b$ is an occurrence in the subword $L^{(1)}$. From Lemma~\ref{replacements_property} it follows that the $\LM$-measure of an image of $g$ in every $U_j^{(i)}$ is not greater than $7\varepsilon < \tau - 2\varepsilon$. Therefore, an image of $g$ can not be used in any replacement $U_j^{(i)} \mapsto U_j^{(i + 1)}$ in~\eqref{virt_member_sequence}. Then every replacement in~\eqref{virt_member_sequence} is of the form
\begin{gather*}
 L^{(i)}g_mR^{(i)} \mapsto L^{(i + 1)}g_mR^{(i + 1)},\\
 \textit{where either } L^{(i + 1)} = L^{(i)}, \textit{ or } R^{(i + 1)} = R^{(i)}.
\end{gather*}
The possibility of a transformation $L^{(i)}g_mR^{(i)} \mapsto L^{(i + 1)}g_mR^{(i + 1)}$, where $R^{(i + 1)} = R^{(i)}$, does not depend on $R^{(i)}$. Since $b$ is an occurrence in the subword $L^{(1)}$, it is sufficient to do only transformations $L^{(i)}g_mR^{(i)} \mapsto L^{(i + 1)}g_mR^{(i + 1)}$, where $R^{(i + 1)} = R^{(i)}$, in a sequence that defines $b$. So, we obtain the sequence
\begin{equation*}
U_j^{(1)} = L^{(1)}g_mR^{(1)} \mapsto L^{(i_1)}g_mR^{(1)} \ldots \mapsto L^{(i_s)}g_mR^{(1)}
\end{equation*}
that also defines $b$. We have $U_h = L^{(1)}R^{\prime}$. Thus, the same replacements can be done starting with $U_h$, so, we obtain the sequence
\begin{equation}
U_h = L^{(1)}R^{\prime} \mapsto L^{(i_1)}R^{\prime} \ldots \mapsto L^{(i_s)}R^{\prime}
\end{equation}
such that the image of $b^{\prime}$ in $L^{(i_s)}R^{\prime}$ is of $\LM$-measure $\geqslant \tau$. Therefore, $b^{\prime}$ is a virtual member of the chart of $U_h$.

The case when there are no elements of $\mathcal{M}(U_h)$ that are not separated from $a_h$ from the left side is considered in a similar (but easier) way.

So, we have proved that every virtual member of the chart of $U_j$ is the image of a virtual member of the chart of $U_h$. Hence, $K_{\tau}(U_j) \leqslant K_{\tau}(U_h)$. If $a_j$ is not a virtual member of the chart of $U_j$, then the image of $a_h$ is not a virtual member of the chart of $U_j$, because $a_j$ is not fully covered by images of $\mathcal{M}^{\nfc}(U_h) \setminus \{a_h\}$. Thus, in this case $K_{\tau}(U_j) < K_{\tau}(U_h)$.
\end{proof}

\begin{corollary}
\label{decreasing_property}
Let $U_h$ be a monomial, $a_h$ be a virtual member of its chart, $U_h = La_hR$. Assume $a_h$ and $a_j$ are incident monomials, $U_j = La_jR$. If $a_j$ is a virtual member of the chart of $U_j$, then $N(U_j) = N(U_h)$. If $a_j$ is not a virtual member of the chart of $U_j$, then either $N(U_j) < N(U_h)$, or $N(U_h) = N(U_j)$ and $K_{\tau}(U_j) < K_{\tau}(U_h)$.
\end{corollary}
\begin{proof}
From Lemma~\ref{minimal_coverings_property} it follows that $N(U_j) \leqslant N(U_h)$. Let $a_j$ be a virtual member of the chart of $U_j$, then $a_j \in \mathcal{M}^{\nfc}(U_j)$. Hence, we can apply Lemma~\ref{minimal_coverings_property} to the opposite replacement $a_j\mapsto a_h$ and obtain $N(U_h) \leqslant N(U_j)$. Thus, $N(U_j) = N(U_h)$.

Assume $a_j$ is not a virtual member of the chart of $U_j$. From Lemma~\ref{minimal_coverings_property} it follows that if $a_j$ is fully covered by images of $\mathcal{M}^{\nfc}(U_h) \setminus \{a_h \}$, then $N(U_j) < N(U_h)$. Therefore, if $N(U_h) = N(U_j)$, then $a_j$ is not fully covered by images of $\mathcal{M}^{\nfc}(U_h) \setminus \{a_h \}$. Hence, from Lemma~\ref{virtual_members_property1} it follows that $K_{\tau}(U_j) < K_{\tau}(U_h)$.
\end{proof}

\begin{corollary}
\label{virtual_members_stability}
Let $U_h$ be a monomial, $a_h$ be a virtual member of its chart, $U_h = La_hR$. Assume $a_h$ and $a_j$ are incident monomials, $a_j$ is a virtual member of the chart of $U_j = La_jR$. Assume $b_h \in \mathcal{M}^{\nfc}(U_h)$, $b_h$ is different from $a_h$ and $b_j$ is the image of $b_h$ in $U_j$. Then $b_h$ is a virtual member of the chart of $U_h$ if and only if $b_j$ is a virtual member of the chart of $U_j$. In particular, in this case $K_{\tau}(U_h) = K_{\tau}(U_j)$.
\end{corollary}
\begin{proof}
Since $a_j$ is a virtual member of the chart of $U_j$, $\LM(a_j) \geqslant \tau - 2\varepsilon$ and $a_j \in \mathcal{M}^{\nfc}(U_j)$. Hence, the conditions of Lemma~\ref{virtual_members_property1} hold for the replacement $U_h \mapsto U_j$, $a_h \mapsto a_j$. In the proof of Lemma~\ref{virtual_members_property1} we actually show that if $b_j$ is a virtual member of the chart of $U_j$, then $b_h$ is a virtual member of the chart of $U_h$.

Since $a_j$ is a virtual member of the chart of $U_j$, we can consider the opposite transformation $U_j \mapsto U_h$, $a_j \mapsto a_h$. Then, obviously, the image of $b_j$ in $U_h$ is $b_h$. If we apply Lemma~\ref{virtual_members_property1} in this case, we obtain that if $b_h$ is a virtual member of the chart of $U_h$, then $b_j$ is a virtual member of the chart of $U_j$.
\end{proof}

In the same way as~\eqref{linear_dep_members} we define the set of all the supports of multi-turns of virtual members of the chart
\begin{align}
\label{linear_dep_virt_members}
\mathcal{T}^{\prime}& = \Bigg\lbrace \sum\limits_{j = 1}^k U_j \mid U_h\in \Fr, U_h = La_hR, \ \textit{where} a_h \textit{ is a virtual member}\\ & \quad \textit{of the chart of } U_h, \quad U_h \mapsto \sum\limits_{\substack{j = 1 \\ j\neq h}}^{k} U_j \textit{ is a multi-turn} \nonumber \\
&\quad \textit{ that comes from an elementary multi-turn } a_h \mapsto \sum\limits_{\substack{j = 1 \\ j\neq h}}^{k} a_j \Bigg\rbrace.\nonumber
\end{align}

\begin{proposition}
\label{the_ideal_characterisation2}
The linear subspace $\langle\mathcal{T}^{\prime} \rangle \subseteq \mathbb{Z}_2\Fr$ is equal to the ideal $\Ideal = \langle 1 + v + vw \rangle$.
\end{proposition}
\begin{proof}
From Proposition~\ref{transversality} it easily follows that $\mathcal{T}^{\prime} = \mathcal{T}$. Hence, by Proposition~\ref{the_ideal_characterisation}, $\langle\mathcal{T}^{\prime}\rangle = \Ideal$.
\end{proof}

\subsection{Sequences of transformations of a given monomial}
\label{monomial_transformations_section}
Let $U$ be a monomial, $a_h$ and $b_h$ be virtual members of the chart of $U = L_aa_hR_a = L_bb_hR_b$. To be precise, assume that the beginning of $a_h$ lies from the left of the beginning of $b_h$. Assume $a_h$ and $a_j$, $b_h$ and $b_j$ are incident monomials and consider transformations $a_h \mapsto a_j$, $b_h \mapsto b_j$. If $a_j$ or $b_j$ is equal to $1$ and the resulting monomial $L_aR_a$ or $L_bR_b$ have cancellations, we do not perform them right after the replacement.

Let $\widetilde{a}_h$ be the image of $a_h$ in $L_bb_jR_b$. If $b_j = 1$ and $L_bR_b$ has cancellations, we say that $\widetilde{a}_h$ is the intersection of $a_h$ and $L_b$ (notice that this slightly differs from Definition~\ref{corresponding_occurence}). Then $\widetilde{a}_h = a_hd_1$, where $d_1$ is a generalized fractional power that cancels or prolongs $a_h$ ($d_1$ may be empty). Let $\widetilde{b}_h$ be the image of $b_h$ in $L_aa_jR_a$. If $a_j = 1$ and $L_aR_a$ has cancellations, we say that $\widetilde{b}_h$ is the intersection of $b_h$ and $R_a$. Then $\widetilde{b}_h = d_2b_h$, where $d_2$ is a generalized fractional power that cancels or prolongs $b_h$ ($d_2$ may be empty).

Since the monomials $a_h$ and $a_j$ are incident, the monomials $a_hd_1 = \widetilde{a}_h$ and $a_jd_1$ are also incident. Hence, we consider the replacement $\widetilde{a}_h \mapsto \widetilde{a}_j$ in $L_bb_hR_b$, where $\widetilde{a}_j = a_jd_1$. Similarly, since the monomials $b_h$ and $b_j$ are incident, the monomials $d_2b_h = \widetilde{b}_h$ and $d_2b_j$ are also incident. So, we consider the replacement $\widetilde{b}_h \mapsto \widetilde{b}_j$ in $L_aa_hR_a$, where $\widetilde{b}_j = d_2b_j$. The next lemma states an important property of these transformations.

\begin{lemma}
\label{replacements_diamond_property}
The result of the replacement $\widetilde{a}_h \mapsto \widetilde{a}_j$ in $L_bb_jR_b$ and the result of the replacement $\widetilde{b}_h \mapsto \widetilde{b}_j$ in $L_aa_jR_a$ are equal.
\end{lemma}
\begin{proof}
First assume that $a_h$ and $b_h$ are separated, $U = L_aa_hMb_hR_b$. Then we obtain two sequences of transformations
\begin{equation*}
U = L_aa_hMb_hR_b \mapsto L_aa_jMb_hR_b \mapsto L_aa_jMb_jR_b
\end{equation*}
and
\begin{equation*}
U = L_aa_hMb_hR_b \mapsto L_aa_hMb_jR_b \mapsto L_aa_jMb_jR_b.
\end{equation*}
Obviously, the results are equal.

Assume $a_h$ and $b_h$ touch at a point, $U = L_aa_hb_hR_b$. After the replacement $a_h \mapsto a_j$ we obtain the monomial $ L_aa_jb_hR_b$. The image of $b_h$ in $L_aa_jb_hR_b$ may have different forms depending on $a_j$ (see classification in Section~\ref{mt_configurations}). But in any case the replacement $\widetilde{b}_h \mapsto \widetilde{b}_j$ in $L_aa_jb_hR_b$ can be represented as a replacement $b_h \mapsto b_j$ and the further cancellations (if there are any). So, we obtain
\begin{equation*}
U = L_aa_hb_hR_b \mapsto L_aa_jb_hR_b \mapsto L_aa_jb_jR_b.
\end{equation*}
After the replacement $b_h \mapsto b_j$ we obtain the monomial $ L_aa_hb_jR_b$. Similarly, the replacement $\widetilde{a}_h \mapsto \widetilde{a}_j$ in $L_aa_hb_jR_b$ can be represented as a replacement $a_h \mapsto a_j$ and the further cancellations (if there are any). So, we obtain
\begin{equation*}
U = L_aa_hb_hR_b \mapsto L_aa_hb_jR_b \mapsto L_aa_jb_jR_b
\end{equation*}
and results of both sequences are equal.

Assume $a_h$ and $b_h$ have an overlap $c$. Denote by $a_h^{\prime}$ the intersection of $a_h$ and $L_b$, by $b_h^{\prime}$ the intersection of $b_h$ and $R_a$, then $U = L_aa_h^{\prime}cb_h^{\prime}R_b$. After the replacement $a_h \mapsto a_j$ we obtain the monomial $ L_aa_jb_h^{\prime}R_b$. As above, the replacement $\widetilde{b}_h \mapsto \widetilde{b}_j$ in $L_aa_jb_h^{\prime}R_b$ can be represented as a replacement $b_h^{\prime} \mapsto c^{-1}b_j$ and the further cancellations (if there are any). So, we obtain
\begin{equation*}
U = L_aa_h^{\prime}cb_h^{\prime}R_b \mapsto L_aa_jb_h^{\prime}R_b \mapsto L_aa_jc^{-1}b_jR_b.
\end{equation*}
After the replacement $b_h \mapsto b_j$ we obtain the monomial $L_aa_h^{\prime}b_jR_b$. As above, the replacement $\widetilde{a}_h \mapsto \widetilde{a}_j$ in $L_aa_h^{\prime}b_jR_b$ can be represented as a replacement $a_h^{\prime} \mapsto a_jc^{-1}$ and the further cancellations (if there are any). So, we obtain
\begin{equation*}
U = L_aa_h^{\prime}cb_h^{\prime}R_b \mapsto L_aa_h^{\prime}b_jR_b \mapsto L_aa_jc^{-1}b_jR_b
\end{equation*}
 and results of both sequences are equal.
 \end{proof}

\begin{lemma}
\label{virtual_members_replacements_stability}
Let $U$ be a monomial, $a_h$ and $b_h$ be virtual members of the chart of $U$, $U = L_aa_hR_a = L_bb_hR_b$. Let $a_h$ and $a_j$ be incident monomials, $b_h$ and $b_j$ be incident monomials, and consider transformations $a_h \mapsto a_j$, $b_h\mapsto b_j$. Assume $a_j$ is a virtual member of the chart of $L_aa_jR_a$, $b_j$ is a virtual member of the chart of $L_bb_jR_b$. Denote by $\widetilde{b}_h$ the image of $b_h$ in $L_aa_jR_a$ and consider the transformation $\widetilde{b}_h \mapsto \widetilde{b}_j$, obtained by multiplying of $b_h\mapsto b_j$ by the corresponding generalized fractional power. Let us apply this replacement to the monomial $L_aa_jR_a$ and denote the result by $U_{a, b}$. Then $\widetilde{b}_j$ is a virtual member of the chart of $U_{a, b}$.
\end{lemma}
\begin{proof}
To be precise, assume that the beginning of $a_h$ lies from the left of the beginning of $b_h$. We will consider only the most interesting case when $a_h$ and $b_h$ have an overlap $c$. The cases when $a_h$ and $b_h$ are separate or touch at a point are considered similarly.

Denote by $a_h^{\prime}$ the intersection of $a_h$ and $L_b$, that is, $a_h = a_h^{\prime}c$. Denote by $b_h^{\prime}$ the intersection of $b_h$ and $R_a$, that is, $b_h = cb_h^{\prime}$.
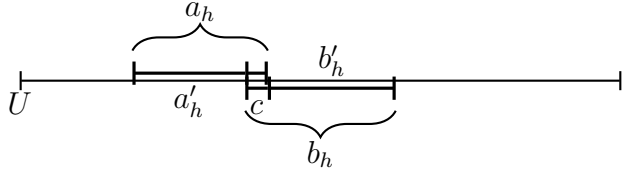
\begin{center}
\begin{tikzpicture}
\draw[|-|, black, thick] (-1,0) to node[at start, below] {$U$} (7,0);
\draw[|-, black, very thick] (0.5,0.1)--(2,0.1) node[midway, below] {$a_h^{\prime}$};
\draw[|-|, black, very thick] (2,0.1)--(2.3,0.1);
\draw [thick, decorate, decoration={brace, amplitude=10pt}] (0.5,0.4)--(2.3,0.4) node [midway, above, yshift=7] {$a_h$};
\draw[|-, black, very thick] (2,-0.1)--(2.3,-0.1) node[midway, below] {$c$};
\draw[|-|, black, very thick] (2.3,-0.1)--(4,-0.1) node[midway, above] {$b_h^{\prime}$};
\draw [thick, decorate, decoration={brace, amplitude=10pt, mirror}] (2,-0.4)--(4,-0.4) node [midway, below, yshift=-7] {$b_h$};
\end{tikzpicture}
\end{center}
Let us perform the replacement $a_h \mapsto a_j$ in $U = L_aa_hR_a$. Then the image of $b_h$ in $L_aa_jR_a$ is equal to $c_1b_h^{\prime}$, where $c_1$ is an overlap with $a_j$ (possibly empty). We put $a_j = a_j^{\prime}c_1$ and $\widetilde{b}_h = c_1b_h^{\prime}$.
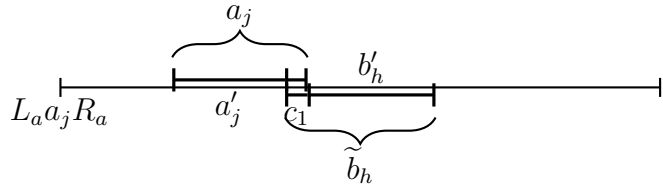
\begin{center}
\begin{tikzpicture}
\draw[|-|, black, thick] (-1,0) to node[at start, below] {$ L_aa_jR_a$} (7,0);
\draw[|-, black, very thick] (0.5,0.1)--(2,0.1) node[midway, below] {$a_j^{\prime}$};
\draw[|-|, black, very thick] (2,0.1)--(2.3,0.1);
\draw [thick, decorate, decoration={brace, amplitude=10pt}] (0.5,0.4)--(2.3,0.4) node [midway, above, yshift=7] {$a_j$};
\draw[|-, black, very thick] (2,-0.1)--(2.3,-0.1) node[midway, below] {$c_1$};
\draw[|-|, black, very thick] (2.3,-0.1)--(4,-0.1) node[midway, above] {$b_h^{\prime}$};
\draw [thick, decorate, decoration={brace, amplitude=10pt, mirror}] (2,-0.4)--(4,-0.4) node [midway, below, yshift=-7] {$\widetilde{b}_h$};
\end{tikzpicture}
\end{center}
Since $a_j$ is a virtual member of the chart of $L_aa_jR_a$, from Corollary~\ref{virtual_members_stability} it follows that $\widetilde{b}_h$ is also a virtual member of the chart of $L_aa_jR_a$. The monomials $b_h$ and $b_j$ are incident. Hence, multiplying them by $c_1c^{-1}$ from the left, we obtain incident monomials $c_1b_h^{\prime}$ and $c_1c^{-1}b_j$. So, we have the replacement $c_1b_h^{\prime} = \widetilde{b}_h \mapsto \widetilde{b}_j$, where $\widetilde{b}_j = c_1c^{-1}b_j$.

Let us perform the replacement $b_h \mapsto b_j$ in $U = L_bb_hR_b$. Then the image of $a_h$ in $L_bb_jR_b$ is equal to $a_h^{\prime}c_2$, where $c_2$ is an overlap with $b_j$ (possibly empty). We put $b_j = c_2b_j^{\prime}$ and $\widetilde{a}_h = a_h^{\prime}c_2$.
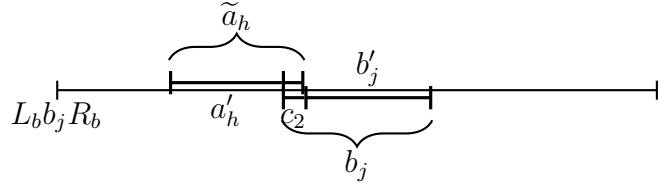
\begin{center}
\begin{tikzpicture}
\draw[|-|, black, thick] (-1,0) to node[at start, below] {$ L_bb_jR_b$} (7,0);
\draw[|-, black, very thick] (0.5,0.1)--(2,0.1) node[midway, below] {$a_h^{\prime}$};
\draw[|-|, black, very thick] (2,0.1)--(2.3,0.1);
\draw [thick, decorate, decoration={brace, amplitude=10pt}] (0.5,0.4)--(2.3,0.4) node [midway, above, yshift=7] {$\widetilde{a}_h$};
\draw[|-, black, very thick] (2,-0.1)--(2.3,-0.1) node[midway, below] {$c_2$};
\draw[|-|, black, very thick] (2.3,-0.1)--(4,-0.1) node[midway, above] {$b_j^{\prime}$};
\draw [thick, decorate, decoration={brace, amplitude=10pt, mirror}] (2,-0.4)--(4,-0.4) node [midway, below, yshift=-7] {$b_j$};
\end{tikzpicture}
\end{center}
Since $b_j$ is a virtual member of the chart of $L_bb_jR_b$, from Corollary~\ref{virtual_members_stability} it follows that $\widetilde{a}_h$ is also a virtual member of the chart of $L_bb_jR_b$. The monomials $a_h$ and $a_j$ are incident. Hence, multiplying them by $c^{-1}c_2$ from the right, we obtain incident monomials $a_h^{\prime}c_2$ and $a_jc^{-1}c_2$. So, we have the replacement $a_h^{\prime}c_2 = \widetilde{a}_h \mapsto \widetilde{a}_j$, where $\widetilde{a}_j = a_jc^{-1}c_2$.

From Lemma~\ref{replacements_diamond_property} it follows that the sequences of replacements
\begin{equation*}
U = L_aa_hR_a \mapsto L_aa_jR_a \mapsto L_aa_j^{\prime}\widetilde{b}_jR_b
\end{equation*}
and
\begin{equation*}
U = L_bb_hR_b \mapsto L_bb_jR_b \mapsto L_a\widetilde{a}_jb_j^{\prime}R_b
\end{equation*}
give the same final result. So, we have
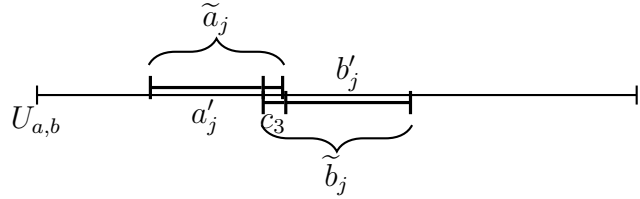
\begin{center}
\begin{tikzpicture}
\draw[|-|, black, thick] (-1,0) to node[at start, below] {$U_{a, b}$} (7,0);
\draw[|-, black, very thick] (0.5,0.1)--(2,0.1) node[midway, below] {$a_j^{\prime}$};
\draw[|-|, black, very thick] (2,0.1)--(2.3,0.1);
\draw [thick, decorate, decoration={brace, amplitude=10pt}] (0.5,0.4)--(2.3,0.4) node [midway, above, yshift=7] {$\widetilde{a}_j$};
\draw[|-, black, very thick] (2,-0.1)--(2.3,-0.1) node[midway, below] {$c_3$};
\draw[|-|, black, very thick] (2.3,-0.1)--(4,-0.1) node[midway, above] {$b_j^{\prime}$};
\draw [thick, decorate, decoration={brace, amplitude=10pt, mirror}] (2,-0.4)--(4,-0.4) node [midway, below, yshift=-7] {$\widetilde{b}_j$};
\end{tikzpicture}
\end{center}
where the overlap $c_3$ is possibly empty. One can easily calculate that $c_3 = c_1c^{-1}c_2$.

Recall that the $\LM$-measure of any overlap of two maximal occurrences of generalized fractional powers is equal either to $\varepsilon$, or to zero. Since $a_j = a_j^{\prime}c_1$ is a virtual member of the chart of $L_aa_jR_a$, $\LM(a_j) \geqslant \tau - 2\varepsilon$. If $\LM(c_1) = \varepsilon$, then, according to Remark~\ref{replacements_property_remark}, $\LM(a_j) \geqslant \tau - \varepsilon$. Hence, $\LM(a_j^{\prime}) \geqslant \tau - 2\varepsilon$. If $\LM(c_1) = 0$, then, clearly, $\LM(a_j^{\prime}) = \LM(a_j) \geqslant \tau - 2\varepsilon$.

We have $\LM(\widetilde{a}_j) \geqslant \LM(a_j^{\prime}) \geqslant \tau - 2\varepsilon$. Since $b_j$ is a virtual member of the chart of $L_bb_jR_b$, there exists a sequence of replacements from Definition~\ref{replacements} starting from $L_bb_jR_b$ such that the $\LM$-measure of the image of $b_j$ in the last monomial of the sequence is not less than $\tau$. Since $\LM(\widetilde{a}_j) \geqslant \tau - 2\varepsilon$, we can consider the replacement $\widetilde{a}_j \mapsto a_h^{\prime}c_2 = \widetilde{a}_h$ in the monomial $U_{a, b}$. Then we obtain the transformation $U_{a, b} \mapsto L_bb_jR_b$. We add this transformation to the beginning of the sequence. Since the image of $\widetilde{b}_j$ in $L_bb_jR_b$ is equal to $b_j$, as a result we obtain the sequence of replacements starting from $U_{a, b}$ such that the $\LM$-measure of the image of $\widetilde{b}_j$ in the last monomial is not less than $\tau$. Hence, $\widetilde{b}_j$ is a virtual member of the chart of $U_{a, b}$.
\end{proof}

Applying Lemma~\ref{replacements_diamond_property} and Lemma~\ref{virtual_members_replacements_stability}, we obtain the following statement, that we will use in the next section.
\begin{corollary}
\label{virtual_members_many_replacements}
Let $U$ be a monomial.
\begin{enumerate}
\item[$(1)$]
Assume we have a sequence of replacements starting from $U$ such that every replacement transforms a virtual member of the chart into a virtual member of the chart. Then any replacement can be moved to any position in the sequence and the final result remains the same. Moreover, after changing of the order of the replacements, every replacement in the obtained sequence still transforms a virtual member of the chart into a virtual member of the chart.
\item[$(2)$]
Let $a_h^{(1)}, \ldots, a_h^{(n)}$ be virtual members of the chart of $U$. Consider a number of replacements $a_h^{(i)} \mapsto a_j^{(i)}$ in $U$ such that each $a_j^{(i)}$ is a virtual member of the chart of the resulting monomial. Suppose these transformation are applied consecutively. Then every transformation in the chain also transforms a virtual member of the chart into a virtual member of the chart. Moreover, we can apply the replacements in any order.
\end{enumerate}
\end{corollary}

\section{The structure of certain subspaces of $\mathbb{Z}_2\Fr / \Ideal$: filtration, grading and tensor products}
\label{structure_calc}
First let us introduce the notion of derived monomials.
\begin{definition}[\textbf{derived monomials}]
\label{derived_monomials}
Let $U$ be a monomial. Consider the following transformations of $U$:
\begin{enumerate}[label={(\arabic*)}]
\item
\label{repl1}
Replacements of a virtual member of the chart by an incident monomial non-equal to $1$. Recall that in this case the result is always a reduced monomial.
\item
\label{repl2}
Replacements of a virtual member of the chart by incident monomial equal to $1$ and further cancellations (in order to obtain the reduced monomial).
\end{enumerate}
Starting with a certain monomial $U$ we consecutively apply transformations \ref{repl1}, \ref{repl2}. All the monomials that we obtain after some sequence of transformations \ref{repl1}, \ref{repl2} (including the monomial $U$ itself) are called \emph{derived monomials of $U$}.
\end{definition}

\begin{definition}
Let $\lbrace U_i \rbrace$ be either finite or countable set of monomials. By $\langle U_1, \ldots, U_k, \ldots\rangle_d$, we denote a subspace of $\mathbb{Z}_2\Fr$ generated by all the derived monomials of the monomials $\lbrace U_i \rbrace$.
\end{definition}

\begin{remark}
Assume $U$ is a monomial and $U \in \langle U_1, \ldots, U_k, \ldots\rangle_d$, where $U_i$ are monomials. Then $U = \sum_{j = 1}^l \widetilde{U}_{i_j}$, where $\widetilde{U}_{i_j}$ is a derived monomial of $U_{i_j}$. Since monomials form a basis of $\mathbb{Z}_2\Fr$, we obtain $U = \widetilde{U}_{i_{j_0}}$. So, we obtain that $U \in \langle U_1, \ldots, U_k, \ldots\rangle_d$ if and only if $U$ is a derived monomial of some $U_i$.
\end{remark}

We have shown that $\Ideal = \langle\mathcal{T}^{\prime}\rangle$. Consider $\sum_{j = 1}^l U_j = T \in \mathcal{T}^{\prime}$, where $U_h \mapsto \sum_{\substack{j = 1 \\ j\neq h}}^{k}U_j$ is a multi-turn of a virtual member of the chart of $U_h$. Then, by definition of $\langle U_h\rangle_d$, we obtain $T \in \langle U_h\rangle_d$.

In this section, we show that $\langle U_1, \ldots, U_k\rangle_d \cap \Ideal$ is generated by supports of multi-turns of monomials from $\langle U_i\rangle_d$ for $i = 1, \ldots, k$. This enables us to construct a linear basis of $\langle U_1, \ldots, U_k\rangle_d / (\langle U_1, \ldots, U_k\rangle_d \cap \Ideal)$ and a linear basis of $\mathbb{Z}_2\Fr / \Ideal$ (and therefore show that $\mathbb{Z}_2\Fr / \Ideal$ is non-trivial).

Our approach can be compared with the more standard one, that uses the Diamond Lemma to control the ring relations (\cite{Bergman}). Instead, we introduce a filtration (and the corresponding grading) on $\langle U_1, \ldots, U_k\rangle_d$ and show its compatibility with the subspaces of linear dependencies (see Lemma~\ref{fall_through_linear_dep} and Theorem~\ref{structure_of_quotient_space}). This allows us to deal with linear dependencies in each graded component independently. The structure of each graded component is described in Proposition~\ref{correspondence_to_tensor_product} and Proposition~\ref{component_subspaces_structure}.

Notice that the Diamond Lemma can also be reformulated in the language of grading.

\subsection{The filtration on spaces $\langle U_1, \ldots, U_k\rangle_d$}
\label{filtration_subsection}
Let $Z$ be a monomial. We introduce the following numerical characteristics of $Z$ ($f$-characteristics of monomials):
\begin{equation}
f(Z) = (N(Z), K_{\tau}(Z)),
\end{equation}
where $N(Z)$ is the number of elements in a minimal covering of $Z$, $K_{\tau}(Z)$ is the number of virtual members of the chart of $Z$. If $Z_1$ and $Z_2$ are monomials, we say that $f(Z_1) < f(Z_2)$ if and only if $N(Z_1) < N(Z_2)$ or $N(Z_1) = N(Z_2)$ and $K_{\tau}(Z_1) < K_{\tau}(Z_2)$.

The characteristics $f$ satisfies the following property.
\begin{lemma}
\label{estimation_value_property}
Assume $U$ and $Z$ are monomials, where $Z$ is a derived monomial of $U$. Then $f(Z) \leqslant f(U)$. Moreover, $f(Z) < f(U)$ if and only if in the corresponding sequence of replacements there exists at least one replacement of the form $La_hR \to La_jR$ such that $a_h$ is a virtual member of the chart of $La_hR$ and $a_j$ is not a virtual member of the chart of $La_jR$.
\end{lemma}
\begin{proof}
It follows directly from Corollary~\ref{decreasing_property} and Corollary~\ref{virtual_members_stability}.
\end{proof}

In this section, let us set
\begin{equation*}
V = \langle U_1, \ldots, U_k\rangle_d, U_1, \ldots, U_k \in \Fr.
\end{equation*}
We will define a decreasing filtration on $V$.

First consider a space generated by one monomial and its derived monomials, namely, $W= \langle U\rangle_d$, $U\in \Fr$. Let us define a subspace $\Low(W) \subseteq W$. We put
\begin{equation}
\label{lower_level_subspace_one}
\Low(W) = \langle Z \in \Fr \mid Z \textit{ is a derived monomial of } U \textit{ such that } f(Z) < f(U) \rangle.
\end{equation}
If the set of monomials with strictly smaller $f$-characteristics than $f(U)$ is empty, by definition, we put $\Low(W) = 0$.

Let a monomial $Z^{\prime} \in \Low(W)$, $Z^{\prime\prime}$ be a derived monomial of $Z^{\prime}$. Then, by definition of $\Low(W)$ and Lemma~\ref{estimation_value_property}, $f(U) > f(Z^{\prime}) \geqslant f(Z^{\prime\prime})$. Hence, $Z^{\prime\prime}$ also belongs to $\Low(W)$, that is, $\Low(W)$ is closed under taking derived monomials.

Now consider a space $Y = \langle Z_1, \ldots, Z_k, \ldots\rangle_d$, where $\lbrace Z_i \rbrace$ is either a finite or infinite set of monomials. By definition, we put
\begin{equation}
\label{lower_level_subspace}
\Low(Y) = \sum\limits_{i} \Low(\langle Z_i \rangle_d).
\end{equation}
Since every $\Low(\langle Z_i \rangle_d)$ is generated by monomials and closed under taking derived monomials, the space $\Low(Y)$ is generated by monomials and closed under taking derived monomials as well. Hence,
\begin{equation}
\label{lower_level_form}
\Low(Y) = \langle Z_1^{\prime}, \ldots, Z_k^{\prime}, \ldots\rangle_d \textit{ for some } Z_i^{\prime} \in \Fr.
\end{equation}

Notice that one generalized fractional power may have an infinite number of incident monomials of $\LM$-measure less than $\tau - 2\varepsilon$, because they may contain different powers of $w$. Recall that such generalized fractional powers are never counted as a virtual members of the chart. Hence, even if the space $Y$ is generated by derived monomials of finite number of monomials, the space $\Low(Y) $ might be generated by derived monomials of countably many monomials.

\medskip

In the sequel, we will widely use the following simple properties of derived monomials.
\begin{lemma}
\label{derived_spaces_equality}
Let $Z_1$ be a monomial, $Z$ be a derived monomial of $Z_1$, $Y \subseteq \mathbb{Z}_2\Fr$ be a space generated by monomials and closed under taking derived monomials. Then the following statements hold:
\begin{enumerate}
\item[$(1)$]
If $Z \in \langle Z_1 \rangle_d \setminus \Low(\langle Z_1 \rangle_d)$, then $\langle Z_1\rangle_d = \langle Z\rangle_d$.
\item[$(2)$]
If $Z \in Y$ and $Z \in \langle Z_1 \rangle_d \setminus \Low(\langle Z_1 \rangle_d)$, then $\langle Z_1 \rangle_d \subseteq Y$.
\end{enumerate}
\end{lemma}
\begin{proof}
Suppose $Z_1^{\prime}$ is a result of transformation~\ref{repl1} (Definition \ref{derived_monomials}) such that $Z_1^{\prime}$ is contained in $\langle Z_1 \rangle_d \setminus \Low(\langle Z_1 \rangle_d)$. Then this transformation is invertible, hence $Z_1$ is also a derived monomial of $Z_1^{\prime}$. Repeating this argument for every transformation in a sequence that connects $Z_1$ and $Z$, we obtain that $Z_1$ is a derived monomial of $Z$. Now assume $U$ is a derived monomial of $Z_1$. Since $Z_1$ is a derived monomial of $Z$, $U$ is also a derived monomial of $Z$. Hence $\langle Z_1\rangle_d = \langle Z\rangle_d$. The first statement of the lemma is proved.

Now let us prove the second statement. Since $Z \in Y$ and $Y$ is generated by monomials and closed under taking derived monomials, we obtain $\langle Z\rangle_d \subseteq Y$. But above we proved $\langle Z\rangle_d = \langle Z_1\rangle_d$, hence, $\langle Z_1\rangle_d \subseteq Y$.
\end{proof}

We defined the subspace $\Low(Y)$ using a set of generators of $Y$ (see formula~\eqref{lower_level_subspace}). Let us show that, in fact, $\Low(Y)$ does not depend on the set of generators of $Y$.
\begin{proposition}
\label{lower_level_gen_independence}
Assume
\begin{equation*}
Y = \langle Z_1, \ldots, Z_k, \ldots\rangle_d = \langle Z_1^{\prime}, \ldots, Z_k^{\prime}, \ldots\rangle_d,
\end{equation*}
where $\lbrace Z_i\rbrace$ and $\lbrace Z_i^{\prime}\rbrace$ are either finite or infinite sets of monomials. Then
\begin{equation*}
\sum\limits_{i}\Low(\langle Z_i\rangle_d) = \sum\limits_{i}\Low(\langle Z_i^{\prime}\rangle_d).
\end{equation*}
\end{proposition}
\begin{proof}
Let us enumerate all monomials from $Y$. Let $\lbrace X_j\rbrace$ be the set of all monomials of the space $Y$. Then, evidently, we obtain the following description of~$Y$:
\begin{equation*}
Y = \langle X_1, \ldots, X_k, \ldots\rangle = \langle X_1, \ldots, X_k, \ldots\rangle_d.
\end{equation*}
Clearly, it is sufficient to show that
\begin{equation*}
\sum\limits_{i}\Low(\langle Z_i\rangle_d) = \sum\limits_{j}\Low(\langle X_j\rangle_d)
\end{equation*}
for an arbitrary set of monomials $\lbrace Z_i \rbrace$ such that $Y = \langle Z_1, \ldots, Z_k, \ldots\rangle_d$.

Since $\lbrace Z_1, \ldots, Z_k, \ldots\rbrace \subseteq \lbrace X_1, \ldots, X_k, \ldots\rbrace$, clearly,
\begin{equation*}
\sum\limits_{i}\Low(\langle Z_i\rangle_d) \subseteq \sum\limits_{j}\Low(\langle X_j\rangle_d).
\end{equation*}

Assume $Z \in \sum_{j}\Low(\langle X_j\rangle_d)$. Monomials form a basis of $\mathbb{Z}_2\Fr$ and every $\Low(\langle X_j\rangle_d)$ is generated by monomials. Hence, $Z$ is a derived monomial of some $X_{j_0}$ such that we have $Z\in \Low(\langle X_{j_0}\rangle_d)$ for some $j_0$. Therefore, there exists a sequence of transformations $\alpha_1, \ldots, \alpha_{s_1}$ of type~\ref{repl1}, \ref{repl2} such that
\begin{equation}
\label{decr_seq}
X_{j_0} \overset{\alpha_1}{\longmapsto} \ldots \overset{\alpha_{s_1}}{\longmapsto} Z.
\end{equation}
Moreover, in the sequence, there exists at least one transformation $\alpha_l : Y_{l-1} \to Y_l$ such that $f(Y_{l}) < f(Y_{l - 1})$.

Since $X_{j_0} \in Y = \langle Z_1, \ldots, Z_k, \ldots\rangle_d$, we have $X_{j_0} \in \langle Z_{i_0}\rangle_d$ for some $i_0$. Hence, there exists a sequence of transformations $\beta_1, \ldots, \beta_{s_2}$ of type~\ref{repl1}, \ref{repl2} such that
\begin{equation}
\label{arbitr_seq}
Z_{i_0} \overset{\beta_1}{\longmapsto} \ldots \overset{\beta_{s_2}}{\longmapsto} X_{j_0}.
\end{equation}
Gluing~\eqref{arbitr_seq} and~\eqref{decr_seq}, we obtain
\begin{equation*}
Z_{i_0} \overset{\beta_1}{\longmapsto} \ldots \overset{\beta_{s_2}}{\longmapsto} X_{j_0} \overset{\alpha_1}{\longmapsto} \ldots \overset{\alpha_{s_1}}{\longmapsto} Z.
\end{equation*}
Consequently, since there exists $\alpha_l$ that decreases the value of the function $f$, we have $Z\in \Low(\langle Z_{i_0}\rangle_d)$. Hence, $Z \in \sum_{i}\Low(\langle Z_i\rangle_d)$ and
\begin{equation*}
\sum\limits_{j}\Low(\langle X_j\rangle_d) \subseteq \sum\limits_{i}\Low(\langle Z_i\rangle_d).
\end{equation*}
This completes the proof.
\end{proof}

Now we define a decreasing filtration on $V = \langle U_1, \ldots, U_k\rangle_d$ in the following way. By definition, put
\begin{align}
\begin{split}
\label{filtration_def}
 \Ft_0V &= V,\\
\Ft_{n + 1}V& = \Low(\Ft_nV).
\end{split}
\end{align}
Since for every $Y = \langle Z_1, \ldots, Z_k, \ldots\rangle_d$, where $\lbrace Z_i \rbrace$ is either a finite or infinite set of monomials, we have description~\eqref{lower_level_form}, therefore formula~\eqref{filtration_def} is applicable for every $n \geqslant 0$. Here we mean $\Low(0) = 0$.

\begin{proposition}
\label{filtration_finite_property}
The filtration defined above has finitely many levels, that is, there exists a number $N$ such that $\Ft_NV = 0$. Moreover, we never have a situation $\Ft_{n + 1}V = \Ft_nV$ for $\Ft_nV \neq 0$.
\end{proposition}
\begin{proof}
Recall that $V = \langle U_1, \ldots, U_k\rangle_d$, $U_1, \ldots, U_k \in \Fr$. We put
\begin{equation*}
N_{\max} = \max\limits_{i \in \lbrace 1, \ldots, k\rbrace}(N(U_i)).
\end{equation*}
Consider $f$-characteristics of monomials in $V$. Let $U$ be an arbitrary monomial from $V$. From Lemma~\ref{estimation_value_property} it follows that $N(U) \leqslant N_{\max}$. Since every virtual member of the chart is of $\LM$-measure not less than $\tau - 2\varepsilon$, it is contained in every covering of $U$. Hence, we evidently obtain $K_{\tau}(U) \leqslant N(U) \leqslant N_{\max}$. Therefore, there are finitely many different values of $f$-characteristics of monomials in $V$. Then
\begin{equation*}
m^{(n)} = \max\limits_{U \in \Ft_nV} f(U)
\end{equation*}
is finite for any $n$ such that $\Ft_nV \neq 0$.

Recall that, by definition, $\Ft_{n + 1}V = \Low(\Ft_nV)$. Then from~\eqref{lower_level_subspace_one} and~\eqref{lower_level_subspace} it follows that $m^{(n)} > m^{(n + 1)}$ if $\Ft_{n + 1}V \neq 0$. So, we have a strictly decreasing sequence of $f$-characteristics
\begin{equation}
\label{decreasing_maximums}
m^{(0)} > \ldots > m^{(n)} > \ldots
\end{equation}
that corresponds to the decreasing sequence of non-trivial spaces
\begin{equation*}
\Ft_0V \supseteq \ldots \supseteq \Ft_nV\supseteq \ldots\, .
\end{equation*}
The sequence~\eqref{decreasing_maximums} can not be infinite, so, it ends up at some step $n_0$. This means that $\Ft_{n_0}V$ is the last non-trivial subspace of the filtration. That is, the filtration has finitely many non-zero levels.

Assume $\Ft_{n + 1}V = \Ft_nV$ for $\Ft_nV \neq 0$. That is, $\Ft_nV = \Low(\Ft_nV)$. Then by induction we have $\Ft_{n + k}V = \Ft_nV \neq 0$ for any $k \in \mathbb{N}$. But we already proved that the filtration has finitely many non-zero levels. This contradiction completes the proof.
\end{proof}

\begin{definition}
\label{subspace_of_dependencies}
Suppose $Y$ is a subspace of $\mathbb{Z}_2\Fr$ linearly generated by an arbitrary set of monomials and closed under taking derived monomials. Every linear dependence from $\mathcal{T}^{\prime}$ is, in fact, a linear dependence between a monomial $U$ and its derived monomials. Hence, any multi-turn of a virtual member of the chart of a monomial from $Y$ generates a linear dependence between the monomials from $Y$, because $Y$ is closed under taking derived monomials. We consider the subspace of $Y$
\begin{align*}
\Dp(Y) = &\left\langle\sum\limits_{j = 1}^k U_j \mid U_h \textit{ is a monomial from }Y,\right.\\
&U_h = La_hR, \textit{ where } a_h \textit{ is a virtual member of the chart of } U_h,\\
&U_h \mapsto \sum\limits_{\substack{j = 1 \\ j\neq h}}^{k} U_j \textit{ is a multi-turn that comes from}\\
&\left.\textit{ an elementary multi-turn } a_h \mapsto \sum\limits_{\substack{j = 1 \\ j\neq h}}^{k} a_j \right\rangle
\end{align*}
and call this subspace \emph{the subspace of dependencies on $Y$}. Using this notion, $\Ideal = \langle \mathcal{T}^{\prime}\rangle = \Dp(\mathbb{Z}_2\Fr)$.

Note that if a monomial $U\in W$ has the empty chart, then there are no multi-turns of $U$. Hence, when $W$ consists only of monomials with the empty chart, by definition we put $\Dp(W) = 0$.
\end{definition}

Since $V$ is closed under taking derived monomials, we consider its subspace $\Dp(V)$. Evidently, for every monomial $U$ the space $\Low(\langle U\rangle_d)$ is linearly generated by monomials and closed under taking derived monomials. Therefore, by the definition, every $\Ft_nV$ is generated by monomials and closed under taking derived monomials. Hence, we can consider its subspace $\Dp(\Ft_nV)$ and define the filtration on $\Dp(V)$ in the following way
\begin{equation*}
\Ft_n\Dp(V) = \Dp(\Ft_nV),
\end{equation*}
that is, $\Ft_n\Dp(V)$ is the vector space generated by linear dependencies coming from monomials of $\Ft_nV$.

\begin{lemma}
\label{intersection_with_low}
Suppose $U$ is a monomial, $U\in \Ft_nV$. If $U \notin \Ft_{n + 1}V$, then
\begin{equation*}
\Ft_{n + 1}V \cap \langle U\rangle_d = \Low(\langle U\rangle_d).
\end{equation*}
\end{lemma}
\begin{proof} Let $\Ft_nV = \langle Z_1, \ldots, Z_k, \ldots\rangle_d$, where $\lbrace Z_i\rbrace$ is either a finite or infinite set of monomials.
Since $U \in \Ft_{n}V$ and $\Ft_{n}V$ is closed under taking derived monomials, we can assume that $\Ft_nV = \langle U, Z_1, \ldots, Z_k, \ldots\rangle_d$. Since the definition of $\Ft_{n + 1}V = \Low(\Ft_nV)$ does not depend on the set of generators of $\Ft_nV$ (see Proposition~\ref{lower_level_gen_independence}), we have
\begin{equation*}
\Ft_{n + 1}V = \Low(\Ft_nV) = \Low(\langle U\rangle_d) + \sum\limits_{i}\Low(\langle Z_i\rangle_d).
\end{equation*}
Let $Z \in \Low(\langle U\rangle_d)$. Then, using the last equality, we immediately obtain $Z \in \Ft_{n + 1}V$. So, $\Low(\langle U\rangle_d) \subseteq \Ft_{n + 1}V \cap \langle U\rangle_d$.

 Assume a monomial $Z \in \Ft_{n + 1}V \cap \langle U\rangle_d$ and $Z \notin \Low(\langle U\rangle_d)$. Since $U \notin \Ft_{n + 1}V$, we have $U \notin \Low(\langle Z_i\rangle_d)$ for all $i$. Monomials form a basis of $\mathbb{Z}_2\Fr$ and every $\Low(\langle Z_{i} \rangle_d)$ is generated by monomials. So, since $Z \in \Ft_{n + 1}V = \sum_{i} \Low(\langle Z_i \rangle_d)$, we have $Z \in \Low(\langle Z_{i_0} \rangle_d)$ for some $i_0$. We assumed that $Z \in \langle U\rangle_d \setminus \Low(\langle U\rangle_d)$, hence, by Lemma~\ref{derived_spaces_equality}, $\langle U\rangle_d \subseteq \Low(\langle Z_{i_0} \rangle_d)$. But that contradicts the assumption that $U \notin \Ft_{n + 1}V$. Therefore, $\Ft_{n + 1}V \cap \langle U\rangle_d \subseteq \Low(\langle U\rangle_d)$.

Thus, finally we obtain $\Ft_{n + 1}V \cap \langle U\rangle_d = \Low(\langle U\rangle_d)$.
\end{proof}

In Definition~\ref{derived_monomials}. we defined derived monomials with the use of a special set of transformations~\ref{repl1} and~\ref{repl2}. However, in the next lemma we will use a slightly wider class of transformations. Let us prove that, using this wider class of transformations of a given monomial, we still will have its derived monomials.

Namely, let $U$ be a monomial, $a_h$ and $b_h$ be virtual members of the chart of $U$. We consider two replacements $a_h \mapsto a_j$, $b_h \mapsto b_j$ ($a_h$ and $a_j$ are incident monomials, $b_h$ and $b_j$ are incident monomials). Assume that first we apply the transformation $a_h \mapsto a_j$, namely, $U = L_aa_hR_a \mapsto L_aa_jR_a$. To be precise, we suppose that the beginning of $b_h$ lies from the right of the beginning of $a_h$. Denote by $b_h^{\prime}$ the intersection of $b_h$ and $R_a$. If $a_h$ and $b_h$ are separated or touch at a point, $b_h^{\prime} = b_h$.
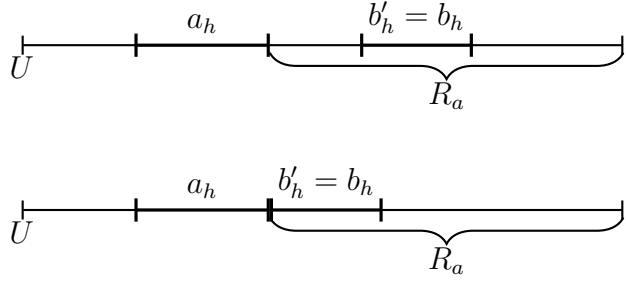
\begin{center}
\begin{tikzpicture}
\draw[|-|, black, thick] (-1,0) to node[at start, below] {$U$} (7,0);
\draw[|-|, black, very thick] (0.5,0)--(2.3,0) node[midway, above] {$a_h$};
\draw[|-|, black, very thick] (3.5,0)--(5,0) node[midway, above] {$b_h^{\prime} = b_h$};
\draw [thick, decorate, decoration={brace, amplitude=10pt, mirror}] (2.3,-0.1)--(7,-0.1) node [midway, below, yshift=-7] {$R_a$};
\end{tikzpicture}
\end{center}
\begin{center}
\begin{tikzpicture}
\draw[|-|, black, thick] (-1,0) to node[at start, below] {$U$} (7,0);
\draw[|-|, black, very thick] (0.5,0)--(2.3,0) node[midway, above] {$a_h$};
\draw[|-|, black, very thick] (2.3,0)--(3.8,0) node[midway, above] {$b_h^{\prime} = b_h$};
\draw [thick, decorate, decoration={brace, amplitude=10pt, mirror}] (2.3,-0.1)--(7,-0.1) node [midway, below, yshift=-7] {$R_a$};
\end{tikzpicture}
\end{center}
If $a_h$ and $b_h$ has an overlap $c$, then $b_h = cb_h^{\prime}$.
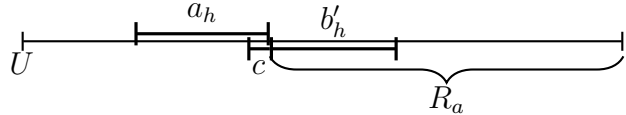
\begin{center}
\begin{tikzpicture}
\draw[|-|, black, thick] (-1,0) to node[at start, below] {$U$} (7,0);
\draw[|-|, black, very thick] (0.5,0.1)--(2.3,0.1) node[midway, above] {$a_h$};
\draw[|-, black, very thick] (2,-0.1)--(2.3,-0.1) node[midway, below] {$c$};
\draw[|-|, black, very thick] (2.3,-0.1)--(4,-0.1) node[midway, above] {$b_h^{\prime}$};
\draw [thick, decorate, decoration={brace, amplitude=10pt, mirror}] (2.3,-0.2)--(7,-0.2) node [midway, below, yshift=-7] {$R_a$};
\end{tikzpicture}
\end{center}
In any case, we can apply the transformation $b_h^{\prime} \mapsto b_j^{\prime}$ to $L_aa_jR_a$, where either $b_j^{\prime} = b_j$ if $a_h$ and $b_h$ are separated or touch at a point, or $b_j^{\prime} = c^{-1}b_j$ if $a_h$ and $b_h$ has an overlap $c$. Here we mean that if $a_j = 1$ and the monomial $L_aR_a$ has cancellations, we do not perform them. Instead, we perform a replacement $b_h^{\prime} \mapsto b_j^{\prime}$.

Assume $a_j \neq 1$ or $a_j = 1$ and $L_aR_a$ has no further cancellations. The replacement of the image of $b_h^{\prime}$ in $L_aa_jR_a$ (see Definition~\ref{corresponding_occurence}) can be represented as the replacement $b_h^{\prime} \mapsto b_j^{\prime}$ and the further cancellations if there are any (they may occur when $b_h^{\prime}$ is not a maximal occurrence of a generalized fractional power in $L_aa_jR_a$). So, if the image of $b_h^{\prime}$ is a virtual member of $L_aa_jR_a$, as a result of the transformation $b_h^{\prime} \mapsto b_j^{\prime}$ in $L_aa_jR_a$ and the further cancellations, we obtain a derived monomial of $U$.

If $a_j$ is a virtual member of the chart of $L_aa_jR_a$, then, by Corollary~\ref{virtual_members_stability}, the image of $b_h^{\prime}$ is always a virtual member of $L_aa_jR_a$. If $a_j$ is not virtual member of the chart of $La_jR$, then the image of $b_h^{\prime}$ may not be a virtual member of $L_aa_jR_a$. The following lemma states that if we apply the replacement $b_h^{\prime} \mapsto b_j^{\prime}$ to $L_aa_jR_a$ in this case, we, nevertheless, obtain a derived monomial of $U$.
\begin{lemma}
\label{additional_derived monomials}
Let $U$ be a monomial, $a_h$ and $b_h$ be virtual members of the chart of $U$. We consider two replacements $a_h \mapsto a_j$, $b_h \mapsto b_j$ ($a_h$ and $a_j$ are incident monomials, $b_h$ and $b_j$ are incident monomials). Assume $U = L_aa_hR_a$, $a_j$ is not a virtual member of the chart of $L_aa_jR_a$, $b_h^{\prime}$ is an intersection of $R_a$ and $b_h$. Then the result of the replacement $b_h^{\prime} \mapsto b_j^{\prime}$ (corresponding to $b_h \mapsto b_j$) in $L_aa_jR_a$ is a derived monomial of $U$, possibly after the further cancellations. Moreover, its $f$-characteristics is strictly less than $f(U)$.
\end{lemma}
\begin{proof}
We will consider only the most interesting case when $a_h$ and $b_h$ has an overlap $c$. Two other cases are similar.
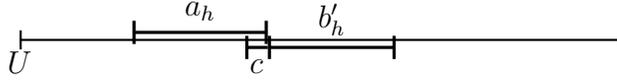
\begin{center}
\begin{tikzpicture}
\draw[|-|, black, thick] (-1,0) to node[at start, below] {$U$} (7,0);
\draw[|-|, black, very thick] (0.5,0.1)--(2.3,0.1) node[midway, above] {$a_h$};
\draw[|-, black, very thick] (2,-0.1)--(2.3,-0.1) node[midway, below] {$c$};
\draw[|-|, black, very thick] (2.3,-0.1)--(4,-0.1) node[midway, above] {$b_h^{\prime}$};
\end{tikzpicture}
\end{center}
First we consider the case when $a_j \neq 1$. Then the resulting monomial $L_aa_jR_a$ is reduced. We have the transformation $b_h \mapsto b_j$. Since $b_h = cb_h^{\prime}$, we also have the transformation $b_h^{\prime} \mapsto c^{-1}b_j$. So, we have a sequence of replacements
\begin{equation}
\label{two_replacements_res}
U = L_aa_hR_a \mapsto L_aa_jR_a = L_aa_jb_h^{\prime}R_b \mapsto L_aa_jc^{-1}b_jR_b.
\end{equation}
If $\LM(b_h) \geqslant \tau + \varepsilon$, then $\LM(b_h^{\prime}) \geqslant \tau$. The image of $b_h^{\prime}$ in $L_aa_jR_a$ prolongs $b_h^{\prime}$. Therefore, its $\LM$-measure is not less than the $\LM$-measure of $b_h^{\prime}$. So, the image of $b_h^{\prime}$ in $L_aa_jR_a$ is of $\LM$-measure $\geqslant \tau$. Therefore, it is a virtual member of the chart of $L_aa_jR_a$ and the result of~\eqref{two_replacements_res} is a derived monomial of $U$, possibly after the further cancellations.

Suppose $\LM(b_h) < \tau + \varepsilon$. Consider an elementary multi-turn $b_h \mapsto \sum_{\substack{k = 1 \\ k \neq h}}^{s}b_k$, where $b_j$ is one of the resulting monomials. In Proposition~\ref{transversality}, we proved that there exists $b_{k_0}$ such that $\LM(b_{k_0}) \geqslant \tau$. However, using the same argument, one can easily prove a slightly stronger statement, namely, that there exists $b_{k_0}$ such that $\LM(b_{k_0}) \geqslant \tau + \varepsilon$ (and, hence, ${k_0} \neq h$).

Assume $U = L_bb_hR_b$. Consider the transformation
\begin{equation*}
U = L_bb_hR_b \mapsto L_bb_{k_0}R_b.
\end{equation*}
Denote by $a_h^{\prime}$ the intersection of $a_h$ and $L_b$ in $L_bb_hR_b$. Then $a_h = a_h^{\prime}c$. Let the image of $a_h$ in $L_bb_{k_0}R_b$ be equal to $a_h^{\prime}c^{\prime}$, where $c^{\prime}$ is possibly empty. Also let us put $b_{k_0} = c^{\prime}b_{k_0}^{\prime}$. Since $b_{k_0}$ is a virtual member of the chart of $L_bb_{k_0}R_b$, from Corollary~\ref{virtual_members_stability} it follows that $a_h^{\prime}c^{\prime}$ is a virtual member of the chart of $L_bb_{k_0}R_b$.
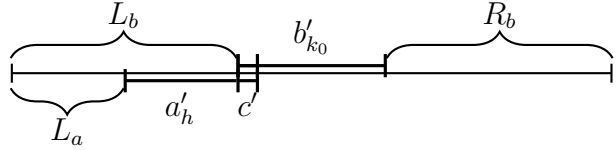
\begin{center}
\begin{tikzpicture}
\draw[|-|, black, thick] (-1,0)--(7,0);
\draw[|-, black, very thick] (0.5,-0.1)--(2,-0.1) node[midway, below] {$a_h^{\prime}$};
\draw[|-|, black, very thick] (2,-0.1)--(2.3,-0.1) node[midway, below] {$c^{\prime}$};
\draw[|-|, black, very thick] (2,0.1)--(2.3,0.1);
\draw[|-|, black, very thick] (2,0.1)--(4,0.1) node[midway, above] {$b_{k_0}^{\prime}$};
\draw [thick, decorate, decoration={brace, amplitude=10pt}] (-1,0.2)--(2,0.2) node [midway, above, yshift=7] {$L_b$};
\draw [thick, decorate, decoration={brace, amplitude=10pt}] (4,0.2)--(7,0.2) node [midway, above, yshift=7] {$R_b$};
\draw [thick, decorate, decoration={brace, amplitude=10pt, mirror}] (-1,-0.2)--(0.5,-0.2) node [midway, below, yshift=-7] {$L_a$};
\end{tikzpicture}
\end{center}

We have a transformation $a_h \mapsto a_j$. Multiplying it by $c^{-1}c^{\prime}$ from the right, we obtain $a_h^{\prime}c^{\prime} \mapsto a_j^{\prime}$, where $a_j^{\prime} = a_jc^{-1}c^{\prime}$. The monomials $b_{k_0}$ and $b_j$ are incident, so, we have a transformation $b_{k_0} \mapsto b_j$. Multiplying it by ${c^{\prime}}^{-1}$ from the left, we obtain $b_{k_0}^{\prime} \mapsto b_j^{\prime}$, where $b_j^{\prime} = {c^{\prime}}^{-1}b_j$. Since $\LM(b_{k_0}) \geqslant \tau + \varepsilon$, clearly, $\LM(b_{k_0}^{\prime}) \geqslant \tau$. Therefore, a prolongation of $b_{k_0}^{\prime}$ is always a virtual member of the chart. Hence, we obtain a sequence of replacements of virtual members of the chart
\begin{align}
\label{two_replacements_transformed_res}
L_bb_{k_0}R_b& = L_aa_h^{\prime}c^{\prime}b_{k_0}^{\prime}R_b \mapsto L_aa_j^{\prime}b_{k_0}^{\prime}R_b \\
&\quad\mapsto L_aa_j^{\prime}b_j^{\prime}R_b = L_aa_jc^{-1}c^{\prime}{c^{\prime}}^{-1}b_jR_b = L_aa_jc^{-1}b_jR_b.\nonumber
\end{align}
So, the results of~\eqref{two_replacements_res} and of~\eqref{two_replacements_transformed_res} are equal. But the result of sequence~\eqref{two_replacements_transformed_res} is a derived monomial of $U$, by definition. Since $a_j$ is not a virtual member of the chart of $L_aa_jR_a$, from Lemma~\ref{virtual_members_replacements_stability} it follows that $a_j^{\prime}$ is not a virtual member of the chart of $L_aa_j^{\prime}b_{k_0}^{\prime}R_b$. Hence, $f(L_aa_jc^{-1}b_jR_b) < f(U)$.

Consider the case when $a_j = 1$. If there are no cancellations in $L_aR_a$, then we argue as above. Suppose there are cancellations in $L_aR_a$, namely, $b_h = cb_{h, m}b_{h, f}$, $L_a = L^{\prime}_ab_{h,m}^{-1}$, where $b_{h, f}$ is possibly empty. For simplicity, assume that $L^{\prime}_ab_{h, f}R_b$ has no further cancellations.
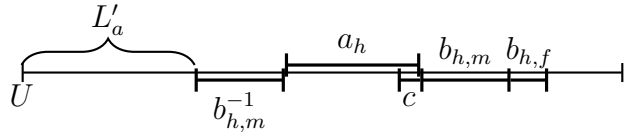
\begin{center}
\begin{tikzpicture}
\draw[|-|, black, thick] (-1,0) to node[at start, below] {$U$} (7,0);
\draw[|-|, black, very thick] (2.5,0.1)--(4.3,0.1) node[midway, above] {$a_h$};
\draw[|-, black, very thick] (4,-0.1)--(4.3,-0.1) node[midway, below] {$c$};
\draw[|-|, black, very thick] (4.3,-0.1)--(5.5,-0.1) node[midway, above] {$b_{h, m}$};
\draw[-|, black, very thick] (5.5,-0.1)--(6,-0.1) node[midway, above] {$b_{h, f}$};
\draw[|-|, black, very thick] (1.3,-0.1)--(2.5,-0.1) node[midway, below] {$b_{h, m}^{-1}$};
\draw [thick, decorate, decoration={brace, amplitude=10pt}] (-1,0.1)--(1.3,0.1) node [midway, above, yshift=7] {$L_a^{\prime}$};
\end{tikzpicture}
\end{center}
As above, let $b_h^{\prime}$ be an intersection of $b_h$ and $R_a$, that is, $b_h = cb_h^{\prime}$, $b_h^{\prime} = b_{h, m}b_{h, f}$. We have a replacement $b_h \mapsto b_j$ and the corresponding replacement $b_h^{\prime} = c^{-1}b_h \mapsto c^{-1}b_j$. After the replacement $L_aa_hR_a \mapsto L_aR_a$ there are cancellations of the resulting monomial. Let us not perform them, instead, we perform the replacement $b_h^{\prime} \mapsto c^{-1}b_j$. So, we have a sequence of replacements
\begin{equation}
\label{two_replacements_res_1}
U = L_aa_hR_a \mapsto L_aR_a = L_ab_h^{\prime}R_b = L_ac^{-1}b_jR_b.
\end{equation}

Again consider an elementary multi-turn $b_h \mapsto \sum_{\substack{k = 1 \\ k \neq h}}^{s}b_k$, where $b_j$ is one of the resulting monomials. Multiplying it by $b_{h, m}^{-1}c^{-1}$ from the left side, we obtain an elementary multi-turn $b_{h, f} \mapsto \sum_{\substack{k = 1 \\ k \neq h}}^{s}b_{h, m}^{-1}c^{-1}b_k$. If $\LM(b_{h, f}) \geqslant \tau$, then its image in $L_a^{\prime}b_{h, f}R_b$ is a virtual member or the chart. Then we again can argue as in the beginning of the proof. Assume $\LM(b_{h, f}) < \tau$. Then, by Proposition~\ref{transversality}, there exists $b_{k_1}$ such that $\LM(b_{h, m}^{-1}c^{-1}b_{k_1}) \geqslant \tau$.

We have the replacement $a_h \mapsto 1$. As above, let $a_h^{\prime}$ be an intersection of $a_h$ and $L_b$, that is, $a_h = a_h^{\prime}c$. Let $a_h^{\prime}c^{\prime\prime}$ be the image of $a_h$ in $L_bb_{k_1}R_b$ ($c^{\prime\prime}$ is possibly empty). That is, $c^{\prime\prime}$ is an overlap of $a_h^{\prime}c^{\prime\prime}$ and $b_{k_1}$, $b_{k_1} = c^{\prime\prime}b_{k_1}^{\prime}$. From Corollary~\ref{virtual_members_stability} it follows that $a_h^{\prime}c^{\prime\prime}$ is a virtual member of the chart of $L_bb_{k_1}R_b$.
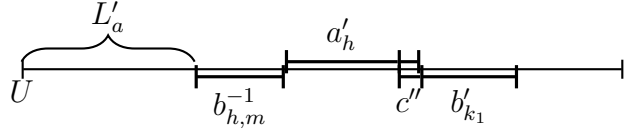
\begin{center}
\begin{tikzpicture}
\draw[|-|, black, thick] (-1,0) to node[at start, below] {$U$} (7,0);
\draw[|-, black, very thick] (2.5,0.1)--(4,0.1) node[midway, above] {$a_h^{\prime}$};
\draw[|-|, black, very thick] (4,0.1)--(4.3,0.1);
\draw[|-, black, very thick] (4,-0.1)--(4.3,-0.1) node[midway, below] {$c^{\prime\prime}$};
\draw[|-|, black, very thick] (4.3,-0.1)--(5.6,-0.1) node[midway, below] {$b_{k_1}^{\prime}$};
\draw[|-|, black, very thick] (1.3,-0.1)--(2.5,-0.1) node[midway, below] {$b_{h, m}^{-1}$};
\draw [thick, decorate, decoration={brace, amplitude=10pt}] (-1,0.1)--(1.3,0.1) node [midway, above, yshift=7] {$L_a^{\prime}$};
\end{tikzpicture}
\end{center}
So, multiplying $a_h \mapsto 1$ by $c^{-1}c^{\prime\prime}$ from the right side, we obtain the replacement $a_h^{\prime}c^{\prime\prime} \mapsto c^{-1}c^{\prime\prime}$.

The monomials $b_{k_1}$ and $b_j$ are incident, hence, we have a transformation $b_{k_1} \mapsto b_j$. So, we also have $b_{h, m}^{-1}c^{-1}b_{k_1} \mapsto b_{h, m}^{-1}c^{-1}b_j$. Since $\LM(b_{h, m}^{-1}c^{-1}b_{k_1}) \geqslant \tau$ its prolongation is always a virtual member of the chart. Therefore, we obtain the sequence of replacements of virtual members of the chart:
\begin{multline}
\label{two_replacements_transformed_res_1}
U = L_bb_hR_b \mapsto L_bb_{k_1}R_b = L_aa_h^{\prime}c^{\prime\prime}b_{k_1}^{\prime}R_b \mapsto\\
\mapsto L_ac^{-1}c^{\prime\prime}b_{k_1}^{\prime}R_b = L_a^{\prime}b_{h, m}^{-1}c^{-1}b_{k_1}R_b \mapsto L_a^{\prime}b_{h, m}^{-1}c^{-1}b_jR_b = L_ac^{-1}b_jR_b.
\end{multline}
So, the result of~\eqref{two_replacements_res_1} is equal to the result of~\eqref{two_replacements_transformed_res_1}. But the result of sequence~\eqref{two_replacements_transformed_res_1} is a derived monomial of $U$, by definition. In~\eqref{two_replacements_transformed_res_1} we had the replacement $a_h^{\prime}c^{\prime\prime} \mapsto c^{-1}c^{\prime\prime}$, where $\LM(c^{-1}c^{\prime\prime}) \leqslant 2\varepsilon$. Hence, $c^{-1}c^{\prime\prime}$ can not be a virtual member of the chart. Therefore, $f(L_ac^{-1}b_jR_b) < f(U)$.
\end{proof}

\begin{lemma}[\textbf{Main Lemma}]
\label{fall_through_linear_dep}
Let $V = \langle U_1, \ldots, U_k\rangle_d, U_1, \ldots, U_k\in \Fr$. Then we have
\begin{equation*}
\Dp(\Ft_nV) \cap \Ft_{n + 1}V = \Dp(\Ft_{n + 1}V).
\end{equation*}
\end{lemma}
\begin{proof}
We have
\begin{equation*}
\Dp(\Ft_nV) = \Dp(\Ft_{n + 1}V) + \langle T^{(n)}_1, \ldots, T^{(n)}_k, \ldots \rangle,
\end{equation*}
where $T^{(n)}_i$ are linear dependencies coming from monomials of $\Ft_nV \setminus \Ft_{n + 1}V$. Evidently, $\Dp(\Ft_nV) \cap \Ft_{n + 1}V \supseteq \Dp(\Ft_{n + 1}V)$. So, we need to show that $\Dp(\Ft_nV) \cap \Ft_{n + 1}V \subseteq \Dp(\Ft_{n + 1}V)$. Since $\Dp(\Ft_{n + 1}V) \subseteq \Ft_{n + 1}V$, it is sufficient to prove that
\begin{equation*}
\langle T^{(n)}_1, \ldots, T^{(n)}_k, \ldots \rangle \cap \Ft_{n + 1}V \subseteq \Dp(\Ft_{n + 1}V).
\end{equation*}

First let us show that it is sufficient to prove Lemma~\ref{fall_through_linear_dep} only for the case when $\Ft_nV$ is generated by one monomial and its derived monomials. Indeed, suppose Lemma~\ref{fall_through_linear_dep} is proved for this case, then let us prove it for the general case. Assume $T_1, \ldots T_l$ are linear dependencies coming from monomials of $\Ft_nV \setminus \Ft_{n + 1}V$ and $\sum_{i = 1}^l T_i \in \Ft_{n + 1}V$. So, all monomials from $\Ft_nV \setminus \Ft_{n + 1}V$ have to cancel out in this sum.

Let $T_i$ be generated by a multi-turn of a monomial $Z_i \in \Ft_nV \setminus \Ft_{n + 1}V$, $i = 1, \ldots, l$. We choose a subset $\lbrace Z_{i_1}, \ldots, Z_{i_k}\rbrace \subseteq \lbrace Z_1, \ldots, Z_l\rbrace$ such that $\langle Z_{i_1}, \ldots, Z_{i_k}\rangle_d = \langle Z_1, \ldots, Z_l\rangle_d$ and $\langle Z_{i_j} \rangle_d \nsubseteq \langle Z_{i_{j^{\prime}}} \rangle_d$ for $j \neq j^{\prime}$. The dependencies $\lbrace T_1, \ldots T_l\rbrace$ can be split into groups $\lbrace T_1^j, \ldots T_{l_j}^j\rbrace$, $j = 1, \dots k$, of dependencies coming from monomials of $\langle Z_{i_j} \rangle_d$. So, we have
\begin{equation}
\label{dep_sum}
\sum\limits_{i = 1}^l T_i = \sum\limits_{j = 1}^k\sum\limits_{i = 1}^{l_j} T_i^j \in \Ft_{n + 1}V.
\end{equation}
Every linear dependence $T_i$ comes from a monomial that belongs to $\Ft_nV \setminus \Ft_{n + 1}V$; hence, by the definition of $\Ft_{n + 1}V$, every $T_i^j$ comes from a monomial that belongs to $\langle Z_{i_j} \rangle_d \setminus \Low(\langle Z_{i_j} \rangle_d)$. From Lemma~\ref{intersection_with_low} it follows that all monomials from $\langle Z_{i_j} \rangle_d \setminus \Low(\langle Z_{i_j} \rangle_d)$ are contained in $\Ft_nV \setminus \Ft_{n + 1}V$ and $\Low(\langle Z_{i_j} \rangle_d) \subseteq \Ft_{n + 1}V$. Hence, the monomials from every $\langle Z_{i_j} \rangle_d \setminus \Low(\langle Z_{i_j} \rangle_d)$, $j = 1, \ldots, k$, have to cancel out in the sum~\eqref{dep_sum}.

Since $\langle Z_{i_j} \rangle_d \nsubseteq \langle Z_{i_{j^{\prime}}} \rangle_d$ for $j \neq j^{\prime}$, from Lemma~\ref{derived_spaces_equality} it follows that if a monomial $Z$ belongs to $\langle Z_{i_j} \rangle_d \setminus \Low(\langle Z_{i_j} \rangle_d)$, then $Z \notin \langle Z_{i_{j^{\prime}}} \rangle_d$ for $j \neq j^{\prime}$. Therefore, the monomials from $\Ft_nV \setminus \Ft_{n + 1}V$ in the sum~\eqref{dep_sum} have to cancel out in every sum $\sum_{i = 1}^{l_j} T_i^j$ separately. Hence, we obtain $\sum_{i = 1}^{l_j} T_i^j \in \Low(\langle Z_{i_j} \rangle_d)$. Since we assumed that the statement of the lemma holds for the spaces $\langle Z_{i_j}\rangle_d$, we have
\begin{equation*}
\sum\limits_{i = 1}^{l_j} T_i^j \in \Dp(\Low(\langle Z_{i_j} \rangle_d)) \subseteq \Dp(\Ft_{n + 1}V).
\end{equation*}
Thus,
\begin{equation*}
\sum\limits_{i = 1}^l T_i = \sum\limits_{j = 1}^k\sum\limits_{i = 1}^{l_j} T_i^j \in \Dp(\Ft_{n + 1}V).
\end{equation*}
So, it remains to prove the lemma only for the case when $\Ft_nV$ is generated by one monomial and its derived monomials.

We start with proving the following statement.
\begin{lemma}
\label{tensor_product_filtration}
Suppose $A_1, \ldots, A_k$ are vector spaces, $\Low(A_i) \subseteq A_i$ is a subspace of $A_i$. Suppose $D_i$ is a subspace of $A_i$ such that $D_i \cap \Low(A_i) = 0$. Consider $A_1\otimes \ldots \otimes A_k$ and its subspaces $A_1\otimes \ldots \otimes D_i \otimes \ldots \otimes A_k$. Define the subspaces
\begin{equation*}
\Low(A_1\otimes \ldots \otimes A_k) = \sum\limits_{i = 1}^{k} A_1\otimes \ldots \otimes \Low(A_i) \otimes \ldots \otimes A_k,
\end{equation*}
and
\begin{equation*}
\Low(A_1\otimes \ldots \otimes D_i \otimes \ldots \otimes A_k) = \sum\limits_{\substack{i^{\prime} = 1 \\ i \neq i^{\prime}}}^{k} A_1\otimes \ldots \otimes D_i \otimes \ldots \otimes \Low(A_{i^{\prime}}) \otimes \ldots \otimes A_k, i = 1, \dots, k.
\end{equation*}

Then
\begin{equation*}
\left(\sum\limits_{i = 1}^{k} A_1\otimes \ldots \otimes D_i \otimes \ldots \otimes A_k \right) \bigcap \Low(A_1\otimes \ldots \otimes A_k) = \sum\limits_{i = 1}^{k} \Low(A_1\otimes \ldots \otimes D_i \otimes \ldots \otimes A_k).
\end{equation*}
\end{lemma}
\begin{proof}
From the condition $D_i \cap \Low(A_i) = 0$ it follows that the sum of this subspaces is a direct sum $D_i \oplus \Low(A_i) \subseteq A_i$. Since $A_i$ is a vector space, the subspace $D_i \oplus \Low(A_i)$ has a direct complement $J_i$ and we obtain $A_i = J_i \oplus D_i \oplus \Low(A_i)$. Hence,
\begin{align*}
\Low(A_1\otimes \ldots \otimes A_k)& = \sum\limits_{i = 1}^{k} (J_1 \oplus D_1 \oplus \Low(A_1))\otimes \ldots \otimes \Low(A_i) \otimes \ldots \otimes (J_k \oplus D_k \oplus \Low(A_k)) \\ &= \sum\limits_{i = 1}^{k} (\bigoplus\limits_{j} B_1^j\otimes \ldots \otimes B_{i - 1}^j \otimes \Low(A_i) \otimes B_{i + 1}^j \otimes \ldots \otimes B_k^j),
\end{align*}
where every $B^j_i$ is either $J_i$, or $D_i$, or $\Low(A_i)$ and we encounter all possible combinations in the sum. Tensor products that contain more then one member $\Low(A_i)$ repeat in the sum. If we take every repeating space only one time, we obtain the direct sum $\Low(A_1\otimes \ldots \otimes A_k) = \bigoplus _{j} B_1^j\otimes \ldots \otimes B_k^j$, where every $B^j_i$ is either $J_i$, or $D_i$, or $\Low(A_i)$, we encounter all combinations in the sum, such that at least one $B_i^j = \Low(A_i)$. Similarly,
\begin{align*}
\sum\limits_{i = 1}^{k} A_1\otimes \ldots &\otimes D_i \otimes \ldots \otimes A_k\\ & = \sum\limits_{i = 1}^{k} (J_1 \oplus D_1 \oplus \Low(A_1))\otimes \ldots \otimes D_i \otimes \ldots \otimes (J_k \oplus D_k \oplus \Low(A_k)) \\ &= \sum\limits_{i = 1}^{k} (\bigoplus\limits_{j} C_1^j\otimes \ldots \otimes C_{i - 1}^j \otimes D_i \otimes C_{i + 1}^j \otimes \ldots \otimes C_k^j),
\end{align*}
where every $C^j_i$ is either $J_i$, or $D_i$, or $\Low(A_i)$ and we encounter all possible combinations in the sum. Tensor products that contain more then one member $D_i$ repeat in the sum. If we take every repeating space only one time, we obtain the direct sum $\sum_{i = 1}^{k} A_1\otimes \ldots \otimes D_i \otimes \ldots \otimes A_k = \bigoplus_{j} C_1^j\otimes \ldots \otimes C_k^j$, where every $C^j_i$ is either $J_i$, or $D_i$, or $\Low(A_i)$, we encounter all combinations in the sum, such that at least one $C_i^j = D_i$. Hence,
\begin{equation*}
\left(\sum\limits_{i = 1}^{k} A_1\otimes \ldots \otimes D_i \otimes \ldots \otimes A_k) \cap \Low(A_1\otimes \ldots \otimes A_k\right) = \bigoplus\limits_{j} K_1^j\otimes \ldots \otimes K_k^j ,
\end{equation*}
where every $K^j_i$ is either $J_i$, or $D_i$, or $\Low(A_i)$, we encounter all combinations in the sum, such that at least one $K_i^j = D_i$ and at least one $K_i^j = \Low(A_i)$. Therefore,
\begin{align*}
&\left(\sum\limits_{i = 1}^{k} A_1\otimes \ldots \otimes D_i \otimes \ldots \otimes A_k) \cap \Low(A_1\otimes \ldots \otimes A_k\right) \\ &\qquad\qquad = \sum\limits_{i = 1}^{k}\sum\limits_{\substack{i^{\prime} = 1 \\ i \neq i^{\prime}}}^{k} A_1\otimes \ldots \otimes D_i \otimes \ldots \otimes \Low(A_{i^{\prime}}) \otimes \ldots \otimes A_k\\ &\qquad\qquad = \sum\limits_{i = 1}^{k} \Low(A_1\otimes \ldots \otimes D_i \otimes \ldots \otimes A_k).
\end{align*}
\end{proof}

Using Lemma~\ref{tensor_product_filtration}, we continue the proof of Lemma~\ref{fall_through_linear_dep}. Assume $U$ is a monomial. Derived monomials of $U$ are defined with the use of certain sequences of replacements of virtual members of the chart. When we perform replacements that preserve $f$-characteristics of monomials, they preserve, roughly speaking, the structure of the chart. Moreover, there is no interaction between the replaced occurrence and the separated virtual members of the chart and there is a very small interaction between the replaced occurrence and its neighbours. This kind of behaviour provides the idea to consider a tensor product of linear spaces that correspond to each place of the chart of $U$.

Assume the monomial $U$ has $k$ virtual members of the chart, that is, $K_{\tau}(U) = k$. Let $a^{(i)}$ be the virtual member of the chart of $U$ placed on the $i$-th position from the beginning of $U$. Let $A_i[U]$ be a subspace of $\mathbb{Z}_2\Fr$ such that
\begin{equation}
\label{tens_prod_components}
A_i[U] = \left\langle a^{(i)}_j \mid a^{(i)} \textit{ and } a^{(i)}_j \textit{ are incident monomials}, j\in \mathbb{N}\right\rangle.
\end{equation}
Suppose $U = L_ia^{(i)}R_i$. We define a subspace $\Low_i[U] \subseteq A_i[U]$ by the following rule:
\begin{align}
\label{tens_prod_components_low}
\Low_i[U] = &\left\langle a^{(i)}_j \mid a^{(i)} \textit{ and } a^{(i)}_j \textit{ are incident monomials}, j\in \mathbb{N},\right.\\
&\left.L_ia^{(i)}_jR_i \in \Low(\langle U\rangle_d)\right\rangle.\nonumber
\end{align}

\begin{remark}
Suppose $U^{\prime}$ is a derived monomial of $U$ such that $U^{\prime} \in \langle U\rangle_d \setminus \Low(\langle U\rangle_d)$. Let us construct the spaces $A_i[U^{\prime}]$ and $\Low_i[U^{\prime}] \subseteq A_i[U^{\prime}]$ corresponding to $U^{\prime}$ as above. Then, obviously, the spaces $A_i[U^{\prime}]$ are generated by the same sets of monomials as the spaces $A_i[U]$ up to shifts of the initial and the final points in the corresponding $v$-diagram by $\varepsilon$. Moreover, using Lemma~\ref{virtual_members_replacements_stability}, one can prove that the spaces $\Low_i[U^{\prime}]$ are also generated by the same sets of monomials as the spaces $\Low_i[U]$ up to shifts of the initial and the final points in the corresponding $v$-diagram by $\varepsilon$.
\end{remark}

Although the precise forms of the spaces $A_i[U^{\prime}]$ and $\Low_i[U^{\prime}] \subseteq A_i[U^{\prime}]$ depend on the monomial $U$, in the sequel, we will omit $[U]$ in the denotation of this spaces when it does not lead to ambiguity.

We construct a linear mapping
\begin{equation*}
\label{mu_def}
\mu[U] : A_1 \otimes \ldots \otimes A_k \to \langle U\rangle_d.
\end{equation*}
Elements $b^{(1)}\otimes \ldots \otimes b^{(k)} \in A_1\otimes \ldots \otimes A_k$, where $b^{(i)} \in A_i$ are generalized fractional powers incident to $a^{(i)}$, form a basis of $A_1\otimes \ldots \otimes A_k$, because generalized fractional powers incident to $a^{(i)}$ form a basis of $A_i$. We will define $\mu[U]$ on these basis elements.

We distinguish between four possibilities:
\begin{enumerate}[label=\textbf{Case~\arabic*}, ref=Case~\arabic*]
\item
\label{mu_case1}
First we define $\mu[U]$ on elements $b^{(1)}\otimes \ldots \otimes b^{(k)}$ such that all $b^{(i)} \in A_i \setminus \Low_i$. It encodes a sequence of replacements that starts from the monomial $U$. Recall that $a^{(1)}, \ldots, a^{(k)}$ are the virtual members of the chart of $U$ enumerated from left to right. Let $U = L_1a^{(1)}R_1$. Then we start with the transformation $L_1a^{(1)}R_1 \mapsto L_1b^{(1)}R_1$. Denote by $\widehat{a}^{(2)}$ the image of $\widehat{a}^{(2)}$ in $L_1b^{(1)}R_1$. Since $\LM(b^{(1)}) > \varepsilon$, either $\widehat{a}^{(2)} = a^{(2)}$, or $\widehat{a}^{(2)} = c^{(2)}a^{(2)}$, or $\widehat{a}^{(2)} = {c^{(2)}}^{-1}a^{(2)}$, where $\LM(c^{(2)}) \leqslant \varepsilon$. The next step is the replacement of $\widehat{a}^{(2)}$ in $L_1b^{(1)}R_1$ corresponding to the transformation $a^{(2)} \mapsto b^{(2)}$. That is, the replacement $\widehat{a}^{(2)} \mapsto \widehat{b}^{(2)}$ in $L_1b^{(1)}R_1$, where either $\widehat{b}^{(2)} = b^{(2)}$ if $\widehat{a}^{(2)} = a^{(2)}$, or $\widehat{b}^{(2)} = c^{(2)}b^{(2)}$ if $\widehat{a}^{(2)} = c^{(2)}a^{(2)}$, or $\widehat{b}^{(2)} = {c^{(2)}}^{-1}b^{(2)}$ if $\widehat{a}^{(2)} = {c^{(2)}}^{-1}a^{(2)}$. In the same way, we continue transformations for every position in the chart. Notice that, by Corollary~\ref{virtual_members_many_replacements}, we actually can perform the replacements in any order.

\item
\label{mu_case2}
Assume $1 \leqslant i_0 \leqslant k$ is a place in the chart of $U$. We define $\mu[U]$ on elements $b^{(1)}\otimes \ldots \otimes b^{(k)}$ such that $b^{(i_0)} \in \Low_{i_0}$ and $b^{(i)} \in A_i \setminus \Low_i$ for $i \neq i_0$. First we replace virtual members of the chart of $U$ in positions different from $i_0$ as we described above. Then we perform a corresponding replacement in the position $i_0$ and all the further cancellations if necessary.

\item
\label{mu_case3}
Assume $1 \leqslant i_0^{\prime} < i_0 \leqslant n$ are two places in the chart of $U$. We define $\mu[U]$ on elements $b^{(1)}\otimes \ldots \otimes b^{(k)}$ such that $b^{(i)} \in A_i \setminus \Low_i$ for $i \neq i_0^{\prime}, i_0$ and $b^{(i_0^{\prime})} \in \Low_{i_0^{\prime}}$, $b^{(i_0)} \in \Low_{i_0}$. First we replace virtual members of the chart of $U$ in positions different from $i_0^{\prime}, i_0$ as we described above. Then we perform the corresponding replacement in the position $i_0$. Assume that as a result we obtain a monomial $L\widehat{b}^{(i_0)}R$. If there are any further cancellations in $L\widehat{b}^{(i_0)}R$ (when $\widehat{b}^{(i_0)} = 1$), we do not perform them right after the replacement. Instead we perform the corresponding replacement in the subword $L$ or $R$ in the place corresponding to $i_0^{\prime}$. Then we perform the cancellations if there are any.

From Lemma~\ref{replacements_diamond_property} it follows that we can perform two last transformations starting from any position $i_0^{\prime}$ or $i_0$ and obtain the same final result. From Lemma~\ref{additional_derived monomials} it follows that the resulting monomial is a derived monomial $U$.

\item
\label{mu_case4}
For elements $b^{(1)}\otimes \ldots \otimes b^{(k)}$ such that there are more than two $b^{(i)} \in \Low_i$, we could continue in a similar way. But, in fact, we do not need to preserve full information about these elements. So, by definition, we put $\mu[U](b^{(1)}\otimes \ldots \otimes b^{(k)}) = 0$ in this case.
\end{enumerate}

From Corollary~\ref{virtual_members_many_replacements}, it follows that in \ref{mu_case1} an element $\mu[U](b^{(1)}\otimes \ldots \otimes b^{(k)}) \in \langle U\rangle_d \setminus \Low(\langle U\rangle_d)$. From Lemma~\ref{estimation_value_property} and Lemma~\ref{additional_derived monomials}, it follows that in \ref{mu_case2} and \ref{mu_case3} an element $\mu[U](b^{(1)}\otimes \ldots \otimes b^{(k)}) \in \Low(\langle U\rangle_d)$.

Let $Z$ be a derived monomial of $U$ such that $Z$ does not belong to $\Low(\langle U\rangle_d)$. By Lemma~\ref{estimation_value_property}, this means that there exists a sequence of replacements of virtual members of the chart such that every replacement preserves $f$-characteristics of monomials. From Corollary~\ref{virtual_members_many_replacements}, it follows that the replacements can be performed in any order and after changing the order every replacement still preserves $f$-characteristics of monomials. Namely, we can first perform all replacements in the first position of the chart, then in the second position, etc, in the $k$-th position. Assume $\widehat{b}^{(1)}, \ldots, \widehat{b}^{(k)}$ are the virtual members of the chart of $Z$ enumerated from left to right. Then the element $\widehat{b}^{(i)}$ may not be incident to $a^{(i)}$ (where $a^{(i)}$ is the corresponding member of the chart of $U$). But, clearly, its ends differ only by a shift by $\varepsilon$. Let $b^{(i)}$ be an incident monomial of $a^{(i)}$, corresponding to $\widehat{b}^{(i)}$. Then
\begin{equation*}
Z = \mu[U](b^{(1)}\otimes \ldots \otimes b^{(k)}).
\end{equation*}
From Corollary~\ref{virtual_members_many_replacements} it follows that every replacement can be moved to the beginning of the sequence, which starts from $U$, and it still preserves the $f$-characteristic of the monomials. That is, every replacement $a^{(i)} \mapsto b^{(i)}$ in the monomial $U$ does not decrease the $f$-characteristic of the resulting monomial. Hence, $b^{(i)} \in A_i \setminus \Low_i$.

Suppose $b^{(1)}\otimes \ldots \otimes b^{(k)}$ and $d^{(1)}\otimes \ldots \otimes d^{(k)}$ are different elements such that $b^{(i)}, d^{(i)}$ are generalized fractional powers, $b^{(i)}, d^{(i)}\in A_i \setminus \Low_i$. Since there exists $b^{(i_0)} \neq d^{(i_0)}$ and two different incident generalized fractional powers can not differ only by a piece equal to a possible overlap, we obtain
\begin{equation*}
\mu[U] (b^{(1)}\otimes \ldots \otimes b^{(k)}) \neq \mu[U](d^{(1)}\otimes \ldots \otimes d^{(k)}).
\end{equation*}
Thus, we have showed that $\mu[U]$ gives a bijective correspondence between elements $b^{(1)}\otimes \ldots \otimes b^{(k)}$ such that generalized fractional powers $b^{(i)} \in A_i \setminus \Low_i$ and monomials from $\langle U\rangle_d \setminus \Low(\langle U\rangle_d)$.

As above, we denote by $\Dp(A_i)$ the subspace of $A_i$ generated by supports of corresponding elementary multi-turns (see Definition~\ref{subspace_of_dependencies}). Recall that $A_i$ is generated by generalized fractional powers incident to one generalized fractional power $a^{(i)}$. Therefore, from the definition of an elementary multi-turn, it follows that a sum of two supports from $\Dp(A_i)$ is again the support of elementary multi-turn. So, actually $\Dp(A_i)$ consists precisely of supports of elementary multi-turns. Recall that a maximal occurrence of a generalized fractional power of $\LM$-measure $\geqslant \tau$ is always a virtual member of the chart. Hence, applying together~\eqref{tens_prod_components_low} and Proposition~\ref{transversality}, we obtain $\Dp(A_i) \cap \Low_i = 0$.

In order to prove Lemma~\ref{fall_through_linear_dep} it remains to show that
\begin{equation}
\label{fall_through_linear_dep_one}
\Dp(\langle U\rangle_d) \cap \Low(\langle U\rangle_d) = \Dp(\Low(\langle U\rangle_d)).
\end{equation}
We already proved that $\Dp(\langle U\rangle_d) \cap \Low(\langle U\rangle_d) \supseteq \Dp(\Low(\langle U\rangle_d))$. So, we will show that $\Dp(\langle U\rangle_d) \cap \Low(\langle U\rangle_d) \subseteq \Dp(\Low(\langle U\rangle_d))$.

Suppose $T_s$, $s = 1, \ldots, m$, are the supports of multi-turns coming from monomials of $\langle U\rangle_d \setminus \Low(\langle U\rangle_d)$ and
\begin{equation*}
\sum\limits_{s= 1}^{m} T_s \in \Low(\langle U\rangle_d).
\end{equation*}
Assume $T_s$ comes from a monomial $U^{(s)}_h \in \langle U\rangle_d \setminus \Low(\langle U\rangle_d)$. Then there exists an element $b^{(1)}_h\otimes \ldots \otimes b^{(k)}_h \in A_1\otimes \ldots \otimes A_n$ such that
\begin{equation*}
\mu[U](b^{(1)}_{h_s}\otimes \ldots \otimes b^{(k)}_{h_s}) = U^{(s)}_h,
\end{equation*}
where $b^{(i)}_{h_s}$ are generalized fractional powers such that $b^{(i)}_{h_s}\notin \Low_i$. Assume this multi-turn comes from an elementary multi-turn of the virtual member of the chart of $U^{(s)}_h$ placed on the $i_s$-th position. The element $b^{(i_s)}_{h_s} \in A_{i_s}$ corresponds to this virtual member of the chart. By the definition of $\mu[U]$ in \ref{mu_case1} and~\ref{mu_case2}, we obtain
\begin{equation*}
T_s = \mu[U](b^{(1)}_{h_s}\otimes \ldots \otimes t^{(i_s)}_s \otimes \ldots \otimes b^{(k)}_{h_s}),
\end{equation*}
where $t^{(i_s)}_s \in \Dp(A_{i_s})$ is the support of the corresponding elementary multi-turn of $b^{(i_s)}_{h_s}$.

So, we have
\begin{equation*}
\sum\limits_{s = 1}^{m} T_s = \mu[U]\left(\sum\limits_{s = 1}^{m} b^{(1)}_{h_s}\otimes \ldots \otimes t^{(i_s)}_s \otimes \ldots \otimes b^{(k)}_{h_s}\right).
\end{equation*}
By the assumption, all monomials in the left-hand sum that do not belong to $\Low(\langle U\rangle_d)$ cancel out. Since $\mu[U]$ gives a bijective correspondence between the elements $b^{(1)}\otimes \ldots \otimes b^{(k)}$ such that generalized fractional powers $b^{(i)} \in A_i \setminus \Low_i$ and the monomials from $\langle U\rangle_d \setminus \Low(\langle U\rangle_d)$, we obtain that all such elements cancel out in the right-hand sum as well. Then from Lemma~\ref{tensor_product_filtration} it follows that
\begin{equation*}
\sum\limits_{s = 1}^{m} b^{(1)}_{h_s}\otimes \ldots \otimes t^{(i_s)}_s \otimes \ldots \otimes b^{(k)}_{h_s} \in \sum\limits_{i = 1}^{k}\sum\limits_{\substack{i^{\prime} = 1 \\ i \neq i^{\prime}}}^{k} A_1\otimes \ldots \otimes \Dp(A_i) \otimes \ldots \otimes \Low_{i^{\prime}} \otimes \ldots \otimes A_k,
\end{equation*}
that is, the sum $\sum_{s = 1}^{m} b^{(1)}_{h_s}\otimes \ldots \otimes t^{(i_s)}_s \otimes \ldots \otimes b^{(k)}_{h_s}$ is equal to a sum of elements of the form
\begin{align*}
&d^{(1)}\otimes \ldots \otimes d^{(i^{\prime}_0 - 1)} \otimes \widetilde{d}^{(i^{\prime}_0)} \otimes d^{(i^{\prime}_0 +1)} \otimes \ldots \otimes d^{(i_0 - 1)} \otimes \widetilde{t}^{(i_0)} \otimes d^{(i_0 + 1)} \otimes \ldots \otimes d^{(k)},\\
&d^{(1)}\otimes \ldots \otimes d^{(i_0 - 1)} \otimes \widetilde{d}^{(i_0)} \otimes d^{(i_0 +1)} \otimes \ldots \otimes d^{(i^{\prime}_0 - 1)} \otimes \widetilde{t}^{(i^{\prime}_0)} \otimes d^{(i^{\prime}_0 + 1)} \otimes \ldots \otimes d^{(k)}
\end{align*}
for different $i_0 \neq i_0^{\prime}$, where generalized fractional powers $d^{(i)} \in A_i \setminus \Low_i$, a generalized fractional power $\widetilde{d}^{(i^{\prime}_0)} \in \Low_{i^{\prime}_0}$, and the support of an elementary multi-turn $\widetilde{t}^{(i_0)} \in \Dp(A_{i_0})$. To be precise, assume $i_0^{\prime} < i_0$.

Assume $\widetilde{t}^{(i_0)} = \sum_{j = 1}^{m_0} e^{(i_0)}_j$. Let us calculate
\begin{equation*}
\mu[U](d^{(1)}\otimes \ldots \otimes \widetilde{d}^{(i^{\prime}_0)} \otimes \ldots \otimes \sum\limits_{j = 1}^{m_0} e^{(i_0)}_j \otimes \ldots \otimes d^{(k)}).
\end{equation*}
From Lemma~\ref{replacements_diamond_property} and Corollary~\ref{virtual_members_many_replacements}, it follows that in the definition of $\mu[U]$ first we can do replacements of all positions of the chart except $i_0$ and $i_0^{\prime}$, then a replacement in the position $i_0$ and then a replacement in the position corresponding to $i_0^{\prime}$ (the exact position may shift if $e^{(i_0)}_j \in \Low_{i_0}$). Suppose $Z$ is a result after the replacement in the position $i_0$. Then $Z = L\widehat{e}^{(i_0)}R$, where $\widehat{e}^{(i_0)}$ may differ from $e^{(i_0)}$ by shifts of its ends by $\varepsilon$. For simplicity, we illustrate the case when the virtual members of the chart on the positions $i^{\prime}_0$ and $i_0$ are separated, but the other cases are analogous to this one.
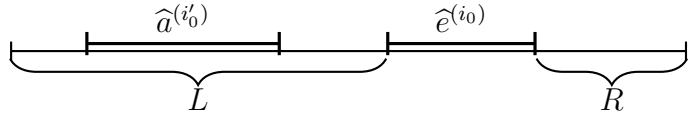
\begin{center}
\begin{tikzpicture}
\draw[|-|, black, thick] (-3,0)--(6,0);
\draw[|-|, black, very thick] (-2,0.1)--(0.6,0.1) node[midway, above] {$\widehat{a}^{(i_0^{\prime})}$};
\draw[|-|, black, very thick] (2, 0.1)--(4,0.1) node[midway, above] {$\widehat{e}^{(i_0)}$};
\draw [thick, decorate, decoration={brace, amplitude=10pt, mirror}] (-3,-0.1)--(2,-0.1) node [midway, below, yshift=-7] {$L$};
\draw [thick, decorate, decoration={brace, amplitude=10pt, mirror}] (4,-0.1)--(6,-0.1) node [midway, below, yshift=-7] {$R$};
\end{tikzpicture}
\end{center}
Then the last replacement in $Z$ can be represented as a replacement of the corresponding maximal occurrence in $L$ and the further cancellations if there are any. Denote the result of the transformation of $L$ by $L^{\prime}$.
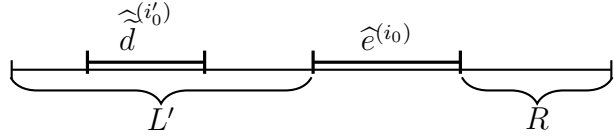
\begin{center}
\begin{tikzpicture}
\draw[|-|, black, thick] (-3,0)--(5,0);
\draw[|-|, black, very thick] (-2,0.1)--(-0.4,0.1) node[midway, above] {$\widehat{\widetilde{d}}^{(i^{\prime}_0)}$};
\draw[|-|, black, very thick] (1, 0.1)--(3,0.1) node[midway, above] {$\widehat{e}^{(i_0)}$};
\draw [thick, decorate, decoration={brace, amplitude=10pt, mirror}] (-3,-0.1)--(1,-0.1) node [midway, below, yshift=-7] {$L^{\prime}$};
\draw [thick, decorate, decoration={brace, amplitude=10pt, mirror}] (3,-0.1)--(5,-0.1) node [midway, below, yshift=-7] {$R$};
\end{tikzpicture}
\end{center}
Hence, we obtain
\begin{equation*}
\mu[U](d^{(1)}\otimes \ldots \otimes \widetilde{d}^{(i^{\prime}_0)} \otimes \ldots \otimes e^{(i_0)}_j \otimes \ldots \otimes d^{(k)}) = L^{\prime}\widehat{e_j}^{(i_0)}R,
\end{equation*}
where the word $L^{\prime}\widehat{e_j}^{(i_0)}R$ is possibly non-reduced. So,
\begin{align*}
\mu[U](d^{(1)}\otimes \ldots \otimes \widetilde{d}^{(i^{\prime}_0)} \otimes \ldots \otimes \sum\limits_{j = 1}^{m_0} e^{(i_0)}_j \otimes \ldots \otimes d^{(k)})& = \sum\limits_{j = 1}^{m_0} L^{\prime}\widehat{e}_j^{(i_0)}R\\ & = L^{\prime}\left(\sum\limits_{j = 1}^{m_0} \widehat{e}^{(i_0)}_j \right)R.
\end{align*}
Since $\sum_{j = 1}^{m_0} e^{(i_0)}_j$ is the support of an elementary multi-turn, the sum $\sum_{j = 1}^{m_0} \widehat{e}^{(i_0)}_j$ is the support of an elementary multi-turn as well (possibly with shifts of the ends of the monomials by $\varepsilon$). Therefore, from Proposition~\ref{the_ideal_characterisation} it follows that $\sum_{j = 1}^{m_0} L^{\prime}\widehat{e}^{(i_0)}_j R$ is the support of a multi-turn (after the cancellation in the monomials).

Since $\widetilde{d}^{(i^{\prime}_0)} \in \Low(A_i)$, we have
\begin{equation*}
\mu[U](d^{(1)}\otimes \ldots \otimes \widetilde{d}^{(i^{\prime}_0)} \otimes \ldots \otimes e^{(i_0)}_j \otimes \ldots \otimes d^{(k)}) \in \Low(\langle U\rangle_d).
\end{equation*}
Then, since $\mu[U](d^{(1)}\otimes \ldots \otimes \widetilde{d}^{(i^{\prime}_0)} \otimes \ldots \otimes \sum\limits_{j = 1}^{m_0} e^{(i_0)}_j \otimes \ldots \otimes d^{(k)})$ is the support of a multi-turn, we have
\begin{equation*}
\mu[U](d^{(1)}\otimes \ldots \otimes \widetilde{d}^{(i^{\prime}_0)} \otimes \ldots \otimes \sum\limits_{j = 1}^{m_0} e^{(i_0)}_j \otimes \ldots \otimes d^{(k)}) \in \Dp(\Low(\langle U\rangle_d)).
\end{equation*}
Thus,
\begin{equation*}
\sum\limits_{s = 1}^{m} T_s \in \Dp(\Low(\langle U\rangle_d))
\end{equation*}
and the equality~\eqref{fall_through_linear_dep_one} holds. This concludes the proof of Lemma~\ref{fall_through_linear_dep}.
\end{proof}

Using Lemma~\ref{fall_through_linear_dep}, we obtain the following proposition.
\begin{proposition}
\label{fall_to_smaller_subspace}
Suppose $X, Y$ are subspaces of $\mathbb{Z}_2\Fr$ generated by monomials and closed under taking derived monomials, $Y \subseteq X$. Then $\Dp(X)\cap Y = \Dp(Y)$.
\end{proposition}
\begin{proof}
Clearly, $\Dp(Y) \subseteq \Dp(X)\cap Y$. Let us show that $\Dp(X)\cap Y \subseteq \Dp(Y)$. Suppose $T_1 +\ldots + T_m \in \Dp(X)\cap Y$, where every $T_i$, $i = 1, \ldots m,$ belongs to the set of generating supports of multi-turns of $\Dp(X)$. Consider $T_{i_0} \in Y$. Since $Y$ is generated by monomials, we obtain that every monomial of $T_{i_0}$ belongs to $Y$. So, $T_{i_0}$ is a linear dependence generated by a monomial from $Y$, that is, $T_{i_0} \in \Dp(Y)$.

So, we may suppose that every $T_i\notin Y$, $i = 1, \ldots, m$. Denote by $X^{\prime}$ the subspace of $X$ generated by the monomials of $T_i$, $i = 1, \ldots, m$, and their derived monomials. Suppose the filtration on $X^{\prime}$ has $N$ non-zero levels. Let us prove that $T_1 +\ldots + T_m \in \Dp(Y)$ by induction on $N$.

First let us do the step of induction. Assume $T_i$ is generated by a multi-turn of a monomial $Z_i$. If $Z_i \in \Low(X^{\prime})$, then all the monomials of $T_i$ belong to $\Low(X^{\prime})$ because $\Low(X^{\prime})$ is closed under taking derived monomials. Assume $Z_i \in X^{\prime} \setminus \Low(X^{\prime})$. Let $Z$ be an arbitrary monomial of $T_i$ such that $Z \in X^{\prime} \setminus \Low(X^{\prime})$. Since $Z \in X^{\prime} \setminus \Low(X^{\prime})$, we have $Z \in \langle Z_i\rangle_d \setminus \Low(\langle Z_i\rangle_d)$. Therefore, from Lemma~\ref{derived_spaces_equality} it follows that $Z_i$ is a derived monomial of $Z$. Assume $Z \in Y$, then all its derived monomials belong to $Y$ because $Y$ is closed under taking derived monomials. So, we have $Z_i \in Y$; hence, all the monomials of $T_i$ are contained in $Y$. This contradicts our assumption that $T_i \notin Y$. Therefore, the monomials of $T_i$ that are contained in $X^{\prime} \setminus \Low(X^{\prime})$ are not contained in $Y$.

Since $T_1 +\ldots + T_m \in Y$, the monomials that are contained in $X^{\prime} \setminus \Low(X^{\prime})$ cancel in the sum $T_1 +\ldots + T_m$; hence, $T_1 +\ldots + T_m \in \Low(X^{\prime})$. Then from Lemma~\ref{fall_through_linear_dep} it follows that
\begin{equation*}
T_1 +\ldots + T_m\in \Dp(\Low(X^{\prime})).
\end{equation*}
Therefore, $T_1 +\ldots + T_m = T_1^{\prime} +\ldots + T_{m^{\prime}}^{\prime}$, where $T_i^{\prime}\in \Low(U^{\prime})$, $i = 1, \ldots, m^{\prime}$, are supports of multi-turns that come from monomials of $\Low(X^{\prime})$. We have
\begin{equation*}
T_1 +\ldots + T_m = T_1^{\prime} +\ldots + T_{m^{\prime}}^{\prime} = \sum\limits_{T^{\prime}_i \in Y} T^{\prime}_i + \sum\limits_{T^{\prime}_i \notin Y} T^{\prime}_i.
\end{equation*}
As above, every element of the first sum $\sum_{T^{\prime}_i \in Y} T^{\prime}_i $ belongs to $\Dp(Y)$. The space $\Low(X^{\prime})$ has $N - 1$ non-zero levels of the filtration; hence the second sum $\sum_{T^{\prime}_i \notin Y} T^{\prime}_i$ belongs to $\Dp(Y)$ by the induction hypothesis, and therefore,
\begin{equation*}
T_1 +\ldots + T_m \in \Dp(W).
\end{equation*}

Let us prove the basis of induction. Consider $N = 1$. As above, we obtain $T_1 +\ldots + T_m\in \Low(X^{\prime})$. But since $N = 1$, we have $\Low(X^{\prime}) = 0$; therefore $T_1 +\ldots + T_m = 0$, and so, it belongs to $\Dp(Y)$.

Thus, we obtain $\Dp(X) \cap Y \subseteq \Dp(Y)$. This concludes the proof.
\end{proof}

Using the mapping $\mu[U]$ that was constructed in the second part of the proof of Lemma~\ref{fall_through_linear_dep}, we obtain the following statement.
\begin{proposition}
\label{correspondence_to_tensor_product}
Let $U$ be a monomial with $k$ virtual members of the chart. Suppose $A_i$ and $\Low_i \subseteq A_i$, $i = 1, \ldots k$, are defined above by~\eqref{tens_prod_components} and~\eqref{tens_prod_components_low} subspaces of $\mathbb{Z}_2\Fr$ corresponding to $U$. Then we have
\begin{equation*}
\langle U\rangle_d / \Low(\langle U\rangle_d) \cong A_1/\Low_1 \otimes \ldots \otimes A_k/\Low_k.
\end{equation*}
Moreover,
\begin{equation*}
\langle U\rangle_d / (\Dp(\langle U\rangle_d ) + \Low(\langle U\rangle_d )) \cong A_1/ (\Dp(A_1) + \Low_1) \otimes \ldots \otimes A_k / (\Dp(A_k) + \Low_k).
\end{equation*}
\end{proposition}
\begin{proof}
Recall that in the proof of Lemma~\ref{fall_through_linear_dep} we constructed a linear mapping
\begin{equation*}
\mu[U]: A_1\otimes \ldots \otimes A_k \to \langle U\rangle_d
\end{equation*}
(see page~\pageref{mu_def}). We define a linear mapping
\begin{equation*}
\overline{\mu}_1[U]: A_1 / \Low_1\otimes \ldots \otimes A_k / \Low_k \to \langle U\rangle_d / \Low(\langle U\rangle_d)
\end{equation*}
by the following rule:
\begin{equation*}
\overline{\mu}_1[U]((b^{(1)} + \Low_1)\otimes \ldots \otimes (b^{(k)} + \Low_k)) = \mu[U](b^{(1)} \otimes \ldots \otimes b^{(k)}) + \Low(\langle U\rangle_d),
\end{equation*}
where $b^{(i)} \in A_i$ are generalized fractional powers. In Lemma~\ref{fall_through_linear_dep}, we proved (using Lemma~\ref{estimation_value_property} and Lemma~\ref{additional_derived monomials}) that
\begin{equation}
\label{mu_low_image}
\mu[U]\left(\sum\limits_{i = 1}^{k} A_1\otimes\ldots \otimes \Low_i\otimes \ldots \otimes A_k\right) \subseteq \Low(\langle U\rangle_d).
\end{equation}
Hence, the mapping $\overline{\mu}_1[U]$ is well-defined. Since $\mu[U]$ gives a bijective correspondence between elements $b^{(1)}\otimes\ldots\otimes b^{(k)}$, where $b^{(i)}$ are generalized fractional powers such that $b^{(i)} \in A_i \setminus \Low_i$, and the monomials from $\langle U\rangle_d \setminus \Low(\langle U\rangle_d)$, the mapping $\overline{\mu}_1[U]$ is bijective. So, $\overline{\mu}_1[U]$ is an isomorphism of linear spaces.

In the same way, we define a linear mapping
\begin{equation*}
\overline{\mu}_2[U]: A_1 / (\Dp(A_1) + \Low_1)\otimes \ldots \otimes A_k / (\Dp(A_k) + \Low_k) \to \langle U\rangle_d / (\Dp(\langle U\rangle_d) + \Low(\langle U\rangle_d)
\end{equation*}
by the following rule:
\begin{align}
\label{isom_to_tensot_prod_dep}
\overline{\mu}_2[U]((b^{(1)}& + \Dp(A_1) + \Low_1)\otimes \ldots \otimes (b^{(k)} + \Dp(A_k) + \Low_k)) \\
&= \mu[U](b^{(1)} \otimes \ldots \otimes b^{(k)}) + \Dp(\langle U\rangle_d) + \Low(\langle U\rangle_d),\nonumber
\end{align}
where $b^{(i)} \in A_i$ are generalized fractional powers. Arguing in the same way as in Lemma~\ref{fall_through_linear_dep}, it is easy to show that
\begin{equation*}
\mu[U]\left(\sum\limits_{i = 1}^{k} A_1\otimes\ldots \otimes \Dp(A_i)\otimes \ldots \otimes A_k\right) \subseteq \Dp(\langle U\rangle_d) + \Low(\langle U\rangle_d).
\end{equation*}
Using this together with~\eqref{mu_low_image}, we see that the mapping $\overline{\mu}_2[U]$ is well-defined. Since $\mu[U]$ gives a bijective correspondence between elements $b^{(1)}\otimes\ldots\otimes b^{(k)}$, where $b^{(i)}$ are generalized fractional powers such that $b^{(i)} \in A_i \setminus \Low_i$, and the monomials from $\langle U\rangle_d \setminus \Low(\langle U\rangle_d)$, one can show that the mapping $\overline{\mu}_2[U]$ is bijective. So, $\overline{\mu}_2[U]$ is an isomorphism of linear spaces.
\end{proof}

\begin{remark}
In Proposition~\ref{correspondence_to_tensor_product}, in fact, we used the following construction. We consider a composition of linear mappings
\begin{align*}
&A_1\otimes \ldots \otimes A_k \overset{\mu[U]}{\longrightarrow} \langle U\rangle_d \overset{\pi_1}{\longrightarrow} \langle U\rangle_d / \Low(\langle U\rangle_d ),\\
&A_1\otimes \ldots \otimes A_k \overset{\mu[U]}{\longrightarrow} \langle U\rangle_d \overset{\pi_2}{\longrightarrow} \langle U\rangle_d / (\Dp(\langle U\rangle_d) + \Low(\langle U\rangle_d)),
\end{align*}
where $\pi_1$ and $\pi_2$ are the canonical homomorphisms. Then, by properties of $\mu[U]$ established in Lemma~\ref{fall_through_linear_dep},
\begin{align*}
&\ker (\pi_1\circ \mu[U]) = \sum\limits_{i = 1}^{k} A_1\otimes\ldots \otimes \Low_i\otimes \ldots \otimes A_k, \\
&\ker (\pi_2\circ\mu[U]) = \sum\limits_{i = 1}^{k} A_1\otimes\ldots \otimes \Low_i\otimes \ldots \otimes A_k + \sum\limits_{i = 1}^{k} A_1\otimes\ldots \otimes \Dp(A_i) \otimes \ldots \otimes A_k.
\end{align*}
Then, using the isomorphism theorem and the definition of a tensor product of vector spaces, we obtain the result of Proposition~\ref{correspondence_to_tensor_product}.
\end{remark}

\subsection{The structure of quotient spaces $\langle U_1, \ldots , U_k\rangle_d / \Dp(\langle U_1, \ldots , U_k\rangle_d)$}
\label{fin_quotient_spaces_section}
The following statement easily follows from Proposition~\ref{fall_to_smaller_subspace}.
\begin{corollary}
\label{finite_gen_subspaces_structure}
Let $V = \langle U_1, \ldots , U_k\rangle_d$, where $U_1, \ldots, U_k$ are monomials. Let $\widehat{V}$ be the corresponding subspace in $\mathbb{Z}_2\Fr / \langle \mathcal{T}^{\prime}\rangle$, that is, $\widehat{V} = (V + \langle\mathcal{T}^{\prime}\rangle) / \langle\mathcal{T}^{\prime}\rangle$. Then $\widehat{V} \cong V / \Dp(V)$.
\end{corollary}
\begin{proof}
From the Isomorphism Theorem, it follows that $\widehat{V} \cong V / (V\cap \langle\mathcal{T}^{\prime}\rangle)$. According to Definition~\ref{subspace_of_dependencies}, we have $\langle\mathcal{T}^{\prime}\rangle = \Dp(\mathbb{Z}_2\Fr)$. Therefore, from Proposition~\ref{fall_to_smaller_subspace} it follows that $V\cap \langle\mathcal{T}^{\prime}\rangle = V\cap \Dp(\mathbb{Z}_2\Fr) = \Dp(V)$; hence, $\widehat{V} \cong V / \Dp(V)$.
\end{proof}

The quotient space $V / \Dp(V)$ inherits the filtration from $V$,
\begin{equation*}
\Ft_n (V / \Dp(V)) = (\Ft_nV + \Dp(V)) / \Dp(V).
\end{equation*}
Suppose $V$ has $N$ non-zero levels of the filtration
\begin{equation*}
V = \Ft_0V \supseteq \Ft_1V \supseteq \ldots \supseteq \Ft_{N-1}V \supseteq \Ft_{N}V = 0.
\end{equation*}
We have the corresponding graded space
\begin{equation*}
\Gr(V / \Dp(V)) = \bigoplus\limits_{n = 0}^{N - 1}(\Ft_n (V / \Dp(V)) / \Ft_{n + 1} (V / \Dp(V))).
\end{equation*}
For any $n = 0, \ldots, N - 1$, we have the mapping $\gr_n: \Ft_nV \to \Ft_nV / \Ft_{n + 1}V$.

\medskip

The following theorem establishes the compatibility of the filtration and the corresponding grading with the linear dependencies on the space $V$.
\begin{theorem}
\label{structure_of_quotient_space}
Let $V = \langle U_1, \ldots, U_k\rangle_d$, $U_1, \ldots, U_k\in \Fr$. Then
\begin{equation*}
\Gr(V / \Dp(V)) \cong \bigoplus\limits_{n = 0}^{N - 1} \Gr_n(V) / \gr_n(\Dp(\Ft_nV)),
\end{equation*}
where
\begin{align*}
&\Gr_n(V) = \Ft_nV / \Ft_{n + 1}V,\\
&\gr_n(\Dp(\Ft_nV)) = (\Dp(\Ft_nV) + \Ft_{n + 1}V) / \Ft_{n + 1}V.
\end{align*}
\end{theorem}
\begin{proof}
Using isomorphism theorems, we obtain
\begin{align*}
\Ft_n (V / \Dp(V)) / \Ft_{n + 1} (V / \Dp(V))& \cong (\Ft_nV + \Dp(V)) / (\Ft_{n + 1}V + \Dp(V)) \\ & = (\Ft_nV + \Ft_{n + 1}V + \Dp(V)) / (\Ft_{n + 1}V + \Dp(V))\\ & \cong \Ft_nV / (\Ft_nV \cap (\Ft_{n + 1}V + \Dp(V)))\nonumber\\ & = \Ft_nV / (\Ft_{n + 1}V + \Ft_nV \cap \Dp(V)).
\end{align*}
Hence,
\begin{equation*}
\Gr(V / \Dp(V)) = \bigoplus\limits_{n = 0}^{N - 1} \Ft_nV / (\Ft_{n + 1}V + \Ft_nV \cap \Dp(V)).
\end{equation*}
On the other hand, we have
\begin{align*}
\bigoplus\limits_{n = 0}^{N - 1}\Gr_nV / \gr_n(\Dp(\Ft_nV)) & = \bigoplus\limits_{n = 0}^{N - 1}(\Ft_nV / \Ft_{n + 1}V) / ((\Dp(\Ft_nV) + \Ft_{n + 1}V) / \Ft_{n + 1}V) \\ &\cong \bigoplus\limits_{n = 0}^{N - 1} \Ft_nV / (\Dp(\Ft_nV) + \Ft_{n + 1}V).
\end{align*}
From Corollary~\ref{fall_to_smaller_subspace} it follows that $\Ft_nV \cap \Dp(V) = \Dp(\Ft_nV)$. Consequently,
\begin{equation*}
\bigoplus\limits_{n = 0}^{N - 1} \Gr_nV / \gr_n(\Dp(\Ft_nV)) \cong \Gr(V / \Dp(V)).
\end{equation*}
\end{proof}

\begin{proposition}
\label{component_subspaces_structure}
Let $V = \langle U_1, \ldots, U_k\rangle_d$, $U_1, \ldots, U_k\in \Fr$. Assume $\Ft_nV$ is non-zero subspace of the filtration, $\Ft_nV = \langle Z_1, \ldots, Z_s, \ldots\rangle_d$, where $\lbrace Z_i\rbrace_{i\in I}$ is either a finite or infinite set of monomials, and let $V_i = \langle Z_i\rangle_d$. Then
\begin{enumerate}[label={(\arabic*)}]
\item[$(1)$]
\label{component_subspaces_structure_1}
there exists a subset of indices $I^{\prime} \subseteq I$ such that
\begin{equation*}
\Ft_nV / (\Dp(\Ft_nV) + \Ft_{n + 1}V) \cong \bigoplus\limits_{i \in I^{\prime}}V_i/(\Dp(V_i) + \Low(V_i)),
\end{equation*}
where $V_i \nsubseteq \Ft_{n + 1}V$ and the spaces $V_i$, $i \in I^{\prime}$ are pairwise different;
\item[$(2)$]
\label{component_subspaces_structure_2}
if $Z \in \Ft_nV$ is a monomial such that $Z \notin \Ft_{n + 1}V$, then $Z$ belongs to precisely one space $V_i$, $i \in I^{\prime}$, and the isomorphism acts as follows
\begin{equation*}
Z + \Dp(\Ft_nV) + \Ft_{n + 1}V \mapsto (0, \ldots, 0, Z + \Dp(V_i) + \Low(V_i), 0, \ldots).
\end{equation*}
\end{enumerate}
\end{proposition}
\begin{proof}
Since $\Ft_nV = \langle Z_1, \ldots, Z_s, \ldots\rangle_d$, $\Ft_nV = \sum_{i = 1}^{\infty} V_i$ if there are infinitely many $V_i$ and $\Ft_nV = \sum_{i = 1}^{m} V_i$ otherwise. In what following we only discuss the case of infinitely many $V_i$, $i \in I$. The other case is similar. So, we have
\begin{equation*}
\Ft_nV / \Ft_{n + 1}V = \sum\limits_{i = 1}^{\infty} \left((V_i + \Ft_{n + 1}V) / \Ft_{n + 1}V\right).
\end{equation*}

Let us choose a special subset of $\lbrace V_i\rbrace_{i \in I}$. We take all the different spaces from $\lbrace V_i\rbrace_{i \in I}$ and remove $V_i$ such that $V_i \nsubseteq \Ft_{n + 1}V$. Since $\Ft_nV \neq 0$, from Proposition~\ref{filtration_finite_property} it follows that $\Ft_nV \neq \Ft_{n + 1}V$. Hence, we do not remove all $V_i$ in this process and obtain a non-empty set $\lbrace V_i\rbrace_{i \in I^{\prime}}$ such that
\begin{equation*}
\Ft_nV / \Ft_{n + 1}V = \sum\limits_{i \in I^{\prime}} \left( (V_i + \Ft_{n + 1}V) / \Ft_{n + 1}V\right).
\end{equation*}

Assume $V_j$ is one of the spaces $\lbrace V_i\rbrace_{i \in I^{\prime}}$, that is, $j \in I^{\prime}$, and
\begin{equation}
\label{spaces_intersection}
\left((V_j + \Ft_{n + 1}V) / \Ft_{n + 1}V \right) \bigcap \sum\limits_{\substack{i \neq j \\ i \in I^{\prime}}} \left((V_i + \Ft_{n + 1}V) / \Ft_{n + 1}V\right) \neq 0.
\end{equation}
Since every $V_i$, $i \in I^{\prime}$, and $\Ft_{n + 1}V$ are generated by monomials, it follows from~\eqref{spaces_intersection} that there exist monomials $Z_j^{\prime} \in V_j$ such that $Z_j^{\prime} \notin \Ft_{n + 1}V$ and a space $V_{i_0}$, $i_0\in I^{\prime}$, $i_0 \neq j$, such that $Z_j^{\prime} \in V_{i_0}$. Since $Z_j^{\prime} \notin \Ft_{n + 1}V$, we have $Z_j^{\prime} \notin \Low(V_{j})$ and $Z_j^{\prime} \notin \Low(V_{i_0})$. Then from Lemma~\ref{derived_spaces_equality} it follows that $V_j = V_{i_0}$. But this contradicts the assumption that all spaces $V_i$, $i \in I^{\prime}$, are pairwise different. Hence,
\begin{equation*}
\left((V_j + \Ft_{n + 1}V) / \Ft_{n + 1}V\right) \bigcap \sum\limits_{\substack{i \neq j \\ i \in I^{\prime}}} \left((V_i + \Ft_{n + 1}V) / \Ft_{n + 1}V\right) = 0.
\end{equation*}
Thus, we obtain a direct sum
\begin{equation}
\label{direct_sum}
\Ft_nV / \Ft_{n + 1}V = \bigoplus\limits_{i \in I^{\prime}} \left((V_i + \Ft_{n + 1}V) / \Ft_{n + 1}V\right).
\end{equation}

Let $Z \in \Ft_nV$ be a monomial, $Z\notin \Ft_{n+1}V$. Notice that, by the same argument as above, we obtain that $Z$ belongs to precisely one of the spaces $V_i$, $i \in I^{\prime}$.

Let $T$ be the support of a multi-turn $U_h \mapsto \sum_{\substack{j = 1 \\ j \neq h}}^k U_j$, where $U_h \in \Ft_nV \setminus \Ft_{n + 1}V$. Hence, $U_h$ belongs to a space $V_{i_0}$, $i_0 \in I^{\prime}$. Then every $U_j$ also belongs to $V_{i_0}$ and $T \in \Dp(V_{i_0})$. Therefore, we have
\begin{equation*}
(\Dp(\Ft_nV) + \Ft_{n + 1}V) / \Ft_{n + 1}V = \bigoplus\limits_{i \in I^{\prime}} \left((\Dp(V_i) + \Ft_{n + 1}V) / \Ft_{n + 1}V\right).
\end{equation*}
Hence,
\begin{align}
\label{direct_sum_components}
\Ft_nV /& (\Dp(\Ft_nV) + \Ft_{n + 1}V) \cong (\Ft_nV / \Ft_{n + 1}V) / \left((\Dp(\Ft_nV) + \Ft_{n + 1}V) / \Ft_{n + 1}V\right) \\
&= \left(\bigoplus\limits_{i \in I^{\prime}} \left((V_i + \Ft_{n + 1}V) / \Ft_{n + 1}V\right) \right) / \left(\bigoplus\limits_{i \in I^{\prime}} \left((\Dp(V_i) + \Ft_{n + 1}V) / \Ft_{n + 1}V\right) \right) \nonumber \\
&\cong \bigoplus\limits_{i \in I^{\prime}} ((V_i + \Ft_{n + 1}V) / \Ft_{n + 1}V) / ((\Dp(V_i) + \Ft_{n + 1}V) / \Ft_{n + 1}V) \nonumber \\
&\cong \bigoplus\limits_{i \in I^{\prime}} (V_i + \Ft_{n + 1}V) / (\Dp(V_i) + \Ft_{n + 1}V).\nonumber
\end{align}
By the Isomorphism Theorem, we obtain
\begin{align}
\label{single_component}
(V_i + \Ft_{n + 1}V) / &(\Dp(V_i) + \Ft_{n + 1}V) = (V_i + \Dp(V_i) + \Ft_{n + 1}V) / (\Dp(V_i) + \Ft_{n + 1}V) \\
&\cong V_i / ((\Dp(V_i) + \Ft_{n + 1}V) \cap V_i) = V_i / (\Dp(V_i) + \Ft_{n + 1}V \cap V_i)\nonumber
\end{align}
Since $V_i$, $i \in I^{\prime}$, is not contained in $\Ft_{n + 1}V$, from Lemma~\ref{intersection_with_low} it follows that $\Ft_{n + 1}V \cap V_i = \Low(V_i)$. Thus, from~\eqref{direct_sum_components} and~\eqref{single_component} it follows that
\begin{equation*}
\Ft_nV / (\Dp(\Ft_nV) + \Ft_{n + 1}V) \cong \bigoplus\limits_{i \in I^{\prime}} V_i/(\Dp(V_i) + \Low(V_i)).
\end{equation*}
So, the first statement is proved.

Assume $Z \in \Ft_nV$ is a monomial, $Z\notin \Ft_{n+1}V$. We noticed above that $Z$ belongs to one space $V_{i_0}$, $i_0 \in I^{\prime}$. Then from the definition of the canonical isomorphisms in~\eqref{direct_sum_components} and~\eqref{single_component}, it easily follows that the final isomorphism maps $Z + \Dp(\Ft_nV) + \Ft_{n + 1}V$ to the corresponding quotient space $V_{i_0}/(\Dp(V_{i_0}) + \Low(V_{i_0}))$, namely,
\begin{equation*}
Z + \Dp(\Ft_nV) + \Ft_{n + 1}V \mapsto (0, \ldots, 0, Z + \Dp(V_{i_0}) + \Low(V_{i_0}), 0, \ldots).
\end{equation*}
So, the second statement is proved.
\end{proof}

\section{Description of a basis in $\mathbb{Z}_2\Fr / \Ideal$}
\label{basis_descr_section}
Let us enumerate all the monomials from $\Fr$, namely, $\Fr = \lbrace U_1, U_2, \ldots, U_k, \ldots\rbrace$. Consider an increasing sequence of subspaces of $\mathbb{Z}_2\Fr$
\begin{equation*}
\langle U_1\rangle_d \subseteq \langle U_1, U_2\rangle_d \subseteq \ldots \subseteq \langle U_1, \ldots, U_k\rangle_d \subseteq \ldots
\end{equation*}
Clearly, the union of all these subspaces gives the whole ring $\mathbb{Z}_2\Fr$. Consider the corresponding increasing sequence of subspaces of $\mathbb{Z}_2\Fr / \langle \mathcal{T}^{\prime} \rangle = \mathbb{Z}_2\Fr / \Ideal$
\begin{align}
\label{gen_incr_subspaces}
(\langle U_1\rangle_d + \langle \mathcal{T}^{\prime}) / \langle \mathcal{T}^{\prime} \rangle \rangle& \subseteq (\langle U_1, U_2\rangle_d + \langle\mathcal{T}^{\prime}\rangle) / \langle \mathcal{T}^{\prime} \rangle \subseteq \ldots \\
&\subseteq (\langle U_1, \ldots, U_k\rangle_d + \langle \mathcal{T}^{\prime} \rangle) / \langle \mathcal{T}^{\prime} \rangle \subseteq \ldots\, .\nonumber
\end{align}
The union of all these subspaces gives the whole ring $\mathbb{Z}_2\Fr / \langle \mathcal{T}^{\prime} \rangle$. By Corollary~\ref{finite_gen_subspaces_structure}, we obtain
\begin{equation*}
(\langle U_1, \ldots, U_k\rangle_d + \langle \mathcal{T}^{\prime} \rangle) / \langle \mathcal{T}^{\prime} \rangle \cong \langle U_1, \ldots, U_k\rangle_d / \Dp(\langle U_1, \ldots, U_k\rangle_d).
\end{equation*}

In Section~\ref{calculate_basis_section}, we will construct a basis $\mathcal{B}_k$ in every subspace $(\langle U_1, \ldots, U_k\rangle + \langle \mathcal{T}^{\prime} \rangle) / \langle \mathcal{T}^{\prime} \rangle$ such that we obtain the increasing sequence
\begin{equation*}
\mathcal{B}_1 \subseteq \mathcal{B}_2 \subseteq \ldots \subseteq \mathcal{B}_k \subseteq \ldots\, .
\end{equation*}
Since the union of the subspaces~\eqref{gen_incr_subspaces} gives the whole ring $\mathbb{Z}_2\Fr / \langle \mathcal{T}^{\prime} \rangle$, the union $\bigcup_{k} \mathcal{B}_k$ is a basis of $\mathbb{Z}_2\Fr / \langle \mathcal{T}^{\prime} \rangle = \mathbb{Z}_2\Fr / \Ideal$.

As above, assume $V = \langle U_1, \ldots, U_k\rangle_d$, $U_1, \ldots, U_k \in \Fr$. Let the filtration on the space $V$ have $N$ non-zero levels. Putting together the results from Section~\ref{fin_quotient_spaces_section}, we obtain
\begin{equation*}
(V + \langle\mathcal{T}^{\prime}\rangle) / \langle\mathcal{T}^{\prime}\rangle \cong \bigoplus\limits_{n = 0}^{N - 1}\bigoplus\limits_{i \in I^{\prime}_n} V_i^{(n)}/(\Dp(V_i^{(n)}) + \Low(V_i^{(n)})),
\end{equation*}
where $V_i^{(n)} = \langle Z_i^{(n)}\rangle_d$, $Z_i^{(n)}$ is a monomial from $\Ft_nV$. In Section~\ref{semicanonical_words_section} we will characterize non-trivial quotient spaces $V_i^{(n)}/(\Dp(V_i^{(n)}) + \Low(V_i^{(n)})$. In Section~\ref{calculate_basis_section} we will construct a basis in every space $V_i^{(n)} / (\Dp(V_i^{(n)}) + \Low(V_i^{(n)}))$ and, using this, we will explicitly describe a basis in the whole ring $\mathbb{Z}_2\Fr / \Ideal$.

\subsection{$\lambda$-semicanonical words}
\label{semicanonical_words_section}
Fix a constant $\frac{1}{2} \ll \lambda \ll 1$ ($\ll$ on the $\varepsilon$-scale). For example, one can use $\lambda \geqslant \frac{2}{3}$. We introduce \emph{$\lambda$-forbidden words}. Recall our notation $v = v_iv_mv_f$, where $v_i$ is some initial part, $v_m$ is some middle part and $v_f$ is some final part of $v$ (any part is allowed to be empty). A word of the form $v_m^{-1}$ is called \emph{$\lambda$-forbidden} if $\LM(v_m) > \lambda$.
\begin{equation}
\label{forbidden_words}
v_m^{-1}, \LM(v_m) > \lambda.
\end{equation}

\begin{definition}
A word is called \emph{$\lambda$-semicanonical} if it does not contain occurrences of $\lambda$-forbidden words.
\end{definition}

We are motivated by an informal analogue for small cancellation groups. Let $G = \langle X \mid \mathcal{R} \rangle$, where $\mathcal{R}$ satisfies small cancellation conditions \cite{LyndonSchupp} with a certain measure $\LM$ on the relators $R_i$. That is $\LM(R_i) = 1$ and $\LM(P) \leqslant \varepsilon$ for any small piece $P$. We call a subword $S$ of $R_i^{\pm 1}$ \emph{$\lambda$-forbidden} if $\LM(S) > \lambda$, where $\lambda$ as above is between $\frac{1}{2}$ and $1$. A word $U$ is \emph{$\lambda$-semicanonical} if it does not contain $\lambda$-forbidden words. In the Cayley graph of $G$ any $\lambda$-semicanonical word is a quasigeodesic \cite{CoornaertDelzantPapadopoulos}, \cite{GhysHarpe}.

A small cancellation group (with appropriate constants) is hyperbolic. Recall that if the elements of a hyperbolic group are represented by quasigeodesics, then their product takes the form of a thin triangle (\cite{Gromov}). In our case, first we will construct a special basis of $\mathbb{Z}_2\Fr / \Ideal$ with the use of $\lambda$-semicanonical words. Then in Section~\ref{multiplication_geometry_section}, we will construct a special set of linear generators of $\mathbb{Z}_2\Fr / \Ideal$ (not linearly independent) such that it contains the basis and we can express the product of two elements of this set as a sum of elements of this set that form thin triangles with the factors.

\begin{proposition}
\label{reduct_power}
Let $a_h$ be a generalized fractional power, $W = \langle a_h\rangle_d$. Then $a_h$ is either $\lambda$-semicanonical, or is equal modulo $\Dp(W)$ to a sum of $\lambda$-semicanonical generalized fractional powers from $W$. Namely, either $a_h$ is $\lambda$-semicanonical, or there exists an elementary multi-turn $a_h \mapsto \sum_{\substack{j = 1 \\ j\neq h}}^k a_j$ such that $a_j$, $j = 1, \ldots, k$, $j\neq h$, are $\lambda$-semicanonical generalized fractional powers and they have shorter word length in $\Fr$ and less $\LM$-measure than $a_h$.
\end{proposition}
\begin{proof}
Let $a_h$ contain a $\lambda$-forbidden subword $b_h$, $a_h = Lb_hR$. Returning to forbidden words~\eqref{forbidden_words}, we consider the following relation in $\mathbb{Z}_2\Fr / \Ideal$:
\begin{equation}
\label{eliminate_vinv}
v_m^{-1} = v_fwv_i + v_fv_i.
\end{equation}
Multiplying the relation $v^{-1} = 1 + w$ by $v_f$ on the left side and by $v_i$ on the right side, we get \eqref{eliminate_vinv}. By definition, the words $v_fwv_i$ and $v_fv_i$ are non $\lambda$-forbidden. If $v_m^{-1}$ is a $\lambda$-forbidden word, that is, $\LM(v_m^{-1}) > \lambda$, then we obtain
\begin{equation*}
\LM(v_fv_i) = \LM(v_fwv_i) = \LM(v_f) + \LM(v_i) < 1 - \lambda < \lambda < \LM(v_m^{-1}).
\end{equation*}
Moreover, since $\vert w\vert \ll \vert v\vert $, we obtain $\vert v_fv_i\vert < \vert v_fwv_i \vert < \vert v_m^{-1}\vert$, where $\vert\cdot\vert$ is word length in $\Fr$.

By Definition~\ref{multiturn_def}, the transformation $v_m^{-1} \mapsto v_fwv_i + v_fv_i$ is an elementary multi-turn and according to the above, it gives the reduction of $\lambda$-forbidden words to sums of shorter (in word length in $\Fr$ and in $\LM$-measure) $\lambda$-semicanonical words.

So, $b_h$ can be substituted by a sum of $\lambda$-semicanonical generalized fractional powers $\sum_{\substack{j = 1 \\ j\neq h}}^{n} b_j = v_fwv_i + v_fv_i$. Since $\LM(b_h) > \varepsilon$, it appears in $v^{-1}$ only once and, therefore, has fixed initial and fixed final point in the corresponding $v$-diagram. Since $a_h = Lb_hR$ is itself a generalized fractional power, we obviously obtain that every word $Lb_jR$ is a generalized fractional power, possibly after cancellations, with the same initial point and the same final point in the corresponding $v$-diagram as the word $a_h$ has. So, since $v_m^{-1} \mapsto v_fwv_i + v_fv_i$ is an elementary multi-turn, $a_h \mapsto \sum_{\substack{j = 1 \\ j\neq h}}^{n}Lb_jR$ is an elementary multi-turn as well. Hence, $Lb_jR$ are derived monomials of $a_h$ and $a_h = \sum_{\substack{j = 1 \\ j\neq h}}^{n} Lb_jR$ modulo $\Dp(W)$.

If the obtained word $Lb_jR$ contains a $\lambda$-forbidden subword, we repeat the process and reduce it using the transformation \eqref{eliminate_vinv}. Successive elementary multi-turns match to addition of the corresponding expressions \eqref{mturn1} --- \eqref{mturn6}. Thus, applying several elementary multi-turns consecutively, we obtain an elementary multi-turn as a result. So, we obtain an elementary multi-turn of $a_h$ after each step of the reduction process.

From~\eqref{eliminate_vinv} it follows that after every transformation that reduces $\lambda$-forbidden subword of a monomial, the word length in $\Fr$ of the resulting monomials is strictly smaller than of the initial monomial. Thus, the reduction process finishes and we obtain the result as an elementary multi-turn $a_h \mapsto \sum_{\substack{j = 1 \\ j\neq h}}^k a_j$, where $j = 1, \ldots, k$, $j\neq h$, are $\lambda$-semicanonical generalized fractional powers with shorter word length in $\Fr$ and less $\LM$-measure than $a_h$.
\end{proof}

\begin{corollary}
\label{lemma_about_generating}
Assume $V = \langle U_1, \ldots, U_k\rangle_d$, $U_1, \ldots, U_k \in \Fr$. Let $U \in V$ be a monomial. Then $U$ is equal to a sum of $\lambda$-semicanonical words from $V$ modulo $\Dp(V)$. In particular, $U$ is equal to a sum of $\lambda$-semicanonical words from $\langle U\rangle_d$ modulo $\Dp(\langle U\rangle_d)$.
\end{corollary}
\begin{proof}
Suppose $U\in \Fr$ and $a_h$ is a maximal occurrence of a generalized fractional power in $U$, $U = La_hR$. Suppose $a_h$ contains a $\lambda$-forbidden word. Then, clearly, $a_h$ is a virtual member of the chart of $U$. From Proposition~\ref{reduct_power}, it follows that there exists an elementary multi-turn $a_h \mapsto \sum_{\substack{j = 1 \\ j\neq h}}^k a_i$, where $j = 1, \ldots, k$, $j\neq h$, are $\lambda$-semicanonical generalized fractional powers with the word length in $\Fr$ strictly smaller than the word length of $a_h$. Then $La_jR$ are derived monomials of $U$ and
\begin{equation*}
U = La_hR = \sum\limits_{\substack{j = 1 \\ j\neq h}}^k La_jR \mod \Dp(\langle U\rangle_d),
\end{equation*}
where the word length in $\Fr$ of all monomials appearing in the right-hand side is strictly smaller than the word length of $U$. We continue the reduction process for the obtained monomials if necessary. Since word length of the monomials strictly decreases after each step of the reduction, the process finishes and we obtain the result as a sum of $\lambda$-semicanonical words.

Assume $U \in V = \langle U_1, \ldots, U_k\rangle_d$. Since $V$ is generated by monomials and closed under taking derived monomials, $\Dp(\langle U\rangle_d) \subseteq \Dp(V)$. Hence, $U$ is equal to the same resulting sum of $\lambda$-semicanonical monomials as above modulo $\Dp(V)$.
\end{proof}

The most important property of $\lambda$-semicanonical word for us is the following.
\begin{proposition}[Non-degeneracy]
\label{non_degeneracy}
Suppose $a_h$ is a $\lambda$-semicanonical generalized fractional power, $a_h \mapsto \sum_{\substack{j = 1 \\ j\neq h}}^{l} a_j$ is an elementary multi-turn. Then there exists $a_j$, $j\neq h$, such that $\LM(a_j) \geqslant \tau$.
\end{proposition}
\begin{proof}
Let $I$ and $F$ be the initial and the final points of the paths corresponding to $a_j$, $j = 1, \ldots, l$, in the $v$-diagram. There are the following possible positions of $I$ and $F$:
\begin{enumerate}[label=\textbf{Case~\arabic*}, ref=Case~\arabic*]
\item
\label{non_degeneracy_small_v_arc}
The points $I$ and $F$ lie on the $v$-arc.
\item
\label{non_degeneracy_small_w_arc}
The points $I$ and $F$ lie on a $w$-arc.
\item
\label{non_degeneracy_small_v_w_arc}
The point $I$ lies on the $v$-arc, the point $F$ lies on a $w$-arc.
\item
\label{non_degeneracy_small_w_v_arc}
The point $I$ lies on a $w$-arc, the point $F$ lies on the $v$-arc.
\end{enumerate}

Consider \ref{non_degeneracy_small_v_arc} for $a_1, \ldots, a_l$. Assume the contrary, that is, $\LM(a_j) < \tau$, $j = 1, \ldots, l$, $j \neq h$. Recall possible forms of monomials of $\LM$-measure less than $\tau$ (we enumerated them in the proof of Proposition~\ref{transversality}). We denote the smallest path from $I$ to $F$ by $b^{\prime}$, the smallest path from $I$ to $O$ by $b_1$, and the smallest path from $O$ to $F$ by $b_2$. Since $\LM(a_j) < \tau$ for $j \neq h$, the smallest path between $I$ and $F$ is necessarily of $\LM$-measure less than $\tau$. Then there are the following possibilities (for a graphical illustration see Proposition~\ref{transversality}).
\begin{enumerate}[label=\ref{small_v_arc}.\arabic*]
\item
\label{non_degeneracy_small_v_arc1}
$b^{\prime}$ contains the point $O$;
\item
\label{non_degeneracy_small_v_arc2}
$b^{\prime}$ does not contain the point $O$ and $\LM(b_1) + \LM(b_2) < \tau$;
\item
\label{non_degeneracy_small_v_arc3}
$b^{\prime}$ does not contain the point $O$ and $\LM(b_1) + \LM(b_2) \geqslant \tau$.
\end{enumerate}

We will study the following particular configuration in~\ref{non_degeneracy_small_v_arc1} in detail. Recall that we use the notation $v = v_iv_mv_f$.
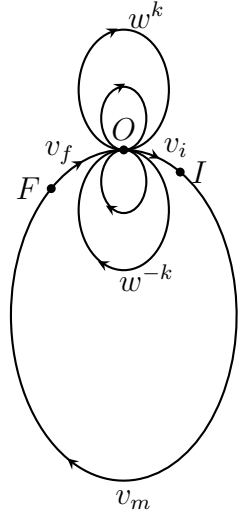
\begin{center}
\begin{tikzpicture}
\begin{scope}
\coordinate (O) at ($(0,0)+(90:1.5 and 2.2)$);
\coordinate (Si) at ($(0,0)+(60:1.5 and 2.2)$);
\coordinate (Se) at ($(0,0)+(130:1.5 and 2.2)$);

\coordinate (vi) at (0.39, 2.125);
\coordinate (vm) at (0.13, -2.19);
\coordinate (vf) at (-0.513, 2.06);

\draw[black, thick, arrow=0.6] (O) arc (90:60:1.5 and 2.2);
\node[right, yshift=3] at (vi) {$v_i$};
\draw[black, thick, arrow=0.6] (Si) arc (60:-360+130:1.5 and 2.2);
\node[below] at (vm) {$v_m$};
\draw[black, thick, arrow=0.5] (Se) arc (130:90:1.5 and 2.2);
\node[left, yshift=3] at (vf) {$v_f$};

\node[circle,fill,inner sep=1.2] at (O) {};
\path let \p1 = (O) in node at (\x1, \y1+8) {$O$};

\node[circle,fill,inner sep=1.2] at (Si) {};
\node[right] at (Si) {$I$};
\node[circle,fill,inner sep=1.2] at (Se) {};
\node[left] at (Se) {$F$};

\draw[black, thick, reversearrow=0.3] (0, 2.2+0.8) ellipse (0.6 and 0.8) node at (0.3, 2.2+1.8) {$w^k$};
\draw[black, thick, reversearrow=0.3] (0, 2.2+0.7*0.6) ellipse (0.5*0.6 and 0.7*0.6);

\draw[black, thick, reversearrow=0.7] (0, 2.2-0.8) ellipse (0.6 and 0.8) node at (0.3, 2.2-1.7) {$w^{-k}$};
\draw[black, thick, reversearrow=0.7] (0, 2.2-0.7*0.6) ellipse (0.5*0.6 and 0.7*0.6);
\node at (0, -3.5) {$\begin{aligned}
&b_1 = v_i^{-1}, b_2 = v_f^{-1},\\
&b^{\prime} = v_i^{-1}v_f^{-1};
\end{aligned}$};
\end{scope}
\end{tikzpicture}
\end{center}
Returning to pictures~\eqref{path_type1}, we obtain that the following list of generalized fractional powers corresponds to the above picture:
\begin{align*}
&v_mv_fM(v, w)v_f^{-1},\\
&v_mv_fM(v, w)v_iv_m,\\
&v_i^{-1}M(v, w)v_iv_m,\\
&v_i^{-1}M(v, w)v_f^{-1}.
\end{align*}
So, every $a_j$, $j = 1, \ldots, l$, is of the above form. Denote possible forms of $a_h$ by
\begin{align*}
&v_mv_fM_h(v, w)v_f^{-1},\\
&v_mv_fM_h(v, w)v_iv_m,\\
&v_i^{-1}M_h(v, w)v_iv_m,\\
&v_i^{-1}M_h(v, w)v_f^{-1}.
\end{align*}
Since $a_h$ is $\lambda$-semicanonical, the monomial $M_h(v, w)$ does not contain negative powers of $v$. By the assumption, $\LM(a_j) < \tau$ for $j\neq h$; hence,
\begin{equation*}
a_j = v_i^{-1}w^{k_j}v_f^{-1}.
\end{equation*}
So, there are the following possible forms of the elementary multi-turn $a_h \mapsto \sum_{\substack{j = 1 \\ j\neq h}}^{l} a_j$ depending on the form of $a_h$:
\begin{align*}
&v_mv_fM_h(v, w)v_f^{-1} \mapsto \sum\limits_{j = 1}^{l} v_i^{-1}w^{k_j}v_f^{-1},\\
&v_mv_fM_h(v, w)v_iv_m \mapsto \sum\limits_{j = 1}^{l} v_i^{-1}w^{k_j}v_f^{-1},\\
&v_i^{-1}M_h(v, w)v_iv_m \mapsto \sum\limits_{j = 1}^{l} v_i^{-1}w^{k_j}v_f^{-1},\\
&v_i^{-1}M_h(v, w)v_f^{-1} \mapsto \sum\limits_{j = 1}^{l} v_i^{-1}w^{k_j}v_f^{-1}.
\end{align*}

We transform every elementary multi-turn above to an elementary multi-turn of a monomial over $v, w$, multiplying every expression above by corresponding generalized fractional powers:
\begin{align*}
&vM_h(v, w) \mapsto \sum\limits_{j = 1}^{l} w^{k_j},\\
&vM_h(v, w)v \mapsto \sum\limits_{j = 1}^{l} w^{k_j},\\
&M_h(v, w)v \mapsto \sum\limits_{j = 1}^{l} w^{k_j},\\
&M_h(v, w) \mapsto \sum\limits_{j = 1}^{l} w^{k_j}.
\end{align*}
That is, we obtain an elementary multi-turn of the form
\begin{equation}
\label{final_multi_turn}
\widetilde{M}_h(v, w) \mapsto \sum\limits_{j = 1}^{l} w^{k_j},
\end{equation}
where the monomial $\widetilde{M}_h(v, w)$ does not contain negative powers of $v$. Since all monomials on the right-hand side are of zero $\LM$-measure, by Proposition~\ref{transversality}, $\LM(\widetilde{M}_h(v, w)) \geqslant \tau$. Hence, the monomial $\widetilde{M}_h(v, w)$ contains at least one positive power of $v$. Then, clearly, the polynomial
\begin{equation*}
P(v, w) = \widetilde{M}(v, w) + \sum\limits_{j = 1}^{l} w^{k_j}
\end{equation*}
does not satisfy condition~\eqref{vanish_in_field}, so, \eqref{final_multi_turn} can not be a multi-turn. We obtain a contradiction.

All the rest of the configurations in \ref{non_degeneracy_small_v_arc} --- \ref{non_degeneracy_small_w_v_arc} are processed in the same way.
\end{proof}

Using Corollary~\ref{lemma_about_generating}, Proposition~\ref{correspondence_to_tensor_product} and Proposition~\ref{non_degeneracy}, we obtain the following important statement.
\begin{corollary}
\label{non_trivial_spaces}
Assume $W = \langle Z\rangle_d$, where $Z$ is a monomial. Then the space $W / (\Dp(W) + \Low(W))$ is non-trivial if and only if there exists a $\lambda$-semicanonical monomial $\widetilde{Z} \in W$ such that $W = \langle \widetilde{Z}\rangle_d$. Moreover, if $X$ is a $\lambda$-semicanonical monomial such that $X \in W \setminus \Low(W)$ (and then $W = \langle X\rangle_d$), then $X + \Dp(W) + \Low(W) \neq 0$ as an element of $W / (\Dp(W) + \Low(W))$.
\end{corollary}
\begin{proof}
Assume the quotient space $W / (\Dp(W) + \Low(W))$ is non-trivial. Consider the set of all monomials from $W \setminus \Low(W)$. If every monomial from this set belongs to $\Dp(W) + \Low(W)$, then the quotient space $W / (\Dp(W) + \Low(W))$ is trivial. Hence, there exists a monomial $Z^{\prime} \in W \setminus \Low(W)$ such that $Z^{\prime} \notin \Dp(W) + \Low(W)$. From Corollary~\ref{lemma_about_generating} it follows that
\begin{equation*}
Z^{\prime} = \sum\limits_{i = 1}^l Z_i \mod \Dp(W),
\end{equation*}
where $Z_i \in W$ are $\lambda$-semicanonical monomials. Since $Z^{\prime} \notin \Dp(W) + \Low(W)$, there exists a monomial $Z_{i_0} \in W \setminus \Low(W)$ in this sum. Hence, from Lemma~\ref{derived_spaces_equality} it follows that $W = \langle Z_{i_0}\rangle_d$.

Assume $W = \langle \widetilde{Z}\rangle_d$, where $\widetilde{Z}$ is $\lambda$-semicanonical monomial. Then, by Proposition~\ref{correspondence_to_tensor_product}, we have
\begin{equation*}
W / (\Dp(W) + \Low(W)) \cong A_1 / (\Dp(A_1) + \Low_1)\otimes \ldots \otimes A_k / (\Dp(A_k) + \Low_k),
\end{equation*}
where $A_1, \ldots, A_k$ are spaces generated by generalized fractional powers and corresponding to each place in the chart of $\widetilde{Z}$ by formula~\eqref{tens_prod_components}. In particular, if $a^{(1)}, \ldots, a^{(k)}$ are all virtual members of the chart of $\widetilde{Z}$ enumerated from left to right, then $a^{(1)} \in A_1, \ldots, a^{(k)} \in A_k$.

Recall that $\Dp(A_i)$ consists of supports of elementary multi-turns. Therefore, by Proposition~\ref{non_degeneracy}, we obtain $a^{(i)} \notin \Dp(A_i) + \Low_i, i = 1, \ldots, k$. Hence,
\begin{equation}
\label{non_degenerated_tensor}
(a^{(1)} + \Dp(A_1) + \Low_1) \otimes \ldots \otimes (a^{(k)} + \Dp(A_k) + \Low_k) \neq 0.
\end{equation}
Then, by definition of the isomorphism
\begin{equation*}
\overline{\mu}_2[\widetilde{Z}] : A_1 / (\Dp(A_1) + \Low_1)\otimes \ldots \otimes A_k / (\Dp(A_k) + \Low_k) \to W / (\Dp(W) + \Low(W)),
\end{equation*}
we have
\begin{multline*}
\overline{\mu}_2[\widetilde{Z}]((a^{(1)} + \Dp(A_1) + \Low_1) \otimes \ldots \otimes (a^{(k)} + \Dp(A_k) + \Low_k)) \\
= \mu[\widetilde{Z}](a^{(1)}\otimes \ldots \otimes a^{(k)}) + \Dp(W) + \Low(W) = \widetilde{Z} + \Dp(W) + \Low(W).
\end{multline*}
Then, by~\eqref{non_degenerated_tensor}, we obtain that $\widetilde{Z} + \Dp(W) + \Low(W) \neq 0$. Thus, the quotient space $W / (\Dp(W) + \Low(W))$ is non-trivial.

Let $X$ be a $\lambda$-semicanonical monomial from $W \setminus \Low(W)$. Then, by Lemma~\ref{derived_spaces_equality}, we have $W = \langle X\rangle_d$. Therefore, by the same argument as above, we obtain that $X + \Dp(W) + \Low(W) \neq 0$. This completes the proof.
\end{proof}

\subsection{How to calculate a basis in $\mathbb{Z}_2\Fr / \Ideal$ using $\lambda$-semicanonical words}
\label{calculate_basis_section}
As before, we denote the space $\langle U_1, \ldots, U_k\rangle_d$, where $U_1, \ldots, U_k$ are monomials, by $V$. Putting together the results of Corollary~\ref{finite_gen_subspaces_structure}, Theorem~\ref{structure_of_quotient_space} and Proposition~\ref{component_subspaces_structure}, we obtain
\begin{equation}
\label{fin_gen_space_final}
(V + \langle\mathcal{T}^{\prime}\rangle) / \langle\mathcal{T}^{\prime}\rangle \cong \bigoplus\limits_{n = 0}^{N - 1}\bigoplus\limits_{i \in I^{\prime}_n} V_i^{(n)}/(\Dp(V_i^{(n)}) + \Low(V_i^{(n)})),
\end{equation}
where $V_i^{(n)} = \langle Z_i^{(n)}\rangle_d$, $Z_i^{(n)}$ is a monomial from $\Ft_nV$, $V_i^{(n)} \nsubseteq \Ft_{n + 1}V$ and the spaces $V_i^{(n)}$ are pairwise different.

\begin{proposition}
\label{basis_correspondence}
Assume
\begin{equation*}
\left\lbrace \overline{W}^{(i, n)}_j \mid \overline{W}^{(i, n)}_j \in V_i^{(n)}/(\Dp(V_i^{(n)}) + \Low(V_i^{(n)})), j\in \mathbb{N}\right\rbrace,
\end{equation*}
is a basis of $V_i^{(n)}/(\Dp(V_i^{(n)}) + \Low(V_i^{(n)}))$, where $n = 0, \ldots, N - 1$ and $i \in I_n^{\prime}$. Let $W^{(i, n)}_j \in V_i^{(n)}$ be an arbitrary representative of the coset $\overline{W}^{(i, n)}_j$. Then
\begin{equation*}
\bigcup\limits_{n = 0}^{N - 1}\bigcup\limits_{i \in I^{\prime}_n}\left\lbrace W^{(i, n)}_j + \langle\mathcal{T}^{\prime}\rangle \mid j \in \mathbb{N}\right\rbrace
\end{equation*}
is a basis of $(V + \langle\mathcal{T}^{\prime}\rangle) / \langle\mathcal{T}^{\prime}\rangle$.
\end{proposition}

\begin{proof}
While the statement is pretty obvious, we prefer to give a proof to recollect the previously stated facts.

In Proposition~\ref{component_subspaces_structure} we proved that there exists an isomorphism
\begin{equation*}
\varphi : \Ft_nV / (\Dp(\Ft_nV) + \Ft_{n + 1}V) \to \bigoplus\limits_{i \in I^{\prime}_n} V_i^{(n)}/(\Dp(V_i^{(n)}) + \Low(V_i^{(n)}))
\end{equation*}
that acts as follows
\begin{equation}
\label{isom_action}
\varphi(U + \Dp(\Ft_nV) + \Ft_{n + 1}V) = (0, \ldots, 0, U + \Dp(V_i^{(n)}) + \Low(V_i^{(n)}), 0, \ldots),
\end{equation}
where $U$ is a monomial, $U \in \Ft_nV \setminus \Ft_{n + 1}V$. Suppose
\begin{equation*}
\left\lbrace \overline{W}^{(i, n)}_j \mid \overline{W}^{(i, n)}_j \in V_i^{(n)}/(\Dp(V_i^{(n)}) + \Low(V_i^{(n)})), j\in \mathbb{N}\right\rbrace
\end{equation*}
is a basis of $V_i^{(n)}/(\Dp(V_i^{(n)}) + \Low(V_i^{(n)}))$. Then from~\eqref{isom_action} it easily follows that
\begin{equation*}
\bigcup\limits_{i \in I^{\prime}}\left\lbrace W^{(i, n)}_j + \Dp(\Ft_nV) + \Ft_{n + 1}V \mid j\in \mathbb{N}\right\rbrace,
\end{equation*}
where $W^{(i, n)}_j \in \Ft_nV$ is an arbitrary representative of the coset $\overline{W}^{(i, n)}_j$, is a basis of $\Ft_nV / (\Dp(\Ft_nV) + \Ft_{n + 1}V)$.

Recall that
\begin{equation*}
\Ft_n(V / \Dp(V)) = (\Ft_nV + \Dp(V)) / \Dp(V)
\end{equation*}
and, by the isomorphism theorems, we have the following canonical isomorphisms:
\begin{align*}
&\pi_1 : \Ft_n(V / \Dp(V)) / \Ft_{n + 1}(V / \Dp(V)) \to (\Ft_nV + \Dp(V)) / (\Ft_{n + 1}V + \Dp(V)),\\
&\pi_2 : (\Ft_nV + \Dp(V)) / (\Ft_{n + 1}V + \Dp(V)) \to \Ft_{n}V / (\Ft_{n + 1}V + \Dp(V) \cap \Ft_nV).
\end{align*}
Let us recall that from Proposition~\ref{fall_to_smaller_subspace} it follows that $\Dp(V) \cap \Ft_nV = \Dp(\Ft_{n}V)$. Therefore, we have\begin{align*}
\Ft_n(V / \Dp(V)) / \Ft_{n + 1}(V / \Dp(V))& \overset{\pi_1} {\longrightarrow} (\Ft_nV + \Dp(V)) / (\Ft_{n + 1}V + \Dp(V)) \\ &\overset{\pi_2}{\longrightarrow}
\Ft_{n}V / (\Ft_{n + 1}V + \Dp(\Ft_nV)).
\end{align*}
Let $W$ be an arbitrary element of $\Ft_nV$ (not necessarily a monomial). Then, by the well known construction of the canonical isomorphisms in the isomorphism theorems, we have
\begin{align*}
\pi_2 \circ \pi_1(W& + \Dp(V) + \Ft_{n + 1}(V / \Dp(V))) \\
&= \pi_2(W + \Dp(V) + \Ft_{n + 1}V) = W + \Ft_{n + 1}V + \Dp(\Ft_nV).
\end{align*}
So, since
\begin{equation*}
\bigcup\limits_{i \in I^{\prime}}\left\lbrace W^{(i, n)}_j + \Dp(\Ft_nV) + \Ft_{n + 1}V \mid j\in \mathbb{N}\right\rbrace
\end{equation*}
is a basis of $\Ft_{n}V / (\Ft_{n + 1}V + \Dp(\Ft_nV))$, we obtain that
\begin{equation}
\label{basis_in_component}
\bigcup\limits_{i \in I^{\prime}}\left\lbrace W^{(i, n)}_j + \Dp(V) + \Ft_{n + 1}(V / \Dp(V)) \mid j\in \mathbb{N}\right\rbrace
\end{equation}
is a basis of $\Ft_n(V / \Dp(V)) / \Ft_{n + 1}(V / \Dp(V))$.

We have the graded space
\begin{equation*}
\Gr(V / \Dp(V)) = \bigoplus\limits_{n = 0}^{N - 1}\Ft_n(V / \Dp(V)) / \Ft_{n + 1}(V / \Dp(V)),
\end{equation*}
where $\Ft_n(V / \Dp(V)) = (\Ft_nV + \Dp(V)) / \Dp(V)$. Assume that
\begin{equation*}
\left\lbrace \overline{Z}^{(n)}_j \mid \overline{Z}^{(n)}_j \in \Ft_n(V / \Dp(V)) / \Ft_{n + 1}(V / \Dp(V)), j\in \mathbb{N}\right\rbrace
\end{equation*}
is a basis of $\Ft_n(V / \Dp(V)) / \Ft_{n + 1}(V / \Dp(V))$, $n = 0, 1, \ldots, N - 1$. Then one can easily show that
\begin{equation*}
\bigcup\limits_{n = 0}^{N - 1}\left\lbrace Z^{(n)}_j \mid Z^{(n)}_j \in \Ft_n(V / \Dp(V)), j\in \mathbb{N}\right\rbrace,
\end{equation*}
where $Z^{(n)}_j$ is an arbitrary representative of the coset $\overline{Z}^{(n)}_j$, forms a basis of $V / \Dp(V)$. Hence, using this observation together with~\eqref{basis_in_component} and Corollary~\ref{finite_gen_subspaces_structure}, we obtain that
\begin{equation*}
\bigcup\limits_{n = 0}^{N - 1}\bigcup\limits_{i \in I^{\prime}_n}\left\lbrace W^{(i, n)}_j + \langle\mathcal{T}^{\prime}\rangle \mid j \in \mathbb{N}\right\rbrace
\end{equation*}
is a basis of $(V + \langle\mathcal{T}^{\prime}\rangle) / \langle\mathcal{T}^{\prime}\rangle$.
\end{proof}

\begin{proposition}
\label{spaces_monomials_correspondence}
Let $\lbrace X_j\rbrace$ be all the $\lambda$-semicanonical monomials of $V$. Then there exists one-to-one correspondence between all the different spaces $\langle X_j\rangle_d$ and all the spaces $V_i^{(n)}$ from~\eqref{fin_gen_space_final} such that $V_i^{(n)} / (\Dp(V_i^{(n)}) + \Low(V_i^{(n)}))$ is non-trivial. Namely,
\begin{enumerate}
\item [$(1)$]
every space $V_i^{(n)}$ such that $V_i^{(n)} / (\Dp(V_i^{(n)}) + \Low(V_i^{(n)})) \neq 0$ is equal to some space $\langle X_j\rangle_d$;
\item[$(2)$]
every space $\langle X_j\rangle_d$ is equal to some space $V_i^{(n)}$ such that $V_i^{(n)} / (\Dp(V_i^{(n)}) + \Low(V_i^{(n)})) \neq 0$.
\end{enumerate}
\end{proposition}
\begin{proof}
Consider a space $V_i^{(n)}$ such that $V_i^{(n)} / (\Dp(V_i^{(n)}) + \Low(V_i^{(n)})) \neq 0$. Then from Corollary~\ref{non_trivial_spaces} it follows that there exists a $\lambda$-semicanonical monomial $X_{j_0} \in V_i^{(n)}$ such that $V_i^{(n)} = \langle X_{j_0}\rangle_d$.

Let $X_j \in V$ be a $\lambda$-semicanonical monomial. Since $\Ft_NV = 0$, there exists a number $n_0$ such that $X_j \in \Ft_{n_0}V$ and $X_j \notin \Ft_{n_0 + 1}V$. Then, by Proposition~\ref{component_subspaces_structure}, $X_j \in V_{i}^{(n_0)}$ for some $i \in I^{\prime}_{n_0}$. Since $X_j \notin \Ft_{n_0 + 1}V$, we have $X_j \notin \Low(V_{i}^{(n_0)})$. Then, by Lemma~\ref{derived_spaces_equality}, $V_{i}^{(n_0)} = \langle X_j\rangle_d$. It remains to notice that, by Corollary~\ref{non_trivial_spaces}, we have $V_i^{(n_0)} / (\Dp(V_i^{(n_0)}) + \Low(V_i^{(n_0)})) \neq 0$.
\end{proof}

The next theorem follows from Proposition~\ref{basis_correspondence} and Proposition~\ref{spaces_monomials_correspondence}.
\begin{theorem}
\label{whole_quotient_ring_structure}
Let $\lbrace X_j\rbrace_{j \in \mathbb{N}}$ be all $\lambda$-semicanonical monomials from $\Fr$. Let $\lbrace V_i\rbrace_{i \in \mathbb{N}}$ be all the different spaces $\lbrace\langle X_j\rangle_d\rbrace_{j \in \mathbb{N}}$ ($V_{i_1} \neq V_{i_2}$ for $i_1 \neq i_2$). If
\begin{equation*}
\left\lbrace \overline{W}^{(i)}_j \mid \overline{W}^{(i)}_j \in V_i/(\Dp(V_i) + \Low(V_i)), j \in \mathbb{N}\right\rbrace
\end{equation*}
is a basis of $V_i/(\Dp(V_i) + \Low(V_i))$, then
\begin{equation*}
\bigcup\limits_{i \in \mathbb{N}} \left\lbrace W^{(i)}_j + \Ideal \mid j \in \mathbb{N}\right\rbrace
\end{equation*}
is a basis of $\mathbb{Z}_2\Fr / \Ideal$, where $W^{(i)}_j$ is an arbitrary representative of the coset $\overline{W}^{(i)}_j $. In particular,
\begin{equation*}
\mathbb{Z}_2\Fr / \Ideal \cong \bigoplus\limits_{i \in \mathbb{N}} V_i/(\Dp(V_i) + \Low(V_i))
\end{equation*}
as vector spaces, and the right-hand side is explicitly described in Proposition~\ref{correspondence_to_tensor_product}.
\end{theorem}
\begin{proof}
Let us enumerate all the monomials from $\Fr$, namely, $\Fr = \lbrace U_1, U_2, \ldots, U_k, \ldots\rbrace$. Then, evidently,
\begin{equation*}
\mathbb{Z}_2\Fr / \Ideal = \bigcup\limits_{k = 1}^{\infty} \left((\langle U_1, \ldots, U_k\rangle_d + \langle\mathcal{T}^{\prime}\rangle) / \langle\mathcal{T}^{\prime}\rangle\right).
\end{equation*}
Consider spaces $\langle U_1, \ldots, U_k\rangle_d \subseteq \langle U_1, \ldots, U_k, U_{k + 1}\rangle_d$. For the quotient space $(\langle U_1, \ldots, U_k\rangle_d + \langle\mathcal{T}^{\prime}\rangle) / \langle\mathcal{T}^{\prime}\rangle$ we have decomposition~\eqref{fin_gen_space_final}, that is,
\begin{equation*}
(\langle U_1, \ldots, U_k\rangle_d + \langle\mathcal{T}^{\prime}\rangle) / \langle\mathcal{T}^{\prime}\rangle \cong \bigoplus\limits_{n = 0}^{N_k - 1}\bigoplus\limits_{i \in I^{\prime}_n} V_i^{(n)}/(\Dp(V_i^{(n)}) + \Low(V_i^{(n)})),
\end{equation*}
where $N_k$ is the number of non-zero levels of the filtration on $\langle U_1, \ldots, U_k\rangle_d$.

Let $\lbrace X_j \rbrace$ be all the $\lambda$-semicanonical monomials from $\langle U_1, \ldots, U_k\rangle_d$. Then, obviously, the set of all $\lambda$-semicanonical monomials from $\langle U_1, \ldots, U_{k + 1}\rangle_d$ can be represented as
\begin{equation*}
\lbrace X_j \rbrace \cup \lbrace \widetilde{X}_j\rbrace,
\end{equation*}
where $\lbrace \widetilde{X}_j\rbrace$ collects the additional $\lambda$-semicanonical monomials that contain in $\langle U_1, \ldots, U_k, U_{k + 1}\rangle_d$ but do not contain in $\langle U_1, \ldots, U_k\rangle_d$ (in particular, the set $\lbrace \widetilde{X}_j\rbrace$ may be empty). Then from Proposition~\ref{spaces_monomials_correspondence}, it follows that decomposition~\eqref{fin_gen_space_final} for $(\langle U_1, \ldots, U_{k + 1}\rangle_d + \langle\mathcal{T}^{\prime}\rangle) / \langle\mathcal{T}^{\prime}\rangle$ can be written as
\begin{align*}
(\langle U_1, \ldots, U_{k + 1}\rangle_d + \langle\mathcal{T}^{\prime}\rangle) / \langle\mathcal{T}^{\prime}\rangle &
\cong \left(\bigoplus\limits_{n = 0}^{N_k - 1}\bigoplus\limits_{i \in I^{\prime}_n} V_i^{(n)}/(\Dp(V_i^{(n)}) + \Low(V_i^{(n)}))\right)\\ & \bigoplus \left(\bigoplus\limits_{n = 0}^{N_{k + 1} - 1}\bigoplus\limits_{h \in I^{\prime\prime}_n} \widetilde{V}_h^{(n)}/(\Dp(\widetilde{V}_h^{(n)}) + \Low(\widetilde{V}_h^{(n)}))\right),
\end{align*}
where $N_{k + 1}$ is the number of non-zero levels of the filtration on $\langle U_1, \ldots, U_k, U_{k + 1}\rangle_d$. The set $I_n^{\prime\prime}$ collects the additional spaces $\widetilde{V}_h^{(n)}$, not appearing in $I_n^{\prime}$, that correspond to additional monomials $\lbrace \widetilde{X}_j\rbrace$.

From Corollary~\ref{non_trivial_spaces}, it follows that the spaces $V_i^{(n)}/(\Dp(V_i^{(n)}) + \Low(V_i^{(n)}))$ and $\widetilde{V}_h^{(n)}/(\Dp(\widetilde{V}_h^{(n)}) + \Low(\widetilde{V}_h^{(n)}))$ are non-trivial. We will choose a basis in every direct summand $V_i^{(n)}/(\Dp(V_i^{(n)}) + \Low(V_i^{(n)}))$ and $\widetilde{V}_h^{(n)}/(\Dp(\widetilde{V}_h^{(n)}) + \Low(\widetilde{V}_h^{(n)}))$ independently. Then from Proposition~\ref{basis_correspondence}, it easily follows that there exists a basis $\mathcal{B}_{k + 1}$ of $(\langle U_1, \ldots, U_{k + 1}\rangle_d + \langle\mathcal{T}^{\prime} \rangle) / \langle\mathcal{T}^{\prime} \rangle$ such that it expands the previously chosen basis $\mathcal{B}_{k}$ of $(\langle U_1, \ldots, U_k\rangle_d + \langle\mathcal{T}^{\prime} \rangle) / \langle\mathcal{T}^{\prime} \rangle$. Therefore, finally, we will construct the increasing sequence of basis
\begin{equation*}
\mathcal{B}_1 \subseteq \mathcal{B}_2 \subseteq \ldots \subseteq \mathcal{B}_k \subseteq \ldots\, .
\end{equation*}
Since the ring $\mathbb{Z}_2\Fr / \langle \mathcal{T}^{\prime} \rangle$ is a union of the subspaces $(\langle U_1, \ldots, U_k\rangle_d + \langle\mathcal{T}^{\prime}\rangle) / \langle\mathcal{T}^{\prime}\rangle$, the union $\bigcup_{k = 1}^{\infty} \mathcal{B}_k$ is a basis of $\mathbb{Z}_2\Fr / \langle \mathcal{T}^{\prime} \rangle = \mathbb{Z}_2\Fr / \Ideal$.
\end{proof}

From Theorem~\ref{whole_quotient_ring_structure} we obtain the claim that served as the main motivation.
\begin{corollary}
The quotient ring $\mathbb{Z}_2\Fr / \Ideal$ is non-trivial.
\end{corollary}

\begin{remark}
It is possible to choose a basis in every direct summand $V_i /(\Dp(V_i) + \Low(V_i))$, where $V_i = \langle Z_i\rangle_d$, $Z_i$ is a $\lambda$-semicanonical monomial, in the following way. In Proposition~\ref{correspondence_to_tensor_product}, we have proved that $V_i/(\Dp(V_i) + \Low(V_i))$ is isomorphic to a tensor product of spaces that corresponds to each virtual member of the chart of $Z_i$. Namely, there exist spaces $A_1, \ldots, A_{k_i}$ defined by~\eqref{tens_prod_components} and the isomorphism
\begin{align*}
\overline{\mu}_2[Z_i] :& A_1 / (\Dp(A_1) + \Low(A_1)) \otimes \ldots \otimes A_{k_i}/ (\Dp(A_{k_i}) + \Low(A_{k_i})) \\
&\to V_i/(\Dp(V_i) + \Low(V_i))
\end{align*}
defined by~\eqref{isom_to_tensot_prod_dep}. So, if we choose a basis in all the different spaces $A_j / (\Dp(A_j) + \Low(A_j))$, this enables us to construct the corresponding basis in every direct summand $V_i /(\Dp(V_i) + \Low(V_i))$.

From Proposition~\ref{lemma_about_generating} it follows that $A_j / (\Dp(A_j) + \Low(A_j))$ is generated by cosets of $\lambda$-semicanonical generalized fractional powers that belong to $A_j$. We choose a basis of $A_j / (\Dp(A_j) + \Low(A_j))$ that consists of such cosets. Let
\begin{equation*}
(b_l^{(1)} + \Dp(A_1) + \Low(A_1))\otimes \ldots \otimes (b_l^{(k_i)} + \Dp(A_{k_i}) + \Low(A_{k_i})), l \in \mathbb{N},
\end{equation*}
be the basis elements, where $b_l^{(1)}, \ldots, b_l^{(k_i)}$ are $\lambda$-semicanonical generalized fractional powers. Then their images
\begin{align*}
\overline{\mu}_2[Z_i](b_l^{(1)} &+ \Dp(A_1) + \Low(A_1))\otimes \ldots \otimes (b_l^{(k_i)} + \Dp(A_{k_i}) + \Low(A_{k_i})) \\ &=
\mu[Z_i](b_l^{(1)} \otimes \ldots \otimes b_l^{(k_i)}) + \Dp(V_i) + \Low(V_i)
\end{align*}
form a basis of $V_i /(\Dp(V_i + \Low(V_i))$.

We put
\begin{equation}
\label{basis_representatives}
W_l^{(i)} = \mu[Z_i](b_l^{(1)} \otimes \ldots \otimes b_l^{(k_i)}) .
\end{equation}
Then from Theorem~\ref{whole_quotient_ring_structure} it follows that
\begin{equation}
\label{basis_from_tensor_product}
\bigcup\limits_{i = 1}^{\infty}\lbrace W_l^{(i)} + \Ideal \mid l\in \mathbb{N}\rbrace
\end{equation}
is a basis of $\mathbb{Z}_2\Fr / \Ideal$. Notice that, by definition of $\mu[Z_i]$, $W_l^{(i)}$ are monomials. From the description given in Section~\ref{mt_configurations}, it follows that the monomials $W_l^{(i)}$ are $(\lambda + 2\varepsilon)$-semicanonical.

Notice that the constructed basis satisfies the following property. Assume $U$ is a monomial and there is the base decomposition
\begin{equation*}
U + \Ideal = \sum\limits_{i, l} W_l^{(i)} + \Ideal.
\end{equation*}
Then every $W_l^{(i)}$ in this decomposition is a derived monomial of $U$, in particular, $f(W_l^{(i)}) \leqslant f(U)$.
\end{remark}

\section{Geometry of the multiplication}
\label{multiplication_geometry_section}
The linear span of all $\lambda$-semicanonical words is denoted by $S_{\lambda}$. However, it turns out that that it is more convenient to define a subspace $\widetilde{S}_{\lambda}$ for which $S_{\lambda} \subseteq \widetilde{S}_{\lambda} \subseteq S_{\lambda + 2\varepsilon}$ that contains all the elements of the basis of $\mathbb{Z}_2\Fr / \Ideal$. The cosets of elements of $\widetilde{S}_{\lambda}$ linearly generate $\mathbb{Z}_2\Fr / \Ideal$, and they display nice multiplication properties. Namely, the result of the multiplication can be represented as a sum of monomials from $\widetilde{S}_{\lambda}$ that form thin triangles with the factors.

\subsection{Definition of the space $\widetilde{S}_{\lambda}$}
\label{stilde_def_section}
Let $U$ be a $(\lambda + 2\varepsilon)$-semicanonical word and $a$ be a virtual member of its chart. Assume $l$ and $r$ are virtual members of the chart of $U$ such that $l$ and $r$ are the left and the right neighbour of $a$ in the chart, respectively. Let $a_l$ be an overlap between $a$ and $l$ (possibly empty), $a_r$ be an overlap between $a$ and $r$ (possibly empty). Then $a = a_la_ma_r$.
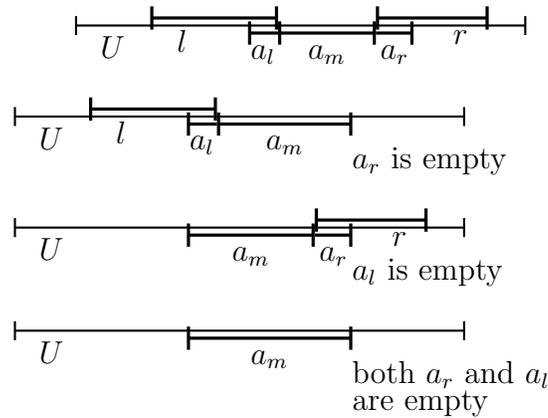
\begin{center}
\begin{tikzpicture}
\draw[|-|, black, thick] (0,0)--(6,0);
\node at (0.5, 0) [below] {$U$};
\draw[|-|, black, very thick] (1,0.1)--(2.7,0.1) node[near start, below] {$l$};
\draw[|-, black, very thick] (2.3,-0.1)--(2.7,-0.1) node[midway, below] {$a_l$};
\draw[|-|, black, very thick] (2.7,-0.1)--(4,-0.1) node[midway, below] {$a_m$};
\draw[-|, black, very thick] (4,-0.1)--(4.5,-0.1) node[midway, below] {$a_r$};
\draw[|-|, black, very thick] (4,0.1)--(5.5,0.1) node[near end, below] {$r$};
\end{tikzpicture}
\end{center}
\begin{center}
\begin{tikzpicture}
\draw[|-|, black, thick] (0,0)--(6,0);
\node at (0.5, 0) [below] {$U$};
\draw[|-|, black, very thick] (1,0.1)--(2.7,0.1) node[near start, below] {$l$};
\draw[|-, black, very thick] (2.3,-0.1)--(2.7,-0.1) node[midway, below] {$a_l$};
\draw[|-|, black, very thick] (2.7,-0.1)--(4.5,-0.1) node[midway, below] {$a_m$};
\node[text width=3cm, align=left, below, yshift=1] at (6, -0.3) {\baselineskip=10pt $a_r$ is empty\par};
\end{tikzpicture}
\end{center}
\begin{center}
\begin{tikzpicture}
\draw[|-|, black, thick] (0,0)--(6,0);
\node at (0.5, 0) [below] {$U$};
\draw[|-|, black, very thick] (2.3,-0.1)--(4,-0.1) node[midway, below] {$a_m$};
\draw[-|, black, very thick] (4,-0.1)--(4.5,-0.1) node[midway, below] {$a_r$};
\draw[|-|, black, very thick] (4,0.1)--(5.5,0.1) node[near end, below] {$r$};
\node[text width=3cm, align=left, below, yshift=1] at (6, -0.3) {\baselineskip=10pt $a_l$ is empty\par};
\end{tikzpicture}
\end{center}
\begin{center}
\begin{tikzpicture}
\draw[|-|, black, thick] (0,0)--(6,0);
\node at (0.5, 0) [below] {$U$};
\draw[|-|, black, very thick] (2.3,-0.1)--(4.5,-0.1) node[midway, below] {$a_m$};
\node[text width=3cm, align=left, below, yshift=1] at (6, -0.3) {\baselineskip=10pt both $a_r$ and $a_l$ are empty\par};
\end{tikzpicture}
\end{center}
\begin{definition}[$\widetilde{S}_{\lambda}$-condition]
\label{stilde_def}
We say that a virtual member of the chart $a$ \emph{satisfies $\widetilde{S}_{\lambda}$-condition} if the subword $a_m$ of $a$ defined above is $\lambda$-semicanonical.

We denote by $\widetilde{S}_{\lambda}$ a linear span of words such that each virtual member of the chart of a word satisfies the $\widetilde{S}_{\lambda}$-condition. Clearly, every monomial that belongs to $\widetilde{S}_{\lambda}$ is $(\lambda + 2\varepsilon)$-semicanonical.
\end{definition}

Notice that unlike words from $S_\lambda$, words from $\widetilde{S}_{\lambda}$ may contain subwords that do not belong to $\widetilde{S}_{\lambda}$.

Recall that in the previous section we constructed a linear basis in $\mathbb{Z}_2\Fr / \Ideal$, see~\eqref{basis_from_tensor_product}. By formula~\eqref{basis_representatives} (since $b^{(1)}_l, \ldots, b^{(k_i)}_l$ are $\lambda$-semicanonical), we obtain that the monomials $W_l^{(i)}$ (the fixed representatives of the basis elements) belong to $\widetilde{S}_{\lambda}$.

Assume a monomial $U_h = La_hR \in \widetilde{S}_{\lambda}$ and let $a_h$ be a virtual member of its chart. Let $U_h \mapsto \sum_{\substack{j = 1 \\ j\neq h}}^{m}U_j$ be a multi-turn that comes from an elementary multi-turn $a_h \mapsto \sum_{\substack{j = 1 \\ j\neq h}}^{m}a_j$ such that $a_j$ are $\lambda$-semicanonical generalized fractional powers. Consider monomials $U_j = La_jR$ such that $a_j$ is a virtual member of the chart of $U_j$. Then, by Corollary~\ref{virtual_members_stability}, we obtain that the images of all virtual members of the chart of $U$ are virtual members of the chart of $U_j$. Hence, by the definition of the space $\widetilde{S}_{\lambda}$, we immediately obtain that these monomials belong to $\widetilde{S}_{\lambda}$ even if the neighbours of $a_h$ are prolonged in $U_j$.

Assume a monomial $U_h = La_hR$ is a $\lambda$-semicanonical word and we perform a multi-turn of $a_h$ as above. If again we consider monomials $U_j = La_jR$ such that $a_j$ is a virtual member of the chart of $U_j$, then $U_j$ is not necessarily $\lambda$-semicanonical, because the neighbours of $a_h$ may prolong in $U_j$ and become $\lambda$-forbidden. That is why the space $\widetilde{S}_{\lambda}$ is more convenient in our further proofs.

\subsection{Behaviour of $\widetilde{S}_{\lambda}$ with respect to multi-turns}
\label{stilde_mt}
As above, assume $U_h = La_hR \in \widetilde{S}_{\lambda}$, $a_h$ is a virtual member of its chart, $U_h \mapsto \sum_{\substack{j = 1 \\ j\neq h}}^{m}U_j$ is a multi-turn that comes from an elementary multi-turn $a_h \mapsto \sum_{\substack{j = 1 \\ j\neq h}}^{m}a_j$ such that $a_j$ are $\lambda$-semicanonical generalized fractional powers. Let us describe in detail monomials $U_j$ such that $a_j$ is not a virtual member of the chart of $U_j$. We split these monomials into three groups in the same way as we did in Section~\ref{mt_configurations}.
\begin{enumerate}
\item
\label{stilde_lower1}
$La_jR$, where $\LM(a_j) > \varepsilon$;
\item
\label{stilde_lower2}
$La_jR$, where $a_j = 1$;
\item
\label{stilde_lower3}
$La_jR$, where $\LM(a_j) \leqslant \varepsilon$, $a_j \neq 1$.
\end{enumerate}

Being a $\lambda$-semicanonical word depends only on the $\LM$-measure of every virtual member of the chart. Hence, a resulting monomial of a multi-turn of a $\lambda$-semicanonical word may become non-$\lambda$-semicanonical only if the $\LM$-measure of some virtual members of the chart increases and $\lambda$-forbidden subwords appear in this monomial. One can see that the $\widetilde{S}_{\lambda}$-condition depends on a virtual member of the chart itself and on its neighbours and the notion of virtual member of the chart depends on the whole monomial. Hence, the $\widetilde{S}_{\lambda}$-condition on a virtual member of the chart may not be hold in a resulting monomial of a multi-turn not only when the virtual member of the chart changes itself, but when the virtual member of the chart stay unchanged but some other maximal occurrence becomes shorter. Let us study the possibilities in detail.

Consider monomials of type~\ref{stilde_lower1}. Let us use notations from Section~\ref{mt_configurations}. We assume that $b_h$ is the left neighbour of $a_h$ in the chart of $U_h$, $b_j$ is a maximal occurrence of a generalized fractional power corresponding to $b_h$ in the chart of $U_j$.
\begin{enumerate}
\item
If $b_h$ and $a_h$ are separated or touch at a point and $\LM(b_h) < \tau$, then $b_j$ may become not a virtual member of the chart and the $\widetilde{S}_{\lambda}$-condition may not hold for its left neighbour.
\begin{center}
\begin{tikzpicture}
\draw[|-|, black, thick] (-3,0)--(6,0);
\node at (-2.5, 0) [below] {$U_h$};
\draw[|-|, black, very thick] (-2,0.1)--(0.6,0.1);
\draw [thick, decorate, decoration={brace, amplitude=10pt}] (-1.8,0.1)--(0.4,0.1) node [midway, above, yshift=7] {$d$};
\draw[|-|, black, very thick] (0,-0.1)--(1.5,-0.1) node[midway, below] {$b_h$};
\draw[|-|, black, very thick] (4,0.1)--(5.5,0.1) node[midway, below] {$a_h$};
\node[text width=3cm, align=left, above, yshift=1] at (-3, 0) {\footnotesize{\baselineskip=10pt $\LM(d) = \lambda + \varepsilon$, $d$ is $\lambda$-forbidden\par}};
\end{tikzpicture}
\end{center}
\begin{center}
\begin{tikzpicture}
\draw[|-|, black, thick] (-3,0)--(6,0);
\node at (-2.5, 0) [below] {$U_j$};
\draw[|-|, black, very thick] (-2,0.1)--(0.6,0.1);
\draw [thick, decorate, decoration={brace, amplitude=10pt}] (-1.8,0.1)--(0.4,0.1) node [midway, above, yshift=7] {$d$};
\draw[|-|, black, very thick] (0,-0.1)--(1.5,-0.1) node[midway, below] {$b_j$};
\draw[|-|, black, very thick] (4,0.1)--(5,0.1) node[midway, below] {$a_j$};
\end{tikzpicture}
\end{center}
\item
If we have configurations~\ref{overlap_disappears} or~\ref{overlap_decreases} (page~\pageref{overlap_disappears}) for $b_j$, then $b_j$ may become not a virtual member of the chart and the $\widetilde{S}_{\lambda}$-condition may not hold for its left neighbour.
\begin{center}
\begin{tikzpicture}
\draw[|-|, black, thick] (0,0)--(6,0);
\node at (0.5, 0) [below] {$U_h$};
\draw[|-|, black, very thick] (1,0.1)--(3.2,0.1);
\draw [thick, decorate, decoration={brace, amplitude=10pt}] (1.5,0.1)--(3,0.1) node [midway, above, yshift=7] {$d$};
\draw[|-|, black, very thick] (2.7,-0.1)--(4.5,-0.1) node[midway, below] {$b_h$};
\draw[|-|, black, very thick] (4,0.1)--(5.5,0.1) node[midway, below] {$a_h$};
\node[text width=3cm, align=left, above, yshift=1] at (0, 0) {\footnotesize{\baselineskip=10pt $\LM(d) = \lambda + \varepsilon$, $d$ is $\lambda$-forbidden\par}};
\end{tikzpicture}
\begin{tikzpicture}
\draw[|-|, black, thick] (0,0)--(6,0);
\node at (0.5, 0) [below] {$U_j$};
\draw[|-|, black, very thick] (1,0.1)--(3.2,0.1);
\draw [thick, decorate, decoration={brace, amplitude=10pt}] (1.5,0.1)--(3,0.1) node [midway, above, yshift=7] {$d$};
\draw[|-|, black, very thick] (2.7,-0.1)--(4,-0.1) node[midway, below] {$b_j$};
\draw[|-|, black, very thick] (4,0.1)--(5,0.1) node[midway, below] {$a_j$};
\end{tikzpicture}
\end{center}
\item
Suppose we have configurations~\ref{overlap_appears}, \ref{overlap_increases} (page~\pageref{overlap_appears}) or $b_j = b_h$. If $b_j$ contains a $\lambda$-forbidden subword $d$ that has an overlap with $a_j$, then the $\widetilde{S}_{\lambda}$-condition may not hold for $b_j$ (because $a_j$ is not a virtual member of the chart).
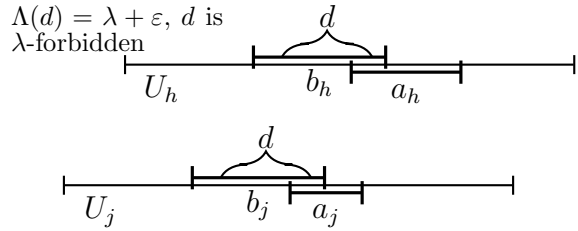
\begin{center}
\begin{tikzpicture}
\draw[|-|, black, thick] (0,0)--(6,0);
\node at (0.5, 0) [below] {$U_h$};
\draw [thick, decorate, decoration={brace, amplitude=10pt}] (2.1,0.1)--(3.3,0.1) node [midway, above, yshift=7] {$d$};
\draw[|-|, black, very thick] (1.7,0.1)--(3.5,0.1) node[midway, below] {$b_h$};
\draw[|-|, black, very thick] (3,-0.1)--(4.5,-0.1) node[midway, below] {$a_h$};
\node[text width=3cm, align=left, above, yshift=1] at (0, 0) {\footnotesize{\baselineskip=10pt $\LM(d) = \lambda + \varepsilon$, $d$ is $\lambda$-forbidden\par}};
\end{tikzpicture}
\begin{tikzpicture}
\draw[|-|, black, thick] (0,0)--(6,0);
\node at (0.5, 0) [below] {$U_j$};
\draw [thick, decorate, decoration={brace, amplitude=10pt}] (2.1,0.1)--(3.3,0.1) node [midway, above, yshift=7] {$d$};
\draw[|-|, black, very thick] (1.7,0.1)--(3.5,0.1) node[midway, below] {$b_j$};
\draw[|-|, black, very thick] (3,-0.1)--(4,-0.1) node[midway, below] {$a_j$};
\end{tikzpicture}
\end{center}
\end{enumerate}
Obviously, the effects for the right neighbour of $a_h$ in the chart are the same.

Consider monomials of type~\ref{stilde_lower2}. First we do cancellations in $U_j = LR$ if possible, after that we obtain $U_j^{\prime} = L^{\prime}R^{\prime}$. Recall that we suppose that $a$ is a maximal occurrence of a generalized fractional power that is a terminal subword of $L^{\prime}$, $b$ is a maximal occurrence of a generalized fractional power that is an initial subword of $R^{\prime}$, $a^{\prime}$ and $b^{\prime}$ are generalized fractional powers that prolong $a$ and $b$ in $L^{\prime}R^{\prime}$ respectively.
\begin{enumerate}
\item
\label{changes_separated}
In any configuration except~\ref{all_separated2}, there may be a maximal occurrence $b_h$, $\LM(b_h) < \tau$, fully contained in the subword $L^{\prime}$ that becomes not a virtual member of the chart of $L^{\prime}R^{\prime}$. Then the $\widetilde{S}_{\lambda}$-condition may not hold for its left neighbour.
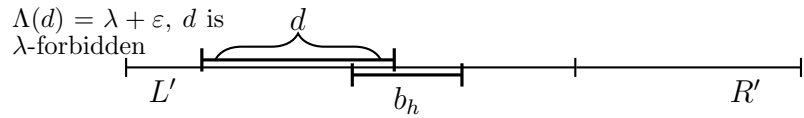
\begin{center}
\begin{tikzpicture}
\draw[|-|, black, thick] (-3,0)--(3, 0);
\draw[-|, black, thick] (3,0)--(6,0) node[near end, below] {$R^{\prime}$};
\node at (-2.5, 0) [below] {$L^{\prime}$};
\draw[|-|, black, very thick] (-2,0.1)--(0.6,0.1);
\draw [thick, decorate, decoration={brace, amplitude=10pt}] (-1.8,0.1)--(0.4,0.1) node [midway, above, yshift=7] {$d$};
\draw[|-|, black, very thick] (0,-0.1)--(1.5,-0.1) node[midway, below] {$b_h$};
\node[text width=3cm, align=left, above, yshift=1] at (-3, 0) {\footnotesize{\baselineskip=10pt $\LM(d) = \lambda + \varepsilon$, $d$ is $\lambda$-forbidden\par}};
\end{tikzpicture}
\end{center}
The same may happen in the subword $R^{\prime}$.
\item
Suppose we have configuration~\ref{strictly_touch2} (page~\pageref{strictly_touch2}). Then $a^{\prime}$ or $b^{\prime}$ may contain a $\lambda$-forbidden subword and do not satisfy the $\widetilde{S}_{\lambda}$-condition. We may obtain this situation if the right neighbour of $a^{\prime}$ in the chart of $L$ cancels out up to the end of $a^{\prime}$ or if the left neighbour of $b^{\prime}$ in the chart of $R$ cancels out up to the beginning of $b^{\prime}$.
\begin{center}
\begin{tikzpicture}
\draw[|-|, black, thick] (0,0)--(6,0);
\node at (0.3, 0) [below] {$L^{\prime}$};
\node at (5.7, 0) [below] {$R^{\prime}$};
\draw [thick, decorate, decoration={brace, amplitude=10pt}] (2.1,0.1)--(3.35,0.1) node [midway, above, yshift=7] {$d$};
\draw[|-|, black, very thick] (1.7,0.1)--(3.5,0.1) node[midway, below] {$a^{\prime}$};
\draw[|-|, black, very thick] (3.5,-0.1)--(5,-0.1) node[midway, below] {$b^{\prime}$};
\node[text width=3cm, align=left, above, yshift=1] at (0, 0) {\footnotesize{\baselineskip=10pt $\LM(d) = \lambda + \varepsilon$, $d$ is $\lambda$-forbidden\par}};
\end{tikzpicture}
\end{center}
\item
\label{stilde_3}
Suppose we have configuration~\ref{strictly_touch2} or~\ref{slightly_enlarge2} (page~\pageref{strictly_touch2}). In this case, $a^{\prime}$ may become too short after cancellations and may not be counted as a virtual member of the chart. If the left neighbour of $a_j$ contains a $\lambda$-forbidden subword $d$ that has an overlap with $a^{\prime}$, then the word $L^{\prime}R^{\prime}$ may not belong to $\widetilde{S}_{\lambda}$.
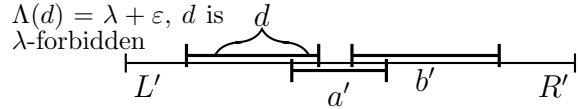
\begin{center}
\begin{tikzpicture}
\draw[|-|, black, thick] (0,0)--(6,0);
\node at (0.3, 0) [below] {$L^{\prime}$};
\node at (5.7, 0) [below] {$R^{\prime}$};
\draw [thick, decorate, decoration={brace, amplitude=10pt}] (1.2,0.1)--(2.45,0.1) node [midway, above, yshift=7] {$d$};
\draw[|-|, black, very thick] (0.8,0.1)--(2.6,0.1);
\draw[|-|, black, very thick] (2.2,-0.1)--(3.5,-0.1) node[midway, below] {$a^{\prime}$};
\draw[|-|, black, very thick] (3,0.1)--(5,0.1) node[midway, below] {$b^{\prime}$};
\node[text width=3cm, align=left, above, yshift=1] at (0, 0) {\footnotesize{\baselineskip=10pt $\LM(d) = \lambda + \varepsilon$, $d$ is $\lambda$-forbidden\par}};
\end{tikzpicture}
\end{center}
We may obtain a similar effect for the right neighbour of $b^{\prime}$.
\item
If we have configuration~\ref{slightly_enlarge2} (page~\pageref{slightly_enlarge2}), then $a^{\prime}$ or $b^{\prime}$ may contain a $\lambda$-forbidden subword that does not satisfy the $\widetilde{S}_{\lambda}$-condition. For example, we may obtain this case when $b^{\prime}$ is not a virtual member of the chart and $a^{\prime}$ contains a $\lambda$-forbidden subword $d$ that has an overlap with $b^{\prime}$.
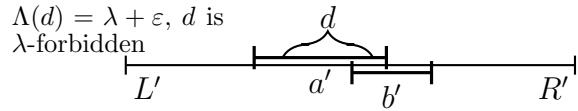
\begin{center}
\begin{tikzpicture}
\draw[|-|, black, thick] (0,0)--(6,0);
\node at (0.3, 0) [below] {$L^{\prime}$};
\node at (5.7, 0) [below] {$R^{\prime}$};
\draw [thick, decorate, decoration={brace, amplitude=10pt}] (2.1,0.1)--(3.3,0.1) node [midway, above, yshift=7] {$d$};
\draw[|-|, black, very thick] (1.7,0.1)--(3.5,0.1) node[midway, below] {$a^{\prime}$};
\draw[|-|, black, very thick] (3,-0.1)--(4.1,-0.1) node[midway, below] {$b^{\prime}$};
\node[text width=3cm, align=left, above, yshift=1] at (0, 0) {\footnotesize{\baselineskip=10pt $\LM(d) = \lambda + \varepsilon$, $d$ is $\lambda$-forbidden\par}};
\end{tikzpicture}
\end{center}
\item
Suppose we have configuration~\ref{merge2} (page~\pageref{merge2}) and $ab$ is not a virtual member of the chart. If the left or the right neighbour of $ab$ contains a $\lambda$-forbidden subword $d$ that has an overlap with $ab$, then the word $L^{\prime}R^{\prime}$ may not belong to $\widetilde{S}_{\lambda}$ (similar to~\ref{stilde_3}).
\begin{center}
\begin{tikzpicture}
\draw[|-|, black, thick] (0,0)--(6,0);
\node at (0.3, 0) [below] {$L^{\prime}$};
\node at (5.7, 0) [below] {$R^{\prime}$};
\draw [thick, decorate, decoration={brace, amplitude=10pt}] (1.1,0.1)--(2.4,0.1) node [midway, above, yshift=7] {$d$};
\draw[|-|, black, very thick] (0.8,0.1)--(2.6,0.1);
\draw[|-|, black, very thick] (2.2,-0.1)--(3.2,-0.1) node[midway, below] {$ab$};
\node[text width=3cm, align=left, above, yshift=1] at (0, 0) {\footnotesize{\baselineskip=10pt $\LM(d) = \lambda + \varepsilon$, $d$ is $\lambda$-forbidden\par}};
\end{tikzpicture}
\end{center}
\item
If we have configuration~\ref{merge2} (page~\pageref{merge2}), the $\widetilde{S}_{\lambda}$-condition may not hold for $ab$.
\end{enumerate}

Consider monomials of type~\ref{stilde_lower3}. Recall that we suppose that $a$ is a maximal occurrence of a generalized fractional power that is a terminal subword of $L$, $b$ is a maximal occurrence of a generalized fractional power that is an initial subword of $R$, $a^{\prime}$ and $b^{\prime}$ are generalized fractional powers that prolong $a$ and $b$ in $La_jR$, respectively.
\begin{enumerate}
\item
In any configuration except~\ref{all_separated1} there may be a maximal occurrence $b_h$, $\LM(b_h) < \tau$, fully contained in the subword $L$ that becomes not a virtual member of the chart of $U_j$. Then the $\widetilde{S}_{\lambda}$-condition may not hold for its left neighbour.
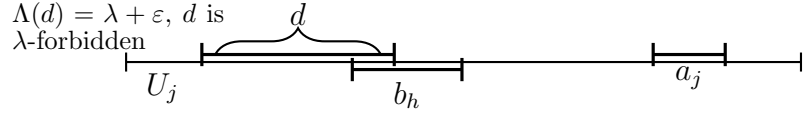
\begin{center}
\begin{tikzpicture}
\draw[|-|, black, thick] (-3,0)--(6,0);
\node at (-2.5, 0) [below] {$U_j$};
\draw[|-|, black, very thick] (-2,0.1)--(0.6,0.1);
\draw [thick, decorate, decoration={brace, amplitude=10pt}] (-1.8,0.1)--(0.4,0.1) node [midway, above, yshift=7] {$d$};
\draw[|-|, black, very thick] (0,-0.1)--(1.5,-0.1) node[midway, below] {$b_h$};
\draw[|-|, black, very thick] (4,0.1)--(5,0.1) node[midway, below] {$a_j$};
\node[text width=3cm, align=left, above, yshift=1] at (-3, 0) {\footnotesize{\baselineskip=10pt $\LM(d) = \lambda + \varepsilon$, $d$ is $\lambda$-forbidden\par}};
\end{tikzpicture}
\end{center}
The same may happen in the subword $R$.
\item
If we obtain configuration~\ref{strictly_touch1} (page~\pageref{strictly_touch1}), then $a^{\prime}$ may be too short to be counted as a virtual member of the chart. In this case, the $\widetilde{S}_{\lambda}$-condition may not hold for its left neighbour.
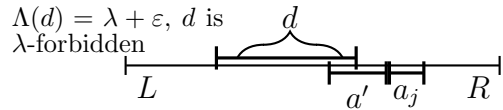
\begin{center}
\begin{tikzpicture}
\draw[|-|, black, thick] (0,0)--(5,0);
\node at (0.3, 0) [below] {$L$};
\node at (4.7, 0) [below] {$R$};
\draw [thick, decorate, decoration={brace, amplitude=10pt}] (1.5,0.1)--(2.9,0.1) node [midway, above, yshift=7] {$d$};
\draw[|-|, black, very thick] (1.2,0.1)--(3.1,0.1);
\draw[|-|, black, very thick] (2.7,-0.1)--(3.5,-0.1) node[midway, below] {$a^{\prime}$};
\draw[|-|, black, very thick] (3.5,-0.1)--(4,-0.1) node[midway, below] {$a_j$};
\node[text width=3cm, align=left, above, yshift=1] at (0, 0) {\footnotesize{\baselineskip=10pt $\LM(d) = \lambda + \varepsilon$, $d$ is $\lambda$-forbidden\par}};
\end{tikzpicture}
\end{center}
Clearly, we may obtain a similar effect for the right neighbour of $b^{\prime}$.
\item
Suppose we have configuration~\ref{slightly_enlarge1} or~\ref{slightly_enlarge_precisely_one} (page~\pageref{slightly_enlarge1}). Then $a^{\prime}$ or $b^{\prime}$ may contain a $\lambda$-forbidden subword $d$ that does not satisfy the $\widetilde{S}_{\lambda}$-condition.
\begin{center}
\begin{tikzpicture}
\draw[|-|, black, thick] (0,0)--(6,0);
\node at (0.3, 0) [below] {$L$};
\node at (5.7, 0) [below] {$R$};
\draw [thick, decorate, decoration={brace, amplitude=10pt}] (2,0.1)--(3.3,0.1) node [midway, above, yshift=7] {$d$};
\draw[|-|, black, very thick] (1.5,0.1)--(3.5,0.1) node[midway, below] {$a^{\prime}$};
\draw[|-|, black, very thick] (3,-0.1)--(4,-0.1) node[midway, below] {$a_j$};
\node[text width=3cm, align=left, above, yshift=1] at (0, 0) {\footnotesize{\baselineskip=10pt $\LM(d) = \lambda + \varepsilon$, $d$ is $\lambda$-forbidden\par}};
\end{tikzpicture}
\begin{tikzpicture}
\draw[|-|, black, thick] (0,0)--(6,0);
\node at (0.3, 0) [below] {$L$};
\node at (5.7, 0) [below] {$R$};
\draw [thick, decorate, decoration={brace, amplitude=10pt}] (2,0.1)--(3.8,0.1) node [midway, above, yshift=7] {$d$};
\draw[|-|, black, very thick] (1.5,0.1)--(4,0.1) node[midway, below] {$a^{\prime}$};
\draw[|-|, black, very thick] (3.6,-0.1)--(5.5,-0.1) node[midway, below] {$b^{\prime}$};
\draw[|-|, black, very thick] (3,-0.1)--(3.6,-0.1) node[midway, below] {$a_j$};
\node[text width=3cm, align=left, above, yshift=1] at (0, 0) {\footnotesize{\baselineskip=10pt $\LM(d) \geqslant \lambda + \varepsilon$, $d$ is $\lambda$-forbidden\par}};
\end{tikzpicture}
\end{center}
\item
If we obtain configuration~\ref{merge1} (page~\pageref{merge1}), the $\widetilde{S}_{\lambda}$-condition may not hold for the virtual member $aa_jb$.
\end{enumerate}

\subsection{Multiplication of words from $\widetilde{S}_\lambda$}
\label{multiplication_def}
In Proposition~\ref{reduct_power} we showed one possible way of transformation of an arbitrary generalized fractional power into a sum of $\lambda$-semicanonical generalized fractional powers. But some of the obtained elements may have $\LM$-measure less than the given threshold $\tau$. In the further argument in Section~\ref{multiplication_def}, it is more convenient to deal with multi-turns such that all the resulting $\lambda$-semicanonical generalized fractional powers are of $\LM$-measure not less than $\tau$. In the following proposition, we construct an elementary multi-turn with this property.
\begin{proposition}
\label{safe_multi_turn}
Suppose $a_h$ is a generalized fractional power. Then there exists an elementary multi-turn $a_h \mapsto \sum_{\substack{j = 1 \\ j\neq h}}^{k} a_j$ such that $a_j$ are $\lambda$-semicanonical generalized fractional powers and $\LM(a_j) \geqslant \tau$, $j \neq h$, i.e., $a_j$ is always counted as a virtual member of the chart (possibly with greater word length and $\LM$-measure than $a_h$).

Moreover, one can choose this multi-turn such that every $a_j$, $j\neq h$, does not contain subwords of $v^{-1}$ of $\LM$-measure greater than $\varepsilon$.
\end{proposition}
\begin{proof}
Let us construct this elementary multi-turn directly using Definition~\ref{multiturn_def}. If we consider $v$ as an element of the field $\mathbb{Z}_2(w)$ such that $v^{-1} = 1 + w$, then in $\mathbb{Z}_2(w)$ we have $1 = vw + v$ and
\begin{equation}
\label{v_inv_safe_mt}
v^{-1} = 1 + w = (vw + v) + (vw + v)w = v + vw^2.
\end{equation}
Therefore, in $\mathbb{Z}_2(w)$ we have $v^{-m} = (v + vw^2)^m$. Hence, the non-commutative Laurent polynomial $x_1^{-m} + (x_1 + x_1x_2^2)^m$, $m > 0$, satisfies ~\eqref{vanish_in_field}. Hence, the transformation in $\mathbb{Z}_2\Fr$
\begin{equation}
\label{v_power_inv_safe_mt}
v^{-m} \mapsto (v + vw^2)^m, m > 0,
\end{equation}
is an elementary multi-turn.

Suppose $M_h(v, w)$ is a non-commutative monomial in $v, w$, that is,
\begin{equation*}
M_h(v, w) = w^{k_1}v^{l_1}\cdots w^{k_n}v^{l_n},
\end{equation*}
where $k_i, l_i \in \mathbb{Z}$. Again consider $M_h(v, w)$ as an element of the field $\mathbb{Z}_2(w)$ such that $v^{-1} = 1 + w$. Then, since $v$ and $w$ commute in $\mathbb{Z}_2(w)$, we have $w^{k_1}v^{l_1}\cdots w^{k_n}v^{l_n} = w^kv^l$, where $k = \sum_{i = 1}^n k_i$, $l = \sum_{i = 1}^n l_i$. Hence, the non-commutative Laurent polynomial $x_2^{k_1}x_1^{l_1}\cdots x_2^{k_n}x_1^{l_n} + x_2^kx_1^l$ satisfies~\eqref{vanish_in_field}. Therefore, the transformation in $\mathbb{Z}_2\Fr$
\begin{equation}
\label{comm_mt}
w^{k_1}v^{l_1}\cdots w^{k_n}v^{l_n} \mapsto w^kv^l, k = \sum\limits_{i = 1}^n k_i, l = \sum\limits_{i = 1}^n l_i,
\end{equation}
is an elementary multi-turn. If $l = \sum_{i = 1}^n l_i < 0$, from~\eqref{v_power_inv_safe_mt} it follows that
\begin{equation}
\label{comm_v_power_inv_mt}
w^{k_1}v^{l_1}\cdots w^{k_n}v^{l_n} \mapsto w^k(v + vw^2)^{-l}
\end{equation}
is an elementary multi-turn. Since $1 \mapsto v + vw$ is an elementary multi-turn, if $l = \sum_{i = 1}^n l_i = 0$, we obtain that
\begin{equation}
\label{comm_one_mt}
w^{k_1}v^{l_1}\cdots w^{k_n}v^{l_n} \mapsto w^k(v + vw)
\end{equation}
is an elementary multi-turn. Using~\eqref{comm_mt}--\eqref{comm_one_mt}, for every $l = \sum_{i = 1}^n l_i$ we obtain a multi-turn
\begin{equation}
\label{safe_mt_general}
M_h(v, w) \mapsto \sum\limits_{\substack{j = 1 \\ j\neq h}}^s M_j(v, w),
\end{equation}
where every $M_j(v, w)$, $j\neq h$, does not contain negative powers of $v$ and contains at least one positive power of $v$.

By definition, a generalized fractional power is a subword in a non-commutative monomial over the words $v$ and $w$. Recall that the types of generalized fractional powers are enumerated in (5) -- (7). Clearly, if we consider generalized fractional powers with the possibility of cancellations, we obtain that $a_h$ before cancellations has one of the following forms:
\begin{align*}
&v_f M_h(v, w)v_i,\\
&v_mv_fM_h(v, w)v_iv_m,\\
&w_fM_h(v, w)v_i,\\
&v_fM_h(v, w)w_i,\\
&w_f M_h(v, w)w_i,\\
&w_mw_fM_h(v, w)w_iw_m,
\end{align*}
where $M_h(v, w)$ is a monomial over the words $v$ and $w$.

Suppose $a_h = v_f M_h(v, w)v_i$, where $v_f M(v, w)_hv_i$ may have cancellations. From \eqref{safe_mt_general} it follows that the transformation $v_fM(v, w)_hv_i \mapsto \sum_{\substack{j = 1 \\ j\neq h}}^s v_fM_j(v, w)v_i$ is an elementary multi-turn after all possible cancellations in monomials are done. That is,
\begin{equation}
\label{safe_mt_case1}
a_h \mapsto \sum\limits_{\substack{j = 1 \\ j\neq h}}^s v_fM_j(v, w)v_i
\end{equation}
is an elementary multi-turn.

Let us check that~\eqref{safe_mt_case1} satisfies the conditions of Proposition~\ref{safe_multi_turn}. Since every $M_j(v, w)$, $j \neq h$, contains only positive powers of $v$ and $v$ has no cancellations with $w$ and $w^{-1}$, every monomial in the right hand side of this multi-turn is of $\LM$-measure not less than $1$, that is, it is especially of $\LM$-measure not less than $\tau$. Since $w^k$ can not contain subwords of $v$ and $v^{-1}$ of $\LM$-measure greater than $\varepsilon$, every monomial in the right hand side does not contain subwords of $v^{-1}$ of $\LM$-measure greater than $\varepsilon$. Thus, multi-turn \eqref{safe_mt_case1} satisfies the conditions of Proposition~\ref{safe_multi_turn}.

In a similar way we deal with $a_h$ of the other form and we obtain multi-turns
\begin{align}
\label{safe_mt_cases}
\begin{split}
&a_h \mapsto \sum\limits_{\substack{j = 1 \\ j\neq h}}^s v_mv_fM_j(v, w)v_iv_m, \ a_h \mapsto \sum\limits_{\substack{j = 1 \\ j\neq h}}^s w_fM_j(v, w)v_i,\\
&a_h \mapsto \sum\limits_{\substack{j = 1 \\ j\neq h}}^s v_fM_j(v, w)w_i, \ a_h \mapsto \sum\limits_{\substack{j = 1 \\ j\neq h}}^s w_fM_j(v, w)w_i,\\
&a_h \mapsto \sum\limits_{\substack{j = 1 \\ j\neq h}}^s w_mw_fM_j(v, w)w_iw_m,
\end{split}
\end{align}
possibly after cancellations in the right hand sides. Notice that only subwords of $w$ may be involved in cancellations. By the same argument as \eqref{safe_mt_case1}, it is proved that multi-turns \eqref{safe_mt_cases} satisfy the conditions of Proposition~\ref{safe_multi_turn}.
\end{proof}

Let $U_1$ and $U_2$ be monomials from $\widetilde{S}_\lambda$. Then their product $U_1U_2$ may not belong to $\widetilde{S}_\lambda$. According to Corollary~\ref{lemma_about_generating}, every monomial is equal to a sum of words from $S_\lambda$ modulo $\Ideal$. Moreover, every monomial is equal to a sum of words from $\widetilde{S}_\lambda$ modulo $\Ideal$, since $S_\lambda \subseteq \widetilde{S}_\lambda$. We will represent $U_1U_2$ as a sum of monomials from $\widetilde{S}_\lambda$. In this section, we analyse the process of reduction of the word $U_1U_2$ to a sum of words from $\widetilde{S}_\lambda$ modulo $\Ideal$ in detail. In parallel we will describe the multiplication diagram that encrypts the history of this process.

The process of the reduction of $U_1U_2$ is as follows.
\begin{enumerate}[label=\textbf{Step~\arabic*}, ref=Step~\arabic*]
\item
\label{cancellation_step}
Cancel $U_1U_2$ if possible. The multiplication diagram that encrypts cancellations is of the form
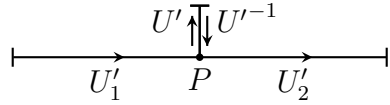
\begin{align*}
&\begin{tikzpicture}
\draw[|-, black, thick, arrow=0.6] (-1,0)--(1.5,0) node[midway, below] {$U_1^{\prime}$};
\draw[-|, black, thick, arrow=0.6] (1.5,0)--(4,0) node[midway, below] {$U_2^{\prime}$};
\draw[-|, black, thick] (1.5,0)--(1.5,0.7) node[near end, right, xshift=2] {${U^{\prime}}^{-1}$};
\draw[-|, black, thick] (1.5,0)--(1.5,0.7) node[near end, left, xshift=-2, yshift=-1] {$U^{\prime}$};
\draw[black, thick, arrow=1] (1.4,0.15)--(1.4,0.55);
\draw[black, thick, arrow=1] (1.6,0.55)--(1.6,0.15);
\node[circle,fill,inner sep=1] at (1.5,0) {};
\node[below] at (1.5,0) {$P$};
\end{tikzpicture}\\
&U_1=U_1^{\prime}U^{\prime}\\
&U_2={U^{\prime}}^{-1}U_2^{\prime}\\
&U= U_1^{\prime}U_2^{\prime} \text{ has no further cancellations}
\end{align*}
\end{enumerate}
Let $U = U_1^{\prime}U_2^{\prime}$. Now we consider the chart of $U$. As we described in Section~\ref{mt_configurations}, we have one of configurations~\ref{all_separated2}---\ref{merge2} (pages~\pageref{all_separated2}---\pageref{merge2}). If we have configuration~\ref{merge2} and the generalized fractional power that is obtained as a result of merging does not satisfy $\widetilde{S}_\lambda$-condition, then we go to \ref{multiturn_step} of the reduction process. So, let us first describe this case.

Denote the generalized fractional power that is obtained as a result of merging by $a_h$, $U = La_hR$.
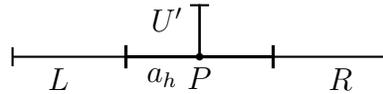
\begin{center}
\begin{tikzpicture}
\draw[|-, black, thick] (-1,0)--(1.5,0) node[near start, below] {$L$};
\draw[-|, black, thick] (1.5,0)--(4,0) node[near end, below] {$R$};
\draw[-|, black, thick] (1.5,0)--(1.5,0.7);
\draw[ black, thick] (1.5,0)--(1.5,0.7) node[near end, left, xshift=-2, yshift=-1] {$U^{\prime}$};
\draw[|-|, black, very thick] (0.5,0) to node[near start, below] {$a_h$} (2.5,0);
\node[circle,fill,inner sep=1] at (1.5,0) {};
\node[below] at (1.5,0) {$P$};
\end{tikzpicture}
\end{center}
Suppose $a_h$ does not satisfy the $\widetilde{S}_\lambda$-condition. Then the next step of the multiplication process is as follows.
\begin{enumerate}[label=\textbf{Step~\arabic*}, ref=Step~\arabic*]
\setcounter{enumi}{1}
\item
\label{multiturn_step}
According to Proposition~\ref{safe_multi_turn}, a $\lambda$-forbidden virtual member $a_h$ can be reduced to a sum of $\lambda$-semicanonical generalized fractional powers using an appropriate elementary multi-turn $a_h \mapsto \sum_{\substack{j = 1 \\ j\neq h}}^{k} a_j$, $\LM(a_j) \geqslant \tau$. So, perform the multi-turn $U = La_hR \mapsto \sum_{\substack{j = 1 \\ j\neq h}}^{k} La_jR$, which comes from this elementary multi-turn.
\end{enumerate}
Since $\LM(a_j) \geqslant \tau$, the resulting monomials $La_jR$ belong to $\widetilde{S}_\lambda$.

The monomial $a_h$ corresponds to a unique path in the $v$-diagram (see Definition~\ref{v_diagram} is Section~\ref{basic_def}) with the initial point $I$ and the final point $F$. The monomials $a_j$, $j = 1, \ldots, k$, correspond to certain paths in this $v$-diagram with the same initial point $I$ and the same final point $F$. Then we add this $v$-diagram to the bottom of the multiplication diagram, gluing the point $I$ with the end of $L$ and point $F$ with the beginning of $R$. The previous part of the multiplication diagram is glued to the added $v$-diagram at the same point on $a_h$ as before.
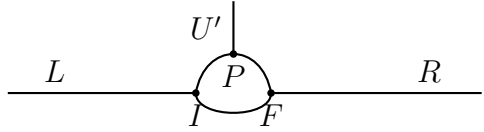
\begin{center}
\begin{tikzpicture}
\coordinate (Begin) at (-1,0);
\coordinate (vb) at (1.5,0);
\coordinate (ve) at (2.5,0);
\coordinate (End) at (5.3,0);
\draw[black, thick, in=90, out=180-90, distance=-10] (vb) to coordinate [pos=0.1] (vel1) (ve);
\path[black, thick, in=90, out=180-90, distance=-10] (vb) to coordinate [pos=0.9] (vbr1) (ve);
\path[black, thick, in=100, out=180-100, distance=20] (vb) to coordinate [pos=0.4] (vb1) (ve);
\path [black, thick, in=100, out=180-100, distance=20](vb) to coordinate [pos=0.6] (ve1) (ve);
\draw[black, thick, in=100, out=180-100, distance=20] (vb) to coordinate [pos=0.5] (vm1) (ve);
\draw[black, thick] (Begin) to node[near start, above] {$L$} (vb);
\draw[black, thick] (ve) to node[near end, above] {$R$} (End);
\node[below] at (vb) {$I$};
\node[below] at (ve) {$F$};
\node[circle,fill,inner sep=1] at (vb) {};
\node[circle,fill,inner sep=1] at (ve) {};
\node[circle,fill,inner sep=1] at (vm1) {};
\node[below] at (vm1) {$P$};
\draw[black, thick] let \p{I}=(vm1) in (vm1) to node[midway, left] {$U^{\prime}$} (\x{I}, \y{I} + 20);
\end{tikzpicture}
\end{center}
So, one can read the monomials $La_jR$, $j = 1, \ldots, k$, in the low level of the obtained diagram. So far for \ref{multiturn_step}.

Now consider configurations~\ref{all_separated2}---\ref{slightly_enlarge2} in $U = U_1^{\prime}U_2^{\prime}$ after \ref{cancellation_step}. Then the monomial $U$ also may not belong to $\widetilde{S}_\lambda$. According to the classification given in Section~\ref{stilde_mt}, $U$ contains at most two subwords that do not satisfy the $\widetilde{S}_\lambda$-condition, one comes from the chart of $U_1$ and the other comes from the chart of $U_2$ (see analysis of monomials of type~\ref{stilde_lower2}, page~\pageref{changes_separated}; although here we deal with a product of two words from $\widetilde{S}_\lambda$, this analysis is still applicable).

First assume that $U$ contains only one virtual member of the chart that does not satisfy the $\widetilde{S}_{\lambda}$-condition. To be precise, assume that it comes from the chart of $U_1$, denote it by $b_h$. Then there are the following possible configurations depending on the position $b_h$ with respect to the point $P$:
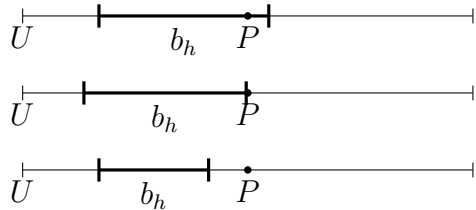
\begin{center}
\begin{tikzpicture}
\draw[|-|, black] (0,0)--(6,0) node[at start, below] {$U$};
\draw[|-|, black, very thick] (1,0)--(3.3,0) node[midway, below] {$b_h$};
\node[circle,fill,inner sep=1] at (3,0) {};
\node[below] at (3,0) {$P$};
\end{tikzpicture}
\end{center}
\begin{center}
\begin{tikzpicture}
\draw[|-|, black] (0,0)--(6,0) node[at start, below] {$U$};
\draw[|-|, black, very thick] (0.8,0)--(3,0) node[midway, below] {$b_h$};
\node[circle,fill,inner sep=1] at (3,0) {};
\node[below] at (3,0) {$P$};
\end{tikzpicture}
\end{center}
\begin{center}
\begin{tikzpicture}
\draw[|-|, black] (0,0)--(6,0) node[at start, below] {$U$};
\draw[|-|, black, very thick] (1,0)--(2.5,0) node[midway, below] {$b_h$};
\node[circle,fill,inner sep=1] at (3,0) {};
\node[below] at (3,0) {$P$};
\end{tikzpicture}
\end{center}
As above we perform a multi-turn $Lb_hR \mapsto \sum_{\substack{j = 1 \\ j\neq h}}^{k} Lb_jR$ from Proposition~\ref{safe_multi_turn} that reduces $b_h$. Since $\LM(b_j) \geqslant \tau$, the resulting monomials $Lb_jR$ belong to $\widetilde{S}_{\lambda}$.

We add a new $v$-diagram to the multiplication diagram as we described above and obtain the final multiplication diagram.
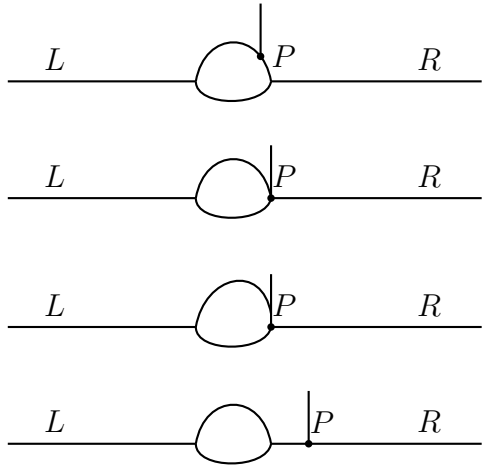
\begin{center}
\begin{tikzpicture}
\coordinate (Begin) at (-1,0);
\coordinate (Middle) at (2.5,0);
\coordinate (End) at (5.3,0);
\coordinate (vbl) at (1.5,0);
\coordinate (vel) at (Middle);
\draw[black, thick, in=80, out=180-90, distance=-10] (vbl) to (vel);
\draw[black, thick, in=100, out=180-100, distance=20] (vbl) to (vel);
\draw[black, thick] (Begin) to node[near start, above] {$L$} (vbl);
\draw[black, thick] (Middle) to node[near end, above] {$R$} (End);
\path[black, thick, in=100, out=180-100, distance=20] (vbl) to coordinate[pos = 0.8] (ShiftedMiddle) (vel);
\draw[black, thick] let \p{I}=(ShiftedMiddle) in (ShiftedMiddle) to (\x{I}, \y{I} + 20);
\node[circle,fill,inner sep=1] at (ShiftedMiddle) {};
\node[right] at (ShiftedMiddle) {$P$};
\end{tikzpicture}
\end{center}
\begin{center}
\begin{tikzpicture}
\coordinate (Begin) at (-1,0);
\coordinate (Middle) at (2.5,0);
\coordinate (End) at (5.3,0);
\coordinate (vbl) at (1.5,0);
\coordinate (vel) at (Middle);
\draw[black, thick, in=80, out=180-90, distance=-10] (vbl) to (vel);
\draw[black, thick, in=100, out=180-100, distance=20] (vbl) to (vel);
\draw[black, thick] (Begin) to node[near start, above] {$L$} (vbl);
\draw[black, thick] (Middle) to node[near end, above] {$R$} (End);
\path[black, thick, in=100, out=180-100, distance=20] (vbl) to (vel);
\draw[black, thick] let \p{I}=(Middle) in (Middle) to (\x{I}, \y{I} + 20);
\node[circle,fill,inner sep=1] at (Middle) {};
\node[above, xshift=5] at (Middle) {$P$};
\end{tikzpicture}
\end{center}
\begin{center}
\begin{tikzpicture}
\coordinate (Begin) at (-1,0);
\coordinate (Middle) at (2.5,0);
\coordinate (End) at (5.3,0);
\coordinate (vbl) at (1.5,0);
\coordinate (vel) at (Middle);
\path let \p1=(Middle) in coordinate (velup) at (\x1, \y1 + 5);
\draw[black, thick, in=80, out=180-90, distance=-10] (vbl) to (vel);
\draw[black, thick, in=100, out=180-100, distance=20] (vbl) to (velup);
\draw[black, thick] (Begin) to node[near start, above] {$L$} (vbl);
\draw[black, thick] (Middle) to node[near end, above] {$R$} (End);
\path[black, thick, in=100, out=180-100, distance=20] (vbl) to (vel);
\draw[black, thick] let \p1=(Middle) in (Middle) to (\x1, \y1 + 20);
\node[circle,fill,inner sep=1] at (Middle) {};
\node[above, xshift=5] at (Middle) {$P$};
\end{tikzpicture}
\end{center}
\begin{center}
\begin{tikzpicture}
\coordinate (Begin) at (-1,0);
\coordinate (ShiftedMiddle) at (3,0);
\coordinate (End) at (5.3,0);
\coordinate (vbl) at (1.5,0);
\coordinate (vel) at (Middle);
\draw[black, thick, in=80, out=180-90, distance=-10] (vbl) to (vel);
\draw[black, thick, in=100, out=180-100, distance=20] (vbl) to (vel);
\draw[black, thick] (Begin) to node[near start, above] {$L$} (vbl);
\draw[black, thick] (Middle) to node[near end, above] {$R$} (End);
\path[black, thick, in=100, out=180-100, distance=20] (vbl) to (vel);
\draw[black, thick] let \p{I}=(ShiftedMiddle) in (ShiftedMiddle) to (\x{I}, \y{I} + 20);
\node[circle,fill,inner sep=1] at (ShiftedMiddle) {};
\node[above, xshift=5] at (ShiftedMiddle) {$P$};
\end{tikzpicture}
\end{center}

Now assume that $U$ contains two virtual members of the chart that do not satisfy the $\widetilde{S}_{\lambda}$-condition, denote them by $b_h$ and $c_h$. Then there are the following possible configurations depending on the position of $b_h$ and $c_h$ with respect to the point $P$.
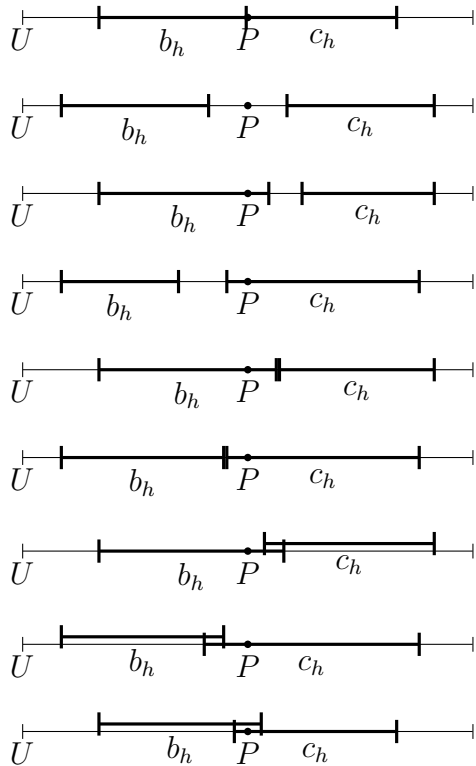
\begin{center}
\begin{tikzpicture}
\draw[|-|, black] (0,0)--(6,0) node[at start, below] {$U$};
\draw[|-|, black, very thick] (1,0)--(3,0) node[midway, below] {$b_h$};
\node[circle,fill,inner sep=1] at (3,0) {};
\node[below] at (3,0) {$P$};
\draw[-|, black, very thick] (3,0)--(5,0) node[midway, below] {$c_h$};
\end{tikzpicture}
\end{center}
\begin{center}
\begin{tikzpicture}
\draw[|-|, black] (0,0)--(6,0) node[at start, below] {$U$};
\draw[|-|, black, very thick] (0.5,0)--(2.5,0) node[midway, below] {$b_h$};
\node[circle,fill,inner sep=1] at (3,0) {};
\node[below] at (3,0) {$P$};
\draw[|-|, black, very thick] (3.5,0)--(5.5,0) node[midway, below] {$c_h$};
\end{tikzpicture}
\end{center}
\begin{center}
\begin{tikzpicture}
\draw[|-|, black] (0,0)--(6,0) node[at start, below] {$U$};
\draw[|-|, black, very thick] (1,0)--(3.3,0) node[midway, below] {$b_h$};
\node[circle,fill,inner sep=1] at (3,0) {};
\node[below] at (3,0) {$P$};
\draw[|-|, black, very thick] (3.7,0)--(5.5,0) node[midway, below] {$c_h$};
\end{tikzpicture}
\end{center}
\begin{center}
\begin{tikzpicture}
\draw[|-|, black] (0,0)--(6,0) node[at start, below] {$U$};
\draw[|-|, black, very thick] (0.5,0)--(2.1,0) node[midway, below] {$b_h$};
\node[circle,fill,inner sep=1] at (3,0) {};
\node[below] at (3,0) {$P$};
\draw[|-|, black, very thick] (2.7,0)--(5.3,0) node[midway, below] {$c_h$};
\end{tikzpicture}
\end{center}
\begin{center}
\begin{tikzpicture}
\draw[|-|, black] (0,0)--(6,0) node[at start, below] {$U$};
\draw[|-|, black, very thick] (1,0)--(3.4,0) node[midway, below] {$b_h$};
\node[circle,fill,inner sep=1] at (3,0) {};
\node[below] at (3,0) {$P$};
\draw[|-|, black, very thick] (3.4,0)--(5.5,0) node[midway, below] {$c_h$};
\end{tikzpicture}
\end{center}
\begin{center}
\begin{tikzpicture}
\draw[|-|, black] (0,0)--(6,0) node[at start, below] {$U$};
\draw[|-|, black, very thick] (0.5,0)--(2.7,0) node[midway, below] {$b_h$};
\node[circle,fill,inner sep=1] at (3,0) {};
\node[below] at (3,0) {$P$};
\draw[|-|, black, very thick] (2.7,0)--(5.3,0) node[midway, below] {$c_h$};
\end{tikzpicture}
\end{center}
\begin{center}
\begin{tikzpicture}
\draw[|-|, black] (0,0)--(6,0) node[at start, below] {$U$};
\draw[|-|, black, very thick] (1,0)--(3.5,0) node[midway, below] {$b_h$};
\node[circle,fill,inner sep=1] at (3,0) {};
\node[below] at (3,0) {$P$};
\draw[|-|, black, very thick] (3.2,0.1)--(5.5,0.1) node[midway, below] {$c_h$};
\end{tikzpicture}
\end{center}
\begin{center}
\begin{tikzpicture}
\draw[|-|, black] (0,0)--(6,0) node[at start, below] {$U$};
\draw[|-|, black, very thick] (0.5,0.1)--(2.7,0.1) node[midway, below] {$b_h$};
\node[circle,fill,inner sep=1] at (3,0) {};
\node[below] at (3,0) {$P$};
\draw[|-|, black, very thick] (2.4,0)--(5.3,0) node[midway, below] {$c_h$};
\end{tikzpicture}
\end{center}
\begin{center}
\begin{tikzpicture}
\draw[|-|, black] (0,0)--(6,0) node[at start, below] {$U$};
\draw[|-|, black, very thick] (1,0.1)--(3.2,0.1) node[midway, below] {$b_h$};
\node[circle,fill,inner sep=1] at (3,0) {};
\node[below] at (3,0) {$P$};
\draw[|-|, black, very thick] (2.8,0)--(5,0) node[midway, below] {$c_h$};
\end{tikzpicture}
\end{center}

Then first, as above, we perform a multi-turn $Lb_hR \mapsto \sum_{\substack{j = 1 \\ j\neq h}}^{k} Lb_jR$ from Proposition~\ref{safe_multi_turn} that reduces $b_h$. Since $\LM(b_j) \geqslant \tau$, in every monomial $Lb_jR$ the image of $c_h$ is the only virtual member of the chart that does not satisfy the $\widetilde{S}_{\lambda}$-condition. Denote the image of $c_h$ in the monomial $Lb_jR$ by $c_h^{(j)}$. So, in every monomial $Lb_jR$ we perform a multi-turn from Proposition~\ref{safe_multi_turn} that reduces $c_h^{(j)}$. As a result, we obtain a sum of monomials from $\widetilde{S}_{\lambda}$.

After the reduction of $c_h^{(j)}$, we glue a new $v$-diagram to the multiplication diagram obtained after the reduction of $b_h$ and obtain the final multiplication diagram. Let us illustrate the first three configurations of $b_h$ and $c_h$. The multiplication diagrams for the remain configurations are constructed similarly.
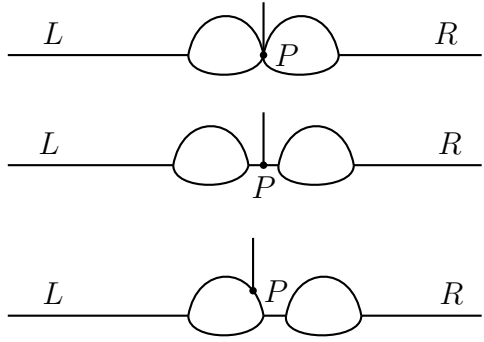
\begin{center}
\begin{tikzpicture}
\coordinate (Begin) at (-1,0);
\coordinate (Middle) at (2.4,0);
\coordinate (End) at (5.3,0);
\coordinate (vbl) at (1.4,0);
\coordinate (vel) at (Middle);
\draw[black, thick, in=80, out=180-90, distance=-10] (vbl) to (vel);
\draw[black, thick, in=100, out=180-100, distance=20] (vbl) to (vel);
\coordinate (vbr) at (Middle);
\coordinate (ver) at (3.4,0);
\draw[black, thick, in=180-90, out=80, distance=-10] (vbr) to (ver);
\draw[black, thick, in=100, out=180-100, distance=20] (vbr) to (ver);
\draw[black, thick] (Begin) to node[near start, above] {$L$} (vbl);
\draw[black, thick] (ver) to node[near end, above] {$R$} (End);
\draw[black, thick] let \p{I}=(Middle) in (Middle) to (\x{I}, \y{I} + 20);
\node[circle,fill,inner sep=1] at (Middle) {};
\node[right] at (Middle) {$P$};
\end{tikzpicture}
\end{center}
\begin{center}
\begin{tikzpicture}
\coordinate (Begin) at (-1,0);
\coordinate (Middle1) at (2.2,0);
\coordinate (Middle) at (2.4,0);
\coordinate (Middle2) at (2.6,0);
\coordinate (End) at (5.3,0);
\coordinate (vbl) at (1.2,0);
\coordinate (vel) at (Middle1);
\draw[black, thick, in=80, out=180-90, distance=-10] (vbl) to (vel);
\draw[black, thick, in=100, out=180-100, distance=20] (vbl) to (vel);
\coordinate (vbr) at (Middle2);
\coordinate (ver) at (3.6,0);
\draw[black, thick, in=180-90, out=80, distance=-10] (vbr) to (ver);
\draw[black, thick, in=100, out=180-100, distance=20] (vbr) to (ver);
\draw[black, thick] (Begin) to node[near start, above] {$L$} (vbl);
\draw[black, thick] (ver) to node[near end, above] {$R$} (End);
\draw[black, thick] (Middle1) to (Middle2);
\draw[black, thick] let \p{I}=(Middle) in (Middle) to (\x{I}, \y{I} + 20);
\node[circle,fill,inner sep=1] at (Middle) {};
\node[below] at (Middle) {$P$};
\end{tikzpicture}
\end{center}
\begin{center}
\begin{tikzpicture}
\coordinate (Begin) at (-1,0);
\coordinate (Middle1) at (2.5,0);
\coordinate (Middle2) at (2.7,0);
\coordinate (End) at (5.3,0);
\coordinate (vbl) at (1.4,0);
\coordinate (vel) at (Middle);
\draw[black, thick, in=80, out=180-90, distance=-10] (vbl) to (vel);
\draw[black, thick, in=100, out=180-100, distance=20] (vbl) to (vel);
\coordinate (vbr) at (Middle2);
\coordinate (ver) at (3.7,0);
\draw[black, thick, in=180-90, out=80, distance=-10] (vbr) to (ver);
\draw[black, thick, in=100, out=180-100, distance=20] (vbr) to (ver);
\draw[black, thick] (vel) to (vbr);
\draw[black, thick] (Begin) to node[near start, above] {$L$} (vbl);
\draw[black, thick] (ver) to node[near end, above] {$R$} (End);
\path[black, thick, in=100, out=180-100, distance=20] (vbl) to coordinate[pos = 0.8] (ShiftedMiddle) (vel);
\draw[black, thick] let \p{I}=(ShiftedMiddle) in (ShiftedMiddle) to (\x{I}, \y{I} + 20);
\node[circle,fill,inner sep=1] at (ShiftedMiddle) {};
\node[right] at (ShiftedMiddle) {$P$};
\end{tikzpicture}
\end{center}

So, finally we obtain the multiplication diagram as a thin triangle. The result of the multiplication is a sum of words from $\widetilde{S}_\lambda$ that correspond to certain paths in the low level of the obtained thin triangle.

Thereby, we proved the following theorem.
\begin{theorem}
\label{thin_triangles_theorem}
Assume $U_1 + \Ideal, U_2 + \Ideal \in \mathbb{Z}_2\Fr / \Ideal$, where $U_1$ and $U_2$ are monomials from $\widetilde{S}_\lambda$. Then there exist monomials $Z_1, \ldots, Z_k$ from $\widetilde{S}_\lambda$ such that $U_1U_2 + \Ideal = \sum_{i = 1}^{k} Z_i + \Ideal$, and $U_1$, $U_2$ and $Z_i$ form a thin triangle consists of $v$-diagrams for every $i = 1, \ldots, k$.
\end{theorem}

\begin{remark}
If we use not an elementary multi-turn from Proposition~\ref{safe_multi_turn}, but rather an arbitrary elementary multi-turn $a_h \mapsto \sum_{\substack{j = 1 \\ j\neq h}}^{k} a_j$ in order to reduce $a_h$ at \ref{multiturn_step}, some $a_j$ may be of $\LM$-measure $\leqslant \varepsilon$. In this case, we may obtain further merging of virtual members of the chart of $La_jR$. Namely, the cases~\ref{merge2} and~\ref{merge1} (pages~\pageref{merge2}, \pageref{merge1}), in which the neighbours of $a_j$ in the chart merge, are possible. For example, for~\ref{merge1} we have
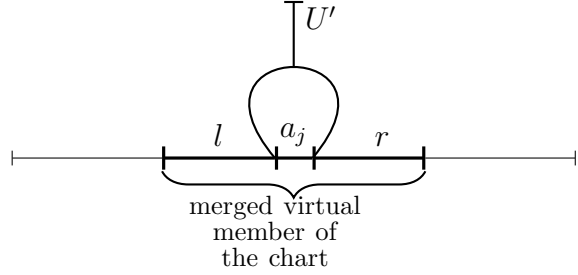
\vspace{0.1cm}
\begin{center}
\begin{tikzpicture}
\coordinate (S) at (0,0);
\coordinate (Li) at (2, 0);
\coordinate (ai) at (3.5, 0);
\coordinate (Ri) at (4, 0);
\coordinate (Rf) at (5.5, 0);
\coordinate (E) at (7.5, 0);
\draw[|-, black] (S)--(Li);
\draw[|-,black, very thick] (Li)--(ai) node [midway, above] {$l$};
\draw[|-, black, very thick] (ai)--(Ri) node [midway, above] {$a_j$};
\draw[black, thick, in=50, out=180-50, distance=60] (ai) to coordinate[pos=0.5] (A) (Ri);
\draw[-|, black, thick] let \p1 = (A) in (A)--(\x1, \y1+25) node[near end, right] {$U^{\prime}$};
\draw[|-|,black, very thick] (Ri)--(Rf) node [midway, above, xshift=5] {$r$};
\draw[-|, black] (Rf)--(E);
\draw [thick, decorate, decoration={brace, mirror, amplitude=10pt, raise=5pt}] (Li) to (Rf);
\node[text width=3cm, align=center, below, yshift=-10] at (3.5, 0) {\footnotesize{\baselineskip=10pt merged virtual member of the chart\par}};
\end{tikzpicture}
\end{center}
\vspace{0.1cm}
If there exists $a_j = 1$, we may obtain cancellations in the resulting monomial that contain $a_j = 1$.

So, when we reduce $U_1U_2$, we obtain a sequence of \ref{cancellation_step} and \ref{multiturn_step}. We can think of it as of a branching process, starting from the monomial $U_1U_2$. After each execution of \ref{cancellation_step} we obtain one monomial, after each execution of \ref{multiturn_step} we obtain a sum of monomials. Each monomial in the sum is treated independently and potentially gives us a new branch of the process. Since we may obtain merging of virtual members of the chart or cancellations at most in one resulting monomial, the process stops in all branches except at most one after each step. From Proposition~\ref{filtration_finite_property}, it follows that finally the process stops in all branches.

For every branch of the process, we can construct a multiplication diagram in the same way as above. If we do an arbitrary multi-turn in order to reduce the monomials obtained after the end of the sequence of \ref{cancellation_step} and \ref{multiturn_step} in a branch, we may obtain a multiplication diagram that is not a thin triangle, but has a form of tree. Nevertheless, the reduction of these monomials can be done using multi-turns described in Proposition~\ref{safe_multi_turn}, and in this case the multiplication diagram for every branch of the process again is a thin triangle but of more complicated form
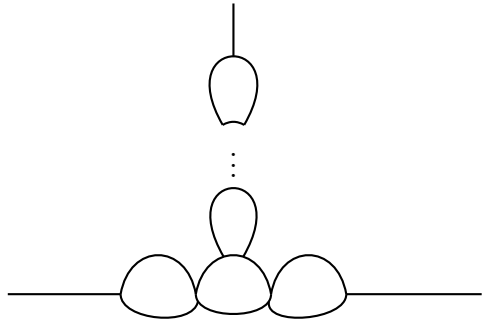
\begin{center}
\begin{tikzpicture}
\coordinate (Begin) at (-1,0);
\coordinate (vb) at (1.5,0);
\coordinate (ve) at (2.5,0);
\coordinate (End) at (5.3,0);
\draw[black, thick, in=90, out=180-90, distance=-10] (vb) to coordinate [pos=0.1] (vel1) (ve);
\path[black, thick, in=90, out=180-90, distance=-10] (vb) to coordinate [pos=0.9] (vbr1) (ve);
\path[black, thick, in=100, out=180-100, distance=20] (vb) to coordinate [pos=0.4] (vb1) (ve);
\path [black, thick, in=100, out=180-100, distance=20](vb) to coordinate [pos=0.6] (ve1) (ve);
\draw[black, thick, in=100, out=180-100, distance=20] (vb) to coordinate [pos=0.5] (vm1) (ve);
\coordinate (vbl) at (0.5,0);
\coordinate (vel) at (1.5,0);
\draw[black, thick, in=80, out=180-90, distance=-10] (vbl) to (vel1);
\draw[black, thick, in=100, out=180-100, distance=20] (vbl) to (vel);
\coordinate (vbr) at (2.5,0);
\coordinate (ver) at (3.5,0);
\draw[black, thick, in=180-90, out=80, distance=-10] (vbr1) to (ver);
\draw[black, thick, in=100, out=180-100, distance=20] (vbr) to (ver);
\draw[black, thick] (Begin)--(vbl);
\draw[black, thick] (ver)--(End);
\path[black, thick, in=60, out=180-60, distance=40] (vb1) to coordinate [pos=0.4] (vb2) (ve1);
\path [black, thick, in=60, out=180-60, distance=40](vb1) to coordinate [pos=0.6] (ve2) (ve1);
\draw[black, thick, in=60, out=180-60, distance=40] (vb1) to coordinate [pos=0.5] (vm2) (ve1);
\node[above] at (vm2) {$\vdots$};
\path[black, thick, in=60, out=180-60, distance=40] (vb2) to coordinate [pos=0.4] (vb3) (ve2);
\path [black, thick, in=60, out=180-60, distance=40](vb2) to coordinate [pos=0.6] (ve3) (ve2);
\path[black, thick, in=60, out=180-60, distance=40] (vb2) to coordinate [pos=0.5] (vm3) (ve2);
\path[black, thick, in=60, out=180-60, distance=40] (vb3) to coordinate [pos=0.4] (vb4) (ve3);
\path [black, thick, in=60, out=180-60, distance=40](vb3) to coordinate [pos=0.6] (ve4) (ve3);
\draw[black, thick, in=60, out=180-60, distance=40] (vb3) to coordinate [pos=0.5] (vm4) (ve3);
\draw[black, thick] (vb3) to [bend left] (ve3);
\draw[black, thick] let \p{I}=(vm4) in (vm4) to (\x{I}, \y{I} + 20);
\end{tikzpicture}
\end{center}
So, we obtain the multiplication diagram of the whole multiplication process as a set of thin triangles. The result of the multiplication is a sum of words from $\widetilde{S}_\lambda$ that correspond to certain paths in the low levels of the obtained thin triangles.
\end{remark}


\begin{thebibliography}{99}
\bibitem{Bergman} G. M. Bergman, \textit{The Diamond Lemma for ring theory}, Adv. Math. \textbf{29} (1978), no. 2, 178-218,

\bibitem{BridsonHaefliger} M. R. Bridson, A. Haefliger, \textit{Metric Spaces of Non-Positive Curvature}, Springer-Verlag Berlin Heidelberg New York, 1999.

\bibitem{CoornaertDelzantPapadopoulos} M. Coornaert, T. Delzant, A. Papadopoulos, \textit{Géométrie et théorie des groupes: les groupes hyperboliques de Gromov}, Lecture Notes in Mathematics \textbf{1441}, Springer, 1991.

\bibitem{GhysHarpe} E. Ghys and P. de la Harpe (eds.), \textit{Sur les Groupes Hyperboliques d'après Mikhael Gromov}, Progress in Math. \textbf{83}, Birkhäuser, 1990.

\bibitem{Gromov} M. Gromov, \textit{Hyperbolic groups}, Essays in Group Theory, S. Gersten (ed.), MSRI publication no. 8, Springer Berlin Heidelberg New York, 1987, 75–263.


\bibitem{Lang} S. Lang, \textit{Algebra}, Springer-Verlag New York, 2002.

\bibitem{LyndonSchupp} R. C. Lyndon and P. E. Schupp, \textit{Combinatorial Group Theory}, Springer, Berlin-Heidelberg-New York, 1977.

\bibitem{Olshanskii} A. Yu. Ol'shanskii, \textit{Geometry of Defining Relations in Groups}, Springer Science + Business Media Dordrecht, 1991.

\bibitem{RipsSela} E. Rips and Z. Sela, \textit{Canonical representatives and equations in hyperbolic groups}, Invent. Math. \textbf{120} (1995), no. 1, 489–512.
\end{thebibliography}
\end{document}